\pdfoutput=1
 \documentclass[a4paper,10pt]{amsart}
\usepackage{amsfonts, amssymb, amsthm, amsmath}
\usepackage{hyperref}

\usepackage{tikz}
\usepackage{enumerate}
\usetikzlibrary{matrix,arrows,decorations.pathmorphing, decorations.pathreplacing, arrows.meta,decorations.markings}

\usepackage{tikz-cd}
\usepackage{yhmath}
\usepackage{MnSymbol}
\usepackage{graphicx}
\usepackage{booktabs,array}
\usepackage[font=small,labelfont=bf]{caption}

\newtheorem{thm}{Theorem}

\newtheorem{lemma}[thm]{Lemma}
\newtheorem{remark}[thm]{Remark}
\newtheorem{defn}[thm]{Definition}
\newtheorem{prop}[thm]{Proposition}

\newtheorem{assumption}[thm]{Assumption}
\newtheorem{construction}[thm]{Construction}

\numberwithin{thm}{section}
\numberwithin{equation}{section}
\DeclareMathOperator{\rend}{End}
\DeclareMathOperator{\Aut}{Aut}
\newcommand{\sk}{\operatorname{Sk}}
\newcommand{\ex}{\mathbf}
\providecommand {\e}[1]{\mathfrak t^{#1}}
\providecommand{\C}[2]{\ensuremath {C^{#1,\underline{#2}}}}
\DeclareMathOperator{\id}{id}
\DeclareMathOperator{\coker}{coker}
\DeclareMathOperator{\Conf}{Conf}
\DeclareMathOperator{\Cone}{Cone}
\newcommand{\dbar}{\bar{\partial}}
\DeclareMathOperator{\ind}{ind}
\providecommand{\abs}[1]{\left\lvert #1\right\rvert}

\tikzset{
    knot/.style={thick, line cap=round, line join=round},
    halo/.style={white, line width=5pt, line cap=round, line join=round}
}

\newcommand{\Lpos}{
\begin{tikzpicture}[baseline=-0.65ex, scale=0.5]
  \draw[thick, ->] (1,-1) -- (-1,1);
  \draw[white, line width=5pt] (-1,-1) -- (1,1);
  \draw[thick, ->] (-1,-1) -- (1,1);
\end{tikzpicture}
}

\newcommand{\Lneg}{
\begin{tikzpicture}[baseline=-0.65ex, scale=0.5]
  \draw[thick, ->] (-1,-1) -- (1,1);
  \draw[white, line width=5pt] (1,-1) -- (-1,1);
  \draw[thick, ->] (1,-1) -- (-1,1);
\end{tikzpicture}
}

\newcommand{\Lzero}{
\begin{tikzpicture}[baseline=-0.65ex, scale=0.5]
  \draw[thick, ->] (-1,-1) .. controls (-0.2,-0.2) and (-0.2,0.2) .. (-1,1);
  \draw[thick, ->] (1,-1) .. controls (0.2,-0.2) and (0.2,0.2) .. (1,1);
\end{tikzpicture}
}

\newcommand{\Unknot}{
\begin{tikzpicture}[baseline=-0.65ex, scale=0.5]
  \draw[knot] (0,0) circle (0.9);
\end{tikzpicture}
}

\newcommand{\Straight}{
\begin{tikzpicture}[baseline=-0.65ex, scale=0.5]
  \draw[knot, ->] (0,-1.2) -- (0,1.2);
\end{tikzpicture}
}

\newcommand{\NegTwist}{
\begin{tikzpicture}[baseline=-0.65ex, scale=0.5]
  \draw[knot, ->] (-0.4, -1.2) to[out=70, in=225] (0,0) to[out=45, in=90] (1, -0.2)
                  to[out=270, in=315] (0,0) to[out=135, in=290] (-0.4, 1.2);

  \begin{scope}
    \clip (-0.4, -0.4) rectangle (0.4, 0.4);
    \draw[halo] (1, -0.2) to[out=270, in=315] (0,0) to[out=135, in=290] (-0.4, 1.2);
    \draw[knot] (1, -0.2) to[out=270, in=315] (0,0) to[out=135, in=290] (-0.4, 1.2);
  \end{scope}
\end{tikzpicture}
}

\newcommand{\PosTwist}{
\begin{tikzpicture}[baseline=-0.65ex, scale=0.5]
  \draw[knot, ->] (-0.4, -1.2) to[out=70, in=225] (0,0) to[out=45, in=90] (1, -0.2)
                  to[out=270, in=315] (0,0) to[out=135, in=290] (-0.4, 1.2);

  \begin{scope}
    \clip (-0.4, -0.4) rectangle (0.4, 0.4);
    \draw[halo] (-0.4, -1.2) to[out=70, in=225] (0,0) to[out=45, in=90] (1, -0.2);
    \draw[knot] (-0.4, -1.2) to[out=70, in=225] (0,0) to[out=45, in=90] (1, -0.2);
  \end{scope}
\end{tikzpicture}
}

\author{Brett Parker}
\email{ }
\thanks{I wish to thank Tobias Ekholm for helpful conversations, suggestions, and explanations. In drafting and editing this paper, and searching the literature, I made use of the AI models ChatGPT, Gemini, and Claude. AI was also extremely helpful for finding and fixing a large, even, number of sign errors.   I am not confident that I would have found all such errors myself. }

\title[Orientations for skein-valued curve counting]{Orientations for skein-valued curve counting}

\begin{document}
\begin{abstract}
Ekholm and Shende's skeins-on-branes formalism provides invariant counts of holomorphic curves with boundary on a Maslov-0 Lagrangian within a Calabi–Yau 3-fold. This paper proves the orientation relations required for worldsheet skein relations to cancel codimension-one boundary strata. However, standard Cauchy–Riemann Fredholm orientations reveal an unexpected sign puzzle: the orientation relation at a curve with a  hyperbolic  boundary crossing  depends on whether the crossing lies on distinct boundary components or the same component. This domain-dependent sign contradicts the uniform crossing relations of the HOMFLYPT skein algebra. We show that this discrepancy is resolved by passing to an anti-holomorphic orientation convention.

 We establish all required orientation comparisons in the general setting of an oriented Lagrangian submanifold of a 6-dimensional symplectic manifold.
\end{abstract}

\maketitle
\tableofcontents

\section{Introduction}\label{sec:introduction}

Let $L$ be an oriented Lagrangian submanifold of a $6$-dimensional symplectic manifold $X$. The skeins-on-branes formalism~\cite{ekholm2019skeins,worldsheetskein} allows us to define invariant counts of holomorphic curves with boundary on $L$, valued in a worldsheet skein module~\cite{worldsheetskein}. With suitable extra data, this module maps to the HOMFLYPT skein module of oriented links in $L$. The original work of Ekholm and Shende~\cite{ekholm2019skeins} constructs such skein-valued invariants for Maslov $0$ Lagrangians in symplectic Calabi--Yau $3$-folds. Worldsheet skein-valued counting should work more generally, without the Maslov $0$ or Calabi--Yau conditions. This paper calculates signs for the worldsheet skein module and proves that orientations of the moduli space of curves satisfy the required relations.

Given a suitable normal-crossing degeneration of $(X,L)$, we expect these skein-valued invariants to be recoverable from relative skein-valued invariants via a tropical gluing formula. There is, however, a tension between the skein algebra and tropical curve counting, illustrated in Figure~\ref{fig:tc1}. Consider the Lagrangian $U^1\times U^1\times\mathbb R_{>0}\subset (\mathbb C^*)^3$. The skein module of this ``thickened torus'' is an algebra, calculated by Morton and Samuelson~\cite{MortonSamuelson2017}, with a commutator coefficient of
\[
V_{\Delta}=q^{\Delta/2}-q^{-\Delta/2},
\]
which should correspond to an enumerative tropical vertex contribution~\cite{3d},
\[
V_{\Delta}=\frac{q^{\Delta/2}-q^{-\Delta/2}}{i}\ .
\]

\begin{figure}[htbp]
\begingroup
\definecolor{tcL}{RGB}{0,112,91}
\definecolor{tcAlpha}{RGB}{53,109,164}
\definecolor{tcBeta}{RGB}{182,106,44}
\newcommand{\tcLocus}{%
  \draw[tcL,line width=.8pt] (0,0,-1.17)--(0,0,1.28);
  \draw[tcL,line width=.6pt,densely dotted]
    (0,0,-1.17)--(0,0,-1.44) (0,0,1.28)--(0,0,1.55);
  \node[tcL,anchor=west] at (.12cm,1.22cm) {$L$};
}
\begin{tikzpicture}[
  scale=.75,
  x={(.90cm,.18cm)},y={(-.45cm,.30cm)},z={(0cm,1cm)},
  line cap=round,line join=round,>=stealth,
  alpha/.style={draw=tcAlpha,line width=.85pt},
  beta/.style={draw=tcBeta,line width=.85pt},
  every node/.style={font=\scriptsize,inner sep=1.2pt}
]
  \tcLocus
  \draw[alpha,->] (0,0,.50)--(-3,-1.5,.50);
  \draw[beta,->] (0,0,-.50)--(-1.5,-4.5,-.50);
  \fill[tcAlpha] (0,0,.50) circle[radius=1.25pt];
  \fill[tcBeta] (0,0,-.50) circle[radius=1.25pt];
  \node[tcAlpha,anchor=north] at (-1.49cm,-.54cm) {$(-a,-b,0)$};
  \node[tcBeta,anchor=west] at (.52cm,-1.45cm) {$(-c,-d,0)$};
  \node[font=\small] at (-.30cm,-2.90cm) {$P_{(a,b)}P_{(c,d)}$};

  \begin{scope}[xshift=4.80cm]
    \tcLocus
    \draw[alpha,->] (0,0,-.50)--(-3,-1.5,-.50);
    \draw[beta,->] (0,0,.50)--(-1.5,-4.5,.50);
    \fill[tcAlpha] (0,0,-.50) circle[radius=1.25pt];
    \fill[tcBeta] (0,0,.50) circle[radius=1.25pt];
    \node[tcAlpha,anchor=north] at (-1.49cm,-1.54cm) {$(-a,-b,0)$};
    \node[tcBeta,anchor=west] at (.52cm,.10cm) {$(-c,-d,0)$};
    \node[font=\small] at (-.30cm,-2.90cm) {$P_{(c,d)}P_{(a,b)}$};
  \end{scope}

  \begin{scope}[xshift=9.60cm]
    \tcLocus
    \coordinate (tcVertex) at (-2.1,-2.8,0);
    \draw[line width=.85pt] (0,0,0)--(tcVertex);
    \draw[alpha,->] (tcVertex)--(-4.5,-4.0,0);
    \draw[beta,->] (tcVertex)--(-3.1,-5.8,0);
    \fill (0,0,0) circle[radius=1.25pt];
    \fill (tcVertex) circle[radius=1.45pt];
    \node[anchor=east] at (-.39cm,-.46cm) {$(-a-c,-b-d,0)$};
    \node[tcAlpha,anchor=south] at (-1.91cm,-1.52cm) {$(-a,-b,0)$};
    \node[tcBeta,anchor=west] at (.14cm,-1.94cm) {$(-c,-d,0)$};
    \node at (-.92cm,-1.86cm) {$V_\Delta$};
    \node[font=\small] at (-.30cm,-2.90cm) {$V_\Delta P_{(a+c,b+d)}$};
  \end{scope}
  \node[font=\normalsize] at (2.10cm,-2.90cm) {$-$};
  \node[font=\normalsize] at (6.90cm,-2.90cm) {$=$};
\end{tikzpicture}
\endgroup
\captionof{figure}{Tropical pictures of holomorphic curves with boundaries on $U^1\times U^1\times\mathbb R_{>0}\subset(\mathbb C^*)^3$. The Morton--Samuelson skein algebra implies that
 $V_\Delta=q^{\Delta/2}-q^{-\Delta/2}$, with
$\Delta=ad-bc$, but the tropical gluing formula implies that $V_\Delta=(q^{\Delta/2}-q^{-\Delta/2})/i$. This apparent discrepancy is resolved by counting \emph{anti}-holomorphic curves.}
\label{fig:tc1}
\end{figure}
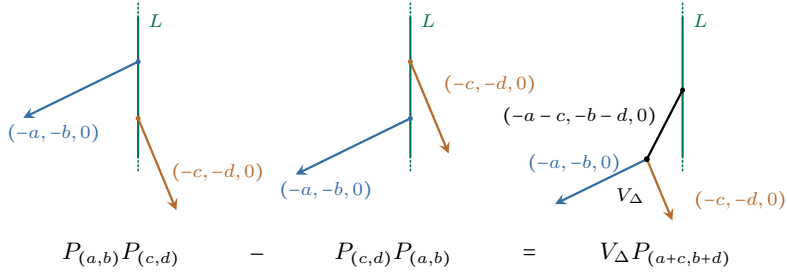

For Gromov--Witten invariants, $q^{1/2}=e^{i\hbar/2}$, but the above contribution of tropical vertices holds quite generally for any curve-counting theory with a similar tropical gluing formula. To remove the incompatible factor of $i$, the tropical gluing formula needs to use different orientation conventions---corresponding to counting anti-holomorphic curves. This anti-holomorphic convention modifies the contribution of tropical vertices to $V_{\Delta}=q^{\Delta/2}-q^{-\Delta/2}$, with $q^{1/2}$ specialising to $e^{\hbar/2}$ in the case of anti-holomorphic stable maps.

This global mismatch between the Morton--Samuelson torus skein algebra and tropical gluing formulae is reflected in a mismatch of local orientation comparisons in the moduli space of curves. To define invariant counts of curves with boundary on $L$, any curve-counting method must account for codimension-one boundaries of the moduli space, which occur when curves develop a hyperbolic or elliptic boundary node on $L$. Such boundaries coincide with codimension-one walls in a simpler moduli space obtained by normalising the node: these walls occur when the boundary of the normalised curve intersects itself in $L$ (in the hyperbolic case), or when the interior of the normalised curve intersects $L$ (in the elliptic case). Skein-valued curve counting requires comparing the boundary orientation with the wall orientation induced by the chambers of the normalised moduli space. There are, however, \emph{two} distinct hyperbolic cases (Figure~\ref{fig:hyperbolic-surgery}):
\begin{enumerate}
\item In the \emph{disconnected crossing} case, two separate boundary components of the normalised curve intersect in $L$; smoothing the corresponding hyperbolic node merges these two boundary components into one.
\item In the \emph{connected crossing} case, a single boundary component of the normalised curve intersects itself in $L$; smoothing the corresponding hyperbolic node splits this boundary component into two.
\end{enumerate}
We show that, under the Fredholm orientation conventions standard in holomorphic curve counting, the orientation comparison between the wall and the smoothing boundary is \emph{opposite} in these two hyperbolic cases (Theorem~\ref{thm:main-orientation-comparison}). This creates an obstruction for skein-valued curve counting, because the local crossing relation in the HOMFLYPT skein module does not distinguish whether the crossing strands belong to the same boundary component or to different components. We resolve this obstruction by passing to the anti-holomorphic orientation convention, which introduces a topological sign twist that reverses the disconnected hyperbolic comparison while preserving the connected hyperbolic and elliptic comparisons, rendering the skein signs uniform across all wall types.

In \cite[Appendix A5]{ekholm2019skeins}, this anti-holomorphic convention is described as changing the orientation of the complex plane (our signs differ from those in version 5 of that paper). In physics language, this choice could be thought of as using an imaginary string coupling constant $g_s=\frac{2\pi i}{k+N}$~\cite[Equation 7.1]{topologicalvertex} to be compatible with the relationship between open string theory in $(X,L)$ and Chern--Simons theory on $L$ \cite{Witten1992,OoguriVafa2000}. We also give an alternative orientation convention in Section~\ref{sec:ramified-covers} using an auxiliary complex structure on the normal bundle to an immersed curve.

With the anti-holomorphic orientation convention, one can generalise Ekholm and Shende's skein-valued counts of bare curves \cite{bare,GhostCensorship} to counts of holomorphic curves with boundary on an oriented $3$-dimensional Lagrangian submanifold $L$ of a compact symplectic manifold. With a choice of Spin structure on $L$, the resulting invariants should take values in the (sign-updated) worldsheet skein module \cite{worldsheetskein}, and with a suitable bounding chain for $L$,  take values in the HOMFLYPT skein module of $L$.

\smallskip
In this paper, we establish the orientation comparisons for all three codimension-one wall types—disconnected hyperbolic, connected hyperbolic, and elliptic—in a general Fredholm setting that requires neither the Calabi--Yau condition on $X$, the Maslov $0$ condition on $L$, nor transversality (Theorem~\ref{thm:main-orientation-comparison}). We formulate the linearised deformation problem for a bordered nodal curve in terms of a coupled bundle pair. A neck-stretching argument  reduces the orientation comparison at any wall to a universal local model. We compute the three universal comparisons using simple, unobstructed holomorphic model families, and check each of them by a second, independent calculation.

Section~\ref{sec:skein-modules} explains the worldsheet skein module of Ekholm, Longhi and Nakamura with our updated signs and shows that twisting by the anti-holomorphic sign correction yields a well-defined homomorphism to the HOMFLYPT skein module. Section~\ref{sec:linear-setting} sets up the abstract Fredholm deformation complexes and the incidence and smoothing exact triples at a wall. Section~\ref{orientation-section} constructs the holomorphic and anti-holomorphic orientations, proves the neck-stretching reduction to local models (yielding a direct proof of the elliptic comparison), and extends the formalism to curves with capped strip and cylindrical ends; detailed determinant-line conventions are collected in Appendix~\ref{app:determinant-lines}, and Appendix~\ref{app:notation} tabulates the main notation. Sections~\ref{sec:disconnected-crossing}, \ref{sec:elliptic-crossing}, and \ref{connected-hyperbolic-section} calculate the orientations using explicit model families for the disconnected hyperbolic, elliptic, and connected hyperbolic walls. Finally, Section~\ref{sec:ramified-covers} orients ramified covers of immersed curves using an auxiliary complex structure on the normal bundle, and uses this to recompute both hyperbolic comparisons simultaneously. With respect to this auxiliary orientation, the smoothing boundary and positive chamber orientations agree at both types of hyperbolic wall; the difference between the connected and disconnected cases in the holomorphic convention, and its absence in the anti-holomorphic convention, comes entirely from comparing the auxiliary orientation with these two conventions (Lemma~\ref{Jcomparison2}).

\subsection{Wall and chamber structure} \label{sec:chamber}

The geometric motivation for a skein relation is that codimension-one boundaries of the moduli space of curves of Euler characteristic $\chi$ correspond to walls within the simpler moduli space of Euler characteristic $\chi+1$. At such a wall, the curve either undergoes a crossing change of its boundary link in $L$ (a hyperbolic crossing) or has an interior point pass through $L$ (an elliptic incidence). At the corresponding boundary, the curve develops a hyperbolic or elliptic boundary node. The skein relation equates the jump between the two chambers across the wall with the boundary contribution of the terminating family, with the sign fixed by the orientation comparison (Figure~\ref{fig:wall-schematic}).

\begin{figure}[htbp]
\centering
\begin{tikzpicture}[
  x=1cm, y=1cm,
  line width=0.8pt,
  line cap=round,
  line join=round,
  every node/.style={font=\small, inner sep=2pt}
]
  \coordinate (wall) at (0,0);

  \draw (-4.1,-0.65)
    .. controls (-2.9,-0.13) and (-1.4,0) .. (wall)
    .. controls (1.5,0) and (2.8,-0.22) .. (4.3,-0.72);

  \draw (wall)
    .. controls (-0.06,1.10) and (0.53,2.55) .. (1.38,3.48);

  \node[anchor=north] at (0,-0.12) {wall};
  \node[anchor=west] at (0.17,0.34) {boundary};
  \node[anchor=west] at (1.13,2.87) {\shortstack{smoothing family\\ Euler characteristic $\chi$}};
  \node[anchor=north] at (2.92,-0.66) {\shortstack{normalised family\\ Euler characteristic $\chi+1$} };
\end{tikzpicture}
\caption{A schematic picture of the moduli space, identifying curves when they agree off $L$. The boundary of the smoothing family is cancelled by a wall in the normalised family.}
\label{fig:wall-schematic}
\end{figure}
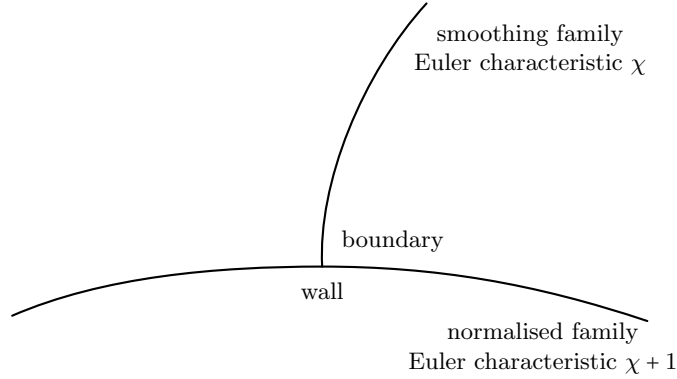

At a wall, there is a smoothing boundary orientation, induced from the terminating moduli space of Euler characteristic $\chi$. Below, we specify locally defined chambers near a wall, used to define wall orientations.

\subsubsection*{1. Hyperbolic walls (Boundary crossing changes)}
A generic curve $v\colon(\Sigma,\partial \Sigma)\to (X,L)$ on a hyperbolic wall is smooth with immersed boundary $\partial v \subset L$ having exactly one self-intersection point $p\in L$, with linearly independent tangent vectors $w_1, w_2$ in $T_pL$ (oriented using the standard boundary orientation of $\partial \Sigma$).
\begin{itemize}
    \item \textbf{Positive-crossing chamber ($\mathcal{M}_{\mathrm{norm}}^+$):}
    The chamber containing deformations of $v$ where the crossing is resolved positively: the ordered tangents $(w_1, w_2)$ of the crossing strands and the displacement vector pointing from strand 2 to strand 1 form an oriented basis of $T_p L$.
    \[\Lpos\]
    \item \textbf{Negative-crossing chamber ($\mathcal{M}_{\mathrm{norm}}^-$):}
    The chamber containing deformations of $v$ with the opposite, negative-crossing resolution.
    \[\Lneg\]
    \item \textbf{Smoothing sector:}
    The smoothing sector consists of deformations of the curve $u$ with a hyperbolic boundary node created by joining the two sheets at $p$. These deformations smooth the resulting hyperbolic boundary node. At the boundary level, this replaces the crossing with its smoothed resolution $\Lzero$.
\end{itemize}

\begin{figure}[htbp]
  \centering
  \begin{minipage}[c]{.48\linewidth}
    \centering
    \includegraphics[width=\linewidth]{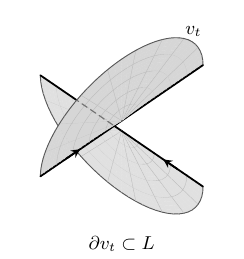}
  \end{minipage}\hfill
  \begin{minipage}[c]{.48\linewidth}
    \centering
    \includegraphics[width=\linewidth]{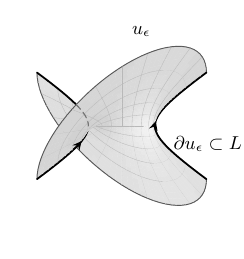}
  \end{minipage}
    \caption{Local views near a hyperbolic wall.
    Left: a curve $v_t$ in the positive-crossing chamber.
    Right: a curve $u_\epsilon$ in the smoothing sector, joining the two local sheets.
   }
  \label{fig:hyperbolic-wall-local}
\end{figure}

\begin{defn}[Types of hyperbolic crossing]
\label{def:hyperbolic-crossing-types}
The crossing is called \emph{connected} if the two crossing preimages lie on the same connected component of $\partial \Sigma$ (smoothing splits one boundary component into two), and \emph{disconnected} if they lie on distinct components (smoothing merges two boundary components into one). The type of a hyperbolic node means the type of the crossing in its normalisation.
\end{defn}

\begin{figure}
\begingroup
\definecolor{surface}{RGB}{247,247,247}
\begin{tikzpicture}[
  x=1cm,y=1cm,font=\small,
  line cap=round,line join=round,
  boundary/.style={draw=black,line width=.85pt},
  contour/.style={draw=black!65,line width=.65pt},
  hidden/.style={boundary,dash pattern=on 2.5pt off 2.5pt},
  flow/.style={-{Latex[length=1.9mm,width=1.25mm]},draw=black!60,line width=.55pt},
  dot/.style={circle,fill=black,inner sep=1.3pt},
  info/.style={font=\footnotesize,text=black!75},
  every node/.style={inner sep=2pt}
]
\node at (0,2.05) {Normalisation};
\node at (4.25,2.05) {Boundary node};
\node at (8.5,2.05) {Smoothing};

\node[anchor=west] at (-1.55,1.47)
  {\textbf{Connected crossing}\quad same boundary component};

\filldraw[fill=surface,boundary] (0,0)
  ellipse[x radius=1.04,y radius=.92];
\node[dot] at (-1.04,0) {};
\node[dot] at (1.04,0) {};
\node[anchor=east] at (-1.09,0) {$p_1$};
\node[anchor=west] at (1.09,0) {$p_2$};
\node at (0,-1.28) {disc};

\draw[flow] (1.65,0) -- (2.65,0);

\begin{scope}[shift={(4.25,0)}]
  \filldraw[fill=surface,boundary] (0,0)
    ellipse[x radius=1.04,y radius=.98];
  \filldraw[fill=white,boundary]
    (1.04,0)
    .. controls (.47,.40) and (.15,.49) .. (-.15,.465)
    .. controls (-.585,.44) and (-.585,-.44) .. (-.15,-.465)
    .. controls (.15,-.49) and (.47,-.40) .. cycle;
  \node[dot] at (1.04,0) {};
  \node[anchor=west] at (1.10,.17) {$q$};
  \node at (0,-1.28) {$p_1\sim p_2=q$};
\end{scope}

\draw[flow] (5.78,0) -- (6.88,0);

\begin{scope}[shift={(8.5,0)}]
  \filldraw[fill=surface,boundary] (0,0)
    ellipse[x radius=1.04,y radius=.98];
  \filldraw[fill=white,boundary]
    (.54,0)
    .. controls (.54,.315) and (.15,.49) .. (-.15,.465)
    .. controls (-.585,.44) and (-.585,-.44) .. (-.15,-.465)
    .. controls (.15,-.49) and (.54,-.315) .. cycle;
  \node at (0,-1.28) {annulus};
  \node[info] at (0,-1.65) {$b_{\mathbb R}: 1\longrightarrow 2$};
\end{scope}

\node[anchor=west] at (-1.55,-2.35)
  {\textbf{Disconnected crossing}\quad different boundary components};

\begin{scope}[shift={(0,-3.95)}]
  \fill[surface]
    (-1.12,-.62)
    .. controls (-1.12,.40) and (-.68,1.05) .. (0,1.05)
    .. controls (.68,1.05) and (1.12,.40) .. (1.12,-.62)
    arc[start angle=0,end angle=-180,x radius=.39,y radius=.20]
    .. controls (.34,-.10) and (.29,.38) .. (0,.38)
    .. controls (-.29,.38) and (-.34,-.10) .. (-.34,-.62)
    arc[start angle=0,end angle=-180,x radius=.39,y radius=.20] -- cycle;
  \draw[contour]
    (-1.12,-.62)
    .. controls (-1.12,.40) and (-.68,1.05) .. (0,1.05)
    .. controls (.68,1.05) and (1.12,.40) .. (1.12,-.62);
  \draw[contour]
    (-.34,-.62)
    .. controls (-.34,-.10) and (-.29,.38) .. (0,.38)
    .. controls (.29,.38) and (.34,-.10) .. (.34,-.62);
  \draw[hidden] (-1.12,-.62)
    arc[start angle=180,end angle=0,x radius=.39,y radius=.20];
  \draw[hidden] (.34,-.62)
    arc[start angle=180,end angle=0,x radius=.39,y radius=.20];
  \draw[boundary] (-1.12,-.62)
    arc[start angle=180,end angle=360,x radius=.39,y radius=.20];
  \draw[boundary] (.34,-.62)
    arc[start angle=180,end angle=360,x radius=.39,y radius=.20];
  \node[dot] at (-.34,-.62) {};
  \node[dot] at (.34,-.62) {};
  \node[anchor=north east] at (-.32,-.85) {$p_1$};
  \node[anchor=north west] at (.32,-.85) {$p_2$};
  \node at (0,-1.48) {annulus};
\end{scope}

\draw[flow] (1.65,-3.95) -- (2.65,-3.95);

\begin{scope}[shift={(4.25,-3.95)}]
  \fill[surface]
    (-1.05,-.62)
    .. controls (-1.05,.40) and (-.60,1.05) .. (0,1.05)
    .. controls (.60,1.05) and (1.05,.40) .. (1.05,-.62)
    arc[start angle=0,end angle=-180,x radius=.525,y radius=.21]
    arc[start angle=0,end angle=-180,x radius=.525,y radius=.21] -- cycle;
  \fill[white]
    (0,-.62)
    .. controls (-.14,-.29) and (-.30,-.05) .. (-.28,.21)
    .. controls (-.24,.64) and (.24,.64) .. (.28,.21)
    .. controls (.30,-.05) and (.14,-.29) .. cycle;
  \draw[contour]
    (-1.05,-.62)
    .. controls (-1.05,.40) and (-.60,1.05) .. (0,1.05)
    .. controls (.60,1.05) and (1.05,.40) .. (1.05,-.62);
  \draw[contour]
    (0,-.62)
    .. controls (-.14,-.29) and (-.30,-.05) .. (-.28,.21)
    .. controls (-.24,.64) and (.24,.64) .. (.28,.21)
    .. controls (.30,-.05) and (.14,-.29) .. cycle;
  \draw[hidden] (-1.05,-.62)
    arc[start angle=180,end angle=0,x radius=.525,y radius=.21];
  \draw[hidden] (0,-.62)
    arc[start angle=180,end angle=0,x radius=.525,y radius=.21];
  \draw[boundary] (-1.05,-.62)
    arc[start angle=180,end angle=360,x radius=.525,y radius=.21];
  \draw[boundary] (0,-.62)
    arc[start angle=180,end angle=360,x radius=.525,y radius=.21];
  \node[dot] at (0,-.62) {};
  \node[anchor=north] at (0,-.91) {$q$};
  \node at (0,-1.48) {$p_1\sim p_2=q$};
\end{scope}

\draw[flow] (5.78,-3.95) -- (6.88,-3.95);

\begin{scope}[shift={(8.5,-3.95)}]
  \fill[surface]
    (-1.05,-.62)
    .. controls (-1.05,.40) and (-.60,1.05) .. (0,1.05)
    .. controls (.60,1.05) and (1.05,.40) .. (1.05,-.62)
    arc[start angle=0,end angle=-180,x radius=1.05,y radius=.29] -- cycle;
  \draw[contour]
    (-1.05,-.62)
    .. controls (-1.05,.40) and (-.60,1.05) .. (0,1.05)
    .. controls (.60,1.05) and (1.05,.40) .. (1.05,-.62);
  \filldraw[fill=white,contour]
    (-.235,.08)
    .. controls (-.09,.31) and (.09,.31) .. (.235,.08)
    .. controls (.12,-.18) and (-.12,-.18) .. cycle;
  \draw[contour] (-.34,.48)
    .. controls (-.32,.24) and (-.29,.17) .. (-.235,.08);
  \draw[contour] (.34,.48)
    .. controls (.32,.24) and (.29,.17) .. (.235,.08);
  \draw[hidden] (-1.05,-.62)
    arc[start angle=180,end angle=0,x radius=1.05,y radius=.29];
  \draw[boundary] (-1.05,-.62)
    arc[start angle=180,end angle=360,x radius=1.05,y radius=.29];
  \node at (0,-1.48) {one-holed torus};
  \node[info] at (0,-1.85) {$b_{\mathbb R}: 2\longrightarrow 1$};
\end{scope}

\end{tikzpicture}
\endgroup
\caption{The two types of hyperbolic boundary surgery.}
\label{fig:hyperbolic-surgery}

\end{figure}
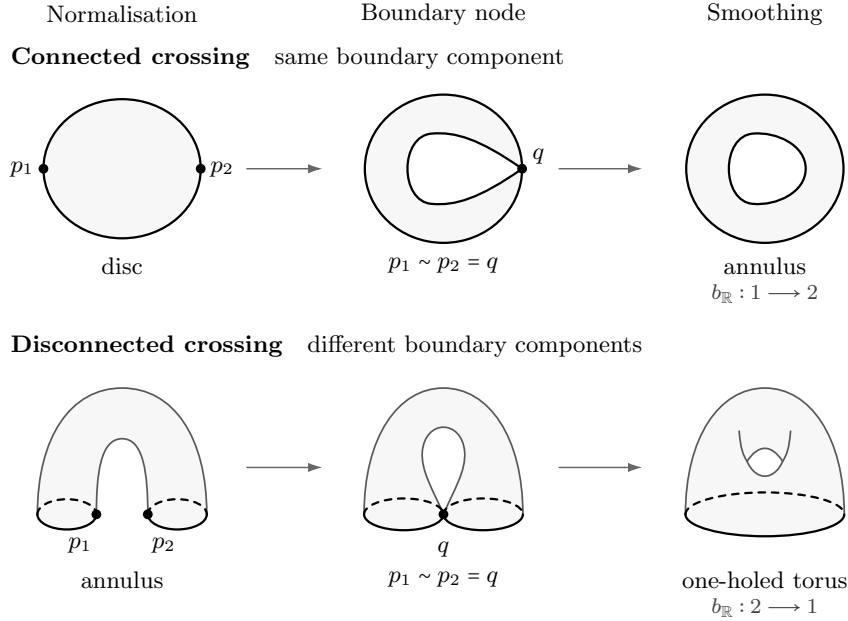

\subsubsection*{2. Elliptic walls (Interior Lagrangian incidences)}
A generic curve $v$ on an elliptic wall is embedded close to $L$, but the interior of the curve intersects $L$ at a single point $q \in \Sigma^\circ$ such that $v$ is holomorphic to leading order at $q$ with $p = v(q) \in L$, and $T_pL\oplus T_qv(T_q\Sigma)$ spans a $5$-dimensional subspace of $T_pX$.

\begin{itemize}
    \item \textbf{Positive-incidence chamber ($\mathcal{M}_{\mathrm{norm}}^+$):}
    The deformations of $v$ moving the image of $q$ in a direction $n$ such that
      \[T_pL\oplus T_qv(T_q\Sigma)\oplus \mathbb R n\]
    is a positively oriented decomposition of $T_pX$.
    \item \textbf{Negative-incidence chamber ($\mathcal{M}_{\mathrm{norm}}^-$):}
    The opposite chamber, consisting of deformations moving $q$ in a direction $n$ such that the above decomposition of $T_pX$ is negatively oriented.
    \item \textbf{Smoothing sector:}
    Deformations of the nodal curve $u$ created by adding an elliptic boundary node to $v$ at $q$. (An alternate model is deformations of $v$ attached at an internal node to a constant disc at $q$.) Smoothing this node opens up a small isolated boundary component around $p$, creating an unknot component
    \[\Unknot.\]
\end{itemize}

Equivalently, we can define the positive- and negative-incidence chambers by locally choosing a bounding $4$-chain for $L$ as follows. Choose smooth, oriented, embedded submanifolds $C_{\pm}$ with $\partial C_{\pm}=L$, intersecting only at $L$, whose union is a smooth submanifold $|C|$ transverse to $v$ at $p$. Curves in the positive-incidence chamber have intersection number $1$ with the chain $C_++C_-$ near $p$, and curves in the negative-incidence chamber have intersection number $-1$ with this chain near $p$. When this chain is globally defined as a $4$-chain $C$ with $\partial C=2L$, $C$ intersects the interior of curves in the positive-incidence chamber with algebraic intersection number $2$ greater than that of curves in the negative-incidence chamber.

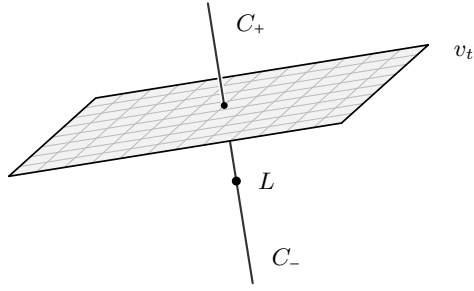
\begin{figure}[htbp]
\centering
\begin{tikzpicture}[x=1cm,y=1cm,line cap=round,line join=round,
  every node/.style={font=\small,inner sep=2pt}]
  \coordinate (L) at (0.16,0.70);
  \coordinate (intersection) at (0,1.70);
  \coordinate (Cplus-end) at (-0.216,3.05);
  \coordinate (Cminus-end) at (0.376,-0.65);

  \draw[black!80,line width=0.80pt] (L) -- (Cplus-end);
  \draw[black!80,line width=0.80pt] (L) -- (Cminus-end);

\fill[black!5] (-2.85000,0.76300) -- (1.55000,1.46700) -- (2.70000,2.50200) -- (-1.70000,1.79800) -- cycle;
\draw[black!27,line width=0.18pt] (-2.48333,0.82167) -- (-1.33333,1.85667);
\draw[black!27,line width=0.18pt] (-2.11667,0.88033) -- (-0.96667,1.91533);
\draw[black!27,line width=0.18pt] (-1.75000,0.93900) -- (-0.60000,1.97400);
\draw[black!27,line width=0.18pt] (-1.38333,0.99767) -- (-0.23333,2.03267);
\draw[black!27,line width=0.18pt] (-1.01667,1.05633) -- (0.13333,2.09133);
\draw[black!27,line width=0.18pt] (-0.65000,1.11500) -- (0.50000,2.15000);
\draw[black!27,line width=0.18pt] (-0.28333,1.17367) -- (0.86667,2.20867);
\draw[black!27,line width=0.18pt] (0.08333,1.23233) -- (1.23333,2.26733);
\draw[black!27,line width=0.18pt] (0.45000,1.29100) -- (1.60000,2.32600);
\draw[black!27,line width=0.18pt] (0.81667,1.34967) -- (1.96667,2.38467);
\draw[black!27,line width=0.18pt] (1.18333,1.40833) -- (2.33333,2.44333);
\draw[black!27,line width=0.18pt] (-2.70625,0.89237) -- (1.69375,1.59638);
\draw[black!27,line width=0.18pt] (-2.56250,1.02175) -- (1.83750,1.72575);
\draw[black!27,line width=0.18pt] (-2.41875,1.15112) -- (1.98125,1.85513);
\draw[black!27,line width=0.18pt] (-2.27500,1.28050) -- (2.12500,1.98450);
\draw[black!27,line width=0.18pt] (-2.13125,1.40987) -- (2.26875,2.11388);
\draw[black!27,line width=0.18pt] (-1.98750,1.53925) -- (2.41250,2.24325);
\draw[black!27,line width=0.18pt] (-1.84375,1.66862) -- (2.55625,2.37263);
\draw[line width=0.65pt] (-2.85000,0.76300) -- (1.55000,1.46700) -- (2.70000,2.50200) -- (-1.70000,1.79800) -- cycle;

  \draw[white,line width=1.9pt] (intersection) -- (Cplus-end);
  \draw[black!80,line width=0.80pt] (intersection) -- (Cplus-end);
  \fill (intersection) circle[radius=1.3pt];
  \fill (L) circle[radius=1.7pt];

  \node[anchor=west] at (0.08,2.78) {$C_+$};
  \node[anchor=west] at (0.55,-0.32) {$C_-$};
  \node[anchor=west] at (0.37,0.70) {$L$};
  \node[anchor=west] at (2.96,2.40) {$v_t$};
\end{tikzpicture}

\caption{A picture of a curve $v_t$ in the positive-incidence chamber close to an elliptic wall. The picture is a $3$-dimensional slice transverse to $L$. For $t$ positive, $v_t$ intersects $C_++C_-$ positively in the region shown.}
\label{elliptic-v}
\end{figure}

\begin{figure}[htbp]
  \centering
  \begin{minipage}[c]{0.48\linewidth}
    \centering
    \includegraphics[width=\linewidth]{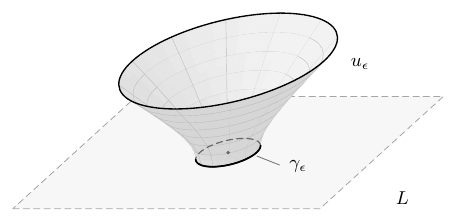}
  \end{minipage}\hfill
  \begin{minipage}[c]{0.48\linewidth}
    \centering
    \includegraphics[width=\linewidth]{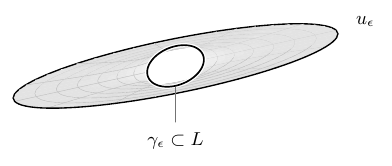}
  \end{minipage}

  \caption{Two views of a curve $u_\epsilon$ in the smoothing sector near an elliptic wall, with boundary
    $\gamma_\epsilon\subset L$ that shrinks to a point at the boundary of this sector.
    Left: a view suppressing the normal direction to $L$.
    Right: a linear projection, showing that the curve departs $L$ in a direction determined by $J$ of the tangent to $\gamma_\epsilon$.}
  \label{fig:u-epsilon}
\end{figure}

\subsubsection*{The orientation comparison theorem}

In Section~\ref{subsec:wall-problems}, we formulate the linearised deformation problem at a curve on a wall as an abstract \emph{wall problem} (Definition~\ref{def:wall-problem}), with wall complex $D_{\mathrm{wall}}$~\eqref{eq:D-wall-def}, normalised complex $D_{\mathrm{norm}}$, and nodal complex $D_{\mathrm{node}}$ (Section~\ref{subsec:nodal-smoothing}). Given compatible orientation data (Definition~\ref{def:compatible-orientation-data}), the holomorphic orientation convention (Definition~\ref{def:holomorphic-orientation}) and the anti-holomorphic orientation convention (Definition~\ref{def:antiholomorphic-orientation}) each induce a \emph{positive chamber orientation} $\mathsf o^+_\partial$, a \emph{negative chamber orientation} $\mathsf o^-_\partial=-\mathsf o^+_\partial$, and a \emph{smoothing boundary orientation} $\mathsf o^{\mathrm{sm}}_\partial$ on the determinant line $\det D_{\mathrm{wall}}$ (and their anti-holomorphic counterparts $\mathsf o^{\pm,\mathrm{ah}}_\partial$ and $\mathsf o^{\mathrm{sm,ah}}_\partial$; see Definition~\ref{def:boundary-orientations}). The same definitions extend to generalised wall problems with capped strip and cylindrical ends on the real blowup $\Sigma^{KN}$ (Section~\ref{subsec:capping-and-generalisations}).

Our main result is the following comparison of orientations on $\det D_{\mathrm{wall}}$, proved by combining Proposition~\ref{prop:wall-model-reduction}, Lemma~\ref{lem:elliptic comparison}, Proposition~\ref{prop:disconnected-hyperbolic-orientation}, and Lemma~\ref{lem:connected-hyperbolic-orientation}. We give a second proof of each case: Section~\ref{sec:elliptic-crossing} recovers the elliptic comparison in an explicit model family, and Section~\ref{sec:ramified-covers} recovers both hyperbolic comparisons.

\begin{thm}[Wall and smoothing orientation comparisons]
\label{thm:main-orientation-comparison}
Let $\Sigma$ be a bordered Riemann surface with no boundary marked points, and let $(E,F,D_\Sigma,\phi)$ be a wall problem (Definition~\ref{def:wall-problem}), equipped with compatible orientation data (Definition~\ref{def:compatible-orientation-data}).
\begin{enumerate}[(1)]
\item \textbf{Holomorphic orientation convention:} Using the holomorphic orientations of Definition~\ref{def:holomorphic-orientation} and Definition~\ref{def:boundary-orientations}, the smoothing boundary orientation $\mathsf o^{\mathrm{sm}}_\partial$ of $\det D_{\mathrm{wall}}$ agrees with:
\begin{enumerate}[(i)]
\item at a disconnected hyperbolic wall, the positive-crossing chamber orientation $\mathsf o^+_\partial$ (Proposition~\ref{prop:disconnected-hyperbolic-orientation});
\item at a connected hyperbolic wall, the negative-crossing chamber orientation $\mathsf o^-_\partial$ (Lemma~\ref{lem:connected-hyperbolic-orientation});
\item at an elliptic wall, the negative-incidence chamber orientation $\mathsf o^-_\partial$ (Lemma~\ref{lem:elliptic comparison}).
\end{enumerate}
\item \textbf{Anti-holomorphic orientation convention:} Using the anti-holomorphic orientations of Definition~\ref{def:antiholomorphic-orientation} and Definition~\ref{def:boundary-orientations}, the smoothing boundary orientation $\mathsf o^{\mathrm{sm,ah}}_\partial$ of $\det D_{\mathrm{wall}}$ agrees with the negative chamber orientation $\mathsf o^{-,\mathrm{ah}}_\partial$ in all three cases:
\[
\mathsf o^{\mathrm{sm,ah}}_\partial = \mathsf o^{-,\mathrm{ah}}_\partial.
\]
\end{enumerate}
\end{thm}

\begin{table}[htbp]
\centering
\small
\renewcommand{\arraystretch}{1.25}
\begin{tabular}{@{}>{\raggedright\arraybackslash}p{0.22\linewidth}>{\raggedright\arraybackslash}p{0.22\linewidth}>{\raggedright\arraybackslash}p{0.22\linewidth}>{\raggedright\arraybackslash}p{0.22\linewidth}@{}}
\toprule
& \textbf{Elliptic} & \textbf{Connected hyperbolic} & \textbf{Disconnected hyperbolic}\\
\midrule
At the wall
 & an interior point of the curve meets $L$
 & the boundary crosses itself on one boundary component
 & two boundary components cross\\
Smoothing the node
 & creates a small boundary circle: $b_{\mathbb R}$ increases by $1$
 & splits a boundary component: $b_{\mathbb R}$ increases by $1$
 & merges two boundary components: $b_{\mathbb R}$ decreases by $1$\\
\midrule
Holomorphic: $\mathsf o^{\mathrm{sm}}_\partial=$
 & $\mathsf o^-_\partial$ & $\mathsf o^-_\partial$ & $\mathsf o^+_\partial$\\
Skein relation (\S\ref{sec:skein-modules})
 & $[v_1]-[v_{-1}]=[u_1]$ & $[v_1]-[v_{-1}]=[u_1]$ & $[v_1]-[v_{-1}]=-[u_1]$\\
Anti-holomorphic: $\mathsf o^{\mathrm{sm},\mathrm{ah}}_\partial=$
 & $\mathsf o^{-,\mathrm{ah}}_\partial$ & $\mathsf o^{-,\mathrm{ah}}_\partial$ & $\mathsf o^{-,\mathrm{ah}}_\partial$\\
$\mathcal J$-complex (\S\ref{sec:ramified-covers}): $\mathsf o^{\mathrm{sm},\mathcal J}_\partial=$
 & --- & $\mathsf o^{+,\mathcal J}_\partial$ & $\mathsf o^{+,\mathcal J}_\partial$\\
\midrule
First calculation
 & Lemma~\ref{lem:elliptic comparison}: direct, in the local model
 & Lemma~\ref{lem:connected-hyperbolic-orientation}: cusp family joining this wall to an elliptic wall (\S\ref{connected-hyperbolic-section})
 & Prop.~\ref{prop:disconnected-hyperbolic-orientation}: two half-cylinders glued at a crossing (\S\ref{sec:disconnected-crossing})\\
Second calculation
 & \S\ref{sec:elliptic-crossing}: half-cylinders with an elliptic nodal limit
 & \S\ref{sec:ramified-covers}: ramified covers of a half-cylinder (Prop.~\ref{prop:ramified-hyperbolic})
 & \S\ref{sec:ramified-covers}: ramified covers of a half-cylinder (Prop.~\ref{prop:ramified-hyperbolic})\\
\bottomrule
\end{tabular}
\caption{The orientation comparisons of Theorem~\ref{thm:main-orientation-comparison} at the three types of wall. In every case $\chi$ decreases by $1$ under smoothing. In the anti-holomorphic convention, the twisted classes $[u]_{\mathrm{ah}}=\epsilon_{\mathrm{ah}}(u)[u]$ satisfy the uniform relation $[v_1]_{\mathrm{ah}}-[v_{-1}]_{\mathrm{ah}}=[u_1]_{\mathrm{ah}}$ (Remark~\ref{rem:antiholomorphic-skein-signs}).}
\label{tab:wall-comparisons}
\end{table}

Theorem~\ref{thm:main-orientation-comparison} compares Fredholm determinant orientations at the nodal curve $u$ and its normalisation $v$ without requiring transversality. In the transverse case, it gives the corresponding comparisons of ordinary boundary orientations between the smoothing sector and the normalised chambers. In particular, passing to the anti-holomorphic orientation convention changes the relative sign in the disconnected hyperbolic case while preserving the connected hyperbolic and elliptic signs (see Remark~\ref{rem:antiholomorphic-skein-signs}), ensuring that the two hyperbolic cases obey the same skein relation. Table~\ref{tab:wall-comparisons} summarises the three comparisons and where each is calculated.

\section{Worldsheet skein modules and HOMFLYPT comparison}
\label{sec:skein-modules}

In this section, we define a version of the worldsheet skein module of Ekholm, Longhi, and Nakamura~\cite[Section 3.2.1]{worldsheetskein}, with updated signs encoding the wall and chamber conventions of Section~\ref{sec:chamber}. To present the algebraic structure and updated signs in their simplest form, we restrict to the case of a compact $6$-dimensional symplectic manifold $X$ and a compact, oriented Lagrangian submanifold $L\subset X$.  In Section~\ref{skein-homomorphism-section}, we construct a homomorphism to the HOMFLYPT skein module of $L$, again following \cite{worldsheetskein}.

\subsection{The worldsheet skein module for compact \texorpdfstring{$X$}{X}}
\label{subsec:worldsheet-conventions}

Let $X$ be a compact $6$-dimensional symplectic manifold with a compatible almost complex structure $J$, and let $L\subset X$ be a compact, oriented Lagrangian submanifold.

\begin{defn}[Generic worldsheets]\label{def:generic-worldsheet}
A \emph{generic worldsheet} in $(X,L)$ is a smooth embedded map $u\colon(\Sigma,\partial\Sigma)\to (X,L)$ from a smooth compact (non-nodal, possibly disconnected)  bordered Riemann surface, satisfying:
\begin{enumerate}[(i)]
\item $u$ intersects $L$ only along its boundary $\partial\Sigma$; and
\item $u$ satisfies the Cauchy--Riemann equation to first order along $\partial\Sigma$.
\end{enumerate}
\end{defn}

\begin{defn}[Worldsheet skein module]\label{def:worldsheet-skein-modules}
The \emph{worldsheet skein module} $\sk(X,L)$ is the $\mathbb R$-module generated by generic  worldsheets $u$ in $(X,L)$, subject to:
\begin{enumerate}[(i)]
\item \emph{Isotopy:} $[u_0]=[u_1]$ if $u_0$ and $u_1$ are joined by a smooth isotopy of generic worldsheets; and
\item \emph{Local skein relations:} the elliptic, connected hyperbolic, and disconnected hyperbolic relations below.
\end{enumerate}
\end{defn}

For each of the three codimension-one wall events described in Section~\ref{sec:chamber}, let $\{v_t\}_{t\in[-1,1]}$ be a smooth family of maps that are generic embedded worldsheets for $t\neq 0$ and  experience a single transverse wall event at $t=0$, so that $v_t$ is in the positive chamber for $t>0$ and the negative chamber for $t<0$.

Let $u_0$ be the embedded curve with a boundary node whose partial normalisation at this node is $v_0$, and
let  $\{u_\epsilon\}_{\epsilon\in[0,1]}$ be a family of maps, with this boundary node smoothed for $\epsilon>0$, such that $u_\epsilon$ is a generic embedded worldsheet for $\epsilon>0$. In each case,
\[
\chi(u_1)=\chi(v_{\pm 1})-1.
\]

\begin{itemize}
\item \textbf{Elliptic skein relation:}
At the wall, $v_0$ has a single transverse interior intersection with $L$ at $q\in\Sigma^\circ$ where $dv_0$ is injective and complex-linear, with $v_t$ passing from the negative-incidence chamber ($t<0$) to the positive-incidence chamber ($t>0$), as illustrated in Figure~\ref{elliptic-v}. The nodal curve $u_0$ is obtained by introducing an elliptic boundary node at $q$, and $u_\epsilon$ ($\epsilon>0$) opens a small boundary unknot $\gamma_\epsilon\subset L$ around $v_0(q)$ (Figure~\ref{fig:u-epsilon}). The elliptic relation is
\begin{equation}\label{eq:elliptic-embedded}
[v_1]-[v_{-1}]=[u_1].
\end{equation}

\item \textbf{Hyperbolic skein relations:}
At the wall, $\partial v_0$ has a single transverse self-intersection in $L$, with $v_t$ passing from the negative-crossing chamber ($t<0$) to the positive-crossing chamber ($t>0$), while $u_0$ has a hyperbolic boundary node at this crossing with normalisation $v_0$, and $u_\epsilon$ ($\epsilon>0$) is its smoothing (Figure~\ref{fig:hyperbolic-wall-local}). Depending on the crossing type (Definition~\ref{def:hyperbolic-crossing-types}), we impose:
\begin{align}
\text{\emph{Connected crossing} }\quad & [v_1]-[v_{-1}]=[u_1], \label{eq:connected-embedded}\\
\text{\emph{Disconnected crossing}}\quad & [v_1]-[v_{-1}]=-[u_1]. \label{eq:disconnected-embedded}
\end{align}
\end{itemize}

With the holomorphic orientation convention of Definition~\ref{def:holomorphic-orientation} induced by a choice of  Spin structure $\mathfrak s$ on $L$, the signs in \eqref{eq:elliptic-embedded}--\eqref{eq:disconnected-embedded} are those required for the skein relations to cancel codimension-one boundary strata of the moduli space: the negative chamber and smoothing boundary orientations agree in the elliptic and connected hyperbolic cases (Lemma~\ref{lem:elliptic comparison} and Lemma~\ref{lem:connected-hyperbolic-orientation}), and the positive chamber and smoothing boundary orientations agree in the disconnected hyperbolic case (Proposition~\ref{prop:disconnected-hyperbolic-orientation}); see also Section~\ref{sec:ramified-covers} for a uniform calculation of both hyperbolic cases.

\begin{remark}[The anti-holomorphic sign correction]
\label{rem:antiholomorphic-skein-signs} In Section~\ref{subsec:antiholomorphic-orientations}, we define an anti-holomorphic orientation convention which differs by a sign $\epsilon_{\mathrm{ah}}$; see \eqref{eq:antiholomorphic-sign-index}.
For a generic worldsheet $u$ (with no marked points) this sign is
\[
\epsilon_{\mathrm{ah}}(u):=(-1)^{(3\chi(u)-b_{\mathbb R}(u))/2},
\]
where $b_{\mathbb R}$ is the number of boundary components.
In all three wall events, $\chi(u_1)=\chi(v_{\pm 1})-1$. Smoothing adds one boundary component in the elliptic and connected hyperbolic cases ($b_{\mathbb R}(u_1)=b_{\mathbb R}(v_{\pm 1})+1$), and removes one in the disconnected hyperbolic case ($b_{\mathbb R}(u_1)=b_{\mathbb R}(v_{\pm 1})-1$). Consequently,
\[
\epsilon_{\mathrm{ah}}(u_1)=
\begin{cases}
\epsilon_{\mathrm{ah}}(v_{\pm 1}), & \text{elliptic or connected hyperbolic},\\
-\epsilon_{\mathrm{ah}}(v_{\pm 1}), & \text{disconnected hyperbolic}.
\end{cases}
\]
Setting $[u]_{\mathrm{ah}}:=\epsilon_{\mathrm{ah}}(u)[u]$, all three local skein relations take the uniform shape
\begin{equation}\label{eq:antiholomorphic-embedded-relation}
[v_1]_{\mathrm{ah}}-[v_{-1}]_{\mathrm{ah}}=[u_1]_{\mathrm{ah}}.
\end{equation}
After this twist, both hyperbolic relations match \cite[Definition 3.1]{worldsheetskein}. Our labelling of the elliptic chambers is opposite to that of \cite[v2, Definition 3.1]{worldsheetskein}. See Remark~\ref{rem:alternative-orientation-conventions} for the effect of alternative orientation conventions. Our version is required for the HOMFLYPT homomorphism below.
\end{remark}

\subsection{The homomorphism to the HOMFLYPT skein module}
\label{skein-homomorphism-section}

In this section, we define a homomorphism from the worldsheet skein module to the HOMFLYPT skein module, following \cite{worldsheetskein} but applying simplified geometric assumptions for expository purposes.

Let $\sk(L)$ denote the framed HOMFLYPT skein module of oriented framed links in $L$ over $\mathbb R[a^{\pm 1},z^{\pm 1}]$, defined by the local skein relations
\[
\left( \Lpos \right) - \left( \Lneg \right) = z \left( \Lzero \right), \qquad a - a^{-1} = z \, \Unknot,
\]
\[
\PosTwist = a \, \Straight, \qquad \NegTwist = a^{-1} \, \Straight.
\]

To map $\sk(X,L)$ to $\sk(L)$ following~\cite[Section 3.2.2]{worldsheetskein}, we require a global bounding $4$-chain for $2L$.

\begin{assumption}[Bounding-chain data]\label{ass:bounding-chain}
There is a $4$-chain $C$ in $X$ such that:
\begin{enumerate}[(i)]
\item $\partial C=2L$; and
\item on a neighbourhood $U$ of $L$, we have $C|_U=C_+ + C_-$, where $C_\pm\subset U$ are properly embedded, oriented smooth submanifolds with $\partial C_\pm=L$, intersecting only along $L$, such that $C_+\cup C_-$ is a smooth submanifold of $U$.
\end{enumerate}
\end{assumption}

Choose a nonvanishing vector field $\xi$ on $L$ such that $J\xi$ is tangent to $C_+\cup C_-$ along $L$ and points into $C_+$. For a generic worldsheet $u$ whose boundary $\partial u$ is transverse to $\xi$, let $\langle\partial u\rangle\in\sk(L)$ denote the oriented boundary link $\partial u$ framed by $\xi$. Since $\partial u$ is transverse to $\xi$, the lifts of $u$ and $C$ to the real oriented blowup $\mathrm{bl}_L X$ have disjoint boundaries; let
\[
u\cap C\in\mathbb Z
\]
denote their algebraic intersection number (equivalently, the algebraic intersection of the interior of $u$ with $C$).

For a generic embedded worldsheet $u$ with $\partial u$ transverse to $\xi$, set
\begin{equation}\label{eq:embedded-homfly-map}
\rho_{C}(u) := \epsilon_{\mathrm{ah}}(u)\, z^{-\chi(u)} a^{u\cap C} \langle\partial u\rangle.
\end{equation}

\begin{lemma}\label{skein-comparison}
 Equation~\eqref{eq:embedded-homfly-map} defines a module homomorphism
\[
\rho_{C}\colon \sk(X,L)\longrightarrow \sk(L).
\]
\end{lemma}

\begin{proof}
In an isotopy of worldsheets $u$, we can lift $u$ to $\mathrm{bl}_L X$ because $u$ is immersed near $L$.
During a generic isotopy of worldsheets, $\partial u$ remains transverse to $\xi$ away from isolated parameter values where a strand of $\partial u$ becomes tangent to $\pm\xi$.  At such a tangency, the lift of $u$ to $\mathrm{bl}_LX$ crosses the boundary of the lift of $C_+$ or $C_-$, changing $u\cap C$ by $\pm 1$ while introducing a compensating $\mp 1$ framing twist in $\langle\partial u\rangle$. By the local calculation of~\cite[Lemma 2.12]{ekholm2019skeins} (checked explicitly in our cusp model in Remark~\ref{framing-model} and Equation~\eqref{eq:connected-framing-invariance}), the product $a^{u\cap C}\langle\partial u\rangle$ is invariant in $\sk(L)$. Since $\chi(u)$ and $b_{\mathbb R}(u)$ are topological invariants of the domain, $\rho_{C}(u)$ is invariant under isotopy.

By Remark~\ref{rem:antiholomorphic-skein-signs}, checking compatibility with \eqref{eq:elliptic-embedded}, \eqref{eq:connected-embedded}, and \eqref{eq:disconnected-embedded} is equivalent to checking that $\rho_{C  }([v_1]_{\mathrm{ah}}) - \rho_{C  }([v_{-1}]_{\mathrm{ah}}) = \rho_{C  }([u_1]_{\mathrm{ah}})$, where $\rho_{C  }([w]_{\mathrm{ah}}) = z^{-\chi(w)}a^{w\cap C}\langle\partial w\rangle$.

For the elliptic relation, our chamber convention gives $v_{\pm 1}\cap C = u_1\cap C \pm 1$, while $\chi(u_1)=\chi(v_{\pm 1})-1$ and the small boundary circle around $v_0(q)$ is a $0$-framed unknot (as verified in the local model in Section~\ref{subsec:elliptic-bounding-chain}), so \[\langle\partial u_1\rangle = \langle\partial v_{\pm 1}\rangle \cup \Unknot.\] Using $a - a^{-1} = z\,\Unknot$, we obtain
\[
\begin{aligned}
\rho_{C  }([v_1]_{\mathrm{ah}}) - \rho_{C  }([v_{-1}]_{\mathrm{ah}})
&= z^{-\chi(v_1)} a^{u_1\cap C}(a - a^{-1})\langle\partial v_1\rangle \\
&= z^{-\chi(v_1)+1} a^{u_1\cap C}\langle\partial u_1\rangle
= \rho_{C  }([u_1]_{\mathrm{ah}}).
\end{aligned}
\]

For both the connected and disconnected hyperbolic relations, the interior intersections with $C$ are unchanged across the wall ($v_{\pm 1}\cap C = u_1\cap C$), while the framed boundary links undergo a crossing change and smoothing:
\[
\langle\partial v_1\rangle - \langle\partial v_{-1}\rangle = z\langle\partial u_1\rangle.
\]
Since $\chi(u_1)=\chi(v_{\pm 1})-1$, we have $z^{-\chi(v_{\pm 1})}z = z^{-\chi(u_1)}$, so
\[
\rho_{C  }([v_1]_{\mathrm{ah}}) - \rho_{C  }([v_{-1}]_{\mathrm{ah}})
= z^{-\chi(v_1)}a^{u_1\cap C}\bigl(\langle\partial v_1\rangle - \langle\partial v_{-1}\rangle\bigr)
= \rho_{C  }([u_1]_{\mathrm{ah}}). \qedhere
\]
\end{proof}

\section{Abstract Fredholm problems}      \label{sec:linear-setting}

In this section, we formulate the linear comparison problem in the abstract setting of real bundle pairs over bordered, possibly nodal Riemann surfaces, without strip ends.

The deformation of a smooth bordered holomorphic curve $u:(\Sigma,\partial \Sigma)\to (X,L)$ involves two coupled problems:
\begin{itemize}
    \item \textbf{Map variations:} Governed by the fixed-domain linearisation of the $\dbar$  operator
    \[
    D_\Sigma \colon \Omega^0(u^*TX, u^*TL) \longrightarrow \Omega^{0,1}(u^*TX).
    \]
    \item \textbf{Domain variations:} Governed by the Cauchy--Riemann operator on vector fields tangent to the boundary,
    \[
    \Conf_\Sigma \colon \Omega^0(T\Sigma, T\partial\Sigma) \longrightarrow \Omega^{0,1}(T\Sigma).
    \]
    Its kernel consists of infinitesimal automorphisms, and its cokernel parametrises variations of conformal structure.
\end{itemize}

The derivative $du$ defines a chain map from $\Conf_\Sigma$ to $D_\Sigma$. The full linearisation $D_u$, including variations of the domain modulo infinitesimal reparametrisations, is represented by the mapping cone
\[
\Cone(\Conf_\Sigma \xrightarrow{\,du\,} D_\Sigma).
\]
For a stable curve, this complex has cohomology concentrated in degree zero and degree one, so this three-term complex is quasi-isomorphic to a Fredholm operator.

Below, we extend this description to domains with nodes and marked points,  and replace $(u^*TX, u^*TL)$ with a real bundle pair $(E,F)$ over $(\Sigma,\partial \Sigma)$, equipped with a doubleable Cauchy--Riemann operator (Assumption~\ref{ass:doubleable}). Section~\ref{subsec:nodal-surfaces} describes the domains and their logarithmic tangent bundles, Section~\ref{subsec:intro-linear-setup} the bundle pairs and coupled deformation complexes, and Sections~\ref{subsec:wall-problems}--\ref{subsec:nodal-smoothing} the wall problem with its incidence and smoothing exact triples.

\subsection{Nodal bordered surfaces and the logarithmic tangent bundle}\label{subsec:nodal-surfaces}

Let $(\Sigma, \partial\Sigma)$ be a compact, bordered, nodal Riemann surface with $n$ interior marked points and $m$ boundary marked points. Let $2\Sigma:=(\Sigma, j)\cup_{\partial \Sigma}(\Sigma,-j)$ be its double, with anti-holomorphic involution $\sigma$ swapping the two halves. We consider nodes of three types:
\begin{itemize}
    \item \emph{Interior nodes}, where two interior branches of $\Sigma$ meet;
    \item \emph{Hyperbolic boundary nodes}, where two branches of $\partial \Sigma$ meet, corresponding to a hyperbolic real node on the doubled surface $2\Sigma$;
    \item \emph{Elliptic boundary nodes}, which are isolated real nodes of $2\Sigma$ where $\sigma$ permutes the two branches. On $\Sigma$, we can also think of an elliptic boundary node as a special marked point, considered to be an isolated component of $\partial \Sigma$.

    \end{itemize}

To ensure that variations of complex structure and vector fields behave continuously across degenerations and node smoothings, we work with the \emph{logarithmic tangent bundle} $T^{\log}\Sigma$.

 Near a marked point $z=0$, $T^{\log}\Sigma$ is generated by the real and imaginary parts of $z\partial_z$, and the logarithmic cotangent space is generated by the real and imaginary parts of $z^{-1}dz$. Geometrically, this corresponds to thinking of an interior marked point as a cylindrical end, and a boundary marked point as a strip end.

 \begin{defn}
Let $v$ be a real section of the logarithmic tangent bundle,
and let $z$ be a holomorphic coordinate vanishing at a marked
point $p$. Write
\[
v(z)=a(z)z.
\]
The residue of $v$ at $p$ is
\[
\operatorname{res}_p v:=a(p)\in\mathbb C.
\]
\end{defn}

The residue does not depend on the choice of holomorphic coordinate, and canonically identifies $T^{\log}_p\Sigma$ with $\mathbb C$. In notation identifying a real vector field with its $(1,0)$ part we can alternately write  $\operatorname{res}_pa(z)z\partial_z=a(0)$.

 At a node with local holomorphic coordinates $z$ and $w$ with $zw = 0$, the  logarithmic tangent bundle\footnote{A nodal curve arises as the underlying scheme of a log smooth curve over a log point, which despite being called a ``point", can be thought of as parametrising infinitesimal gluing data at nodes~\cite[Section 5]{elc}. Our log tangent bundle corresponds to the vertical tangent bundle of such a log curve; see Remark~\ref{log residue}. } is generated by the vector field
\[
(z\partial_z, - w\partial_w).
\]
 More precisely, locally, a section of $T^{\log}\Sigma$ consists of a pair of logarithmic vector fields $(a(z)z\partial_z, b(w)w\partial_w)$ on the two branches satisfying the residue-matching condition
\[
a(0) + b(0) = 0.
\]

The above suffices to define the logarithmic tangent space of a closed Riemann surface. The logarithmic tangent space of a bordered Riemann surface $\Sigma$ is the restriction of $T^{\log}2\Sigma$ to $\Sigma\subset 2\Sigma$. Over $\partial \Sigma$, we have the real subbundle $T^{\log}\partial \Sigma\subset T^{\log}\Sigma\rvert_{\partial \Sigma}$, consisting of the fixed locus of the involution, or alternately, generated by real logarithmic vector fields tangent to the boundary.

If $p$ is a boundary marked point and $v$ is a section of $T^{\log}\partial\Sigma$, then  $\operatorname{res}_pv\in\mathbb R$, and the restriction of the real part of $ z\partial_z$ to $T^{\log}_p\Sigma$ is a canonical inward pointing vector generating $T_p^{\log}\partial \Sigma$.

We can choose coordinates $z,w$ centred on a real node of $2\Sigma$ so that $\sigma(z,w)=(\bar z,\bar w)$ in the hyperbolic case, and $\sigma(z,w)=(\bar w,\bar z)$ in the elliptic case. Near such a node, $T^{\log}\partial \Sigma$ is generated by the real part of $(z\partial_z,-w\partial_w)$ in the hyperbolic case, and the imaginary part of $z\partial_z$ in the elliptic case. In the elliptic case, we write $\partial_\vartheta:=-iz\partial_z$, where $z$ is the coordinate on the branch in $\Sigma$; it can be thought of as generating rotation around an infinitesimal boundary component at the node, with the usual boundary orientation. See Section~\ref{subsec:elliptic-orientation-triple}. A logarithmic vector field $v$ is tangent to the boundary at an elliptic node $q$ if $\operatorname{Re}\operatorname{res}_q v=0$, and tangent to the boundary at a hyperbolic boundary node joining points $p_1$ and $p_2$ if $\operatorname{Im}\operatorname{res}_{p_1}v=0$ and it also satisfies the residue-matching condition $\operatorname{res}_{p_1}v+\operatorname{res}_{p_2}v=0$.

\subsubsection{Infinitesimal gluing data and smoothing nodes}
\label{sec:gluing}

In Section~\ref{sec:conf}, we define a Fredholm operator $\Conf_\Sigma \colon \Omega^0(T^{\log}\Sigma, T^{\log}\partial\Sigma) \to \Omega^{0,1}(T^{\log}\Sigma)$ governing, in the case that $\Sigma$ is smooth,  variations of the complex structure of $\Sigma$. Using the logarithmic tangent space ensures suitable continuity of this operator in families that form nodes, essentially because the logarithmic tangent bundle forms a vector bundle over the total space of such families. When $\Sigma$ has nodes, $\Conf_\Sigma$ governs variations of complex structure together with infinitesimal gluing information at nodes, or equivalently, variations of complex structure on an  associated exploded~\cite{iec} curve $\ex \Sigma$.

Accordingly, we will think of our Riemann surfaces as having an extra choice of infinitesimal gluing data at nodes. With a choice of such data on a closed Riemann surface $\Sigma$, there is an exploded curve~\cite[Definition 8.3]{iec} $\ex \Sigma$ with smooth part $\Sigma$, and the tangent bundle of $\ex \Sigma$ is the pullback of the logarithmic tangent bundle of $\Sigma$. Around a node, corresponding to coordinates $z$ and $w$ of $\Sigma$ with $zw=0$,  $\ex \Sigma$ has coordinates $\tilde z$ and $\tilde w$ satisfying the relation
\[\tilde z\tilde w=\tilde \tau\]
where  $\tilde z$ and $\tilde w$ take values in $\mathbb C^*\e{\mathbb R_{\geq 0}}$~\cite[Section 2]{iec}, and $\tilde \tau=c\e l\in \mathbb C^*\e{\mathbb R_{>0}}$ is constant.  We think of $\mathfrak t$ as infinitesimal, so a value of $\tilde z$ or $\tilde w$ with positive exponent corresponds to a zero of $z$ or $w$.

The extra information at this node is  $c\in \mathbb C^*$, thought of as an infinitesimal gluing parameter, and the real positive length $l$, thought of as large scale gluing information and the length of an internal edge of an associated tropical curve. Intrinsically, when a node joins the points $p_1$ and $p_2$ on the normalised surface $\Sigma_n$, the infinitesimal gluing parameter is given by a nonzero element $v_1\otimes v_2$ of $T_{p_1}\Sigma_n\otimes_{\mathbb C} T_{p_2}\Sigma_n$, and $c=(v_1z)(v_2w)$.

Allowing $\tilde\tau$ to vary within  $\mathbb C^*\e{\mathbb R_{\geq 0}}$, we obtain a family of curves $\ex \Sigma_{\tilde \tau}$, corresponding to a family $\Sigma_\tau$ of curves with local plumbing coordinates
\[zw=\tau,\]
defined for $\tau\in \mathbb C$ sufficiently small. While the map $(z,w)\mapsto zw$  is not a submersion at $z=w=0$, the family $\ex \Sigma_{\tilde \tau}$ is a smooth family of exploded curves, and the map $(\tilde z,\tilde w)\mapsto \tilde z\tilde w$ is a submersion. Reflecting this, the logarithmic tangent bundle of $\Sigma_{\tau}$, generated by $(z\partial_z,-w\partial_w)$, is a smooth complex line bundle over the total space of this family.

There is also a natural notion of a bordered exploded curve $(\ex \Sigma,\partial\ex \Sigma)$ with smooth part $(\Sigma,\partial\Sigma)$. Near a hyperbolic boundary node of $\Sigma$, with local coordinates
\[z\in\mathbb H, \quad w\in\bar{\mathbb H},\quad zw=0\]
in the upper and lower complex half-planes, $\ex \Sigma$ has coordinates $\tilde z$ and $\tilde w$ satisfying
\[\tilde z\tilde w=\tilde \epsilon, \quad \tilde\epsilon\in \mathbb R^*_{>0}\e{\mathbb R_{>0}},\]
where
\[\tilde z\in (\mathbb H\cap \mathbb C^*)\e{\mathbb R_{\geq 0}},\quad \tilde w\in (\bar{\mathbb H}\cap \mathbb C^*)\e{\mathbb R_{\geq 0}} .\]
Allowing $\tilde \epsilon$ to vary, we obtain a family of exploded manifolds $\ex \Sigma_{\tilde \epsilon}$  with smooth part a family of curves $\Sigma_\epsilon$ smoothing this boundary node, with local plumbing coordinates
\[zw=\epsilon, \quad \epsilon\in \mathbb R_{\geq 0}\]
defined for $\epsilon$ sufficiently small.

We also choose infinitesimal gluing information at elliptic boundary nodes. Centred on an elliptic real node of  $2\Sigma$, there are coordinates $z$ and $w$ on the branches such that $\sigma(z,w)=(\bar w,\bar z)$. After a choice of infinitesimal gluing information compatible with this involution, we obtain coordinates $\tilde z$ and $\tilde w$ in $\mathbb C^*\e{\mathbb R_{\geq 0}}$ satisfying
\[\tilde z\tilde w=\tilde \epsilon, \quad \tilde\epsilon\in \mathbb R^*_{>0}\e{\mathbb R_{>0}},\]
and there is an infinitesimal boundary component where $|\tilde z|^2=\tilde \epsilon$.

\begin{remark}\label{disc model}There is an alternate model $\Sigma'$ for a curve with an elliptic boundary node at $q\in \Sigma$, consisting of a disc attached at an interior node to $q$ (Figure~\ref{fig:disc-model}). Given a choice of infinitesimal gluing information at the elliptic node as above, with $z$ the coordinate on $\Sigma$, there are induced coordinates $w'$ with $|w'|\leq 1$ on the disc with induced infinitesimal gluing data at the internal node in the form $\tilde z\tilde w'=\sqrt{\tilde \epsilon}$. This expands the infinitesimal boundary at the elliptic boundary node to the  boundary $|w'|=1$ of the attached disc.  The corresponding map of exploded curves $\pi:\ex \Sigma'\to \ex \Sigma$ is called a refinement~\cite[Section 10]{iec}, and is related to logarithmic modification.

The logarithmic tangent bundle of $\Sigma'$ is naturally isomorphic to the pullback of the logarithmic tangent bundle of $\Sigma$ via $\pi$. In particular, if $v$ is a logarithmic vector  field on $\Sigma$ tangent to the boundary at $q$, then we define $\pi^*v$ to be the unique logarithmic vector field on $\Sigma'$ that agrees with $v$ on $\Sigma\subset \Sigma'$, and is holomorphic and tangent to the boundary on the extra disc component.   When we study Cauchy--Riemann operators on bundle pairs, $\pi^*$ will induce natural quasi-isomorphisms, so we can use either model as convenient.
\end{remark}

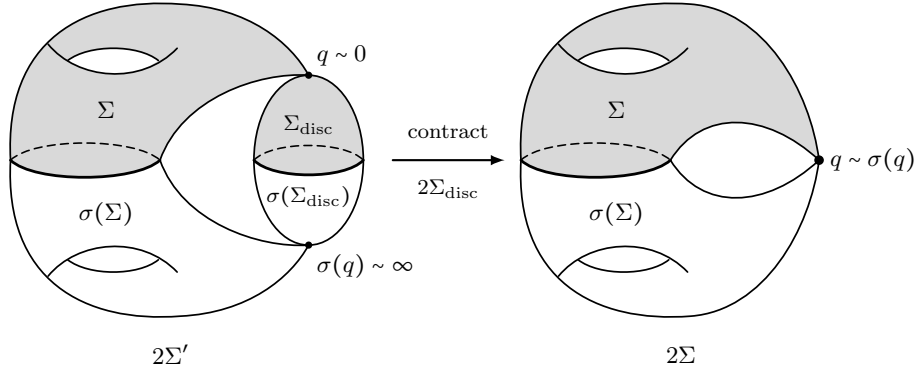
\begin{figure}[htbp]
\begin{tikzpicture}[
  x=1.08cm,y=1.08cm,
  line cap=round,line join=round,
  surface/.style={line width=0.65pt},
  real/.style={line width=1.05pt},
  rear/.style={line width=0.6pt,densely dashed},
  every node/.style={font=\small},
  >={latex}
]
  \begin{scope}
    \path[fill=white]
      (0,0) .. controls (-0.08,1.47) and (0.90,2.00) .. (2.08,1.91)
      .. controls (2.91,1.87) and (3.43,1.44) .. (3.65,1.04)
      .. controls (3.04,1.08) and (2.10,0.76) .. (1.83,0)
      .. controls (2.10,-0.76) and (3.04,-1.08) .. (3.65,-1.04)
      .. controls (3.43,-1.44) and (2.91,-1.87) .. (2.08,-1.91)
      .. controls (0.90,-2.00) and (-0.08,-1.47) .. (0,0) -- cycle;
    \path[fill=black!15]
      (0,0) .. controls (-0.08,1.47) and (0.90,2.00) .. (2.08,1.91)
      .. controls (2.91,1.87) and (3.43,1.44) .. (3.65,1.04)
      .. controls (3.04,1.08) and (2.10,0.76) .. (1.83,0)
      .. controls (1.52,-0.28) and (0.29,-0.28) .. (0,0) -- cycle;
    \draw[surface]
      (0,0) .. controls (-0.08,1.47) and (0.90,2.00) .. (2.08,1.91)
      .. controls (2.91,1.87) and (3.43,1.44) .. (3.65,1.04)
      .. controls (3.04,1.08) and (2.10,0.76) .. (1.83,0)
      .. controls (2.10,-0.76) and (3.04,-1.08) .. (3.65,-1.04)
      .. controls (3.43,-1.44) and (2.91,-1.87) .. (2.08,-1.91)
      .. controls (0.90,-2.00) and (-0.08,-1.47) .. (0,0);
    \path[fill=white]
      (0.70,1.22) .. controls (0.96,1.42) and (1.56,1.42) .. (1.81,1.22)
      .. controls (1.50,0.99) and (1.01,0.99) .. (0.70,1.22);
    \draw[surface]
      (0.46,1.42) .. controls (0.82,0.98) and (1.60,0.91) .. (2.04,1.37);
    \draw[surface]
      (0.70,1.22) .. controls (0.96,1.42) and (1.56,1.42) .. (1.81,1.22);
    \draw[surface]
      (0.46,-1.42) .. controls (0.82,-0.98) and (1.60,-0.91) .. (2.04,-1.37);
    \draw[surface]
      (0.70,-1.22) .. controls (0.96,-1.42) and (1.56,-1.42) .. (1.81,-1.22);
    \draw[rear] (0,0) .. controls (0.29,0.28) and (1.52,0.28) .. (1.83,0);
    \draw[real] (0,0) .. controls (0.29,-0.28) and (1.52,-0.28) .. (1.83,0);
    \node at (1.17,0.62) {$\Sigma$};
    \node at (1.17,-0.66) {$\sigma(\Sigma)$};
    \path[fill=white] (3.65,0) ellipse[x radius=0.67,y radius=1.04];
    \path[fill=black!15]
      (2.98,0) arc[start angle=180,end angle=0,x radius=0.67,y radius=1.04]
      .. controls (4.09,-0.24) and (3.21,-0.24) .. (2.98,0);
    \draw[surface] (3.65,0) ellipse[x radius=0.67,y radius=1.04];
    \draw[rear] (2.98,0) .. controls (3.21,0.24) and (4.09,0.24) .. (4.32,0);
    \draw[real] (2.98,0) .. controls (3.21,-0.24) and (4.09,-0.24) .. (4.32,0);
    \node[font=\footnotesize] at (3.65,0.43) {$\Sigma_{\mathrm{disc}}$};
    \node[font=\footnotesize] at (3.65,-0.45) {$\sigma(\Sigma_{\mathrm{disc}})$};
    \fill (3.65,1.04) circle[radius=1.35pt];
    \fill (3.65,-1.04) circle[radius=1.35pt];
    \node[anchor=south west,inner sep=2pt] at (3.68,1.08) {$q\sim0$};
    \node[anchor=north west,inner sep=2pt] at (3.68,-1.08) {$\sigma(q)\sim\infty$};
    \node at (1.96,-2.38) { $2\Sigma'$};
  \end{scope}
  \draw[->,line width=0.8pt] (4.68,0) -- (6.05,0);
  \node[font=\footnotesize] at (5.36,0.35) {contract};
  \node[font=\footnotesize] at (5.36,-0.35) {$2\Sigma_{\mathrm{disc}}$};
  \begin{scope}[xshift=6.75cm]
    \path[fill=white]
      (0,0) .. controls (-0.08,1.47) and (0.90,2.00) .. (2.08,1.91)
      .. controls (2.98,1.84) and (3.54,0.85) .. (3.65,0)
      .. controls (3.00,0.60) and (2.22,0.63) .. (1.83,0)
      .. controls (2.22,-0.63) and (3.00,-0.60) .. (3.65,0)
      .. controls (3.54,-0.85) and (2.98,-1.84) .. (2.08,-1.91)
      .. controls (0.90,-2.00) and (-0.08,-1.47) .. (0,0) -- cycle;
    \path[fill=black!15]
      (0,0) .. controls (-0.08,1.47) and (0.90,2.00) .. (2.08,1.91)
      .. controls (2.98,1.84) and (3.54,0.85) .. (3.65,0)
      .. controls (3.00,0.60) and (2.22,0.63) .. (1.83,0)
      .. controls (1.52,-0.28) and (0.29,-0.28) .. (0,0) -- cycle;
    \draw[surface]
      (0,0) .. controls (-0.08,1.47) and (0.90,2.00) .. (2.08,1.91)
      .. controls (2.98,1.84) and (3.54,0.85) .. (3.65,0)
      .. controls (3.00,0.60) and (2.22,0.63) .. (1.83,0)
      .. controls (2.22,-0.63) and (3.00,-0.60) .. (3.65,0)
      .. controls (3.54,-0.85) and (2.98,-1.84) .. (2.08,-1.91)
      .. controls (0.90,-2.00) and (-0.08,-1.47) .. (0,0);
    \path[fill=white]
      (0.70,1.22) .. controls (0.96,1.42) and (1.56,1.42) .. (1.81,1.22)
      .. controls (1.50,0.99) and (1.01,0.99) .. (0.70,1.22);
    \draw[surface]
      (0.46,1.42) .. controls (0.82,0.98) and (1.60,0.91) .. (2.04,1.37);
    \draw[surface]
      (0.70,1.22) .. controls (0.96,1.42) and (1.56,1.42) .. (1.81,1.22);
    \draw[surface]
      (0.46,-1.42) .. controls (0.82,-0.98) and (1.60,-0.91) .. (2.04,-1.37);
    \draw[surface]
      (0.70,-1.22) .. controls (0.96,-1.42) and (1.56,-1.42) .. (1.81,-1.22);
    \draw[rear] (0,0) .. controls (0.29,0.28) and (1.52,0.28) .. (1.83,0);
    \draw[real] (0,0) .. controls (0.29,-0.28) and (1.52,-0.28) .. (1.83,0);
    \node at (1.17,0.62) {$\Sigma$};
    \node at (1.17,-0.66) {$\sigma(\Sigma)$};
    \fill (3.65,0) circle[radius=1.8pt];
    \node[anchor=west,inner sep=4pt] at (3.65,0) {$q\sim\sigma(q)$};
    \node at (1.96,-2.38) { $2\Sigma$};
  \end{scope}
\end{tikzpicture}
\caption{An alternate model $\Sigma'$ for $\Sigma$ with an elliptic node at $q$ is a constant disc attached to $q$ at an internal node.}
\label{fig:disc-model}
\end{figure}

Again, we can smooth an elliptic boundary node by allowing $\epsilon$ to vary in $\mathbb R^*_{>0}\e{\mathbb R_{\geq 0}}$. In each of the above cases, the logarithmic tangent bundle and $T^{\log}\partial \Sigma_{\epsilon}$ form smooth bundles over the total space of this family.

In each case, we can double $(\ex \Sigma_{\tilde \epsilon},\partial \ex \Sigma_{\tilde \epsilon})$ to obtain a closed exploded curve  $2\ex \Sigma_{\tilde \epsilon}$, with an involution that acts by complex conjugation on our (doubled) local coordinates $\tilde z$ and $\tilde w$ in the hyperbolic case, and acts by permutation and complex conjugation in the elliptic case.
This doubled family sits inside the smooth family of exploded curves $2\ex \Sigma_{\tilde \epsilon e^{i\theta}}$ locally defined by $\tilde z\tilde w= \tilde \epsilon e^{i\theta}$. There  are compatible anti-holomorphic involutions on the total space and base of this family acting by complex conjugation on  $e^{i\theta}$, and our family of doubled curves $2\ex\Sigma_{\tilde \epsilon}$ is the restriction of this family to $e^{i\theta}=1$. (The other half of the real locus, where $e^{i\theta}=-1$, removes the isolated boundary component in the elliptic case, and  locally glues different halves of $2\ex \Sigma$ in the hyperbolic case, so the real locus may no longer separate $2\ex\Sigma_{\tilde \tau}$ into two halves.)

Similarly, all nodes of $(\Sigma,\partial \Sigma)$ can be smoothed in a family $\Sigma_{(\tau_k,\epsilon_j)}$ with gluing coordinates $\tau_k\in\mathbb C$ and $\epsilon_j\in\mathbb R_{\geq 0}$ sufficiently small, and this family can be naturally embedded inside a log smooth doubled family $2\Sigma_{(\tau_k,\tau'_k, \epsilon_j e^{i\theta_j})}$ of curves, with a compatible anti-holomorphic involution on the total space and base (acting by $(\tau_k,\tau'_k, \epsilon_j e^{i\theta_j})\mapsto (\bar\tau'_k,\bar \tau_k, \epsilon_j e^{-i\theta_j})$). This log smooth family of curves is itself the smooth part of a smooth family $2\ex \Sigma_{(\tilde \tau_k,\tilde \tau'_k, \tilde \epsilon_j e^{i\theta_j})}$ of exploded manifolds. Accordingly, we can apply the regularity results of \cite{reg} to families of Cauchy--Riemann operators on such a family $\Sigma_{(\tau_k,\epsilon_j)}$ smoothing the nodes of $\Sigma$.

\begin{remark}\label{rem:smoothing direction} If we drop the residue-matching condition on a vector field, its flow will not preserve the infinitesimal gluing data. Such vector fields arise naturally on the total space of a family in which the gluing data vary.
Let $\pi:2\hat\Sigma\to\mathbb C$ be the complexified
doubled plumbing family, with parameter
$\tau=\epsilon e^{i\theta}$ sufficiently small and local equation $zw=\tau$.
The anti-holomorphic involution on the total space covers
complex conjugation on the base.
 Choose a real logarithmic vector field $V$ on its total space, invariant under the involution, and  with
$\pi_*V=\epsilon\partial_\epsilon$. Near the plumbing region, we can take
\[V=\left(\frac 12 z\partial_z,\frac 12 w\partial_w\right)\]
(so the flow of $V$ for time $t$ scales $z$ by $e^{t/2}$). Away from the
plumbing region, the family is a product, and we extend
$V$ by patching it with the product lift of
$\epsilon\partial_\epsilon$.

On the normalisation of the central fibre, $V$ restricts
to logarithmic vector fields $v_1,v_2$ compactly supported near the node, with
\[
\operatorname{res}_{p_1}v_1
=\operatorname{res}_{p_2}v_2=\frac12.
\]

Let $j$ be the complex structure on the vertical logarithmic
tangent bundle. Since $V$ projects to
$\epsilon\partial_\epsilon$, its flow identifies nearby fibres,
and $L_Vj$ is the resulting infinitesimal variation of their
complex structures. Under the identification of such variations
with tangent-valued $(0,1)$-forms, this variation is represented by
\[
\frac{1}{2i}L_Vj=\dbar_{\mathrm{vert}}V.
\]
Its restriction to the central fibre is
$(\dbar v_1,\dbar v_2)$, which therefore represents the
logarithmic smoothing variation induced by
$\epsilon\partial_\epsilon$.

\end{remark}

\begin{remark}\label{log residue}From the perspective of log geometry, the central fibre of the log smooth family  $2\Sigma_{(\tau_k,\tau'_k, \epsilon_j e^{i\theta_j})}$ is a log curve  over a log point $\Sigma^\dagger\to p^\dagger$. From the perspective of exploded manifolds, this log point parametrises choices of infinitesimal gluing data.  A projectable logarithmic vector field on $\Sigma^\dagger$
induces a vector in $T^{\log}p^\dagger$. Its component in
the gluing direction at a node is the sum of the two
branch residues. The relative logarithmic tangent bundle
consists of vertical fields, whose restrictions to the
normalisation satisfy the residue-matching condition.
\end{remark}

\begin{construction}[Stretching the neck] \label{stretch} We can stretch the neck around a smooth real analytic embedded circle in the interior of $(\Sigma,\partial \Sigma)$ or a smooth boundary component with no marked points. In the case of an embedded circle $C$,  choose a local holomorphic coordinate $z$ on a neighbourhood of $C \subset \Sigma$ such that $|z|=1$ on $C$, and let $w=1/z$. There is an associated plumbing family $\Sigma_{\tau}$ with local plumbing coordinates $zw=\tau$ for $|\tau|\leq 1$. A homotopy stretching the neck around $C$ is the family $\Sigma_{1-t}$ for $t\in [0,1]$, with $C$ contracted to an interior node at $t=1$. Similarly, in the case of a smooth boundary component, stretching the neck around the corresponding circle in $2\Sigma$ gives a family $\Sigma_{1-t}$ of bordered Riemann surfaces, with an elliptic boundary node when $t=1$.

When we apply this neck-stretching construction, we will simultaneously homotope a Cauchy--Riemann operator on a bundle pair $(E,F)$ that has a chosen trivialisation around $C$ or the boundary component, and is the standard $\dbar$ operator there. We extend the bundle trivially over the stretched region, with the standard operator $\dbar$ there. In the case of a coupled bundle pair, with coupling map $\phi\colon(T^{\log}\Sigma,T^{\log}\partial\Sigma)\to(E,F)$, if $\phi(\partial_\vartheta)$ is constant in our trivialisation of $E$, where $\partial_\vartheta=iz\partial_z$, then we extend the coupling to remain constant on the stretched region.
\end{construction}

\subsubsection{Metric, nodal, and real blowup spaces}\label{subsec:domain-topologies}

There are three natural topological spaces mapping to a bordered curve $(\Sigma,\partial \Sigma)$ with infinitesimal gluing data at nodes. Figure~\ref{fig:domain-topologies} shows them near nodes, and Figure~\ref{fig:domain-topologies-ends} near marked points and ends.

\begin{figure}[htbp]
  \centering
  \includegraphics[width=0.8\linewidth]{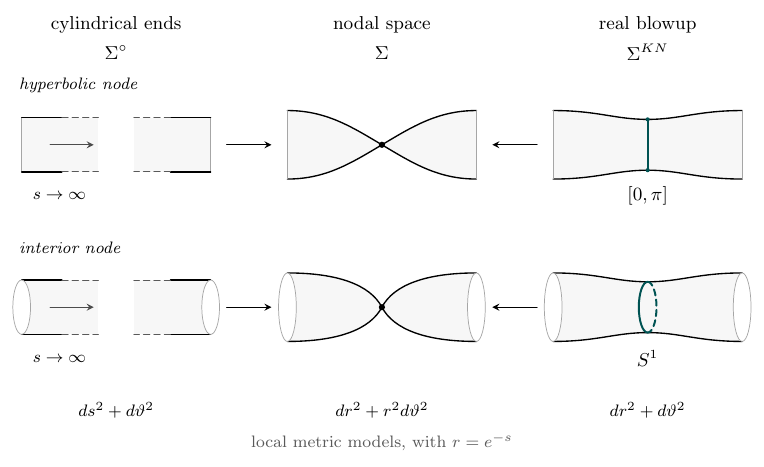}
  \caption{Local pictures of the three spaces near a hyperbolic node
    (top) and an interior node (bottom).
    In $\Sigma^\circ$, the branches have separate infinite strip or
    cylindrical ends. In $\Sigma$, they meet at a single point.
    In $\Sigma^{KN}$, infinitesimal gluing data identify the angular
    intervals or circles on the two branches, producing the common
    interval or circle.}
  \label{fig:domain-topologies}
\end{figure}

The first, called the \emph{curve with cylindrical ends}, is $\Sigma^{\circ}\subset\Sigma$, obtained by removing all marked points and nodes and choosing a metric on the logarithmic tangent space.  If $z=e^{-s-i\vartheta}$ gives coordinates around a marked point, strip end, or  branch of a node, so that $s\to\infty$ at the end, the associated metric is comparable to the Euclidean metric in $(s,\vartheta)$ coordinates. In these coordinates, $\partial_s=-z\partial_z$ points into the end and $\partial_\vartheta=-iz\partial_z$. So, $\Sigma^{\circ}$  is generally noncompact, and locally,  the two sides of a  node are separated from each other. It is this perspective that leads to the terminology of stretching the neck.

The nodal space $\Sigma$ is $\Sigma$ with the \emph{nodal topology}   pulled back from $2\Sigma$. In this nodal topology, the two sides of $\partial \Sigma$ at a hyperbolic node or strip end are identified, so a nonvanishing continuous section of $T^{\log}\partial\Sigma $ swaps orientation at each strip end.  As the normalisation of $2\Sigma$ is smooth, this space  comes with a natural equivalence class of smooth metrics. In this perspective, stretching the neck around a circle appears to contract the circle.

The third space is called the \emph{real blowup space}
\[\Sigma^{KN}\to\Sigma\ ,\]
related to the Kato--Nakayama space from logarithmic geometry~\cite{kato1999}. Unlike the other two spaces, its topology does not change in connected families of curves~\cite{NakayamaOgusRelativeRounding}. Its construction requires infinitesimal gluing information at nodes.
We construct $\Sigma^{KN}$ by taking the real oriented blowup of each component of $\Sigma$ at each strip end, marked point and node, and then gluing the resulting extra boundary at the nodes using the infinitesimal gluing data. In the case $\tilde z\tilde w= \tilde \epsilon \in \mathbb R^*_{> 0} \e{\mathbb R_{>0}}$, we identify the boundaries so that $\arg z+\arg w=0$. The topology on $\Sigma^{KN}$ is induced by a metric that degenerates at boundary nodes and strip ends. In the coordinates $z=e^{-s-i\vartheta}$, this metric is comparable to the Euclidean metric in $(e^{-s},\vartheta)$. In this real blowup space, each hyperbolic node or boundary marked point corresponds to an interval, parametrised by an interval of length $\pi$ in $\vartheta$, and each elliptic boundary node, marked point, or interior node corresponds to a circle, parametrised by $\vartheta\in\mathbb R/2\pi \mathbb Z$.
Each section of the logarithmic tangent space lifts to a vector field on $\Sigma^{KN}$ that is tangent to these special circles and intervals.  We use $(\partial \Sigma)^{KN}\subset \partial\Sigma^{KN}$  to mean the boundary of $\Sigma^{KN}$, minus the extra boundary intervals from strip ends and extra boundary circles from marked points.

\begin{figure}[htbp]
  \centering
  \includegraphics[width=0.8\linewidth]{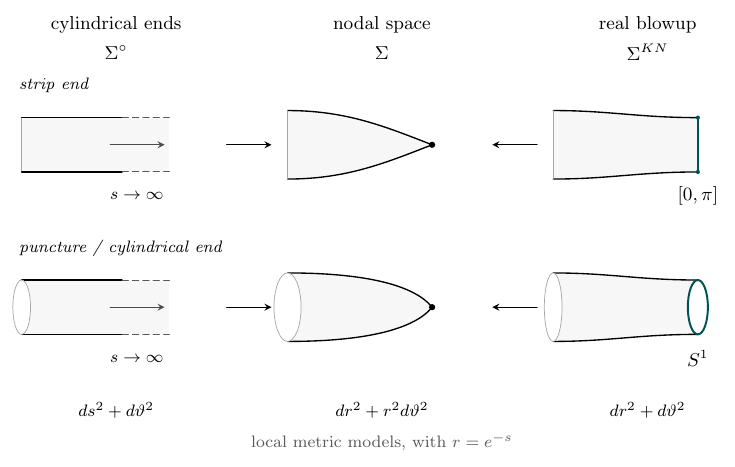}
  \caption{The three spaces near a boundary marked point or strip end (top) and an interior marked
    point or cylindrical end (bottom).
    The additional intervals at strip ends and circles at cylindrical ends
    are excluded from $(\partial\Sigma)^{KN}$. The additional circles count towards the boundary number $b(\Sigma)$, but do not count towards the number $b_{\mathbb R}(\Sigma)$ of real boundary components. }
  \label{fig:domain-topologies-ends}
\end{figure}

\begin{defn}[Euler characteristic and real boundary components]
\label{def:euler-characteristic}
Define the boundary number $b(\Sigma)$ of $\Sigma$ to be
\[b(\Sigma)=|\pi_0(\partial\Sigma^{KN})|\]
and, when $\Sigma$ has $n$ interior marked points, define  the number of real boundary components  of $\Sigma$ to be
\[b_{\mathbb R}(\Sigma)=b(\Sigma)-n.\]

     Define
\begin{equation}\label{eq:euler-characteristic}
\chi(\Sigma):=\frac12\chi\bigl((2\Sigma)^{KN}\bigr),
\end{equation}
where $(2\Sigma)^{KN}$ is the  real oriented blowup of the doubled domain. For a
smooth connected domain of genus $g$, with  $b_{\mathbb R}(\Sigma)$ real boundary components, $n$ interior marked points and $m$ boundary marked points, this is
\begin{equation}\label{eq:euler-characteristic-smooth}
\chi(\Sigma)=2-2g-b_{\mathbb R}(\Sigma) -n-\frac m2.
\end{equation}
\end{defn}

\subsubsection{Domain variations}\label{sec:conf}

Domain variations are governed by the natural $\dbar$ operator on vector fields tangent to the boundary:
\[
\Conf_\Sigma \colon \Omega^0(T^{\log}\Sigma, T^{\log}\partial\Sigma) \longrightarrow \Omega^{0,1}(T^{\log}\Sigma).
\]

Its kernel $\ker\Conf_\Sigma$ consists of holomorphic sections of the logarithmic tangent bundle tangent to $\partial \Sigma$, or ordinary holomorphic vector fields tangent to $\partial\Sigma$,  marked points and nodes, and with matched residues at nodes. We think of these matched residues as preserving infinitesimal gluing data, so $\ker\Conf_\Sigma$ represents infinitesimal automorphisms  preserving marked points and gluing data. Its cokernel $\coker\Conf_\Sigma$ parametrises variations of conformal structure and infinitesimal gluing parameters, or variations of complex structure on $\ex \Sigma$.

For example, if $\Sigma$ is a disc attached by an interior node to a stable curve, $\Sigma$ has no such infinitesimal automorphisms, whereas a disc with an interior marked point has infinitesimal automorphisms generating the rotations.

\begin{figure}[htbp]
\begin{tikzpicture}[
  x=0.82cm,y=0.82cm,
  line cap=round,line join=round,
  outline/.style={draw=black,line width=0.8pt},
  back rim/.style={draw=black,line width=0.6pt},
  surface/.style={fill=black!2}
]
  \coordinate (left) at (0.08,0.85);
  \coordinate (right) at (2.99,0.63);
  \coordinate (node) at (8.25,1.34);
  \coordinate (disc left) at (7.10,0.12);
  \coordinate (disc right) at (9.60,0.21);

  \path[surface,outline]
    (node)
    .. controls (7.64,1.17) and (7.34,0.79) .. (disc left)
    .. controls (7.60,-0.06) and (9.08,-0.06) .. (disc right)
    .. controls (9.46,0.66) and (8.74,1.19) .. (node)
    -- cycle;

  \path[surface,outline]
    (left)
    .. controls (0.62,3.30) and (3.50,6.55) .. (5.78,6.40)
    .. controls (6.80,6.39) and (8.00,6.07) .. (8.42,5.75)
    .. controls (9.27,5.20) and (10.30,3.76) .. (10.26,2.70)
    .. controls (10.24,1.65) and (9.35,1.34) .. (node)
    .. controls (7.18,1.35) and (7.17,1.98) .. (6.38,2.32)
    .. controls (5.60,2.69) and (4.54,2.59) .. (3.85,1.88)
    .. controls (3.30,1.33) and (3.08,0.84) .. (right)
    .. controls (2.46,-0.27) and (0.17,0.12) .. (left)
    -- cycle;

  \path[fill=white]
    (left) .. controls (0.69,1.53) and (2.64,0.90) .. (right)
    .. controls (2.46,-0.27) and (0.17,0.12) .. (left) -- cycle;
  \draw[back rim]
    (left) .. controls (0.69,1.53) and (2.64,0.90) .. (right);
  \draw[outline]
    (right) .. controls (2.46,-0.27) and (0.17,0.12) .. (left);

  \path[fill=white]
    (disc left) .. controls (7.90,0.58) and (9.02,0.59) .. (disc right)
    .. controls (9.08,-0.06) and (7.60,-0.06) .. (disc left) -- cycle;
  \draw[back rim]
    (disc left) .. controls (7.90,0.58) and (9.02,0.59) .. (disc right);
  \draw[outline]
    (disc left) .. controls (7.60,-0.06) and (9.08,-0.06) .. (disc right);

  \draw[outline]
    (4.78,4.73) .. controls (5.43,3.96) and (6.03,3.93) .. (7.00,4.67);
  \draw[outline]
    (5.41,4.25) .. controls (5.63,4.57) and (6.10,4.61) .. (6.34,4.42);

  \fill (node) circle[radius=1.25pt];
\end{tikzpicture}

\caption{The pictured bordered Riemann surface has no infinitesimal automorphisms that preserve infinitesimal gluing information at the node. It has Euler characteristic $-2$.}
\label{elliptic-fig}
\end{figure}
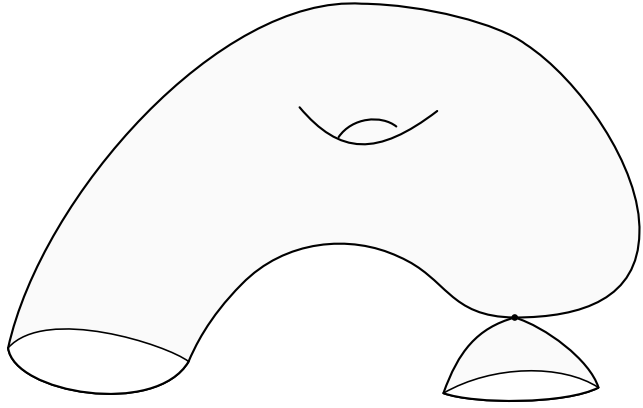

Here and below, for a vector bundle $E$, $\Omega^{0,1}(E)$ means $E$-valued ordinary $(0,1)$-forms on the normalisation of $\Sigma$, with nodes and marked points forgotten, and no implied matching condition at nodes.

With our definition of Euler characteristic, the index of $\Conf_{\Sigma}$ is
\[\ind\Conf_{\Sigma}=3\chi(\Sigma)+n+m/2\]
when $\Sigma$ has $n$ interior marked points  and $m$ boundary marked points.

\subsection{Bundle pairs and Cauchy--Riemann deformation complexes}
\label{subsec:intro-linear-setup}

A rank $k$ \emph{bundle pair} $(E, F)$ over $(\Sigma, \partial\Sigma)$ consists of:
\begin{enumerate}[(i)]
    \item A smooth complex vector bundle $E \to \Sigma$ of rank $k$;
    \item A totally real, smooth  subbundle $F \subset E|_{\partial\Sigma}$ of real rank $k$.
\end{enumerate}
At each node, the fibres of $E$ on the two branches are identified; at boundary nodes, this identification respects the real subbundle $F$. So,  sections of $(E, F)$ have matching values at nodes.

Apart from the rank $1$ bundle pair $(T^{\log}\Sigma, T^{\log}\partial \Sigma)$, we will generally be interested in rank $3$ bundle pairs. We will also usually restrict to the case that $F$ is oriented.

A \emph{real Cauchy--Riemann operator} on $(E, F)$ is a first-order differential operator
\[
D_\Sigma \colon \Omega^0(E, F) \longrightarrow \Omega^{0,1}(E)
\]
satisfying the Leibniz rule $D_\Sigma(f\xi) = (\dbar f)\otimes \xi + f D_\Sigma\xi$ for real-valued smooth functions $f$. In local trivialisations, $D_\Sigma = \dbar + A$, where $A \in \Omega^{0,1}(\rend_{\mathbb{R}} E)$ is a smooth ordinary (non-logarithmic) $(0,1)$-form valued in $\rend_{\mathbb R}E$.

\begin{assumption}[Doubleable Cauchy--Riemann operators]\label{ass:doubleable}
Except in Section~\ref{subsec:capping-and-generalisations}, all real Cauchy--Riemann operators considered in this paper are \emph{doubleable}: the bundle pair $(E,F)$ doubles to a smooth complex vector bundle $2E$ over $2\Sigma$ with a $\mathbb C$-antilinear lift of the anti-holomorphic involution $\sigma$, such that $2E\rvert_{\Sigma}=E$ and the fixed set of $\sigma$ over $\partial\Sigma$ is $F$, and $D_\Sigma$ is the restriction of a smooth real Cauchy--Riemann operator $D_{2\Sigma}$ on $2E$ that commutes with $\sigma$. Correspondingly, $\Omega^0(E,F)$ denotes the restrictions to $\Sigma$ of smooth $\sigma$-invariant sections of $2E$, and $\Omega^{0,1}(E)$ the restrictions of smooth $\sigma$-invariant $(0,1)$-forms, so that $D_\Sigma$ is $D_{2\Sigma}$ restricted to the $\sigma$-invariant parts of its domain and codomain.
\end{assumption}

The operator $\Conf_\Sigma$ is doubleable in this sense, with double $\Conf_{2\Sigma}$. Not all real Cauchy--Riemann operators are doubleable, but this distinction does not matter for orientation purposes, as the space of real Cauchy--Riemann operators on a real bundle pair is contractible, and contains doubleable Cauchy--Riemann operators as an affine subspace.  Section~\ref{subsec:capping-and-generalisations} explains how to orient more general Cauchy--Riemann operators.

A real Cauchy--Riemann operator $D_\Sigma$ is Fredholm in the sense that it is quasi-isomorphic to a finite-dimensional map.

Define the Maslov index $\mu(E,F)$ to be $c_1(2E)$, which is even when $F$ is oriented. With our definition of Euler characteristic,
\[\ind D_{\Sigma}= k(\chi(\Sigma)+n+m/2) +\mu(E,F)\]
when $\Sigma$ has $n$ interior marked points and $m$ boundary marked points, and $k$ is the rank of $(E,F)$.

\begin{remark}[Gluing Cauchy--Riemann operators]\label{rem:smoothing-family-regularity}
Given a family of curves $\Sigma_{\tau_k,\epsilon_j}$ smoothing the nodes of $(\Sigma,\partial \Sigma)$ as in Section~\ref{sec:gluing}, we can extend $(E,F)$ and $D_\Sigma$ to a smooth family of bundle pairs and doubleable Cauchy--Riemann operators. Since $E$ and the vertical logarithmic cotangent bundle $T^*_{\mathrm{vert}}$ are smooth vector bundles over the total space, smooth sections of $\Omega^{0,1}(E)$ over the family correspond to smooth sections of $(T^*_{\mathrm{vert}}\otimes E)^{0,1}$ vanishing at the nodes (with additional smoothness at nodes). This provides a natural notion of a smooth, finite-dimensional obstruction subbundle $V\subset \Omega^{0,1}(E)$ over the parameter space.

Applying \cite[Theorem 1.2]{reg} to the doubled family (equivariantly with respect to the involution $\sigma$), if $V_p$ is transverse to $D_{\Sigma_p}$ at a parameter $p$, then $V_q$ remains transverse to $D_{\Sigma_q}$ for all $q$ in a neighbourhood of $p$. Moreover, the preimages $K_q:=D^{-1}_{\Sigma_q}(V_q)$ assemble into a $\C{\infty}{1}$ vector bundle $K$, rendering the family of operators $D_{\Sigma_q}$ locally quasi-isomorphic to the family of finite-dimensional linear maps $K_q\to V_q$.

For this paper, we only require that $K_q\to V_q$ is a continuous family of finite-dimensional linear maps, ensuring that $\det D_{\Sigma_q}$ forms a continuous line bundle across the smoothing family—a fact that also follows from standard gluing analysis with significantly weaker assumptions on our Cauchy--Riemann operators.
\end{remark}

\subsubsection{The coupled deformation complex}

To couple domain variations to bundle variations, we fix a complex-linear bundle morphism
\[
\phi \colon T^{\log}\Sigma \longrightarrow E, \qquad \text{with } \phi(T^{\log}\partial\Sigma) \subset F,
\]
satisfying the compatibility condition that it induces a map of complexes:
\[
\begin{tikzcd} \Omega^0(T^{\log}\Sigma,T^{\log}\partial\Sigma)\dar{\Conf_\Sigma} \rar{\phi_0} &  \Omega^0(E,F)\dar{D_\Sigma}
\\ \Omega^{0,1}(T^{\log}\Sigma) \rar{\phi_1} & \Omega^{0,1}(E)
\end{tikzcd}
\]
The full deformation problem is represented by the mapping cone
\begin{equation}
\label{eq:full-def-cone}
D_{\Sigma, \phi} := \Cone\bigl(\Conf_\Sigma \xrightarrow{\;\phi\;} D_\Sigma\bigr),
\end{equation}
regarded as a three-term complex placed in degrees $-1, 0, 1$:
\[
\Omega^0(T^{\log}\Sigma, T^{\log}\partial\Sigma) \xrightarrow{\;(\phi_0, \Conf_\Sigma)\;} \Omega^0(E,F) \oplus \Omega^{0,1}(T^{\log}\Sigma) \xrightarrow{\;D_\Sigma - \phi_1\;} \Omega^{0,1}(E).
\]
When $(\phi, \Conf_\Sigma)$ is injective, this complex is quasi-isomorphic to a two-term complex. We refer to this as the stable case.

\begin{defn}[Coupled bundle pair]
\label{def:coupled-bundle-pair}
A \emph{coupled bundle pair} $(E,F,D,\phi)$ on $\Sigma$
consists of a rank $3$ bundle pair $(E,F)$ with $F$ oriented, a doubleable real
Cauchy--Riemann operator $D=D_\Sigma$ on $(E,F)$, and a complex-linear
coupling map $\phi\colon T^{\log}\Sigma\to E$ satisfying the compatibility condition above, with
$\phi(T^{\log}\partial\Sigma)\subset F$.
 Its \emph{deformation complex}
is the Fredholm complex
\[
D_{\Sigma,\phi}:=\Cone\bigl(\Conf_\Sigma\xrightarrow{\;\phi\;}D_\Sigma\bigr)
\]
in degrees $-1,0,1$, as in \eqref{eq:full-def-cone}.
\end{defn}

Given a continuous family of domains $\Sigma_t$, with a continuous family of coupled bundle pairs, the corresponding family of Fredholm complexes $D_{\Sigma_t,\phi_t}$ is quasi-isomorphic to a continuous family of finite-dimensional complexes. This remains true even if we allow nodes to form or be smoothed in the family, as in continuous base changes of the smoothing families from Section~\ref{sec:gluing}.

\subsection{Wall problems}
\label{subsec:wall-problems}

A wall problem abstracts the linearisation at a curve lying on a
codimension-one wall: a curve whose boundary crosses itself
(hyperbolic), or whose interior meets $L$ (elliptic).

\begin{defn}[Wall problems]\label{def:wall-problem}
A \emph{wall problem}  is a coupled bundle pair
$(E,F,D,\phi)$ on $\Sigma$, together with one of the following.
\begin{itemize}
\item[(Hyperbolic)]  Distinct (smooth, unmarked) points $p_1,p_2\in\partial\Sigma$
  and a complex-linear isomorphism $I\colon E_{p_1}\to E_{p_2}$ that
  restricts to an orientation-preserving isomorphism
  $F_{p_1}\to F_{p_2}$.  The wall problem is
  \emph{connected} if $p_1$ and $p_2$ lie on the same component of
  $\partial\Sigma^{KN}$, and \emph{disconnected} otherwise; compare
  Definition~\ref{def:hyperbolic-crossing-types}.
\item[(Elliptic)]  A (smooth, unmarked) point $q$ in the interior of $\Sigma$,
 and an oriented, $3$-dimensional totally real subspace
  $F_q\subset E_q$.
\end{itemize}
We call $\Sigma$ the \emph{normalised surface}, and write
$D_{\mathrm{norm}}:=D_{\Sigma,\phi}$.
\end{defn}

In the elliptic case, $E_q=F_q\oplus JF_q$ as $F_q$ is totally real; write $\operatorname{pr}_{JF_q}$ for the projection to $JF_q$ along $F_q$. As $E$ has rank $3$, crossings and incidences have
codimension one: $\dim F_{p_1}=3=2+1$ and $\dim JF_q=3=2+1$.

The \emph{incidence map} of a  wall problem  is
\begin{alignat}{3}
\label{eq:K-hyp-def}
&K\colon T_{p_1}\partial\Sigma\oplus T_{p_2}\partial\Sigma\to F_{p_1},
&\qquad& K(v_1,v_2):=\phi(v_1)-I^{-1}\phi(v_2)
&\qquad&\text{(hyperbolic)},\\
\label{eq:K-ell-def}
&K\colon T_q\Sigma\to JF_q,
&& K(v):=\operatorname{pr}_{JF_q}\phi(v)
&&\text{(elliptic)}.
\end{alignat}
The wall problem is \emph{transverse} if $K$ is injective. In that case
$\coker K$ is a real line. A \emph{positive normal} is a vector
$n_+$ such that:
\begin{itemize}
\item (hyperbolic) $(\phi(v_1),I^{-1}\phi(v_2),n_+)$ is an oriented
  basis of $F_{p_1}$, where $v_i\in T_{p_i}\partial\Sigma$ are positively
  oriented for the boundary orientation;
\item (elliptic) $n_+\in JF_q$, and $E_q=F_q\oplus\phi(T_q\Sigma)\oplus\mathbb R n_+$ is an
  oriented decomposition, where $E_q$ and $T_q\Sigma$ carry their
  complex orientations.
\end{itemize}
The class $[n_+]\in \coker K$ is nonzero, and the ray
$\mathbb R_{>0}[n_+]$ does not depend on these choices.

\begin{remark}[Chambers]\label{rem:abstract-chambers}
For a curve $v$ on a wall, a
deformation of $v$ moves the crossing or incidence by the image of its
variation in $\coker K$. In the hyperbolic case, the class of
$\xi(p_1)-I^{-1}\xi(p_2)$ is the displacement of strand~$1$ from
strand~$2$. In the elliptic case, the class of $\operatorname{pr}_{JF_q}\xi(q)$ is the
displacement of $v(q)$ off $L$. The side $\mathbb R_{>0}[n_+]$ is the
positive-crossing, respectively positive-incidence, chamber
$\mathcal M^+_{\mathrm{norm}}$ of Section~\ref{sec:chamber}, and its outward normal is $-[n_+]$.

In the hyperbolic case the ray of $[n_+]$ is unchanged if the labels of
$p_1$ and $p_2$ are exchanged and $I$ is replaced by $I^{-1}$. The two
resulting quotients $F_{p_1}$ and $F_{p_2}$ are identified by $-I$, and
$-I$ sends a positive normal to a positive normal for the relabelled
wall problem.
\end{remark}

\subsection{The wall complex and incidence exact triple}\label{subsec:incidence-triple}

The wall complex describes deformations of a curve on the wall that preserve the crossing or incidence; it is a subcomplex of $D_{\mathrm{norm}}$ with quotient $\Cone(K)$.

Let $\Conf_{\Sigma;p_1,p_2}$ be the restriction of  $\Conf_{\Sigma}$ to vector fields vanishing at
$p_1,p_2$, and let $\Conf_{\Sigma;q}$ be the restriction of $\Conf_{\Sigma}$ to vector fields vanishing at $q$.
These operators are quasi-isomorphic to the corresponding operators on the surfaces with marked points added. Let $\Sigma_{p_1,p_2}$ indicate $\Sigma$ with extra boundary marked points at $p_1$ and $p_2$, and $\Sigma_q$ indicate $\Sigma$ with an extra interior marked point at $q$. In coordinates $z$ centred on a marked point, a logarithmic vector field is in the form $a(z)z\partial_z$, so defines an ordinary vector field vanishing at that point, and there are canonical injective quasi-isomorphisms:
\begin{equation}\label{conf models}\Conf_{\Sigma_{p_1,p_2}}\xrightarrow{\cong} \Conf_{\Sigma;p_1,p_2},\quad \Conf_{\Sigma_{q}}\xrightarrow{\cong} \Conf_{\Sigma;q}\ .\end{equation}
 In each case the kernel of these operators consists of
infinitesimal automorphisms fixing the added points, and the cokernel
includes their motions. Let $D_{\Sigma;p_1\sim p_2}$ and
$D_{\Sigma;q\in F_q}$ be the restrictions of $D_\Sigma$ to the sections
\[
\{\xi\in\Omega^0(E,F):\xi(p_2)=I\xi(p_1)\}
\qquad\text{and}\qquad
\{\xi\in\Omega^0(E,F):\xi(q)\in F_q\},
\]
respectively. The coupling $\phi$ sends vector fields vanishing at the
added points to sections vanishing there, so it restricts to these
operators. The \emph{wall complex} is
\begin{equation}\label{eq:D-wall-def}
D_{\mathrm{wall}}:=
\begin{cases}
\Cone\bigl(\Conf_{\Sigma;p_1,p_2}\xrightarrow{\;\phi\;}
D_{\Sigma;p_1\sim p_2}\bigr) & \text{(hyperbolic)},\\[2pt]
\Cone\bigl(\Conf_{\Sigma;q}\xrightarrow{\;\phi\;}D_{\Sigma;q\in F_q}\bigr)
& \text{(elliptic)}.
\end{cases}
\end{equation}
In Section~\ref{subsec:nodal-smoothing}, we use the canonically quasi-isomorphic complex $D'_{\mathrm{wall}}$, defined using $\Conf_{\Sigma_{p_1,p_2}}$ and $\Conf_{\Sigma_q}$ in place of $\Conf_{\Sigma;p_1,p_2}$ and $\Conf_{\Sigma;q}$.

Evaluation at the added points gives a map of short exact sequences of
two-term complexes:
\[
\begin{tikzcd}[column sep=large]
0\arrow[r]&\Conf_{\Sigma;\bullet}\arrow[r]\arrow[d,"\phi"]
&\Conf_\Sigma\arrow[r,"\mathrm{ev}"]\arrow[d,"\phi"]
&(A\to0)\arrow[r]\arrow[d,"K"]&0\\
0\arrow[r]&D_{\Sigma;\bullet}\arrow[r]
&D_\Sigma\arrow[r,"\mathrm{ev}"]&(B\to0)\arrow[r]&0
\end{tikzcd}
\]
Here $K\colon A\to B$ is \eqref{eq:K-hyp-def} or \eqref{eq:K-ell-def}.
The evaluation maps are $v\mapsto(v(p_1),v(p_2))$ and
$\xi\mapsto\xi(p_1)-I^{-1}\xi(p_2)$ in the hyperbolic case, and
$v\mapsto v(q)$ and $\xi\mapsto\operatorname{pr}_{JF_q}\xi(q)$ in the elliptic case. Taking
cones of the vertical maps gives the \emph{incidence exact triple}
\begin{equation*}\tag{$\mathfrak T_{\mathrm{inc}}$}\label{eq:wall-triangle}
0\longrightarrow D_{\mathrm{wall}}\longrightarrow D_{\mathrm{norm}}
\longrightarrow\Cone(K)\longrightarrow0.
\end{equation*}
If the wall problem is transverse, the cohomology of $\Cone(K)$ is the line $\coker K$,
concentrated in degree $0$. Hence $\det\Cone(K)=\det \coker K$ and
$\ind D_{\mathrm{wall}}=\ind D_{\mathrm{norm}}-1$.

\subsection{The nodal problem and smoothing exact triple}
\label{subsec:nodal-smoothing}

The nodal problem models the curve with a boundary node at which the smoothing family terminates; its deformation complex is related to $D_{\mathrm{wall}}$ by the smoothing exact triple~\eqref{eq:smoothing-triangle}.

Corresponding to a wall problem, we define a coupled bundle pair $(E_{\mathrm{node}},\allowbreak F_{\mathrm{node}},\allowbreak D_{\Sigma_{\mathrm{node}}},\allowbreak\phi)$ on a nodal Riemann surface $\Sigma_{\mathrm{node}}$. The nodal domains are defined as follows:

\begin{itemize}
\item[(Hyperbolic)] Let $\Sigma_{\mathrm{node}}:=\Sigma/(p_1\sim p_2)$, with a hyperbolic
  boundary node at $p_1\sim p_2$.

\item[(Elliptic)] Let $\Sigma_{\mathrm{node}}$ be $\Sigma$ with an elliptic boundary node glued at $q$, so
\[
2\Sigma_{\mathrm{node}}
:=(2\Sigma)/(q\sim\sigma(q)).
\]

\end{itemize}
In each case, we can arbitrarily choose infinitesimal gluing data at the new node. The associated Fredholm problems are not affected by the choice of such gluing data.

The bundle pairs $(E_{\mathrm{node}},F_{\mathrm{node}})$ and Cauchy--Riemann operators $D_{\Sigma_{\mathrm{node}}}$ over $\Sigma_{\mathrm{node}}$ are defined using the quotient map $\pi: \Sigma\to\Sigma_{\mathrm{node}}$ as follows:
\begin{itemize}
\item[(Hyperbolic)] Using $I:E_{p_1}\to E_{p_2}$, the pair $(E,F)$ descends
  to a bundle pair $(E_{\mathrm{node}},F_{\mathrm{node}})$. So, we have
  $(E,F)=\pi^*(E_{\mathrm{node}},F_{\mathrm{node}})$. Pullback via $\pi$ defines a canonical isomorphism between $\Omega^0(E_{\mathrm{node}},F_{\mathrm{node}})$ and sections $\zeta$ of $(E,F)$ such that $\zeta(p_2)=I\zeta(p_1)$. Moreover, pullback via $\pi$ identifies $\Omega^{0,1}(E_{\mathrm{node}})$ with $\Omega^{0,1}(E)$, so canonically identifies the domains and codomains  of $D_{\Sigma_{\mathrm{node}}}$ and $D_{\Sigma; p_1\sim p_2}$. We then define
  \begin{equation}\label{eq:D Sigma hyp}D_{\Sigma_{\mathrm{node}}}=D_{\Sigma;p_1\sim p_2}\ .\end{equation}
\item[(Elliptic)] We define the bundle pair $(E_{\mathrm{node}},F_{\mathrm{node}})$ over $\Sigma_{\mathrm{node}}$ so that $\pi^*E_{\mathrm{node}}=E$, and $\pi^*(F_{\mathrm{node}})=F\cup F_q$. Pullback via $\pi$ defines a canonical isomorphism between $\Omega^0(E_{\mathrm{node}},F_{\mathrm{node}})$ and sections $\zeta$ of $(E,F)$ such that $\zeta(q)\in F_q$, and a canonical isomorphism between $\Omega^{0,1}(E_{\mathrm{node}})$ and $\Omega^{0,1}(E)$, so identifies the domains and codomains of $D_{\Sigma_{\mathrm{node}}}$ and $D_{\Sigma; q\in F_q}$. We then define
 \begin {equation}\label{eq:D Sigma node}D_{\Sigma_{\mathrm{node}}}=D_{\Sigma;q\in F_q}\ .\end{equation}

\end{itemize}

In both cases, a section $v$ of $T^{\log}\Sigma_{\mathrm{node}}$ pulls back to a unique logarithmic vector field $\pi^*v$ on $\Sigma$ that  vanishes at the added points and agrees with $v$ away from these added points, so $\phi(\pi^*v)$ satisfies the extra matching conditions and pushes forward to a  section of
$(E_{\mathrm{node}},F_{\mathrm{node}})$. This defines our coupled bundle pair  on
$\Sigma_{\mathrm{node}}$, with deformation complex
\[
D_{\mathrm{node}}:=
\Cone\bigl(\Conf_{\Sigma_{\mathrm{node}}}\xrightarrow{\;\phi\;}
D_{\Sigma_{\mathrm{node}}}\bigr).
\]
Since $\Conf_{\Sigma_{\mathrm{node}}}$ includes infinitesimal gluing data
at the node, $D_{\mathrm{node}}$ includes the smoothing direction. By
\eqref{eq:smoothing-triangle} below,
$\ind D_{\mathrm{node}}=\ind D_{\mathrm{wall}}+1=\ind D_{\mathrm{norm}}$.

\subsubsection{The smoothing triple}

\

The smoothing direction comes from varying the infinitesimal
gluing data. We first identify this variation in the conformal
deformation problem, then include the bundle variations.

\textbf{Hyperbolic case.}
The domain of $\Conf_{\Sigma_{\mathrm{node}}}$ is the kernel of
\[
\rho:=\operatorname{res}_{p_1}+\operatorname{res}_{p_2}
\]
on the domain of $\Conf_{\Sigma_{p_1,p_2}}$, and the two
operators have the same codomain. We therefore have the
following exact triple of operators:
\begin{equation}\label{hyp conf triple}
\begin{tikzcd}[column sep=small,row sep=large]
\Omega^0(T^{\log}\Sigma_{\mathrm{node}},
         T^{\log}\partial\Sigma_{\mathrm{node}})
  \arrow[r,hook]
  \arrow[d,"\Conf_{\Sigma_{\mathrm{node}}}"']
&
\Omega^0(T^{\log}\Sigma_{p_1,p_2},
         T^{\log}\partial\Sigma_{p_1,p_2})
  \arrow[r,"\rho"]
  \arrow[d,"\Conf_{\Sigma_{p_1,p_2}}"]
&
\mathbb R_{\mathrm{res}}
  \arrow[d]
\\
\Omega^{0,1}(T^{\log}\Sigma_{\mathrm{node}})
  \arrow[r,"\cong"]
&
\Omega^{0,1}(T^{\log}\Sigma_{p_1,p_2})
  \arrow[r]
&
0 .
\end{tikzcd}
\end{equation}

The connecting homomorphism sends
$r\in\mathbb R_{\mathrm{res}}$ to
\[
[\dbar v]\in\coker\Conf_{\Sigma_{\mathrm{node}}},
\qquad \rho(v)=r.
\]
We may choose $v$ supported near $p_1,p_2$ and holomorphic
near these points. Any two choices differ by a vector field
with matched residues, so this class is independent of the
choice. By Remark~\ref{rem:smoothing direction}, it represents
the logarithmic smoothing variation
$r\epsilon\partial_\epsilon$ in the plumbing family
$zw=\epsilon$.

Using the identification
$D_{\Sigma_{\mathrm{node}}}=D_{\Sigma;p_1\sim p_2}$,
the inclusion in \eqref{hyp conf triple} identifies
$D_{\mathrm{node}}$ as a subcomplex of
\[
D'_{\mathrm{wall}}
:=
\Cone\bigl(
\Conf_{\Sigma_{p_1,p_2}}
\xrightarrow{\;\phi\;}
D_{\Sigma;p_1\sim p_2}
\bigr).
\]
This complex is canonically quasi-isomorphic to
$D_{\mathrm{wall}}$ by \eqref{conf models}.
The resulting exact triple of cones is displayed below.
\[
\begin{tikzcd}[
  ampersand replacement=\&,
  column sep=1.1em,
  row sep=2.4em
]
{}
\& D_{\mathrm{node}}
\& D'_{\mathrm{wall}}
\& \mathbb R_{\mathrm{res}}
\\[-1em]
-1
\& \Omega^0(T^{\log}\Sigma_{\mathrm{node}},
             T^{\log}\partial\Sigma_{\mathrm{node}})
   \arrow[r,hook]
   \arrow[d,"d^{-1}"']
\& \Omega^0(T^{\log}\Sigma_{p_1,p_2},
             T^{\log}\partial\Sigma_{p_1,p_2})
   \arrow[r,"\rho"]
   \arrow[d,"d^{-1}"]
\& \mathbb R
   \arrow[d]
\\
0
\& \begin{gathered}
     \Omega^0(E_{\mathrm{node}},F_{\mathrm{node}})\\
     {}\oplus\Omega^{0,1}(T^{\log}\Sigma_{\mathrm{node}})
   \end{gathered}
   \arrow[r,equal]
   \arrow[d,"d^0"']
\& \begin{gathered}
     \Omega^0(E_{\mathrm{node}},F_{\mathrm{node}})\\
     {}\oplus\Omega^{0,1}(T^{\log}\Sigma_{\mathrm{node}})
   \end{gathered}
   \arrow[r]
   \arrow[d,"d^0"]
\& 0
   \arrow[d]
\\
1
\& \Omega^{0,1}(E_{\mathrm{node}})
   \arrow[r,equal]
\& \Omega^{0,1}(E_{\mathrm{node}})
   \arrow[r]
\& 0 .
\end{tikzcd}
\]
Here we have used the preceding identifications to write
the common terms on $\Sigma_{\mathrm{node}}$.
The vertical differentials are
\[
d^{-1}v=(\phi_0v,\dbar v),
\qquad
d^0(\xi,\kappa)
=D_{\Sigma_{\mathrm{node}}}\xi-\phi_1\kappa.
\]

\textbf{Elliptic case.}
Similarly, $\Conf_{\Sigma_{\mathrm{node}}}$ is a subcomplex
of $\Conf_{\Sigma_q}$ with the same codomain, and its domain
is the kernel of
\[
\rho:= 2\operatorname{Re}\operatorname{res}_q.
\]
Indeed, on the double, an invariant vector field with
residue $a$ at $q$ has residue $\bar a$ at $\sigma(q)$,
so the residue-matching condition is
$a+\bar a=0$.

We therefore have an exact triple of $\Conf$ operators
as in \eqref{hyp conf triple}, with $\Sigma_q$ replacing
$\Sigma_{p_1,p_2}$. Its connecting homomorphism sends
$r$ to $[\dbar v]$, where
$\operatorname{Re}\operatorname{res}_qv=r/2$.
The sum of the residues on the two branches of the double
is $r$, so this class represents the logarithmic smoothing
variation $r\epsilon\partial_\epsilon$ in the doubled
plumbing family $zw=\epsilon$.

Using
$D_{\Sigma_{\mathrm{node}}}=D_{\Sigma;q\in F_q}$,
we obtain the same exact triple of cones, now with
\[
D'_{\mathrm{wall}}
:=
\Cone\bigl(
\Conf_{\Sigma_q}
\xrightarrow{\;\phi\;}
D_{\Sigma;q\in F_q}
\bigr).
\]
Again, \eqref{conf models} identifies this complex
canonically up to quasi-isomorphism with
$D_{\mathrm{wall}}$. In the displayed cone triple,
only the degree $-1$ term of $D'_{\mathrm{wall}}$
and the definition of $\rho$ change.

In both cases, we write the resulting smoothing triple as
\begin{equation*}
\tag{$\mathfrak T_{\mathrm{sm}}$}
\label{eq:smoothing-triangle}
0\longrightarrow D_{\mathrm{node}}
\longrightarrow D'_{\mathrm{wall}}
\xrightarrow{\;\rho\;}
\mathbb R_{\mathrm{res}}
\longrightarrow0,
\end{equation*}
where $\mathbb R_{\mathrm{res}}$ is concentrated in degree
$-1$. Its determinant factor is therefore
$(\det\mathbb R_{\mathrm{res}})^{[-1]}$.

\begin{remark}[The smoothing direction]
\label{rem:residue-sign}
The connecting homomorphism of
\eqref{eq:smoothing-triangle} sends $r\in\mathbb R_{\mathrm{res}}$ to
\[
\delta(r)=[(\phi_0v,\dbar v)]
\in H^0(D_{\mathrm{node}}),
\qquad \rho(v)=r.
\]
The domain component is the smoothing class described
above. In particular,  Remark~\ref{rem:smoothing direction} implies that the domain component of $\delta(1)$ corresponds to the inward normal direction $\epsilon\partial_\epsilon$, so, to comply with an outward-normal orientation convention, we give $\mathbb R_{\mathrm{res}}$ the negative orientation.
\end{remark}

\section{Orientations}
\label{orientation-section}

In this section, we construct the holomorphic and anti-holomorphic orientations for coupled bundle pairs (Sections~\ref{subsec:determinant-summary}--\ref{subsec:antiholomorphic-orientations}), establish their compatibility with gluing along interior nodes (Lemma~\ref{lem:gluing-orientation}), and define the chamber and smoothing boundary orientations on the wall complex $D_{\mathrm{wall}}$. Using excision at separating interior nodes, we then show that the relative sign between these wall orientations reduces to a universal local model for each of the three wall types (Proposition~\ref{prop:wall-model-reduction}) and compute the elliptic comparison directly. Finally, in Section~\ref{subsec:capping-and-generalisations}, we explain how to apply this orientation formalism in more general geometric settings, where holomorphic curves can have strip ends asymptotic to Reeb chords.

Our approach to orientation is a more detailed version of the approach sketched in \cite{ekholm2019skeins},  building on the well-developed literature on orientations for Cauchy--Riemann operators with Lagrangian boundary conditions \cite[Chapter 8]{FOOO2009}, \cite{SolomonIntersectionTheory}, \cite{GeorgievaOrientabilityOpenGW}, \cite{GeorgievaZingerRealGWConstruction}, \cite{WehrheimWoodwardOrientationsQuilts}, \cite{EkholmEtnyreSullivanOrientations}.

\subsection{Determinant lines: notation and properties}
\label{subsec:determinant-summary}

We use graded determinant lines with the conventions recorded in
Appendix~\ref{app:determinant-lines}. Our conventions are a slight generalisation of Zinger's conventions~\cite{ZingerDetLine}, similar to \cite[Definition 2.20]{KnudsenDetFunctors}, allowing Fredholm complexes with more than two terms. The following notation and properties
suffice for the orientation arguments below.

For a finite-dimensional real vector space $V$, write
$\det V=\Lambda^{\dim V}V$, with index $\dim V$. Interchanging graded
lines of indices $p$ and $q$ introduces the Koszul sign $(-1)^{pq}$.
We denote the right and left duals of $\det V$ by  $\det V^{[1]}$ and $\det V^{[-1]}$ respectively, both with index $-\dim V$.

Their evaluations occur in the orders
$\det V\otimes\det V^{[1]}$ and
$\det V^{[-1]}\otimes\det V$, with the pairings
\[(e_1\wedge \dotsb \wedge e_n )\otimes (e_n^{[1]}\wedge\dotsb \wedge e_1^{[1]})\mapsto 1, \quad (e_n^{[-1]}\wedge\dotsb \wedge e_1^{[-1]})\otimes  (e_1\wedge \dotsb \wedge e_n )\mapsto 1.\]
 We keep these duals distinct. In this reversed duals convention the complex orientation of $\det E^{[\pm 1]}$ has the reversed order and is determined by
 \[ (Je_1)^{[\pm 1]}\wedge e_1^{[\pm 1]}\wedge (Je_2)^{[\pm 1]}\wedge e_2^{[\pm 1]}\wedge (Je_3)^{[\pm 1]}\wedge e_3^{[\pm 1]},\]
consistent with the complex dual $(ie_1)^*=-i(e_1^*)$.
A Fredholm map $D$ is placed in degrees $0$ and $1$, and has determinant
\begin{equation}\label{eq:determinant-summary}
\det D=\det\ker D\otimes\det(\coker D)^{[1]},
\end{equation}
of index $\ind D$. For a bounded Fredholm complex, the determinant is
the ordered graded determinant of its cohomology.

For a bounded complex of finite-dimensional vector spaces $(V_\bullet, d_\bullet)$, there is a canonical \emph{Euler isomorphism}
\[
\Phi_{V_\bullet} \colon \det d_\bullet \xrightarrow{\;\cong\;} \det V_\bullet := \bigotimes_k \det V_k^{[k]},
\]
identifying the determinant of cohomology $\det d_\bullet$ with the determinant of the graded vector space $\det V_\bullet$.  Its signs and compatibility with exact triples are detailed in Appendices~\ref{app:euler-isomorphism} and \ref{app:additivity}. We can use the Euler isomorphism to make determinants into continuous real line bundles in families; in the case of a family of infinite-dimensional Fredholm maps, we choose a continuous quasi-isomorphism with a family of finite-dimensional complexes; see Remark~\ref{rem:smoothing-family-regularity}.

For an exact triple $\mathfrak T$ of Fredholm complexes
$0\to D'_\bullet\to D_\bullet\to D''_\bullet\to0$, we use the additivity
isomorphism
\begin{equation}\label{eq:additivity-summary}
\Psi_{\mathfrak T}:\det D'_\bullet\otimes\det D''_\bullet
\longrightarrow\det D_\bullet.
\end{equation}
It agrees with the ordered wedge isomorphism when cohomology is
concentrated in degree $0$, is natural under quasi-isomorphisms of exact
triples, and varies continuously under homotopies. Thus the comparison
of two orientations constructed from such triples can be transported
along a homotopy to a model problem. Its construction and the signs
 are given in Appendix~\ref{app:additivity}.

\subsection{The elliptic boundary node exact triple}
\label{subsec:elliptic-orientation-triple}

A curve whose boundary consists entirely of elliptic nodes can be
compared with a closed curve by replacing each node with an interior
marked point. We use the resulting exact triple to orient the limiting
problem in our construction of the standard orientation.

Let $\Sigma$  have  elliptic boundary nodes
$q_1,\ldots,q_b$, with no other boundary components, and let
$\Sigma_c$ be its closed partial normalisation at these nodes, retaining these nodes as
interior marked points. Below, $q$ always denotes one of these nodes, and sums and products indexed by $q$ run over them. Let $(E,F,D,\phi)$ be a coupled bundle pair
on $\Sigma$. Write $E_c$ for the pullback of $E$ to $\Sigma_c$, and $D_{\Sigma_c}$ and $\phi_c$ for the operator
and coupling on $\Sigma_c$ obtained by removing the boundary conditions at the nodes. As $F_q$ is totally real, $E_q=F_q\oplus JF_q$.

In the real blowup $\Sigma^{KN}$ of Section~\ref{subsec:domain-topologies}, each node $q$ is a boundary circle $S^1_q$; see Figure~\ref{fig:Q-geometry}. If $z$ is a holomorphic coordinate on $\Sigma_c$ centred at $q$, then
\[
\nu_q:=-z\partial_z
\qquad\text{and}\qquad
\partial_\vartheta:=j\nu_q=-iz\partial_z
\]
are the outward normal and boundary orientation of $S^1_q$: $\nu_q$ points out of $\Sigma$ into the node, $\partial_\vartheta$ is clockwise around $q$ in $\Sigma_c$, and $T^{\log}_q\Sigma_c=\mathbb R(\nu_q)\oplus\mathbb R(\partial_\vartheta)$, with $T^{\log}_q\partial\Sigma=\mathbb R(\partial_\vartheta)$. We use the same symbol $\partial_\vartheta:=iz\partial_z$ for the anticlockwise rotation of the boundary of the unit disc and, more generally, for the vector field tangent to a smooth boundary component given by a holomorphic coordinate $z$ identifying a collar of that component with an annulus $\{r\leq|z|\leq1\}$. In every case $\partial_\vartheta=j\nu$ for the outward normal $\nu$, and $\partial_\vartheta$ orients the boundary.

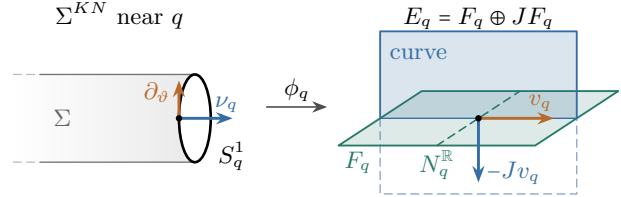
\begin{figure}[htbp]
\centering
\begingroup
\definecolor{qfigF}{HTML}{287A68}
\definecolor{qfigCurve}{HTML}{356DA4}
\definecolor{qfigTau}{HTML}{B66A2C}
\begin{tikzpicture}[
  line cap=round,line join=round,
  every node/.style={font=\small,inner sep=2pt},
  tip/.style={-{Stealth[length=2.2mm,width=1.6mm]}},
  tau/.style={draw=qfigTau,line width=1.1pt,tip},
  nu/.style={draw=qfigCurve,line width=1.1pt,tip},
  map/.style={draw=black!70,line width=.6pt,-{Stealth[length=2mm,width=1.5mm]}}
]
\begin{scope}
  \def\r{.58}\def\a{.21}\def\L{2.4}
  \shade[left color=white,right color=black!9]
    (-\L,-\r)--(0,-\r) arc[start angle=-90,end angle=90,x radius=\a,y radius=\r]
    --(-\L,\r)--cycle;
  \fill[white] (0,0) ellipse[x radius=\a,y radius=\r];
  \draw[black!55,line width=.5pt] (-\L+.35,\r)--(0,\r) (-\L+.35,-\r)--(0,-\r);
  \draw[black!25,line width=.5pt,densely dashed] (-\L,\r)--(-\L+.35,\r) (-\L,-\r)--(-\L+.35,-\r);
  \draw[line width=1pt] (0,0) ellipse[x radius=\a,y radius=\r];
  \node[text=black!60] at (-1.75,0) {$\Sigma$};
  \node[anchor=west] at (\a+.02,-.5) {$S^1_q$};
  \draw[tau] (-\a,0) -- (-\a,.5);
  \draw[nu]  (-\a,0) -- (.5,0);
  \fill (-\a,0) circle[radius=1.2pt];
  \node[text=qfigTau,anchor=east] at (-\a-.04,.34) {$\partial_\vartheta$};
  \node[text=qfigCurve,anchor=south] at (.42,0) {$\nu_q$};
  \node[anchor=base] at (-1.0,1.25) {$\Sigma^{KN}$ near $q$};
\end{scope}
\begin{scope}[xshift=3.75cm,x={(1cm,0cm)},y={(.55cm,.36cm)},z={(0cm,1cm)}]
  \draw[qfigCurve!60,line width=.5pt,densely dashed]
    (-1.3,0,0)--(-1.3,0,-1.0)--(1.3,0,-1.0)--(1.3,0,0);
  \fill[qfigF,fill opacity=.16] (-1.3,-1,0)--(1.3,-1,0)--(1.3,1,0)--(-1.3,1,0)--cycle;
  \draw[qfigF,line width=.7pt] (-1.3,-1,0)--(1.3,-1,0)--(1.3,1,0)--(-1.3,1,0)--cycle;
  \draw[qfigF,line width=.6pt,densely dashed] (0,-1,0)--(0,1,0);
  \fill[qfigCurve,fill opacity=.18] (-1.3,0,0)--(1.3,0,0)--(1.3,0,1.15)--(-1.3,0,1.15)--cycle;
  \draw[qfigCurve,line width=.7pt] (-1.3,0,0)--(-1.3,0,1.15)--(1.3,0,1.15)--(1.3,0,0);
  \draw[qfigCurve,line width=.5pt] (-1.3,0,0)--(1.3,0,0);
  \draw[tau] (0,0,0) -- (1.0,0,0);
  \draw[nu]  (0,0,0) -- (0,0,-.85);
  \fill (0,0,0) circle[radius=1.2pt];
  \node[text=qfigTau,anchor=south] at (.82,0,.02) {$v_q$};
  \node[text=qfigCurve,anchor=west] at (.03,0,-.72) {$-Jv_q$};
  \node[text=qfigCurve,anchor=west] at (-1.25,0,.85) {curve};
  \node[text=qfigF,anchor=north west] at (-1.3,-1,0) {$F_q$};
  \node[text=qfigF,anchor=north,inner sep=1.5pt] at (0,-1,0) {$N^{\mathbb R}_q$};
  \node[anchor=base] at (0,0,1.25) {$E_q=F_q\oplus JF_q$};
\end{scope}
\draw[map] (.95,.15) -- node[above]{$\phi_q$} (1.75,.15);
\end{tikzpicture}
\endgroup
\caption{The geometry of $Q_{q,\phi}$ in the stable case. Left: at the node $q$, $\Sigma^{KN}$ has a boundary circle $S^1_q$, with outward normal $\nu_q$ and boundary orientation $\partial_\vartheta$. Right: $\phi_q$ sends $\partial_\vartheta$ to the velocity $v_q\in F_q$ of $S^1_q$ and $\nu_q$ to $-Jv_q\in JF_q$, so $H^0(Q_{q,\phi})=JF_q/\mathbb R(Jv_q)$ is $J$ applied to the directions $N^{\mathbb R}_q=F_q/\mathbb R(v_q)$ in $F_q$ transverse to $v_q$.}
\label{fig:Q-geometry}
\end{figure}

A logarithmic vector field $v$ on $\Sigma_c$ descends to a vector field on $\Sigma$ tangent to the boundary precisely when each $v(q)\in\mathbb R(\partial_\vartheta)$, and a section $\xi$ of $E_c$ descends to a section of $(E,F)$ precisely when each $\xi(q)\in F_q$. Write
\[
\rho\colon\Omega^0(T^{\log}\Sigma_c)\longrightarrow\bigoplus_q\mathbb R(\nu_q),
\qquad
\operatorname{ev}\colon\Omega^0(E_c)\longrightarrow\bigoplus_q JF_q
\]
for the $\mathbb R(\nu_q)$ and $JF_q$ components of $v(q)$ and $\xi(q)$; so $\rho(v)=-(\operatorname{Re}\operatorname{res}_qv)\,\nu_q$. Then $\Conf_\Sigma$ and $D_\Sigma$ are the subcomplexes of $\Conf_{\Sigma_c}$ and $D_{\Sigma_c}$ with quotients $\bigoplus_q\mathbb R(\nu_q)$ and $\bigoplus_q JF_q$, concentrated in the degree of the domain. As $\phi(\partial_\vartheta)\in F_q$, we have $\phi(\nu_q)=-J\phi(\partial_\vartheta)\in JF_q$, so $\operatorname{ev}\circ\phi_c=\phi\circ\rho$, and taking cones gives the exact triple
\begin{equation*} \tag{$\mathfrak T_{\mathrm{ell}}$}\label{eq:elliptic-orientation-triple}
0\longrightarrow D_{\Sigma,\phi}
\longrightarrow D_{\Sigma_c,\phi_c}
\longrightarrow \bigoplus_q Q_{q,\phi}
\longrightarrow0,
\qquad
Q_{q,\phi}:=\bigl(\mathbb R(\nu_q)\xrightarrow{\;\phi\;}JF_q\bigr),
\end{equation*}
with $Q_{q,\phi}$ in degrees $-1$ and $0$. Explicitly,
\[
\begin{tikzcd}[column sep=1.1em,row sep=2.4em]
-1\colon&
\Omega^0(T^{\log}\Sigma,T^{\log}\partial\Sigma)
  \arrow[r]\arrow[d,"d^{-1}"']
&
\Omega^0(T^{\log}\Sigma_c)
  \arrow[r,"\rho"]\arrow[d,"d^{-1}"]
&
\displaystyle\bigoplus_q\mathbb R(\nu_q)
  \arrow[d,"\phi"]
\\
0\colon&
\begin{gathered}
\Omega^0(E,F)\\
{}\oplus\Omega^{0,1}(T^{\log}\Sigma)
\end{gathered}
  \arrow[r]\arrow[d,"d^0"']
&
\begin{gathered}
\Omega^0(E_c)\\
{}\oplus\Omega^{0,1}(T^{\log}\Sigma_c)
\end{gathered}
  \arrow[r,"\operatorname{ev}"]\arrow[d,"d^0"]
&
\displaystyle\bigoplus_q JF_q
  \arrow[d]
\\
1\colon&
\Omega^{0,1}(E)
  \arrow[r]
&
\Omega^{0,1}(E_c)
  \arrow[r]
&
0,
\end{tikzcd}
\]
where $d^{-1}v=(\phi_0v,\Conf v)$ and $d^0(\xi,\kappa)=D\xi-\phi_1\kappa$ are the cone differentials, and the degree-$0$ quotient map is $(\xi,\kappa)\mapsto\operatorname{ev}(\xi)$.

\begin{defn}[Orientation of $\det Q_{q,\phi}$] \label{def:Q orientation}
Orient $\mathbb R(\nu_q)$ by $\nu_q$, and $JF_q$ by applying $J$ to the given orientation of $F_q$. As $E$ has rank $3$, the latter is the quotient orientation of $JF_q\cong E_q/F_q$ determined by $\det(E_q/F_q)\otimes\det F_q\to\det E_q$. The Euler isomorphism
\[
\det Q_{q,\phi}
\xrightarrow{\;\cong\;}
\det\mathbb R(\nu_q)^{[-1]}
\otimes\det JF_q
\]
then orients $\det Q_{q,\phi}$, which has index $3-1=2$.
\end{defn}

\begin{remark}[The geometry of $Q_{q,\phi}$]\label{rem:Q-geometry}
 The map \(\phi:\mathbb R(\nu_q)\to JF_q\) is the restriction of the coupling to the outward logarithmic normal line.  As $\phi_q\colon T^{\log}_q\Sigma_c\to E_q$ is complex-linear, either $\phi_q=0$ or $\phi_q$ is injective.
\begin{itemize}
\item \emph{The stable case.} If $\phi_q$ is injective, then $v_q:=\phi(\partial_\vartheta)\in F_q$ is nonzero.  This case occurs at the new nodes of Definition~\ref{def:holomorphic-orientation} when $\phi(\partial_\vartheta)$ is nowhere zero on $\partial\Sigma$, for example, for the standard punctured disc of Lemma~\ref{lem:standard-punctured-disc} with a constant nonzero coupling map, which models the half-cylinders of Section~\ref{half-cylinder-section}. The vector  $v_q$ is the velocity of $S^1_q$, and the curve leaves $F_q$ in the direction $Jv_q=-\phi(\nu_q)$. So $Q_{q,\phi}$ is quasi-isomorphic to
\[
H^0(Q_{q,\phi})=JF_q/\mathbb R(\nu_q)=J F_q/\mathbb R(-Jv_q),
\]
the $2$-dimensional space of directions normal to both $F_q$ and the curve. It is given the standard quotient orientation, so that $\phi(\nu_q)$ followed by an oriented basis of $H^0(Q_{q,\phi})$ is an oriented basis of $JF_q$.

\item \emph{The unstable case.} If $\phi_q=0$, $H^{-1}(Q_{q,0})=\mathbb R(\nu_q)$ and $H^0(Q_{q,0})=JF_q$. This case occurs at the new node of the nodal problem of an elliptic wall problem (Section~\ref{subsec:nodal-smoothing}), or if we homotope the coupling map to $0$.
\end{itemize}
\end{remark}

The additivity isomorphism for
\eqref{eq:elliptic-orientation-triple} is
\begin{equation}\label{eq:en comparison}\Psi_{\mathfrak T_{\mathrm{ell}}}:
\det D_{\Sigma,\phi}\otimes
\bigotimes_q\det Q_{q,\phi}
\xrightarrow{\;\cong\;}
\det D_{\Sigma_c,\phi_c}.
\end{equation}
An orientation of the closed-curve problem thus determines an
orientation of $D_{\Sigma,\phi}$. The result is independent of
the ordering of the elliptic nodes, because each node factor
has even index.

In particular, when $D_{\Sigma_c}$ is complex-linear, the cone
$D_{\Sigma_c,\phi_c}$ has its complex orientation, and
\eqref{eq:en comparison} orients $D_{\Sigma,\phi}$. As all of these orientations vary continuously along the homotopy $s\mapsto s\phi$, the result is the transport of the orientation obtained from the model with $\phi=0$, in which every $Q_{q,0}$ has zero differential. We will obtain the standard orientation for a general
bordered curve by homotoping its deformation problem to such
a model.

\begin{remark}[The constant-disc model]
\label{rem:disc-refinement-orientation}
The orientation construction above can also be described using
the constant-disc model of Remark~\ref{disc model}. Replace each
elliptic boundary node $q$ of $\Sigma$ by a disc attached at an
interior node, and write $\pi\colon\Sigma'\to\Sigma$ for the projection contracting the added discs. On the associated exploded curves, this is a refinement, so $T^{\log}\Sigma'\cong\pi^*T^{\log}\Sigma$, and pulling back the coupled bundle pair gives a natural quasi-isomorphism
\[
\pi^*\colon D_{\Sigma,\phi}\longrightarrow D_{\Sigma',\pi^*\phi}.
\]
Indeed, on the disc attached at $q$, the pulled-back data are the constant pair $(E_q,F_q)$, the operator $\dbar$ and the constant coupling $\phi_q$. Pullback extends sections and vector fields constantly over the disc (a vector field with residue $a$ at $q$ extends as $-a\,w\partial_w$, as required by residue matching), and extends target forms by zero; the quotient complexes are acyclic because the disc operators are surjective with kernels the constants.

Each added disc is a copy of the standard punctured disc of Lemma~\ref{lem:standard-punctured-disc}, whose boundary rotation corresponds to $\partial_\vartheta\in T^{\log}_q\Sigma_c$, so it contributes the complex $C^{\mathbb R}_q:=\bigl(\mathbb R(\partial_\vartheta)\xrightarrow{\;\phi\;}F_q\bigr)$, oriented by $\partial_\vartheta$ and the orientation of $F_q$. Gluing it to $\Sigma_c$ at $q$ contributes the complex-oriented gluing complex $C_q:=\bigl(T^{\log}_q\Sigma_c\xrightarrow{\;\phi\;}E_q\bigr)$ of Lemma~\ref{lem:gluing-orientation}. As $T^{\log}_q\Sigma_c=\mathbb R(\partial_\vartheta)\oplus\mathbb R(\nu_q)$ and $E_q=F_q\oplus JF_q$, we have $C_q=C^{\mathbb R}_q\oplus Q_{q,\phi}$, and the proof of Lemma~\ref{lem:standard-punctured-disc} shows that this splitting preserves orientations. So, by Lemma~\ref{lem:gluing-orientation}, the orientation defined by $\mathfrak T_{\mathrm{ell}}$ agrees under $\pi^*$ with the orientation obtained by gluing these disc factors to the complex-oriented closed problem; as each disc factor has index $2$, their ordering does not matter. Consequently, the homotopy used to define the standard orientation may be carried out in either model.
\end{remark}

\subsection{Orientations of coupled bundle pairs}
\label{subsec: bundle pair orientation}

In this section, we define orientations for coupled bundle pairs in the case without strip ends. So, $(\Sigma,\partial\Sigma)$ will be a bordered Riemann surface with no boundary marked points.

In the geometric situation of a holomorphic curve $u:(\Sigma,\partial \Sigma)\to (X,L)$ with $(E,F)=u^*(TX,TL)$,  a Spin structure on $L$ distinguishes a homotopy class of trivialisation of $u^*TL\rvert_{\partial \Sigma}$. Such a trivialisation of $F$ can be used to orient $\det D_{\Sigma}$.

\begin{defn}[Orientation of bundle pairs]\label{def:bundle-pair-orientation}
 Let  $(E,F)$ be a bundle pair  over a bordered Riemann surface $(\Sigma,\partial \Sigma)$ with no boundary marked points. Define the orientation module $\mathcal O_{(E,F)}$ of $(E,F)$ to be the rank-one $\mathbb Z$-module generated by oriented trivialisations $\tau:F\cong \partial\Sigma\times \mathbb R^3$, under the relation that if $\tau$ and $\tau'$ are related by $g\colon \partial\Sigma \to GL^+(3,\mathbb R)$,
\[[\tau]=(-1)^{\langle g^*\eta,[\partial\Sigma]\rangle}[\tau'],\]
where $\eta$ is the generator of
$H^1(GL^+(3,\mathbb R);\mathbb Z/2)\cong\mathbb Z/2$.
An orientation of the bundle pair $(E,F)$ is a generator $\mathfrak o$ of $\mathcal O_{(E,F)}$.
\end{defn}

The following is a version of the standard approach to orienting Cauchy--Riemann operators with real boundary conditions; see \cite[Chapter 8]{FOOO2009}, \cite[Proposition 2.8]{SolomonIntersectionTheory}, \cite[Lemma 6.37]{LiuOpenGW}. In outline: after smoothing any hyperbolic boundary nodes, we put the data in standard form near $\partial\Sigma$, stretch the neck so that each boundary component becomes an elliptic boundary node (Figure~\ref{fig:orientation-neck-stretching}), orient the closed normalisation using its complex-linear part, and orient the nodal problem from this using the elliptic boundary node triple~\eqref{eq:elliptic-orientation-triple} and Definition~\ref{def:Q orientation}.

\begin{defn}[The holomorphic orientation]
\label{def:holomorphic-orientation}
Let $(E,F,D,\phi)$ be a coupled bundle pair on $\Sigma$ with
orientation $\mathfrak o$. The \emph{holomorphic orientation}
$\mathsf o_{\mathfrak o}(D_{\Sigma,\phi})$ of $\det D_{\Sigma,\phi}$ is the
orientation constructed as follows.

In the case that $\Sigma$ has no hyperbolic boundary nodes:
\begin{enumerate}[(i)]
\item If $\phi(\partial_\vartheta)$ vanishes somewhere on $\partial\Sigma$, reduce to the case $\phi=0$ by transport along the linear homotopy $s\mapsto D_{\Sigma,s\phi}$, $s\in[0,1]$. In either case, choose a collar of $\partial\Sigma$ with a holomorphic coordinate $z$ as in Construction~\ref{stretch}, so that $\partial_\vartheta=iz\partial_z$ there, and a trivialisation $\tau$ of $F$ with $[\tau]=\mathfrak o$, extended to a trivialisation of $E$ over the collar in which $\phi(\partial_\vartheta)$ is constant.\footnote{If $\phi(\partial_\vartheta)$ is nowhere zero, shrink the collar so that $\phi(\partial_\vartheta)$ is nowhere zero on it. Complete $\phi(\partial_\vartheta)$ to an oriented frame of $F$, rotating the last two vectors once around $\phi(\partial_\vartheta)$ along one boundary component as necessary to match $\mathfrak o$. Then extend this frame to a complex frame of $E$ over the collar whose first vector is $\phi(\partial_\vartheta)$, and which doubles to a smooth $\sigma$-invariant frame of $2E$.}
\item Homotope $D_\Sigma$ over the collar, through doubleable operators, to the standard operator $\dbar$, keeping $\phi$ fixed. Then stretch the neck along each smooth boundary circle of  $\Sigma$ as in Construction~\ref{stretch}, keeping the data in standard form and the coupling constant; see Figure~\ref{fig:orientation-neck-stretching}. The limit is a nodal surface $\Sigma'$ with an elliptic node $q$ for each component of $\partial\Sigma$. Let $\Sigma_c$ be its closed normalisation, which replaces each elliptic node with an interior marked point. If $\phi\neq0$, each such $q$ is stable in the sense of Remark~\ref{rem:Q-geometry}, and $v_q$ is the constant value of $\phi(\partial_\vartheta)$.
\item Orient $\det D_{\Sigma_c,\phi_c}$ on $\Sigma_c$ using the linear homotopy of $D_{\Sigma_c}$ to the complex-linear map $D^{(1,0)}_{\Sigma_c}:=\frac 12(D_{\Sigma_c}-JD_{\Sigma_c}J)$. That is, transport the complex orientation of $\Cone(\Conf_{\Sigma_c}\xrightarrow{\phi_c}D^{(1,0)}_{\Sigma_c})$ back along this homotopy to define the complex orientation of $\det D_{\Sigma_c,\phi_c}$.
\item Orient $\det D_{\Sigma',\phi}$ by the additivity isomorphism~\eqref{eq:en comparison} of the exact triple~\eqref{eq:elliptic-orientation-triple}, using the above complex orientation of $\det D_{\Sigma_c,\phi_c}$ and Definition~\ref{def:Q orientation}. Transport this
  orientation back along the homotopies.
\end{enumerate}

In the case that $\Sigma$ has hyperbolic boundary nodes, first homotope the coupling map $\phi$ to $0$, then smooth these boundary nodes in a family $\Sigma_\epsilon$ as in Section~\ref{sec:gluing}, extending $(E,F)$ and $D_\Sigma$ to a smooth family of doubleable operators over this family of domains. Gluing analysis~\cite[Theorem 1.2]{reg} then implies that the resulting family of Cauchy--Riemann operators $D_{\Sigma_\epsilon}$ is quasi-isomorphic to a continuous family of finite-dimensional maps, so we can transport orientations of $\det D_{\Sigma_{\epsilon}}$, $\det\Conf_{\Sigma_{\epsilon}}$ and $\det D_{\Sigma_{\epsilon},0}$ back along this homotopy.  Any trivialisation of $F$ over boundary components of $\Sigma$ extends to a trivialisation over $\partial\Sigma_{\epsilon}$. The holomorphic orientation of $\det D_{\Sigma,\phi}$ is the transport of the holomorphic orientation of $\det D_{\Sigma_{\epsilon},0}$ back along this homotopy.
\end{defn}

\begin{remark}[Constant couplings]\label{rem:constant-coupling}
The homotopies in Definition~\ref{def:holomorphic-orientation} preserve the compatibility $D_\Sigma\circ\phi_0=\phi_1\circ\Conf_\Sigma$ of a coupled bundle pair. On the collar, write $D_\Sigma=\dbar+A$ in our trivialisation, with $\phi(\partial_\vartheta)=v$ constant. As $\partial_\vartheta$ is holomorphic, compatibility says exactly that $A(fv)=0$ for every complex-valued function $f$, so it is preserved by the homotopy $\dbar+(1-t\beta)A$, $t\in[0,1]$, where $\beta$ is a cutoff function supported in the collar and equal to $1$ near $\partial\Sigma$. As our trivialisation doubles smoothly, $\dbar$ and $A$ are doubleable, and so is this homotopy. On $\Sigma_c$, applying compatibility to $iw$ for vector fields $w$ shows that $D_{\Sigma_c}$ is complex-linear on the image of $\phi_0$, so it agrees there with $D^{(1,0)}_{\Sigma_c}$, and the linear homotopy in (iii) preserves compatibility.

If $\phi(\partial_\vartheta)$ is nowhere zero, Definition~\ref{def:holomorphic-orientation} may also be applied after homotoping $\phi$ to $0$, with the same result. Indeed, steps (ii)--(iv) can be run for the couplings $s\phi$, $s\in[0,1]$, with the same trivialisation and homotopies; as noted after~\eqref{eq:en comparison}, the orientations in (iv) vary continuously in $s$, so the resulting orientations of $\det D_{\Sigma,s\phi}$ do too. At $s=0$, this is the construction after homotoping $\phi$ to $0$.
\end{remark}

  \begin{figure}[htbp]
  \centering

  \begingroup
\colorlet{neckcollar}{teal!65!black}
\begin{tikzpicture}[x=1cm,y=1cm,line cap=round,line join=round,
  line width=.65pt,every node/.style={font=\small,inner sep=2pt}]

\def\neckhandle#1#2{%
  \begin{scope}[shift={(#1,#2)}]
    \fill[white]
      (-.256,-.1005) .. controls (-.11,.065) and (.11,.065) .. (.256,-.1005)
      .. controls (.10,-.155) and (-.10,-.155) .. cycle;
    \draw (-.60,.12) .. controls (-.22,-.23) and (.22,-.23) .. (.60,.12);
    \draw (-.256,-.1005) .. controls (-.11,.065) and (.11,.065) .. (.256,-.1005);
  \end{scope}%
}

\def\neckcollarband#1#2{%
  \begin{scope}[shift={(#1,#2)}]
    \def\bandpath{(0,.45) -- (.48,.45)
      arc[start angle=90,end angle=-90,x radius=.14cm,y radius=.45cm]
      -- (0,-.45) -- cycle}
    \fill[neckcollar!13] \bandpath;
    \draw[neckcollar,line width=.6pt]
      (.48,.45) arc[start angle=90,end angle=-90,x radius=.14cm,y radius=.45cm];
    \filldraw[fill=white,line width=.75pt]
      (0,0) ellipse[x radius=.22cm,y radius=.45cm];
  \end{scope}%
}

\begin{scope}
  \def\surfaceoutline{
    (.40,3.10) -- (1.00,3.10)
    .. controls (1.66,3.09) and (1.86,3.45) .. (2.61,3.45)
    .. controls (3.71,3.45) and (4.49,2.66) .. (4.49,1.69)
    .. controls (4.49,.71) and (3.91,.10) .. (3.05,.10)
    .. controls (2.18,.10) and (1.89,.20) .. (1.18,.20)
    -- (.40,.20) -- (.40,1.10) -- (.98,1.10)
    .. controls (1.98,1.10) and (1.94,2.20) .. (1.00,2.20)
    -- (.40,2.20) -- cycle}
  \fill[black!2] \surfaceoutline;
  \neckcollarband{.40}{2.65}
  \neckcollarband{.40}{.65}
  \draw (.40,3.10) -- (1.00,3.10)
    .. controls (1.66,3.09) and (1.86,3.45) .. (2.61,3.45)
    .. controls (3.71,3.45) and (4.49,2.66) .. (4.49,1.69)
    .. controls (4.49,.71) and (3.91,.10) .. (3.05,.10)
    .. controls (2.18,.10) and (1.89,.20) .. (1.18,.20)
    -- (.40,.20);
  \draw (.40,1.10) -- (.98,1.10)
    .. controls (1.98,1.10) and (1.94,2.20) .. (1.00,2.20)
    -- (.40,2.20);
  \neckhandle{2.88}{1.77}
  \node at (2.31,3.91) {$\Sigma$};
\end{scope}

\draw[->,>=stealth,line width=.6pt] (4.91,1.72) -- (5.67,1.72);

\begin{scope}[shift={(6.00,0)}]
  \def\ellipticoutline{
    (.45,2.65) -- (1.45,3.10)
    .. controls (1.90,3.30) and (2.27,3.50) .. (2.77,3.50)
    .. controls (3.73,3.50) and (4.14,2.72) .. (4.14,1.80)
    .. controls (4.14,.76) and (3.53,.16) .. (2.64,.16)
    .. controls (2.00,.16) and (1.77,.20) .. (1.45,.20)
    -- (.45,.65) -- (1.45,1.10)
    .. controls (2.35,1.10) and (2.35,2.20) .. (1.45,2.20)
    -- cycle}
  \fill[black!2] \ellipticoutline;
  \foreach \cy in {2.65,.65}{%
    \begin{scope}[shift={(0,\cy)}]
      \def\conecollar{
        (.45,0) -- (1.30,.3825)
        arc[start angle=90,end angle=-90,x radius=.12cm,y radius=.3825cm]
        -- cycle}
      \fill[neckcollar!13] \conecollar;
      \draw[neckcollar,line width=.6pt]
        (1.30,.3825)
        arc[start angle=90,end angle=-90,x radius=.12cm,y radius=.3825cm];
    \end{scope}
  }
  \draw \ellipticoutline;
  \fill (.45,2.65) circle[radius=1.15pt];
  \fill (.45,.65) circle[radius=1.15pt];
  \node[anchor=south east] at (.41,2.78) {$q_1$};
  \node[anchor=north east] at (.41,.52) {$q_2$};
  \neckhandle{2.96}{1.99}
  \node at (2.59,3.96) {$\Sigma'$};
\end{scope}
\end{tikzpicture}
\endgroup

  \caption{The neck-stretching degeneration used to construct the orientation,
    illustrated with two boundary components.
    We stretch the neck along each boundary component, keeping the data in standard form on the shaded collars.
    In the nodal limit $\Sigma'$, elliptic boundary nodes are drawn as cone points. An equivalent model replaces these elliptic boundary nodes with constant discs, attached at internal nodes to $\Sigma_c$, seen by stretching the neck along the inner boundary of the collar neighbourhood.
   }
  \label{fig:orientation-neck-stretching}
\end{figure}
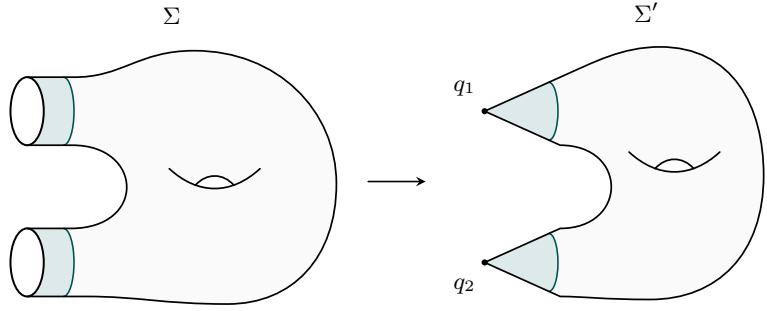

The above construction, after homotoping $\phi$ to $0$ if necessary (Remark~\ref{rem:constant-coupling}), can be applied continuously to a family of coupled bundle pairs over a family of bordered Riemann surfaces with a trivialisation of $F$ over the boundary. The resulting orientation is therefore continuous in families.

\begin{lemma}[Dependence on the boundary trivialisation]
\label{lem:orientation-trivialisation}
Let $\tau$ and $\tau'=g\tau$ be oriented trivialisations of $F$.
The orientations constructed using $\tau$ and $\tau'$ differ
by the sign
\[
(-1)^{\langle g^*\eta,[\partial\Sigma]\rangle}.
\]
So, the construction respects the defining relations of
$\mathcal O_{(E,F)}$.
\end{lemma}

\begin{proof}
Below, we construct a family of bundle pairs parametrised by a circle, whose associated determinant line is orientable if and only if $\tau$ and $\tau'$ induce the same orientation. We then  apply \cite[Proposition 3.1]{GeorgievaOrientabilityOpenGW}, which determines whether the determinant line of a family of Cauchy--Riemann operators is orientable.

First, homotope $\phi$ to zero and smooth any  nodes, reducing to the case of a smooth bordered Riemann surface $(\Sigma,\partial\Sigma)$.

Given $\tau$ and $\tau'=g\tau$, we construct a family of Cauchy--Riemann operators using a clutching construction as follows. Without loss of generality, the gauge transformation $g$ takes values in $SO(3)\subset SU(3)$, so we have $g:\partial \Sigma\to SO(3)\subset SU(3)$. Using a homotopy of $g$ to the identity within $SU(3)$,  we can then construct a bundle automorphism $G$ of $(E,F)$, supported near the boundary so that $\tau=\tau'\circ G\rvert_F$.   Pull back $(E,F)$ to the trivial family of Riemann surfaces parametrised by the interval, then glue at the endpoints of this interval to produce a family of bundle pairs $(E',F')$ over the circle $S^1$, with monodromy $G$. Choose a continuous family of doubleable Cauchy--Riemann operators on this family of bundle pairs.

The second Stiefel--Whitney class $w_2(F')\in H^2(\partial \Sigma\times S^1,\mathbb Z/2)$ measures the obstruction to a Spin structure on $F'$, and this vanishes on each connected component if and only if $\tau$ and $\tau'$ represent the same Spin structure on $F$ over the corresponding boundary component. So,
\[
\left\langle w_2(F'),[\partial \Sigma \times S^1]\right\rangle
=
\left\langle g^*\eta,[\partial\Sigma]\right\rangle .
\]
Proposition 3.1 of \cite{GeorgievaOrientabilityOpenGW} states that the determinant line of a family of Cauchy--Riemann operators on $(E',F')$ is orientable if and only if $\left\langle w_2(F'),[\partial \Sigma \times S^1]\right\rangle=0$. The monodromy of the determinant line is trivial if and only if $\tau$ and $\tau'$ induce the same orientation of $\det D_{\Sigma}$, so $\tau$ and $\tau'$ induce the same orientation on $D_{\Sigma}$ if and only if $\left\langle g^*\eta,[\partial\Sigma]\right\rangle=0$.

As  $\Conf_{\Sigma}$ is fixed in this family, the monodromy of its determinant line is trivial, and therefore the monodromy of $\det D_{\Sigma,0}$ agrees with the monodromy of $\det D_\Sigma$. It follows that the orientations of $\det D_{\Sigma,\phi}$ induced by $\tau$ and $\tau'$ are related by
 \[\mathsf o_{[\tau']}=(-1)^{\left\langle g^*\eta,[\partial\Sigma]\right\rangle}\mathsf o_{[\tau]}\ .\]

\end{proof}

\begin{remark}\label{rem:holomorphic-orientation-properties}
The holomorphic orientation  $\mathsf o_{\mathfrak o}$ of $\det D_{\Sigma,\phi}$ has
the following properties.
\begin{enumerate}[(a)]
\item It is independent of the homotopies and the collar, and of any
  ordering of the components of $\partial\Sigma$, and, by Remark~\ref{rem:constant-coupling}, of whether $\phi$ is first homotoped to $0$. Each elliptic node contributes
  a factor of even index $(-1)+3$.
\item It varies continuously in families of coupled bundle pairs
  with continuously varying orientations, including families that form or smooth nodes.
\item It depends on the trivialisation $\tau$ in
  Definition~\ref{def:holomorphic-orientation} only through
  $\mathfrak o=[\tau]$, and $\mathsf o_{-\mathfrak o}=-\mathsf o_{\mathfrak o}$.
\end{enumerate}
\end{remark}

\begin{lemma}[The standard punctured disc]\label{lem:standard-punctured-disc}
Let $\mathbb D_0=\{z\in\mathbb C:|z|\leq 1\}$ be the unit disc with an interior marked point at $0$, equipped with the bundle pair $(\mathbb C^3,\mathbb R^3)$, the operator $\dbar$, a constant coupling map $\phi$, and the standard trivialisation $(e_1,e_2,e_3)$ of $\mathbb R^3$. Here $\partial_\vartheta=iz\partial_z$ is the anticlockwise boundary rotation. As $\Conf_{\mathbb D_0}$ and $\dbar$ are surjective with kernels $\mathbb R(\partial_\vartheta)$ and $\mathbb R^3$, the complex $D_{\mathbb D_0,\phi}$ is quasi-isomorphic to
\[
C^{\mathbb R}_\phi:=\bigl(\mathbb R(\partial_\vartheta)\xrightarrow{\;\phi\;}\mathbb R^3\bigr)
\]
in degrees $-1$ and $0$. The holomorphic orientation of $\det D_{\mathbb D_0,\phi}$ is induced by $\partial_\vartheta$ and $e_1\wedge e_2\wedge e_3$ via the Euler isomorphism. So for $\phi=0$ it is $\partial_\vartheta^{[-1]}\otimes(e_1\wedge e_2\wedge e_3)$, and for $\phi\neq0$ a basis $([v_1],[v_2])$ of $H^0(D_{\mathbb D_0,\phi})=\mathbb R^3/\mathbb R(\phi(\partial_\vartheta))$ is oriented if and only if $(\phi(\partial_\vartheta),v_1,v_2)$ is an oriented basis of $\mathbb R^3$.
\end{lemma}

\begin{proof}
Stretching the neck along $\partial\mathbb D_0$, keeping the coupling constant, just enlarges the disc, and in the limit gives $\mathbb{CP}^1$ with an interior marked point at $0$ and an elliptic boundary node at $q=\infty$, with closed normalisation $\Sigma_c=\mathbb{CP}^1_{0,\infty}$. At $q$, the outward normal is $\nu_\infty=z\partial_z$, so $\partial_\vartheta=i\nu_\infty$. As $\Conf_{\Sigma_c}$ and $\dbar$ are surjective with kernels $\mathbb C(\nu_\infty)$ and $\mathbb C^3$, the triple~\eqref{eq:elliptic-orientation-triple} is quasi-isomorphic to
\[
0\longrightarrow C^{\mathbb R}_\phi\longrightarrow C_\phi\longrightarrow Q_{\infty,\phi}\longrightarrow0,
\qquad
C_\phi:=\bigl(\mathbb C(\nu_\infty)\xrightarrow{\;\phi\;}\mathbb C^3\bigr),
\]
which is split, $C_\phi=C^{\mathbb R}_\phi\oplus Q_{\infty,\phi}$, because $\phi(\partial_\vartheta)\in\mathbb R^3$ and $\phi(\nu_\infty)\in i\mathbb R^3$. As all orientations involved vary continuously along $s\mapsto s\phi$, it suffices to check that, for $\phi=0$, the graded wedge of the given orientation of $C^{\mathbb R}_0$ and the orientation $\nu_\infty^{[-1]}\otimes(ie_1\wedge ie_2\wedge ie_3)$ of $Q_{\infty,0}$ from Definition~\ref{def:Q orientation} is the complex orientation of $C_0$. Moving $\det\mathbb R^3$ past $\det\mathbb R(\nu_\infty)^{[-1]}$ introduces the Koszul sign $(-1)^3$, so this graded wedge is
\[
\begin{aligned}
&-\bigl(\partial_\vartheta^{[-1]}\wedge\nu_\infty^{[-1]}\bigr)\otimes(e_1\wedge e_2\wedge e_3\wedge ie_1\wedge ie_2\wedge ie_3)\\
&\qquad=(\nu_\infty\wedge\partial_\vartheta)^{[-1]}\otimes(e_1\wedge ie_1\wedge e_2\wedge ie_2\wedge e_3\wedge ie_3),
\end{aligned}
\]
which is the complex orientation, as $\partial_\vartheta=i\nu_\infty$.
\end{proof}

\subsection{The anti-holomorphic orientation convention}
\label{subsec:antiholomorphic-orientations}

Orientations in this convention differ from holomorphic orientations by the explicit sign $\epsilon_{\mathrm{ah}}$ of~\eqref{eq:antiholomorphic-sign-index}. This sign changes under a disconnected hyperbolic smoothing but not under an elliptic or connected hyperbolic smoothing (Remark~\ref{rem:antiholomorphic-skein-signs}), which is what makes the skein relations uniform.

Each holomorphic curve $u:(\Sigma,j)\to(X,J)$ corresponds to an
anti-holomorphic curve $\bar u:(\bar\Sigma,-j)\to(X,J)$, where
$\bar\Sigma$ has the same underlying surface as $\Sigma$, with the
opposite complex structure, and $\bar u$ is the same underlying map. The moduli space of anti-holomorphic curves has its own natural orientation.
In the anti-holomorphic case, the equation $\dbar_Ju=0$ becomes $\partial_J\bar u=0$.
This identifies the two deformation problems as real Fredholm
complexes, but changes the complex structures used to orient them.

Similarly, each coupled bundle pair $(E,F,D_\Sigma,\phi)$ corresponds to an anti-holomorphic coupled bundle pair $(E,F,D_{\bar \Sigma},\phi)$.  For this, we  retain $(E,F)$
and its orientation data, and regard $D_\Sigma$ as
\[
D_{\bar\Sigma}:\Omega^0(E,F)
\longrightarrow\Omega^{1,0}_{\bar\Sigma}(E),
\]
using the identification
$\Omega^{0,1}_{\Sigma}(E)=\Omega^{1,0}_{\bar\Sigma}(E)$.
The coupling is now complex-antilinear:
\[
\phi:T^{\log}\bar\Sigma\longrightarrow E,
\qquad \phi\circ(-j)=-J\circ\phi.
\]
It still maps $T^{\log}\partial\bar\Sigma$ into $F$ and induces
a real chain map from $\Conf_{\bar\Sigma}$ to $D_{\bar\Sigma}$.
In degree one this map is again composition with $\phi$, now taking
tangent-valued $(0,1)$-forms to $E$-valued $(1,0)$-forms.

In the case that $\Sigma$ has no boundary, the identification $D_{\Sigma}\cong D_{\bar\Sigma}$ is complex-linear, and therefore preserves the natural complex orientation. In contrast, identifying vector fields and their $\dbar$-derivatives as real
tensors gives a conjugate-linear chain isomorphism
$\Conf_\Sigma\cong\Conf_{\bar\Sigma}$
which therefore modifies the natural complex orientation by a factor of $(-1)^{(\ind\Conf_\Sigma)/2}$.
These two identifications induce a canonical real isomorphism
\[
D_{\Sigma,\phi}\cong
D_{\bar\Sigma,\phi}
:=\Cone\bigl(\Conf_{\bar\Sigma}
\xrightarrow{\;\phi\;}D_{\bar\Sigma}\bigr),
\]
which therefore changes the canonical complex orientation by the same sign.

For domains with boundary, extending this convention also requires
a choice of orientation for the local boundary factors.

\begin{defn}[The anti-holomorphic orientation]
\label{def:antiholomorphic-orientation}
Apply the homotopy construction of
Definition~\ref{def:holomorphic-orientation} to the corresponding
anti-holomorphic problem. After homotoping $\phi$ to zero, use the
complex orientations of the fixed-domain and conformal factors
on the closed normalisation. In the elliptic boundary node exact
triple, retain the orientation of $JF_q$ from Definition~\ref{def:Q orientation},
but orient each line $\mathbb R(\nu_q)$ by the \emph{inward} normal $-\nu_q$. Thus each node factor has the
opposite orientation to that used in the holomorphic convention.

The resulting orientation, transported back along the homotopies
and the canonical real isomorphism above, is the
\emph{anti-holomorphic orientation}
$\mathsf o^{\mathrm{ah}}_{\mathfrak o}(D_{\Sigma,\phi})$.
\end{defn}

To compare the two conventions,  let $\Sigma_c$ be the closed normalisation in the elliptic
node model. The conformal exact triple gives
\[
\ind\Conf_{\Sigma_c}=\ind\Conf_\Sigma+b_{\mathbb R}(\Sigma).
\]
Conjugation on the closed conformal factor therefore contributes
$(-1)^{(\ind\Conf_\Sigma+b_{\mathbb R}(\Sigma))/2}$, and reversing the $b_{\mathbb R}(\Sigma)$ node
factors contributes $(-1)^{b_{\mathbb R}(\Sigma)}$. Consequently,
\[
\mathsf o^{\mathrm{ah}}_{\mathfrak o}(D_{\Sigma,\phi})
=\epsilon_{\mathrm{ah}}(\Sigma)\,
\mathsf o_{\mathfrak o}(D_{\Sigma,\phi}),
\]
where
\begin{equation}
\label{eq:antiholomorphic-sign-index}
\epsilon_{\mathrm{ah}}(\Sigma)
:=(-1)^{(\ind\Conf_\Sigma-b_{\mathbb R}(\Sigma))/2}
=(-1)^{(3\chi(\Sigma)+n-b_{\mathbb R}(\Sigma))/2}.
\end{equation}
Here $n$ is the number of interior marked points, and we have only described the case with no
boundary marked points. For a smooth connected domain of genus
$g$, the exponent is $3-3g-2b_{\mathbb R}-n$, so it is an integer
and the sign is $(-1)^{1-g-n}$. In particular, adding a real
boundary component without changing $g$ or $n$ leaves this sign
unchanged. For nodal domains the comparison follows by smoothing
and continuity.

The choice of the inward normal in the definition is a boundary
normalisation of the convention. Keeping the outward normals $\nu_q$
instead would multiply the anti-holomorphic
orientation by $(-1)^{b_{\mathbb R}(\Sigma)}$.

\begin{remark}[Changing orientation conventions]
\label{rem:alternative-orientation-conventions}

Our orientation convention requires choosing an orientation for each elliptic node factor $\det Q_{q,\phi}$ in (\ref{eq:en comparison}), or equivalently, an orientation for the determinant line of the model Fredholm operator $D_{\rm disc}$ on the disc; see Remark~\ref{rem:disc-refinement-orientation}.  The opposite convention changes orientations by $(-1)^{b_{\mathbb R}(u)}$, where $b_{\mathbb R}(u)$ is the number of real boundary components.  This flips the relationship between the negative chamber and smoothing boundary wall orientations, because in each wall problem the number of boundary components changes by $\pm 1$. Welschinger used this opposite convention to orient discs in \cite{WelschingerOpenGW6}, resulting in an opposite sign comparison at the disconnected hyperbolic walls that occur in such a moduli space.

Another orientation convention could use the homotopy to a standard model to orient both $ D_{\Sigma}$ and $\Conf_{\Sigma}$ separately, with both orientations depending on an ordering of real boundary components. This would also modify the orientation convention by a sign depending on the parity of the real boundary components.

Both of the above modifications  fail to reconcile the sign discrepancy between the connected and disconnected crossing cases observed in Proposition~\ref{prop:disconnected-hyperbolic-orientation}
and Lemma~\ref{lem:connected-hyperbolic-orientation}.

If we instead use the anti-holomorphic orientation convention,  the relationship between the negative chamber  and smoothing boundary orientations of $\det D_{\rm wall}$ becomes identical for both  connected and disconnected crossing cases.

Section~\ref{sec:ramified-covers} explains this uniformity for ramified covers of an immersed curve. There, an auxiliary $\mathcal J$-complex orientation gives the same comparison at connected and disconnected hyperbolic walls, and differs from the holomorphic orientation by $(-1)^{1-g+b_{\mathbb R}}$ and from the anti-holomorphic orientation by $(-1)^{b_{\mathbb R}+n}$ (Lemma~\ref{Jcomparison2}). Under a hyperbolic smoothing, the second sign always changes, whereas the first changes only in the connected case.
\end{remark}

\subsection{Gluing compatibility}

We now check that the orientations are compatible with normalising interior nodes (Lemma~\ref{lem:gluing-orientation}), disjoint unions (Remark~\ref{rem:disjoint-union-orientation}), and adding interior marked points (Remark~\ref{rem:marked-point-orientation}). The first two allow us to discard components of the domain away from the wall (Lemma~\ref{lem:wall-excision}).

\begin{lemma}[Gluing orientation comparison]\label{lem:gluing-orientation}
Let $\Sigma$ be a bordered Riemann surface with no boundary marked points and a chosen interior node $\mathrm{n}$, and let $\pi\colon \Sigma'\to \Sigma$ be the partial normalisation of $\Sigma$ at $\mathrm{n}$, with interior marked points $p_1$ and $p_2$ recording the preimages of $\mathrm{n}$. Given a coupled bundle pair $(E,F,D_\Sigma,\phi)$ over $\Sigma$, let $(E',F',D_{\Sigma'},\phi')$ denote the pulled-back coupled bundle pair over $\Sigma'$.

Choosing an ordering $(p_1,p_2)$ of the preimages of $\mathrm{n}$ defines a natural gluing exact triple
\begin{equation*}\tag{$\mathfrak T_{\mathrm{n}}$}\label{eq:interior-gluing-triple}
0 \longrightarrow D_{\Sigma,\phi} \longrightarrow D_{\Sigma',\phi'} \longrightarrow \bigl(T^{\log}_{\mathrm{n}}\Sigma\xrightarrow{\;\phi_{\mathrm{n}}\;} E_{\mathrm{n}}\bigr) \longrightarrow 0,
\end{equation*}
where the two-term complex $\bigl(T^{\log}_{\mathrm{n}}\Sigma\xrightarrow{\;\phi_{\mathrm{n}}\;} E_{\mathrm{n}}\bigr)$ is placed in degrees $-1$ and $0$. The additivity isomorphism
\[
\Psi_{\mathfrak T_{\mathrm{n}}}\colon \det D_{\Sigma,\phi}\otimes \det \bigl(T^{\log}_{\mathrm{n}}\Sigma\xrightarrow{\;\phi_{\mathrm{n}}\;} E_{\mathrm{n}}\bigr) \xrightarrow{\;\cong\;} \det D_{\Sigma',\phi'}
\]
combined with the Euler isomorphism $\det \bigl(T^{\log}_{\mathrm{n}}\Sigma\xrightarrow{\;\phi_{\mathrm{n}}\;} E_{\mathrm{n}}\bigr) \xrightarrow{\;\cong\;} (\det T^{\log}_{\mathrm{n}}\Sigma)^{[-1]}\otimes \det E_{\mathrm{n}}$ defines a natural gluing isomorphism
\begin{equation}\label{eq:interior-gluing-isomorphism}
\det D_{\Sigma,\phi}\otimes (\det T^{\log}_{\mathrm{n}}\Sigma)^{[-1]}\otimes \det E_{\mathrm{n}} \xrightarrow{\;\cong\;} \det D_{\Sigma',\phi'}
\end{equation}
which is independent of the ordering of $p_1$ and $p_2$.

Given an oriented trivialisation of $F=F'$ over $\partial \Sigma=\partial \Sigma'$, the gluing isomorphism \eqref{eq:interior-gluing-isomorphism} preserves orientations when:
\begin{enumerate}[(i)]
\item $\det D_{\Sigma,\phi}$ and $\det D_{\Sigma',\phi'}$ are given their induced holomorphic orientations, and $\det T^{\log}_{\mathrm{n}}\Sigma$ and $\det E_{\mathrm{n}}$ are given their complex orientations; or
\item $\det D_{\Sigma,\phi}$ and $\det D_{\Sigma',\phi'}$ are given their induced anti-holomorphic orientations, $\det E_{\mathrm{n}}$ is given its complex orientation, and $\det T^{\log}_{\mathrm{n}}\Sigma$ is given the complex conjugate orientation.
\end{enumerate}
\end{lemma}

\begin{proof}
We construct the triple and check that it is independent of the ordering, then compare orientations by following the homotopies of Definition~\ref{def:holomorphic-orientation} to a closed complex-linear problem, and finally deduce the anti-holomorphic case from~\eqref{eq:antiholomorphic-sign-index}.

Under the canonical identification $T^{\log}\Sigma' = \pi^*T^{\log}\Sigma$, the fibres $T^{\log}_{p_1}\Sigma'$ and $T^{\log}_{p_2}\Sigma'$ are both identified with $T^{\log}_{\mathrm{n}}\Sigma$ (in residue coordinates, $\operatorname{res}_{p_1}$ and $-\operatorname{res}_{p_2}$ define the same complex-linear isomorphism $T^{\log}_{\mathrm{n}}\Sigma\cong\mathbb C$). A section $v$ of $(T^{\log}\Sigma',T^{\log}\partial\Sigma')$ descends to $(T^{\log}\Sigma,T^{\log}\partial\Sigma)$ if and only if $v(p_1)-v(p_2)=0$ in $T^{\log}_{\mathrm{n}}\Sigma$ (equivalently, $\operatorname{res}_{p_1}v+\operatorname{res}_{p_2}v=0$), while $\Omega^{0,1}(T^{\log}\Sigma)=\Omega^{0,1}(T^{\log}\Sigma')$. Likewise, a section $\xi$ of $(E',F')$ descends to $(E,F)$ if and only if $\xi(p_1)-\xi(p_2)=0$ in $E_{\mathrm{n}}$, while $\Omega^{0,1}(E)=\Omega^{0,1}(E')$. Since $\phi'$ is pulled back from $\phi$, we have $\phi'(v)(p_1)-\phi'(v)(p_2)=\phi_{\mathrm{n}}(v(p_1)-v(p_2))$. Taking cones of the vertical maps in the resulting commutative diagram of exact triples
\[
\begin{tikzcd}
0 \rar & \Conf_{\Sigma} \rar\dar{\phi} & \Conf_{\Sigma'} \rar{\operatorname{ev}_{\mathrm{n}}}\dar{\phi'} & (T^{\log}_{\mathrm{n}}\Sigma\to 0) \rar\dar{\phi_{\mathrm{n}}} & 0 \\
0 \rar & D_{\Sigma} \rar & D_{\Sigma'} \rar{\operatorname{ev}_{\mathrm{n}}} & (E_{\mathrm{n}}\to 0) \rar & 0
\end{tikzcd}
\]
produces the exact triple \eqref{eq:interior-gluing-triple}. Reversing the ordering of $p_1$ and $p_2$ multiplies the evaluation maps to $T^{\log}_{\mathrm{n}}\Sigma$ and $E_{\mathrm{n}}$ by $-1$, so acts trivially on the determinants, as these spaces are even dimensional.  Therefore the gluing isomorphism \eqref{eq:interior-gluing-isomorphism} is independent of the ordering of $p_1$ and $p_2$.

The gluing isomorphism is continuous in homotopies, so can be followed through the homotopies defining the holomorphic orientation in Definition~\ref{def:holomorphic-orientation}. In particular, we homotope the coupling map to $0$, smooth hyperbolic nodes, use our trivialisation to homotope $(E,F,D_\Sigma)$ to standard data  $(\mathbb C^3,\mathbb R^3,\dbar)$ in a collar neighbourhood of the boundary (not containing our node $\rm n$), and stretch the neck along each boundary component. The partial normalisation at our node and the pullback of coupled bundle pairs remain defined throughout this homotopy, giving a pulled-back homotopy and a continuous family of gluing isomorphisms.

In the limit, we obtain a nodal surface $\Sigma_{\mathrm{lim}}$ whose only boundary components are elliptic boundary nodes $q_1,\dots,q_b$ disjoint from $\mathrm{n}$. Write $\Sigma'_{\mathrm{lim}}$ for the partial normalisation of $\Sigma_{\mathrm{lim}}$ at $\mathrm{n}$. Let $\Sigma_c$ and $\Sigma'_c$ denote their closed normalisations at $q_1,\dots,q_b$; then $\Sigma'_c$ is the partial normalisation of $\Sigma_c$ at $\mathrm{n}$. These fit into the following commuting grid of exact triples:
\[
\begin{tikzcd}
0 \rar & D_{\Sigma_{\mathrm{lim}},0} \rar\dar & D_{\Sigma_c,0} \rar\dar & \bigoplus_q Q_{q,0} \rar\dar{=} & 0 \\
0 \rar & D_{\Sigma'_{\mathrm{lim}},0} \rar\dar & D_{\Sigma'_c,0} \rar\dar & \bigoplus_q Q_{q,0} \rar & 0 \\
& C_{\mathrm{n},0} \rar{=} & C_{\mathrm{n},0} & &
\end{tikzcd}
\]
Here the rows are the elliptic boundary node triples  \eqref{eq:elliptic-orientation-triple}, the first two columns are the interior gluing triples $\mathfrak T_{\mathrm{n}}$, and $C_{\mathrm{n},0} = \bigl(T^{\log}_{\mathrm{n}}\Sigma\xrightarrow{\;0\;} E_{\mathrm{n}}\bigr)$. Since $\det C_{\mathrm{n},0}$ and every $\det Q_{q,0}$ have even index, commuting these factors in the determinant-line additivity isomorphisms introduces no Koszul signs; see Lemma~\ref{lem:exact-squares}. By Definition~\ref{def:holomorphic-orientation}(iv), $\det D_{\Sigma_{\mathrm{lim}},0}$ and $\det D_{\Sigma'_{\mathrm{lim}},0}$ are oriented from $\det D_{\Sigma_c,0}$ and $\det D_{\Sigma'_c,0}$ via the rows using the same orientations on $\det Q_{q,0}$. It therefore suffices to check that the middle column preserves orientations on the closed surfaces $\Sigma_c$ and $\Sigma'_c$.

On $\Sigma_c$, the linear homotopy from $D_{\Sigma_c}$ to its complex-linear part $D^{(1,0)}_{\Sigma_c}$ pulls back to the linear homotopy from $D_{\Sigma'_c}$ to $D^{(1,0)}_{\Sigma'_c}$ on $\Sigma'_c$. When $D_{\Sigma_c}$ is complex-linear, the exact triple
\[
0 \longrightarrow D_{\Sigma_c,0} \longrightarrow D_{\Sigma'_c,0} \longrightarrow C_{\mathrm{n},0} \longrightarrow 0
\]
is an exact triple of complex Fredholm complexes, so the additivity isomorphism preserves the complex orientation. As the gluing isomorphism is orientation preserving in this limiting case, and continuous under homotopies, it follows that the gluing isomorphism was orientation preserving at the start of the homotopy.

The anti-holomorphic case can be proved similarly, but also follows easily from the holomorphic case using \eqref{eq:antiholomorphic-sign-index}. For a surface $S$ without boundary marked points, the anti-holomorphic orientation is related to the holomorphic orientation by
\[
\mathsf o^{\mathrm{ah}}_{\mathfrak o}(D_{S,\phi}) = \epsilon_{\mathrm{ah}}(S)\,\mathsf o_{\mathfrak o}(D_{S,\phi}), \qquad \epsilon_{\mathrm{ah}}(S) = (-1)^{(\ind\Conf_S - b_{\mathbb R}(S))/2}.
\]
The exact sequence $0\to\Conf_{\Sigma}\to\Conf_{\Sigma'}\to (T^{\log}_{\mathrm{n}}\Sigma\to 0)\to 0$ gives
\[
\ind\Conf_{\Sigma'} = \ind\Conf_{\Sigma} + \dim_{\mathbb R} T^{\log}_{\mathrm{n}}\Sigma = \ind\Conf_{\Sigma} + 2,
\]
while $b_{\mathbb R}(\Sigma') = b_{\mathbb R}(\Sigma)$ because normalising an interior node does not affect the number of real boundary components.  Hence $\epsilon_{\mathrm{ah}}(\Sigma') = -\epsilon_{\mathrm{ah}}(\Sigma)$. Using the complex conjugate orientation  of $\det T^{\log}_{\mathrm{n}}\Sigma$  compensates for this sign change, so the gluing isomorphism \eqref{eq:interior-gluing-isomorphism} also preserves orientations in the anti-holomorphic convention.

\end{proof}

\begin{remark}[Disjoint union orientation comparison]\label{rem:disjoint-union-orientation}
Suppose that $\Sigma = \Sigma_1\amalg \Sigma_2$ is a disjoint union of bordered Riemann surfaces with no boundary marked points, carrying a coupled bundle pair $(E,F,D_\Sigma,\phi)$ that restricts to $(E_i,F_i,D_{\Sigma_i},\phi_i)$ on $\Sigma_i$. The direct-sum decomposition $D_{\Sigma,\phi} = D_{\Sigma_1,\phi_1}\oplus D_{\Sigma_2,\phi_2}$ induces the canonical additivity isomorphism
\begin{equation}\label{eq:disjoint-union-isomorphism}
\Psi\colon \det D_{\Sigma_1,\phi_1}\otimes \det D_{\Sigma_2,\phi_2} \xrightarrow{\;\cong\;} \det D_{\Sigma,\phi}.
\end{equation}
Given an oriented trivialisation of $F$ over $\partial\Sigma = \partial\Sigma_1\amalg \partial\Sigma_2$, this isomorphism preserves both the induced holomorphic orientations and the induced anti-holomorphic orientations. Indeed, for the holomorphic convention, the homotopies of Definition~\ref{def:holomorphic-orientation} proceed independently on the two components, and in the elliptic boundary node limit both the closed complex-oriented factors $D_{(\Sigma_i)_c,0}$ and the node factors $Q_{q,0}$ have even index, so no Koszul signs arise from reordering them. For the anti-holomorphic convention, the correction sign \eqref{eq:antiholomorphic-sign-index} is multiplicative under disjoint union because $\ind\Conf_{\Sigma}=\ind\Conf_{\Sigma_1}+\ind\Conf_{\Sigma_2}$ and $b_{\mathbb R}(\Sigma)=b_{\mathbb R}(\Sigma_1)+b_{\mathbb R}(\Sigma_2)$, giving $\epsilon_{\mathrm{ah}}(\Sigma)=\epsilon_{\mathrm{ah}}(\Sigma_1)\epsilon_{\mathrm{ah}}(\Sigma_2)$.
\end{remark}

\begin{remark}[Adding an interior marked point]\label{rem:marked-point-orientation}
Let $\Sigma_p$ be obtained from $\Sigma$ by adding a new interior marked point at a smooth interior point $p\in\Sigma$, and let $(E,F,D_{\Sigma_p},\phi)$ be the induced coupled bundle pair on $\Sigma_p$. Using the canonical quasi-isomorphism $\Conf_{\Sigma_p}\xrightarrow{\;\cong\;}\Conf_{\Sigma;p}$ from \eqref{conf models}, evaluation of vector fields at $p$ gives an exact triple of deformation complexes
\begin{equation}\label{eq:marked-point-triple}
0 \longrightarrow D_{\Sigma_p,\phi} \longrightarrow D_{\Sigma,\phi} \longrightarrow T_p\Sigma \longrightarrow 0,
\end{equation}
where $T_p\Sigma$ is placed in degree $-1$, and thus an additivity isomorphism
\begin{equation}\label{eq:marked-point-isomorphism}
\Psi\colon \det D_{\Sigma_p,\phi}\otimes (\det T_p\Sigma)^{[-1]} \xrightarrow{\;\cong\;} \det D_{\Sigma,\phi}.
\end{equation}
Given an oriented trivialisation of $F$ over $\partial\Sigma$, the isomorphism \eqref{eq:marked-point-isomorphism} preserves orientations when:
\begin{enumerate}[(i)]
\item $\det D_{\Sigma_p,\phi}$ and $\det D_{\Sigma,\phi}$ carry their holomorphic orientations, and $\det T_p\Sigma$ carries its complex orientation; or
\item $\det D_{\Sigma_p,\phi}$ and $\det D_{\Sigma,\phi}$ carry their anti-holomorphic orientations, and $\det T_p\Sigma$ carries the complex conjugate orientation.
\end{enumerate}
Indeed, (i) follows by transporting \eqref{eq:marked-point-triple} along the homotopies of Definition~\ref{def:holomorphic-orientation} to the complex-linear closed problem (since $\det Q_{q,0}$ and $(\det T_p\Sigma)^{[-1]}$ have even index), while (ii) follows from \eqref{eq:antiholomorphic-sign-index} since $\ind\Conf_{\Sigma_p} = \ind\Conf_\Sigma - 2$ and $b_{\mathbb R}(\Sigma_p) = b_{\mathbb R}(\Sigma)$, so that $\epsilon_{\mathrm{ah}}(\Sigma_p) = -\epsilon_{\mathrm{ah}}(\Sigma)$.
\end{remark}

\subsection{Wall problem orientations}

The incidence triple~\eqref{eq:wall-triangle} and the smoothing triple~\eqref{eq:smoothing-triangle} give two ways to orient the wall complex, once $D_{\mathrm{norm}}$ and $D_{\mathrm{node}}$ are oriented. Viewing the wall as the boundary of a chamber in the normalised problem uses
\[
\Psi_{\mathfrak T_{\mathrm{inc}}}\colon
\det D_{\mathrm{wall}}\otimes\det\Cone(K)
\xrightarrow{\;\cong\;}
\det D_{\mathrm{norm}},
\]
where, for a transverse wall problem, $\det\Cone(K)=\det\coker K$ is the incidence normal line. Viewing the nodal problem as the boundary of its smoothings uses
\[
\Psi_{\mathfrak T_{\mathrm{sm}}}\colon
\det D_{\mathrm{node}}\otimes
(\det\mathbb R_{\mathrm{res}})^{[-1]}
\xrightarrow{\;\cong\;}
\det D'_{\mathrm{wall}},
\]
where $\mathbb R_{\mathrm{res}}$ represents the logarithmic smoothing direction. A canonical quasi-isomorphism $D'_{\mathrm{wall}}\simeq D_{\mathrm{wall}}$ identifies the determinant lines, and our main results determine the relative sign of the two resulting orientations of $\det D_{\mathrm{wall}}$.

\begin{defn}[Compatible orientation data]
\label{def:compatible-orientation-data}
 \emph{Compatible orientation data} for a hyperbolic wall problem $(E,F,D,\phi,p_1,p_2,I)$ is a trivialisation  $\tau$ of $F$ in which the identification $I:F_{p_1}\to F_{p_2}$ is the identity, and the induced  trivialisation $\tau_{\mathrm{node}}$ of $F_{\mathrm{node}}\to\Sigma_{\mathrm{node}}$. Compatible orientation data for an elliptic wall problem $(E,F,D,\phi,q,F_q)$ is a trivialisation $\tau_{\mathrm{node}}$ of $F_{\mathrm{node}}$ that is a constant frame of $F_q$ over the new boundary component, and the induced trivialisation $\tau$ of $F$. In both cases, write $\pi\colon\partial\Sigma\to\partial\Sigma_{\mathrm{node}}$ for the quotient map (hyperbolic) or the inclusion (elliptic), so that $\tau=\pi^*\tau_{\mathrm{node}}$.
\end{defn}

\begin{defn}[Orientations of $D_{\mathrm{norm}}$ and $D_{\mathrm{node}}$]
\label{def:norm-node-orientations}
Let $(\tau,\tau_{\mathrm{node}})$ be compatible orientation data, and set
\[\mathsf o_{\mathrm{norm}}:=\mathsf o_{[\tau]}(D_{\mathrm{norm}}), \quad \mathsf o_{\mathrm{node}}:=\mathsf o_{[\tau_{\mathrm{node}}]}(D_{\mathrm{node}}).\]
Similarly, set
\[\mathsf o^{\rm ah}_{\mathrm{norm}}:=\mathsf o^{\rm ah}_{[\tau]}(D_{\mathrm{norm}}), \quad \mathsf o^{\rm ah}_{\mathrm{node}}:=\mathsf o^{\rm ah}_{[\tau_{\mathrm{node}}]}(D_{\mathrm{node}}).\]
\end{defn}

\begin{defn}[Incidence orientation of $\det\Cone(K)$] \label{def:incidence-orientation} Orient the domain of the incidence map $K$ from
Section~\ref{subsec:wall-problems} by the ordered boundary
orientations in the hyperbolic case, and by the complex
orientation of $T_q\Sigma$ in the elliptic case. Orient its codomain
by the given orientation of $F_{p_1}$ in the hyperbolic case,
and by applying $J$ to the given orientation of $F_q$ in the elliptic case, as in Definition~\ref{def:Q orientation}. The Euler isomorphism~\eqref{eq:euler-isomorphism}  gives an orientation
$\mathsf o_{\mathrm{inc}}$ of $\det\Cone(K)$.
\end{defn}
When $K$ is injective, the incidence orientation corresponds to the outward normal
$-[n_+]$ of the positive chamber, where $n_+$ is the positive
normal specified in Section~\ref{subsec:wall-problems}.

\begin{defn}[Wall orientations]
\label{def:boundary-orientations}
Fix compatible orientation data, and use the resulting orientations
$\mathsf o_{\mathrm{norm}}$ and $\mathsf o_{\mathrm{node}}$ from
Definition~\ref{def:norm-node-orientations}. Identify
$\det D'_{\mathrm{wall}}$ with $\det D_{\mathrm{wall}}$ by the
canonical quasi-isomorphism.

\begin{enumerate}[(i)]
\item The \emph{positive chamber orientation}
$\mathsf o^+_\partial$ is determined by
\begin{equation}\label{eq:chamber-orientation}
(-1)^{\ind D_{\mathrm{wall}}}
\Psi_{\mathfrak T_{\mathrm{inc}}}
\bigl(\mathsf o^+_\partial\otimes\mathsf o_{\mathrm{inc}}\bigr)
=\mathsf o_{\mathrm{norm}}.
\end{equation}
The sign places the normal factor first. The \emph{negative
chamber orientation} is $\mathsf o^-_\partial:=-\mathsf o^+_\partial$. As $\ind D_{\mathrm{wall}}$ is odd, the negative chamber orientation is determined by
\begin{equation}\label{eq:negative-chamber-orientation}\Psi_{\mathfrak T_{\mathrm{inc}}}
\bigl(\mathsf o^-_\partial\otimes\mathsf o_{\mathrm{inc}}\bigr)
=\mathsf o_{\mathrm{norm}}.\end{equation}
Similarly, the anti-holomorphic version of the negative chamber orientation is determined by
\begin{equation}\label{eq:ah-chamber}\Psi_{\mathfrak T_{\mathrm{inc}}}
\bigl(\mathsf o^{-,\rm ah}_\partial\otimes\mathsf o_{\mathrm{inc}}\bigr)
=\mathsf o^{\rm ah}_{\mathrm{norm}}.\end{equation}
In the more general case that $\ind D_{\rm norm}$ may be odd, the right-hand sides of \eqref{eq:negative-chamber-orientation} and \eqref{eq:ah-chamber} acquire a sign $(-1)^{\ind D_{\rm norm}}$.
\item The \emph{smoothing boundary orientation}
$\mathsf o^{\mathrm{sm}}_\partial$ and corresponding anti-holomorphic smoothing boundary orientation $\mathsf o^{\mathrm{sm},\rm ah}_\partial$ are determined by
\begin{equation}\label{eq:smoothing-orientation}
\mathsf o^{\mathrm{sm}}_\partial
=\Psi_{\mathfrak T_{\mathrm{sm}}}
\bigl(\mathsf o_{\mathrm{node}}\otimes r_{\mathrm{out}}^{[-1]}\bigr), \quad \mathsf o^{\mathrm{sm},\rm ah}_\partial
=\Psi_{\mathfrak T_{\mathrm{sm}}}
\bigl(\mathsf o^{\rm ah}_{\mathrm{node}}\otimes r_{\mathrm{out}}^{[-1]}\bigr),
\end{equation}
where $r_{\mathrm{out}}=-1\in\mathbb R_{\mathrm{res}}$ is the
negative-residue generator and $r_{\mathrm{out}}^{[-1]}$ is its
left dual.  Its connecting class is the outward logarithmic
smoothing direction; see Remark~\ref{rem:residue-sign}. In the more general case that $\ind D_{\rm node}$ may be odd, both definitions in \eqref{eq:smoothing-orientation} should include an extra sign $(-1)^{\ind D_{\rm node}}$.
\end{enumerate}

\end{defn}

\begin{remark}[Outward-normal recipe for unobstructed models]
\label{rem:transverse-wall-recipe}
In the unobstructed model families of Sections~\ref{sec:disconnected-crossing}--\ref{connected-hyperbolic-section}, all cohomology is concentrated in degree~$0$. If $t\in(-1,1)$ is a transverse coordinate on $\mathcal M_{\mathrm{norm}}$ with $t>0$ in $\mathcal M^+_{\mathrm{norm}}$, and $\epsilon\in[0,1)$ is the plumbing coordinate on $\mathcal M_{\mathrm{smooth}}$ ($zw=\epsilon$), then Definition~\ref{def:boundary-orientations} reduces to the standard \emph{outward-normal-first} rule for a basis $w=(w_1,\dots,w_d)$ of $T\mathcal M_{\mathrm{wall}}$:

\begin{center}
\setlength{\fboxsep}{6pt}
\fbox{\small
\begin{tabular}{@{}l@{\quad}l@{\quad}l@{}}
\textbf{Orientation on $T\mathcal M_{\mathrm{wall}}$} & \textbf{Outward normal} & \textbf{Positively oriented frame} \\[2pt]
\hline\\[-7pt]
Positive chamber $\mathsf o^+_\partial = [w]$ & $-\partial_t$ & $(-\partial_t, w_1,\dots,w_d) \in T\mathcal M_{\mathrm{norm}}$ \\[2pt]
Negative chamber $\mathsf o^-_\partial = [w]$ & $+\partial_t$ & $(+\partial_t, w_1,\dots,w_d) \in T\mathcal M_{\mathrm{norm}}$ \\[2pt]
Smoothing boundary $\mathsf o^{\mathrm{sm}}_\partial = [w]$ & $-\epsilon\partial_\epsilon$ & $(-\epsilon\partial_\epsilon, w_1,\dots,w_d) \in T\mathcal M_{\mathrm{smooth}}$
\end{tabular}}
\end{center}

\smallskip\noindent
Indeed, $\mathsf o_{\mathrm{inc}}=[-\partial_t]$ in degree~$0$, and since $d=\ind D_{\mathrm{wall}}$ is odd,~\eqref{eq:negative-chamber-orientation} gives $\mathsf o_{\mathrm{norm}}=\mathsf o^-_\partial\wedge(-\partial_t)=(+\partial_t)\wedge\mathsf o^-_\partial = (-\partial_t)\wedge\mathsf o^+_\partial$. For $\mathfrak T_{\mathrm{sm}}$, the connecting map $\delta$ sends $r_{\mathrm{out}}=-1\in\mathbb R_{\mathrm{res}}$ (in degree~$-1$) to $-\epsilon\partial_\epsilon\in H^0(D_{\mathrm{node}})$ (Remark~\ref{rem:residue-sign}); by~\eqref{eq:euler-isomorphism-formula}, the $-1\to 0$ cancelling pair places $\delta(r_{\mathrm{out}})$ in the \emph{first} slot of $H^0(D_{\mathrm{node}})$, so $\mathsf o^{\mathrm{sm}}_\partial=\Psi_{\mathfrak T_{\mathrm{sm}}}(\mathsf o_{\mathrm{node}}\otimes r_{\mathrm{out}}^{[-1]})$ means $\mathsf o_{\mathrm{node}}=(-\epsilon\partial_\epsilon)\wedge\mathsf o^{\mathrm{sm}}_\partial$.
\end{remark}

\begin{remark} The comparison between the smoothing boundary and negative chamber orientations of $D_{\mathrm{wall}}$ does not depend on the choice of compatible orientation data, because a change affects $\mathsf o_{\rm node}$ and $\mathsf o_{\rm norm}$ by the same sign. Moreover, the comparison between these two orientations is locally constant in continuous families of wall problems over families of bordered Riemann surfaces, including families that form or smooth nodes. The comparison between the anti-holomorphic smoothing boundary and negative chamber orientations is similarly locally constant in families.  \end{remark}

Theorem~\ref{thm:main-orientation-comparison} states that $\mathsf o_{\partial}^{-,\rm ah}=\mathsf o_{\partial}^{\rm sm,  ah}$, whereas $\mathsf o_{\partial}^{-}=\mathsf o_{\partial}^{\rm sm}$ in the case of elliptic and connected hyperbolic wall problems, but $\mathsf o_{\partial}^{+}=\mathsf o^{\rm sm}_{\partial}$ for disconnected hyperbolic wall problems. We prove this by showing that the relative sign is unchanged by excision (Lemma~\ref{lem:wall-excision}), so that it depends only on the type of wall (Proposition~\ref{prop:wall-model-reduction}), and then computing it in each local model.

\begin{lemma}[Excision for wall orientation comparisons]\label{lem:wall-excision}
Let $(E,F,D_\Sigma,\phi)$ be a wall problem on a bordered Riemann surface $\Sigma$ with no boundary marked points. In both the holomorphic and anti-holomorphic orientation conventions, the relative sign between the smoothing boundary orientation and the negative chamber orientation of $\det D_{\mathrm{wall}}$ satisfies the following invariance properties:
\begin{enumerate}[(i)]
\item \emph{(Normalising an interior node)} If $\mathrm{n}$ is an interior node of $\Sigma$, the relative sign is unchanged by passing to the pulled-back wall problem on the partial normalisation $\Sigma'$ of $\Sigma$ at $\mathrm{n}$.
\item \emph{(Discarding a disjoint component)} If $\Sigma = \Sigma_1\amalg \Sigma_2$ is a disjoint union such that the wall points ($p_1,p_2$ in the hyperbolic case, or $q$ in the elliptic case) lie in $\Sigma_1$, the relative sign is unchanged by restricting the wall problem to $\Sigma_1$.
\end{enumerate}
In particular, if $\mathrm{n}$ is a separating interior node of $\Sigma$, normalising $\mathrm{n}$ and discarding the component not containing the wall points leaves the orientation comparison unchanged.
\end{lemma}

\begin{proof}
Fix compatible orientation data on $\Sigma$, and equip $\Sigma'$ and $\Sigma_1$ with the induced compatible orientation data.

For (i), since the interior node $\mathrm{n}$ is disjoint from the wall points, the evaluation maps at $\mathrm{n}$ commute with the incidence map in $\mathfrak T_{\mathrm{inc}}$ and the residue map in $\mathfrak T_{\mathrm{sm}}$. Writing $C_{\mathrm{n},\phi}=\bigl(T^{\log}_{\mathrm{n}}\Sigma\xrightarrow{\;\phi_{\mathrm{n}}\;} E_{\mathrm{n}}\bigr)$, we obtain two commuting $3\times 3$ grids of exact triples:
\[
\begin{tikzcd}[column sep=small]
 D_{\mathrm{wall}}(\Sigma) \rar\dar & D_{\mathrm{norm}}(\Sigma) \rar\dar & \Cone(K) \dar{=}  \\
 D_{\mathrm{wall}}(\Sigma') \rar\dar & D_{\mathrm{norm}}(\Sigma') \rar\dar & \Cone(K)   \\
 C_{\mathrm{n},\phi} \rar{=} & C_{\mathrm{n},\phi} & &
\end{tikzcd}
\quad
\begin{tikzcd}[column sep=small]
 D_{\mathrm{node}}(\Sigma) \rar\dar & D'_{\mathrm{wall}}(\Sigma) \rar\dar & \mathbb R_{\mathrm{res}} \dar{=} \\
  D_{\mathrm{node}}(\Sigma') \rar\dar & D'_{\mathrm{wall}}(\Sigma') \rar\dar & \mathbb R_{\mathrm{res}}  \\
 C_{\mathrm{n},\phi} \rar{=} & C_{\mathrm{n},\phi} & &
\end{tikzcd}
\]
Here the rows are $\mathfrak T_{\mathrm{inc}}$ on the left and $\mathfrak T_{\mathrm{sm}}$ on the right, while the first two columns of each grid are the interior gluing triples  \eqref{eq:interior-gluing-triple}. By Lemma~\ref{lem:gluing-orientation}, the vertical gluing isomorphisms
\begin{gather*}
\det D_{\mathrm{norm}}(\Sigma)\otimes \det C_{\mathrm{n},\phi} \xrightarrow{\;\cong\;} \det D_{\mathrm{norm}}(\Sigma'),
\\
\det D_{\mathrm{node}}(\Sigma)\otimes \det C_{\mathrm{n},\phi} \xrightarrow{\;\cong\;} \det D_{\mathrm{node}}(\Sigma')
\end{gather*}
preserve orientations (in either the holomorphic or anti-holomorphic convention, with the corresponding orientation of $\det C_{\mathrm{n},\phi}$). Because $\det C_{\mathrm{n},\phi}$ has even index $6-2=4$, commuting it past the normal factors $\det\Cone(K)$ and $(\det\mathbb R_{\mathrm{res}})^{[-1]}$ in the determinant-line additivity isomorphisms for the two grids introduces no Koszul signs; see Lemma~\ref{lem:exact-squares}. So, the vertical gluing isomorphism
\[
\det D_{\mathrm{wall}}(\Sigma)\otimes \det C_{\mathrm{n},\phi} \xrightarrow{\;\cong\;} \det D_{\mathrm{wall}}(\Sigma')
\]
from the additivity isomorphism of the first column,
simultaneously sends $\mathsf o^-_\partial(\Sigma)\otimes\mathsf o(C_{\mathrm{n},\phi})$ to $\mathsf o^-_\partial(\Sigma')$ via the first grid and $\mathsf o^{\mathrm{sm}}_\partial(\Sigma)\otimes\mathsf o(C_{\mathrm{n},\phi})$ to $\mathsf o^{\mathrm{sm}}_\partial(\Sigma')$ via the second grid.

For (ii), the deformation complexes split as direct sums with $D_{\Sigma_2,\phi_2}$, giving the same two $3\times 3$ grids with $\Sigma_1$ in the top row, $\Sigma$ in the middle row, and $D_{\Sigma_2,\phi_2}$ in place of $C_{\mathrm{n},\phi}$ in the bottom row. Since $\Sigma_2$ has no boundary marked points and $\mu(E_2,F_2)$ is even, $\ind D_{\Sigma_2,\phi_2}=\ind D_{\Sigma_2}-\ind\Conf_{\Sigma_2}=2n_2+\mu(E_2,F_2)$ is even, where $n_2$ is the number of interior marked points of $\Sigma_2$. By Remark~\ref{rem:disjoint-union-orientation}, the vertical direct-sum isomorphisms
\[
\det D_{\bullet}(\Sigma_1)\otimes \det D_{\Sigma_2,\phi_2} \xrightarrow{\;\cong\;} \det D_{\bullet}(\Sigma)
\]
preserve the orientations of $D_{\mathrm{norm}}$ and $D_{\mathrm{node}}$, and commuting the even-index factor $\det D_{\Sigma_2,\phi_2}$ past $\det\Cone(K)$ and $(\det\mathbb R_{\mathrm{res}})^{[-1]}$ introduces no Koszul signs. Hence the induced isomorphism $\det D_{\mathrm{wall}}(\Sigma_1)\otimes \det D_{\Sigma_2,\phi_2}\xrightarrow{\;\cong\;}\det D_{\mathrm{wall}}(\Sigma)$ preserves both wall orientations.
\end{proof}

\begin{prop}[Reduction to local standard models]\label{prop:wall-model-reduction}
For any wall problem on a bordered Riemann surface $\Sigma$ with no boundary marked points, the relative sign between the smoothing boundary orientation and the negative chamber orientation of $\det D_{\mathrm{wall}}$ (in either the holomorphic or anti-holomorphic convention) depends only on whether the wall problem is elliptic, connected hyperbolic, or disconnected hyperbolic.
\end{prop}

\begin{proof}
\begin{figure}[htbp]
\centering
\resizebox{\linewidth}{!}{%
\begingroup
\definecolor{neckteal}{RGB}{0,112,101}
\colorlet{neckfill}{neckteal!12}
\begin{tikzpicture}[x=.88cm,y=.88cm,
  font=\small,line cap=round,line join=round,
  every node/.style={inner sep=2pt},
  rim/.style={draw=black,line width=.8pt},
  contour/.style={draw=black!60,line width=.6pt},
  seam/.style={draw=neckteal,line width=.7pt},
  back seam/.style={seam,dash pattern=on 1.7pt off 1.7pt},
  latitude/.style={draw=neckteal!55!black,line width=.35pt},
  flow/.style={-{Latex[length=1.7mm,width=1.2mm]},draw=black!65,line width=.55pt},
  wall/.style={circle,fill=black,inner sep=1.15pt},
  interior mark/.style={circle,fill=neckteal,inner sep=1.25pt},
  detail/.style={font=\footnotesize},
  operation/.style={font=\scriptsize,align=center,text=black!75}
]

\def\neckhandle#1#2{%
  \begin{scope}[shift={(#1,#2)}]
    \fill[white]
      (-.256,-.1005) .. controls (-.11,.065) and (.11,.065) .. (.256,-.1005)
      .. controls (.10,-.155) and (-.10,-.155) .. cycle;
    \draw[contour] (-.60,.12)
      .. controls (-.22,-.23) and (.22,-.23) .. (.60,.12);
    \draw[contour] (-.256,-.1005)
      .. controls (-.11,.065) and (.11,.065) .. (.256,-.1005);
  \end{scope}%
}

\def\initialsurface{%
  \path[fill=black!2]
    (.40,2.95) -- (1.00,2.95)
    .. controls (1.66,2.95) and (1.86,3.25) .. (2.55,3.25)
    .. controls (3.55,3.25) and (4.25,2.55) .. (4.25,1.65)
    .. controls (4.25,.75) and (3.65,.15) .. (2.85,.15)
    .. controls (2.05,.15) and (1.80,.23) .. (1.15,.23)
    -- (.40,.23) -- (.40,1.07) -- (.98,1.07)
    .. controls (1.92,1.07) and (1.92,2.11) .. (1.00,2.11)
    -- (.40,2.11) -- cycle;
  \draw[contour] (.40,2.95) -- (1.00,2.95)
    .. controls (1.66,2.95) and (1.86,3.25) .. (2.55,3.25)
    .. controls (3.55,3.25) and (4.25,2.55) .. (4.25,1.65)
    .. controls (4.25,.75) and (3.65,.15) .. (2.85,.15)
    .. controls (2.05,.15) and (1.80,.23) .. (1.15,.23)
    -- (.40,.23);
  \draw[contour] (.40,1.07) -- (.98,1.07)
    .. controls (1.92,1.07) and (1.92,2.11) .. (1.00,2.11)
    -- (.40,2.11);
  \filldraw[fill=white,rim] (.40,2.53) ellipse[x radius=.20,y radius=.42];
  \filldraw[fill=white,rim] (.40,.65) ellipse[x radius=.20,y radius=.42];
  \neckhandle{2.55}{1.65}
}

\def\collar#1{%
  \begin{scope}[shift={(.40,#1)}]
    \fill[neckfill]
      (0,.42) -- (.64,.42)
      arc[start angle=90,end angle=-90,x radius=.16,y radius=.42]
      -- (0,-.42) -- cycle;
    \draw[back seam] (.64,.42)
      arc[start angle=90,end angle=270,x radius=.16,y radius=.42];
    \draw[seam] (.64,.42)
      arc[start angle=90,end angle=-90,x radius=.16,y radius=.42];
    \draw[contour] (0,.42) -- (.64,.42);
    \draw[contour] (0,-.42) -- (.64,-.42);
    \filldraw[fill=white,rim] (0,0) ellipse[x radius=.20,y radius=.42];
  \end{scope}%
}

\def\nodalcap#1{%
  \begin{scope}[shift={(.40,#1)}]
    \path[fill=neckfill,contour]
      (0,.42) .. controls (.55,.42) and (.88,.20) .. (1.10,0)
      .. controls (.88,-.20) and (.55,-.42) .. (0,-.42);
    \filldraw[fill=white,rim] (0,0) ellipse[x radius=.20,y radius=.42];
  \end{scope}%
}

\def\modeldisc#1#2{%
  \begin{scope}[shift={(.40,#1)}]
    \path[fill=neckfill,contour]
      (0,.42) .. controls (.55,.42) and (1.25,.42) .. (1.25,0)
      .. controls (1.25,-.42) and (.55,-.42) .. (0,-.42);
    \filldraw[fill=white,rim] (0,0) ellipse[x radius=.20,y radius=.42];
    \node[interior mark] at (.99,0) {};
    \node[detail,anchor=east,text=neckteal] at (.91,0) {$#2$};
  \end{scope}%
}

\def\twopoints#1{%
  \node[wall] at ({.40+.20*cos(145)},{#1+.42*sin(145)}) {};
  \node[anchor=east] at ({.34+.20*cos(145)},{#1+.42*sin(145)}) {$p_1$};
  \node[wall] at ({.40+.20*cos(215)},{#1+.42*sin(215)}) {};
  \node[anchor=east] at ({.34+.20*cos(215)},{#1+.42*sin(215)}) {$p_2$};
}
\def\onepoint#1#2{%
  \node[wall] at (.20,#1) {};
  \node[anchor=east] at (.14,#1) {$#2$};
}
\def\steps{%
  \draw[flow] (4.55,1.65) -- (5.40,1.65);
  \node[operation,anchor=south] at (4.975,1.83) {stretch};
  \draw[flow] (10.70,1.65) -- (11.65,1.65);
  \node[operation,anchor=south] at (11.175,1.83) {normalise};
  \node[operation,anchor=north] at (11.175,1.47) {discard\\remainder};
}

\node at (2.18,13.00) {Domain $\Sigma$};
\node at (8.05,13.00) {Nodal limit $\Sigma_{\mathrm{lim}}$};
\node at (13.33,13.00) {Model $\Sigma_{\mathrm{mod}}$};

\begin{scope}[shift={(0,8.50)}]
  \node[anchor=west,font=\small\bfseries] at (-.18,3.76)
    {Disconnected hyperbolic};
  \initialsurface
  \collar{2.53}\collar{.65}
  \onepoint{2.53}{p_1}\onepoint{.65}{p_2}
  \node[detail,text=neckteal] at (.76,3.24) {$U_1$};
  \node[detail,text=neckteal] at (.76,-.06) {$U_2$};
  \steps

  \begin{scope}[shift={(5.90,0)}]
    \nodalcap{2.53}\nodalcap{.65}
    \onepoint{2.53}{p_1}\onepoint{.65}{p_2}
    \path[fill=black!2,contour]
      (1.50,2.53) .. controls (1.95,2.95) and (2.35,3.25) .. (2.95,3.25)
      .. controls (3.85,3.25) and (4.45,2.55) .. (4.45,1.65)
      .. controls (4.45,.75) and (3.85,.15) .. (2.95,.15)
      .. controls (2.35,.15) and (1.95,.35) .. (1.50,.65)
      .. controls (2.05,1.00) and (2.35,1.15) .. (2.35,1.59)
      .. controls (2.35,2.03) and (2.05,2.18) .. cycle;
    \neckhandle{3.25}{1.65}
    \node[detail,text=black!65] at (3.05,2.55) {$\Sigma_{\mathrm{rem}}$};
    \node[interior mark] at (1.50,2.53) {};
    \node[interior mark] at (1.50,.65) {};
    \node[detail,anchor=south,text=neckteal] at (1.50,2.68) {$n_1$};
    \node[detail,anchor=north,text=neckteal] at (1.50,.50) {$n_2$};
  \end{scope}

  \begin{scope}[shift={(12.48,0)}]
    \modeldisc{2.53}{p_{0,1}}\modeldisc{.65}{p_{0,2}}
    \onepoint{2.53}{p_1}\onepoint{.65}{p_2}
    \node[detail] at (1.00,1.65) {$\mathbb D_1\amalg\mathbb D_2$};
  \end{scope}
\end{scope}

\begin{scope}[shift={(0,4.25)}]
  \node[anchor=west,font=\small\bfseries] at (-.18,3.76)
    {Connected hyperbolic};
  \initialsurface
  \collar{2.53}\twopoints{2.53}
  \node[detail,text=neckteal] at (.76,3.24) {$U$};
  \steps

  \begin{scope}[shift={(5.90,0)}]
    \nodalcap{2.53}\twopoints{2.53}
    \path[fill=black!2]
      (1.50,2.53) .. controls (1.95,2.95) and (2.35,3.25) .. (2.95,3.25)
      .. controls (3.85,3.25) and (4.45,2.55) .. (4.45,1.65)
      .. controls (4.45,.75) and (3.85,.15) .. (2.95,.15)
      .. controls (2.15,.15) and (1.85,.23) .. (1.20,.23)
      -- (.60,.23) -- (.60,1.07) -- (1.15,1.07)
      .. controls (2.05,1.07) and (2.15,2.05) .. cycle;
    \draw[contour]
      (1.50,2.53) .. controls (1.95,2.95) and (2.35,3.25) .. (2.95,3.25)
      .. controls (3.85,3.25) and (4.45,2.55) .. (4.45,1.65)
      .. controls (4.45,.75) and (3.85,.15) .. (2.95,.15)
      .. controls (2.15,.15) and (1.85,.23) .. (1.20,.23) -- (.60,.23);
    \draw[contour] (.60,1.07) -- (1.15,1.07)
      .. controls (2.05,1.07) and (2.15,2.05) .. (1.50,2.53);
    \filldraw[fill=white,rim] (.60,.65) ellipse[x radius=.20,y radius=.42];
    \neckhandle{3.25}{1.65}
    \node[detail,text=black!65] at (3.05,2.55) {$\Sigma_{\mathrm{rem}}$};
    \node[interior mark] at (1.50,2.53) {};
    \node[detail,anchor=south,text=neckteal] at (1.50,2.68) {$n$};
  \end{scope}

  \begin{scope}[shift={(12.48,0)}]
    \modeldisc{1.65}{p_0}\twopoints{1.65}
    \node[detail] at (1.00,.92) {$\mathbb D$};
  \end{scope}
\end{scope}

\begin{scope}
  \node[anchor=west,font=\small\bfseries] at (-.18,3.76) {Elliptic};
  \initialsurface
  \filldraw[fill=neckfill,seam] (3.62,1.65) circle[radius=.42];
  \node[wall] at (3.62,1.65) {};
  \node[anchor=west] at (3.67,1.65) {$q$};
  \node[detail,text=neckteal] at (3.62,2.30) {$U$};
  \steps

  \begin{scope}[shift={(5.90,0)}]
    \path[fill=black!2]
      (.40,2.95) -- (.95,2.95)
      .. controls (1.50,2.95) and (1.75,3.25) .. (2.35,3.25)
      .. controls (3.05,3.25) and (3.25,2.35) .. (3.55,1.65)
      .. controls (3.25,.95) and (3.05,.15) .. (2.35,.15)
      .. controls (1.75,.15) and (1.50,.23) .. (.95,.23)
      -- (.40,.23) -- (.40,1.07) -- (.92,1.07)
      .. controls (1.75,1.07) and (1.75,2.11) .. (.95,2.11)
      -- (.40,2.11) -- cycle;
    \draw[contour]
      (.40,2.95) -- (.95,2.95)
      .. controls (1.50,2.95) and (1.75,3.25) .. (2.35,3.25)
      .. controls (3.05,3.25) and (3.25,2.35) .. (3.55,1.65)
      .. controls (3.25,.95) and (3.05,.15) .. (2.35,.15)
      .. controls (1.75,.15) and (1.50,.23) .. (.95,.23) -- (.40,.23);
    \draw[contour] (.40,1.07) -- (.92,1.07)
      .. controls (1.75,1.07) and (1.75,2.11) .. (.95,2.11) -- (.40,2.11);
    \filldraw[fill=white,rim] (.40,2.53) ellipse[x radius=.20,y radius=.42];
    \filldraw[fill=white,rim] (.40,.65) ellipse[x radius=.20,y radius=.42];
    \neckhandle{2.25}{1.65}
    \node[detail,text=black!65] at (2.30,2.55) {$\Sigma_{\mathrm{rem}}$};
    \filldraw[fill=neckfill,contour] (4.10,1.65) circle[radius=.55];
    \draw[latitude,dash pattern=on 1.3pt off 1.5pt] (3.55,1.65)
      arc[start angle=180,end angle=0,x radius=.55,y radius=.16];
    \draw[latitude] (3.55,1.65)
      arc[start angle=180,end angle=360,x radius=.55,y radius=.16];
    \node[interior mark] at (3.55,1.65) {};
    \node[detail,anchor=south,text=neckteal] at (3.55,1.87) {$n$};
    \node[wall] at (4.14,1.36) {};
    \node[anchor=west] at (4.18,1.36) {$q$};
  \end{scope}

  \begin{scope}[shift={(12.48,0)}]
    \filldraw[fill=neckfill,contour] (1.00,1.65) circle[radius=.82];
    \draw[latitude,dash pattern=on 1.3pt off 1.5pt] (.18,1.65)
      arc[start angle=180,end angle=0,x radius=.82,y radius=.20];
    \draw[latitude] (.18,1.65)
      arc[start angle=180,end angle=360,x radius=.82,y radius=.20];
    \node[interior mark] at (.69,1.31) {};
    \node[detail,anchor=north,text=neckteal] at (.69,1.26) {$p_0$};
    \node[wall] at (1.30,1.31) {};
    \node[anchor=north] at (1.30,1.26) {$q$};
    \node[detail] at (1.00,.57) {$\mathbb{CP}^1$};
  \end{scope}
\end{scope}

\filldraw[fill=neckfill,seam] (1.92,-.61) rectangle (2.18,-.40);
\node[detail,anchor=west] at (2.28,-.505) {retained neighbourhood / model};
\end{tikzpicture}
\endgroup%
}
\caption{Reduction to the local standard models. Stretching along the inner boundary of the shaded neighbourhood $U$ produces the separating interior nodes $n_i$ or $n$. Normalising these nodes and discarding $\Sigma_{\mathrm{rem}}$ leaves the shaded components, with the retained preimages of the nodes marked as $p_{0,i}$ or $p_0$. }
\label{fig:wall-model-reduction}
\end{figure}
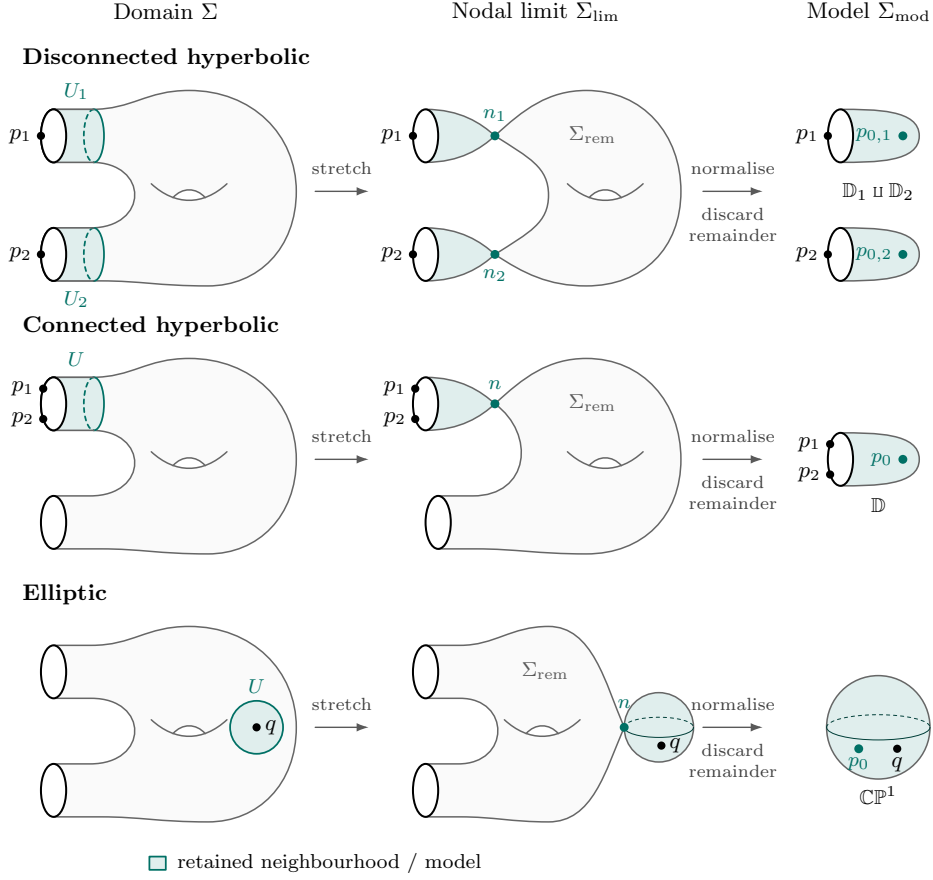

Fix compatible orientation data $(\tau,\tau_{\mathrm{node}})$ for the wall problem on $\Sigma$. Since the relative sign between $\mathsf o^{\mathrm{sm}}_\partial$ and $\mathsf o^-_\partial$ is locally constant in continuous families of wall problems, we can reduce to a standard local model by the following sequence of homotopies and excisions:
\begin{enumerate}[(i)]
\item Apply the linear homotopy $s\mapsto s\phi$ ($s\in[0,1]$) to homotope the coupling map $\phi$ to $0$, then smooth any hyperbolic boundary nodes (away from the wall points).
\item Choose a neighbourhood $U\subset\Sigma$ of the wall points whose boundary in $\Sigma$ is a union of smooth interior circles disjoint from all nodes and marked points of $\Sigma$:
\begin{itemize}
\item In the elliptic case, take $U$ to be a smooth disc neighbourhood of the interior point $q$.
\item In the connected hyperbolic case, take $U$ to be a collar neighbourhood of the boundary component of $\Sigma$ containing $p_1$ and $p_2$.
\item In the disconnected hyperbolic case, take $U$ to be the union of collar neighbourhoods of the two boundary components of $\Sigma$ containing $p_1$ and $p_2$.
\end{itemize}
\item Using the compatible orientation data, homotope the bundle pair and Cauchy--Riemann operator over $U$, through doubleable operators, to the standard data $(\mathbb C^3,\mathbb R^3,\dbar)$, where $F_q=\mathbb R^3$ in the elliptic case and $I=\id_{\mathbb C^3}$ in the hyperbolic cases.
\item Stretch the neck of $\Sigma$ along the interior boundary circle(s) $\partial U\setminus\partial\Sigma$, keeping the data on $U$ in the standard form $(\mathbb C^3,\mathbb R^3,\dbar)$ with $\phi=0$. In the nodal limit (Figure~\ref{fig:wall-model-reduction}), the domain is a surface $\Sigma_{\mathrm{lim}}$ obtained by attaching a special component (or pair of components) $\Sigma_{\mathrm{mod}}$ to a complementary surface $\Sigma_{\mathrm{rem}}$ along one or two separating interior nodes.
\item By Lemma~\ref{lem:wall-excision}, normalising the separating interior node(s) between $\Sigma_{\mathrm{mod}}$ and $\Sigma_{\mathrm{rem}}$ and discarding $\Sigma_{\mathrm{rem}}$ leaves the relative sign between the smoothing boundary orientation and the negative chamber orientation unchanged.
\end{enumerate}

After this excision, every wall problem of a given type reduces to the corresponding standard model problem on $\Sigma_{\mathrm{mod}}$ with bundle pair $(\mathbb C^3,\mathbb R^3)$, operator $\dbar$, coupling map $\phi=0$, and $I=\id_{\mathbb C^3}$ or $F_q=\mathbb R^3$, where:
\begin{itemize}
\item \emph{(Elliptic model)} $\Sigma_{\mathrm{mod}}$ is a sphere $\mathbb{CP}^1$ with one interior marked point $p_0$ and the wall point $q\neq p_0$;
\item \emph{(Connected hyperbolic model)} $\Sigma_{\mathrm{mod}}$ is a disc $\mathbb D$ with one interior marked point $p_0$ and distinct boundary points $p_1,p_2\in\partial\mathbb D$;
\item \emph{(Disconnected hyperbolic model)} $\Sigma_{\mathrm{mod}} = \mathbb D_1\amalg\mathbb D_2$ is a disjoint union of two discs, each with one interior marked point, with $p_1\in\partial\mathbb D_1$ and $p_2\in\partial\mathbb D_2$.
\end{itemize}
Since the relative sign is completely determined by this standard model in each of the three cases, it depends only on the type of the wall problem.

\end{proof}

We will calculate the relationship between wall orientations in the three cases using enlightening examples in Sections~\ref{sec:disconnected-crossing}--\ref{connected-hyperbolic-section}. It is also possible  to calculate the comparison in all three models directly (this exercise is well within the capability of current AI models); we give the elliptic case below. Section~\ref{sec:ramified-covers} gives a second calculation of both hyperbolic cases, so each of the three comparisons is computed in two independent ways.

\begin{lemma}\label{lem:elliptic comparison} Let $(E,F,D_{\Sigma},\phi, F_q)$ be an elliptic wall problem. Then the smoothing boundary and negative chamber orientations of $D_{\rm wall}$ agree, both in the holomorphic and anti-holomorphic conventions.

\end{lemma}
\begin{proof}
We relate the smoothing boundary orientation to the orientation of $D_{\Sigma_q,0}$ using the elliptic boundary node and smoothing triples, relate the negative chamber orientation to the orientation of $D_{\Sigma,0}$ using the incidence triple, and compare the two using the marked-point triple of Remark~\ref{rem:marked-point-orientation}.

By Proposition~\ref{prop:wall-model-reduction}, we may assume that $\Sigma$ has no boundary and $\phi=0$. Then $\Sigma_{\rm node}$ has a single elliptic node $q$, with closed normalisation $\Sigma_q$, and by Definition~\ref{def:holomorphic-orientation} the elliptic boundary node exact triple $D_{\rm node}\to D_{\Sigma_q,0}\to Q_{q,0}$, where $Q_{q,0}=\bigl(\mathbb R(\nu_q)\xrightarrow{0}JF_q\bigr)$, gives an oriented isomorphism
\begin{equation}\label{eq:1}\det D_{\Sigma_q,0} \cong \det D_{\rm node}\otimes \det \mathbb R(\nu_q)^{[-1]}\otimes \det JF_q,\end{equation}
with $\mathbb R(\nu_q)$ and $JF_q$ oriented as in Definition~\ref{def:Q orientation}. The intermediate complex $D'_{\rm wall}$ of the smoothing triple~\eqref{eq:smoothing-triangle} satisfies $D'_{\rm wall}/D_{\rm node}=\mathbb R(\nu_q)$, identified with $\mathbb R_{\rm res}$ by sending $\nu_q$ to $2r_{\rm out}$, and $D_{\Sigma_q,0}/D'_{\rm wall}=JF_q$. So applying the smoothing boundary orientation to the first two factors of~\eqref{eq:1}, and the quasi-isomorphism $D'_{\rm wall}\simeq D_{\rm wall}$, gives an oriented isomorphism
\[\det D_{\Sigma_q,0}\cong\det D_{\rm wall}\otimes\det JF_q.\]
On the other hand, $\Cone(K)$ is the complex $T_q\Sigma\xrightarrow{0}JF_q$ in degrees $-1$ and $0$, and $JF_q$ carries the same orientation in Definitions~\ref{def:Q orientation} and~\ref{def:incidence-orientation}. With the negative chamber orientation, the incidence triple~\eqref{eq:wall-triangle} therefore gives an oriented isomorphism $\det D_{\rm wall}\otimes\det T_q\Sigma^{[-1]}\otimes\det JF_q\cong\det D_{\Sigma,0}$; see~\eqref{eq:negative-chamber-orientation}. Moving the even factor $\det T_q\Sigma^{[-1]}$, these combine to an isomorphism
\[\det D_{\Sigma,0}\cong \det D_{\rm wall}\otimes \det JF_q\otimes \det T_q\Sigma^{[-1]}\cong \det D_{\Sigma_q,0}\otimes \det T_q\Sigma^{[-1]},\]
which preserves orientations if and only if the negative chamber and smoothing boundary orientations agree. It is the additivity isomorphism of the triple $D_{\Sigma_q,0}\to D_{\Sigma,0}\to T_q\Sigma$ of Remark~\ref{rem:marked-point-orientation}, which preserves orientations by part~(i) of that remark. So the two orientations agree.

The anti-holomorphic case follows, as the anti-holomorphic correction sign~\eqref{eq:antiholomorphic-sign-index} is the same for $\Sigma_{\rm node}$ and $\Sigma$.
\end{proof}

\subsection{Generalisations: bundle pairs on \texorpdfstring{$\Sigma^{KN}$}{Sigma\textasciicircum KN} and capping ends}
\label{subsec:capping-and-generalisations}

In geometric applications—such as curves with strip ends asymptotic to Reeb chords or cylindrical ends with nonconstant asymptotic operators—the linearised Cauchy--Riemann problem does not naturally form a smooth bundle pair over the nodal surface $\Sigma$. Instead, the boundary conditions and asymptotic operators depend on the angular coordinates around the punctures and nodes. To accommodate these settings, we formulate the linearised problem over the real blowup space $\Sigma^{KN}$.

\subsubsection{Bundle pairs over \texorpdfstring{$\Sigma^{KN}$}{Sigma\textasciicircum KN} and weaker regularity}

Let $\Sigma$ be a bordered, possibly nodal Riemann surface with infinitesimal gluing data at its nodes, and let $\Sigma^{KN}$ be its real blowup space (Section~\ref{subsec:domain-topologies}). A \emph{generalised coupled bundle pair} $(E,F,D_\Sigma,\phi)$ over $\Sigma^{KN}$ consists of:
\begin{enumerate}[(i)]
\item a complex vector bundle $E\to\Sigma^{KN}$ of rank $3$;
\item a totally real, oriented subbundle $F\subset E|_{(\partial\Sigma)^{KN}}$ of real rank $3$;
\item a real Cauchy--Riemann operator $D_\Sigma$ and coupling morphism $\phi\colon T^{\log}\Sigma\to E$ with $\phi(T^{\log}\partial\Sigma)\subset F$ inducing a map of complexes $\Conf_{\Sigma}\to D_\Sigma$.
\end{enumerate}

In this setting $D_\Sigma$ and $(E,F)$ need not be doubleable.

At interior nodes and hyperbolic boundary nodes, the gluing of the angular circles or intervals in $\Sigma^{KN}$ ensures that $E$ and $F$ are continuous across the glued seam. Rather than requiring $E$, $F$, and the zero-order terms of $D_\Sigma$ to be smooth on $\Sigma$, we only require sufficient regularity on $\Sigma^{KN}$ near the nodes and ends for standard Fredholm and gluing analysis to hold.   In particular, over each end circle or interval in $\partial\Sigma^{KN}\setminus(\partial\Sigma)^{KN}$, the data $(E,F,D_\Sigma)$ converges to an asymptotic bundle pair and asymptotic operator. We assume throughout that these asymptotic operators are appropriately nondegenerate, so that the coupled deformation complex
\[
D_{\Sigma,\phi} := \Cone\bigl(\Conf_\Sigma \xrightarrow{\;\phi\;} D_\Sigma\bigr)
\]
is Fredholm, and its determinant line varies continuously in families that smooth nodes.

\subsubsection{Orientations without ends}

Suppose first that $\Sigma$ has no strip ends and no cylindrical ends, so that $(\partial\Sigma)^{KN}=\partial\Sigma^{KN}$ is a union of closed circles.

An orientation of $(E,F)$ is again a generator $\mathfrak o$ of the orientation module $\mathcal O_{(E,F)}$ of homotopy classes of oriented trivialisations of $F\to\partial\Sigma^{KN}$, subject to the sign relation of Definition~\ref{def:bundle-pair-orientation}. To orient $\det D_{\Sigma,\phi}$:
\begin{enumerate}[(i)]
\item Homotope the coupling map $\phi$ to $0$.
\item Smooth all  nodes of $\Sigma$ in a family $\Sigma_\epsilon$ ($\epsilon>0$), extending $(E,F,D_\Sigma)$ continuously across the smoothing family. Because the regularity at the nodes is sufficient for gluing analysis, the determinant lines form a continuous line bundle over this family.
\item For $\epsilon>0$, the domain $\Sigma_\epsilon$ is a smooth bordered surface without nodes or ends ($\Sigma_\epsilon^{KN}=\Sigma_\epsilon$).  Homotope $D_{\Sigma_\epsilon}$ to a doubleable operator. Then apply the holomorphic (or anti-holomorphic) orientation construction of Definition~\ref{def:holomorphic-orientation} (or Definition~\ref{def:antiholomorphic-orientation}), and transport the resulting orientation back to $\det D_{\Sigma,\phi}$.
\end{enumerate}

Because the wall modifications in $\mathfrak T_{\mathrm{inc}}$ and $\mathfrak T_{\mathrm{sm}}$ take place at smooth points of the normalised domain away from the nodes, the excision and neck-stretching reduction of Proposition~\ref{prop:wall-model-reduction} applies verbatim. Consequently, our three local wall orientation comparisons hold in this setting as well.

\subsubsection{Capping strip and cylindrical ends}

Now suppose that $\Sigma$ has a nonempty set of ends $e$ (boundary punctures/strip ends and interior punctures/cylindrical ends) where the bundle pair on $\Sigma^{KN}$ does not extend smoothly across the puncture in $\Sigma$. To orient $D_{\Sigma,\phi}$, we cap each end, following~\cite{BourgeoisMohnkeOrientations, EkholmEtnyreSullivanOrientations}.

For each end $e$, let $\Sigma_{\mathrm{cap},e}$ be a standard cap domain:
\begin{itemize}
\item For a strip end $e$, $\Sigma_{\mathrm{cap},e}$ is a disc with one boundary puncture, so $\Sigma_{\mathrm{cap},e}^{KN}$ has a strip-end interval $I_e\subset\partial\Sigma_{\mathrm{cap},e}^{KN}$.
\item For a cylindrical end $e$, $\Sigma_{\mathrm{cap},e}$ is a sphere with  one interior puncture, so $\Sigma_{\mathrm{cap},e}^{KN}$ has an end circle $S^1_e=\partial\Sigma_{\mathrm{cap},e}^{KN}$.
\end{itemize}
Choose a coupled bundle pair $(E_{\mathrm{cap},e},F_{\mathrm{cap},e},D_{\mathrm{cap},e},\phi_{\mathrm{cap},e})$ on $\Sigma_{\mathrm{cap},e}^{KN}$ whose asymptotic restriction to $I_e$ or $S^1_e$ matches that of $(E,F,D_\Sigma,\phi)$ at $e$ (with opposite end orientation), and write $D_e$ for its coupled deformation complex on $\Sigma_{\mathrm{cap},e}^{KN}$. In geometric situations, $\phi$ will generally be injective at $e$, so it is most practical to construct $E_{\mathrm{cap},e}$ as an extension of $T^{\log}\Sigma_{\mathrm{cap},e}$, ensuring that a suitable holomorphic coupling map $\phi_{\mathrm{cap}, e}$ exists.

Glue $\Sigma_{\mathrm{cap},e}$ to $\Sigma$ at each end $e$  to form a nodal surface $\hat\Sigma$ without ends (Figure~\ref{fig:Capped-domain}). On the real blowup $\hat\Sigma^{KN}$, the bundle pairs and operators glue along the common intervals $I_e$ and circles $S^1_e$ to define a capped coupled bundle pair $(\hat E,\hat F,D_{\hat\Sigma},\hat\phi)$ without ends.

\begin{figure}

\begingroup
\definecolor{capseam}{RGB}{0,112,91}
\begin{tikzpicture}[
  x=0.9cm,y=0.9cm,
  line cap=round,line join=round,
  silhouette/.style={draw=black,line width=0.8pt},
  boundary/.style={draw=black,line width=0.85pt},
  hidden boundary/.style={draw=black!55,line width=0.5pt,
    dash pattern=on 2.4pt off 2.2pt}
]
  \coordinate (e) at (4.95,0);
  \coordinate (cap mark) at (7.35,0);

  \def\lowerBoundary{
    (3.20,0.60)
    .. controls (4.02,0.60) and (4.40,0.26) .. (e)
    .. controls (4.40,-0.26) and (4.02,-0.60) .. (3.20,-0.60)
    .. controls (2.36,-0.65) and (1.32,-0.65) .. (0.55,-0.57)
    .. controls (-0.18,-0.49) and (-0.33,0.20) .. (0.39,0.49)
    .. controls (0.95,0.716) and (2.00,0.60) .. (3.20,0.60)
    -- cycle
  }
  \def\capBoundary{
    (e)
    .. controls (5.62,0.32) and (6.20,1.30) .. (7.35,1.30)
    .. controls (8.068,1.30) and (8.65,0.718) .. (8.65,0)
    .. controls (8.65,-0.718) and (8.068,-1.30) .. (7.35,-1.30)
    .. controls (6.20,-1.30) and (5.62,-0.32) .. (e)
    -- cycle
  }

  \def\uprightLeft{
    (0.55,5.20)
    .. controls (0.05,4.57) and (-0.23,4.08) .. (-0.18,3.47)
    .. controls (-0.13,2.75) and (0.38,2.18) .. (0.87,1.92)
    .. controls (1.16,1.766) and (1.45,1.34) .. (1.39,0.96)
    .. controls (1.33,0.58) and (1.08,0.29) .. (0.75,0.19)
  }
  \def\uprightRight{
    (3.00,0.17)
    .. controls (2.69,0.16) and (2.42,0.29) .. (2.28,0.48)
    .. controls (2.06,0.779) and (1.98,1.27) .. (2.16,1.60)
    .. controls (2.55,2.315) and (3.15,3.04) .. (3.20,3.48)
    .. controls (3.27,4.096) and (2.99,4.83) .. (2.55,5.20)
  }
  \def\uprightPatch{\uprightLeft -- \uprightRight -- cycle}

  \path[fill=black!2,boundary] \lowerBoundary;
  \path[fill=black!2,boundary] \capBoundary;

  \path[fill=black!2] \uprightPatch;
  \begin{scope}
    \clip \uprightPatch;
    \draw[hidden boundary] \lowerBoundary;
  \end{scope}
  \draw[silhouette] \uprightLeft;
  \draw[silhouette] \uprightRight;

  \path[fill=white] (1.55,5.20) ellipse[x radius=1.00,y radius=0.30];
  \draw[draw=black,line width=0.65pt]
    (2.55,5.20) arc[start angle=0,end angle=180,x radius=1.00,y radius=0.30];
  \draw[boundary]
    (0.55,5.20) arc[start angle=180,end angle=360,x radius=1.00,y radius=0.30];

  \draw[silhouette]
    (1.00,3.89) .. controls (1.19,3.12) and (1.61,3.21) .. (2.13,3.88);
  \draw[silhouette]
    (1.242,3.409) .. controls (1.39,3.63) and (1.61,3.65) .. (1.74,3.56);

  \fill[capseam] (e) circle[radius=1.35pt];
  \node[text=capseam,anchor=north,inner sep=3pt] at (4.95,-0.13) {$e$};

  \fill (cap mark) circle[radius=1.5pt];
  \node[font=\large] at (7.35,0.58) {$\Sigma_{\mathrm{cap},e}$};
  \node[font=\Large] at (4.55,3.40) {$\hat\Sigma$};
\end{tikzpicture}
\endgroup

\caption{The nodal curve $\hat\Sigma$ is obtained by gluing $\Sigma_{\mathrm{cap},e}$ to the strip end $e$. Drawn with the nodal topology so that the strip end has finite length. }
\label{fig:Capped-domain}
\end{figure}
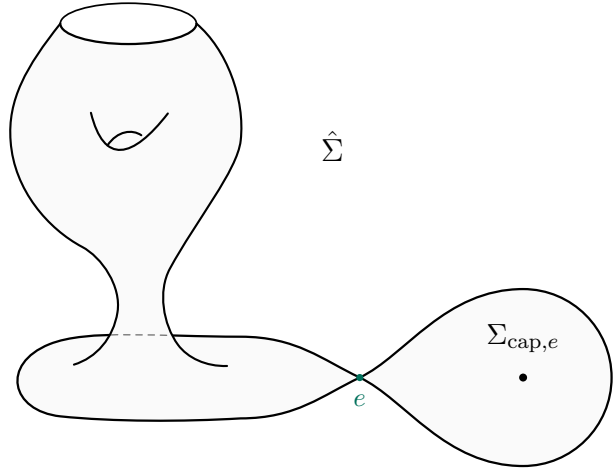

As in the gluing exact triple of Lemma~\ref{lem:gluing-orientation}, requiring sections on $\Sigma^{KN}$ and $\Sigma_{\mathrm{cap},e}^{KN}$ to match along $I_e$ or $S^1_e$ (or equivalently, smoothing the resulting nodes and applying gluing analysis) yields a natural exact triple of Fredholm complexes
\begin{equation}\label{eq:general-capping-triple}
0 \longrightarrow D_{\hat\Sigma,\hat\phi} \longrightarrow D_{\Sigma,\phi} \oplus \bigoplus_e D_e \longrightarrow \bigoplus_e C_e \longrightarrow 0,
\end{equation}
where $C_e$ is the finite-dimensional asymptotic matching complex at $e$, and thus an additivity isomorphism
\begin{equation}\label{eq:general-capped-orientation}
\Psi_{\mathrm{cap}}\colon \det D_{\hat\Sigma,\hat\phi}\otimes \bigotimes_e \det C_e \xrightarrow{\;\cong\;} \det D_{\Sigma,\phi}\otimes \bigotimes_e \det D_e.
\end{equation}

Since $\hat\Sigma$ has no ends, $\det D_{\hat\Sigma,\hat\phi}$ is oriented by an orientation datum $\hat{\mathfrak o}\in\mathcal O_{(\hat E,\hat F)}$, represented by an oriented trivialisation of $\hat F$ over $\partial\hat\Sigma^{KN}$. An \emph{orientation datum} for $(E,F,D_\Sigma,\phi)$ is therefore determined by:
\begin{enumerate}[(i)]
\item an ordering of the ends $e$, together with a choice of capping problem $D_e$ and orientations of $\det D_e$ and $\det C_e$ for each $e$; and
\item an orientation $\hat{\mathfrak o}\in\mathcal O_{(\hat E,\hat F)}$ of the capped bundle pair.
\end{enumerate}
Via~\eqref{eq:general-capped-orientation}, these choices orient $\det D_{\Sigma,\phi}$. In the presence of ends, there will generally not be a globally defined orientation of the determinant line over the moduli space of curves, but invariants can be defined using coefficients in the orientation local system.

Finally, suppose $(E,F,D_\Sigma,\phi)$ is part of a wall problem on $\Sigma^{KN}$ (with the wall points away from the ends). Capping the ends commutes with the local wall modifications at the crossing or incidence point. Holding the ordering of the ends, the capping problems $D_e$, and their orientations fixed on both sides of the wall, all capping determinant factors in~\eqref{eq:general-capped-orientation} appear in the same ordered positions in $\mathfrak T_{\mathrm{inc}}$ and $\mathfrak T_{\mathrm{sm}}$ and cancel in the comparison (using Lemma~\ref{lem:exact-squares}). The relative sign between the smoothing boundary orientation and the chamber orientation on $\det D_{\mathrm{wall}}$ therefore coincides with that of the capped problem on $\hat\Sigma^{KN}$, which is governed by Proposition~\ref{prop:wall-model-reduction}.

\section{Disconnected hyperbolic crossing}\label{sec:disconnected-crossing}

In this section, we consider an explicit example within $(\mathbb C^*)^3$, proving that at a disconnected hyperbolic crossing, the smoothing boundary
orientation agrees with the boundary orientation of the
positive-crossing chamber. A related calculation by Welschinger (using an orientation convention that reverses the sign comparison) appears in \cite[Propositions 1.7 and 1.9]{WelschingerOpenGW6}.  Section~\ref{sec:ramified-covers} gives a second calculation, treating the disconnected and connected hyperbolic cases simultaneously.

\subsection{Geometric setup and orientation data}
\label{subsec:disconnected-setup}

In this section, we work with the ambient space
\[X=(\mathbb C^*)^3\]
with structure at infinity induced by including $X$ in a toric compactification $\bar X$ with log structure determined by the toric boundary divisors. This log scheme $\bar X$ is a logarithmic Calabi--Yau 3-fold, with the following holomorphic volume form:
\[\Omega=-\frac{dz_1}{z_1}\wedge \frac{dz_2}{z_2}\wedge\frac{dz_3}{z_3}.\]
We use a $(\mathbb C^*)^3$-invariant trivialisation of its logarithmic tangent bundle $T^{\log}\bar X$, such that flowing by the constant vector field $(a,b,c)$ corresponds to $(z_1,z_2,z_3)\mapsto(e^{ia} z_1,e^{ib} z_2, e^c z_3)$. We treat the third coordinate differently to be compatible with the following Lagrangian.

\[L=\{(e^{i\theta_1},e^{i\theta_2},e^t)\in X \text{ such that }(\theta_1,\theta_2,t)\in \mathbb T^2\times \mathbb R \}\]

We can choose a K\"ahler form on $\bar X$ that vanishes on $L$.    This makes $L$  a special Lagrangian in the sense that $\Omega\rvert_L$ is a real volume form. The induced orientation on $L$ is
compatible with the trivialisation defined by $(\partial_{\theta_1},\partial_{\theta_2},\partial_t)$.

{\bf Spin structure on $L$:}

The  frame $(\partial_{\theta_1},\partial_{\theta_2},\partial_t)$ defines a Spin structure on $L$. We will call this the trivial Spin structure. This frame is the restriction of $\mathbb R^3\subset \mathbb C^3$ using our trivialisation of $TX$.

\subsubsection{Half-cylinder orientations}\label{half-cylinder-section}

Let $\mathbb D_0$ be the unit disc $\{z\in\mathbb C:|z|\leq 1\}$ with an interior marked point at $0$, and let $\mathbb D^*:=\mathbb D_0\setminus \{0\}$ be the punctured unit disc. The boundary $\partial \mathbb D_0$ is the unit circle, naturally parametrised by $z=e^{i\vartheta}$.

Given $(a,b)\in \mathbb Z^2\setminus \{0\}=H_1(\mathbb T^2)\setminus \{0\}$, consider the space of finite-energy holomorphic maps $u\colon (\mathbb D^*,\partial\mathbb D^*)\longrightarrow ((\mathbb C^*)^3,L)$ such that $u(\partial \mathbb D^*)\subset L$ represents $(a,b)$. Evaluation at $z=1$ determines a diffeomorphism of this space of maps with $L$. Explicitly,
\begin{equation}\label{eq:disconnected-half-cylinders}
u_{\theta_1,\theta_2,t}(z)=(e^{i\theta_1}z^a,e^{i\theta_2}z^b,e^t)\ .
\end{equation}

\begin{figure}[htbp]
\begingroup
\definecolor{tropicalL}{RGB}{0,112,91}
\definecolor{amoebafill}{RGB}{225,233,239}
\begin{tikzpicture}[
  x={(0.78cm,0.36cm)},y={(-0.78cm,0.36cm)},z={(0cm,1cm)},
  line cap=round,line join=round,
  curve/.style={draw=black,line width=0.85pt},
  ambient/.style={draw=black!28,line width=0.4pt,->},
  every node/.style={font=\small,inner sep=2pt}
]
  \path[use as bounding box] (-2.85cm,-2.00cm) rectangle (2.85cm,2.05cm);
  \draw[ambient] (0,0,0) -- (1.35,0,0);
  \draw[ambient] (0,0,0) -- (0,1.35,0);
  \node[text=black!50,anchor=south west,font=\footnotesize]
    at (1.35,0,0) {$\log|z_1|$};
  \node[text=black!50,anchor=south east,font=\footnotesize]
    at (0,1.35,0) {$\log|z_2|$};
  \draw[tropicalL,line width=0.95pt] (0,0,-1.65) -- (0,0,1.62);
  \draw[tropicalL,line width=0.7pt,densely dotted]
    (0,0,-1.65) -- (0,0,-1.90);
  \draw[tropicalL,line width=0.7pt,densely dotted]
    (0,0,1.62) -- (0,0,1.86);
  \node[tropicalL,anchor=west] at (0.10cm,1.57cm) {$L$};
  \draw[curve,->] (0,0,0) -- (-3.000000,-1.500000,0.000000);
  \fill (0,0,0) circle[radius=1.25pt];
  \node[anchor=west] at (0.10cm,-0.02cm) {$(0,0,t)$};
  \node[anchor=east] at (-0.75cm,-0.70cm) {$(-a,-b,0)$};
\end{tikzpicture}
\caption{Both $L$ and the half-cylinder $u$ project to lines using $(\log |z_1|,\log |z_2|,\log |z_3|)$. The image of the half-cylinder is a ray travelling in direction $(-a,-b,0)$.}
\endgroup

\end{figure}
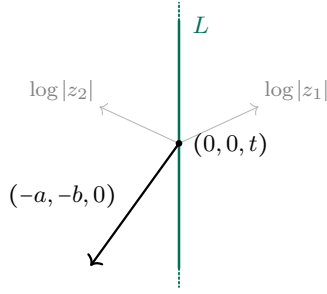

Each map $u=u_{\theta_1,\theta_2,t}$ extends uniquely to a map $u\colon\mathbb D_0\to \bar X$ compatible with the logarithmic structure, defining a coupled bundle pair $(u^*T^{\log}\bar X,u^*TL, D_{\mathbb D_0},du)$ over $\mathbb D_0$. Our trivialisation of $T^{\log}\bar X$ identifies $u^*(T^{\log}\bar X,TL)$ with the constant bundle pair $(\mathbb C^3,\mathbb R^3)$, where $\mathbb R^3\cong T_{u(1)}L$, and the linearisation $D_{\mathbb D_0}$ of the $\dbar$ operator at $u$ is the standard Cauchy--Riemann operator $\dbar$ on $(\mathbb C^3,\mathbb R^3)$. In the trivialisation of $(T^{\log}\mathbb D_0,T^{\log}\partial\mathbb D_0)$ provided by the rotation vector field $\partial_\vartheta = iz\partial_z$, the coupling map $du\colon T^{\log}\mathbb D_0\to u^*T^{\log}\bar X$ is the constant complex-linear map sending $\partial_\vartheta$ to $(a,b,0)\in\mathbb R^3\subset\mathbb C^3$.

 The standard complexes $\Conf_{\mathbb D_0}$ and $D_{\mathbb D_0}$ are surjective with kernels $\mathbb R\partial_\vartheta\cong T_1\partial\mathbb D_0$ and $\mathbb R^3\cong T_{u(1)}L$. The coupled deformation complex $D_{\mathbb D_0,du}=\Cone(\Conf_{\mathbb D_0}\xrightarrow{\,du\,}D_{\mathbb D_0})$ is therefore quasi-isomorphic to the finite-dimensional two-term complex in degrees $-1$ and $0$:
\begin{equation}\label{eq:half-cylinder-finite-cone}
\bigl(\mathbb R\partial_\vartheta \xrightarrow{\;du\;} \mathbb R^3\bigr)
\;\cong\;
\bigl(T_1\partial\mathbb D_0 \xrightarrow{\;du\;} T_{u(1)}L\bigr),
\qquad c\partial_\vartheta \longmapsto (ac,bc,0).
\end{equation}
Since $(a,b)\neq (0,0)$, $du$ is injective, so $D_{\mathbb D_0,du}$ has cohomology concentrated in degree $0$:
\[
H^0(D_{\mathbb D_0,du}) \cong \mathbb R^3/\mathbb R(a,b,0) \cong T_{u(1)}L / du(T_1\partial\mathbb D_0).
\]
 Lemma~\ref{lem:standard-punctured-disc} immediately implies the following.

\begin{lemma}\label{lem:half-cylinder-orientation}
The holomorphic orientation of $\det D_{\mathbb D_0,du} = \det H^0(D_{\mathbb D_0,du})$ using the trivial Spin structure on $L$ is the quotient orientation induced by the short exact sequence
\begin{equation}\label{eq:half-cylinder-quotient-sequence}
0 \longrightarrow T_1\partial\mathbb D_0 \xrightarrow{\;du\;} T_{u(1)}L \longrightarrow H^0(D_{\mathbb D_0,du}) \longrightarrow 0,
\end{equation}
where $T_1\partial\mathbb D_0$ is given the standard anticlockwise boundary orientation $\partial_\vartheta$ and $T_{u(1)}L\cong\mathbb R^3$ is given the orientation $(\partial_{\theta_1},\partial_{\theta_2},\partial_t)$. That is, an ordered basis $([v_1],[v_2])$ of $H^0(D_{\mathbb D_0,du})$ is positively oriented if and only if $\bigl(du(\partial_\vartheta),v_1,v_2\bigr)$ is a positively oriented basis of $T_{u(1)}L$.
\end{lemma}

The moduli space of once-punctured holomorphic discs with boundary on $L$ representing $(a,b)$ is the quotient of the space of maps~\eqref{eq:disconnected-half-cylinders} by the rotation group $\Aut(\mathbb D_0)\cong S^1$, and is thus the quotient of $L$ by the circle action of weight $(a,b)$. If $(a,b)$ is not a primitive vector, this moduli space is a stack; if $(a,b)$ is primitive, Lemma~\ref{lem:half-cylinder-orientation} gives an explicit oriented parametrisation as follows:

Choose $(c,d)\in\mathbb Z^2$ such that $(a,b)\wedge (c,d) = ad-bc = 1$. Then an oriented parametrisation of the moduli space is given by $(\theta,t)\in (\mathbb R/2\pi\mathbb Z)\times\mathbb R$, with
\begin{equation}\label{eq:half-cylinder-oriented-parametrisation}
u_{(\theta,t)}(z) := (e^{ic\theta}z^a, e^{id\theta}z^b, e^t)\ .
\end{equation}

\subsection{The families and their degeneration}
\label{subsec:disconnected-families}

\subsubsection{The normalised family and its crossing wall}
\label{subsec:disconnected-normalised-family}

Consider the moduli space of pairs of once-punctured holomorphic discs with boundary on $L$, representing $(1,0)$ and $(0,1)$. A parametrisation of this moduli space is given by  $(\theta_2,t_1,t_2,\theta_1)$ where the first disc is
\[z\mapsto (z,e^{i\theta_2},e^{t_1})\]
and the second disc is
\[z\mapsto (e^{i\theta_1}, z, e^{t_2}) \ .\]
The boundaries cross when $t_1=t_2$, at the point with coordinates $(\theta_1,\theta_2,t_1=t_2)$ on $L$. The positive-crossing chamber is $t_1\geq t_2$. We compare its boundary orientation with that of the glued family below.

\begin{figure}[htbp]
\begingroup
\definecolor{tropicalL}{RGB}{0,112,91}
\definecolor{amoebafill}{RGB}{225,233,239}
\begin{tikzpicture}[
  x={(0.78cm,0.36cm)},y={(-0.78cm,0.36cm)},z={(0cm,1cm)},
  line cap=round,line join=round,
  curve/.style={draw=black,line width=0.85pt},
  ambient/.style={draw=black!28,line width=0.4pt,->},
  every node/.style={font=\small,inner sep=2pt}
]
  \path[use as bounding box] (-2.85cm,-2.00cm) rectangle (2.85cm,2.05cm);
  \draw[ambient] (0,0,0) -- (1.35,0,0);
  \draw[ambient] (0,0,0) -- (0,1.35,0);
  \node[text=black!50,anchor=south west,font=\footnotesize]
    at (1.35,0,0) {$\log|z_1|$};
  \node[text=black!50,anchor=south east,font=\footnotesize]
    at (0,1.35,0) {$\log|z_2|$};
  \draw[tropicalL,line width=0.95pt] (0,0,-1.65) -- (0,0,1.62);
  \draw[tropicalL,line width=0.7pt,densely dotted]
    (0,0,-1.65) -- (0,0,-1.90);
  \draw[tropicalL,line width=0.7pt,densely dotted]
    (0,0,1.62) -- (0,0,1.86);
  \node[tropicalL,anchor=west] at (0.10cm,1.57cm) {$L$};
  \draw[curve,->] (0.000000,0.000000,0.420000) -- (-3.150000,0.000000,0.420000);
  \draw[curve,->] (0.000000,0.000000,-0.380000) -- (0.000000,-3.150000,-0.380000);
  \fill (0.000000,0.000000,0.420000) circle[radius=1.25pt];
  \fill (0.000000,0.000000,-0.380000) circle[radius=1.25pt];
  \node[anchor=west] at (0.000000,0.000000,0.420000) {$\ t_1$};
  \node[anchor=west] at (0.000000,0.000000,-0.380000) {$\ t_2$};
  \node[anchor=north] at (-1.72cm,-0.82cm) {$(-1,0,0)$};
  \node[anchor=north] at (1.72cm,-1.68cm) {$(0,-1,0)$};
\end{tikzpicture}
\endgroup
\caption{A curve in the positive-crossing chamber of the normalised family.}
\end{figure}
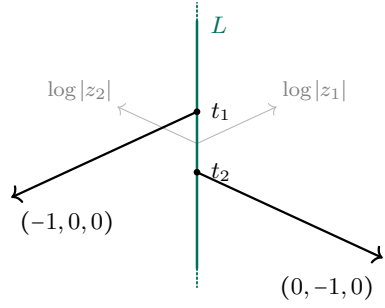

\subsubsection{The glued family and its two limits}\label{glued-section}
We can glue the two discs where they intersect. The glued moduli space is parametrised by $(r,\theta_1,\theta_2,t)\in (0,1]\times L $ as follows:
\begin{equation}\label{eq:disconnected-glued-family}z\mapsto \left( e^{i{\theta_1}}\frac{(z-r)}{rz-1}, e^{i\theta_2} \frac{(z+r)}{rz+1},e^t  \right)\end{equation}
The $r=1$ boundary of this moduli space is the  nodal curve consisting of two discs $z\mapsto(z,e^{i\theta_2},e^t)$ and $z\mapsto (e^{i\theta_1},z,e^t)$ glued at the point $(e^{i\theta_1},e^{i\theta_2},e^t)\in L$.

\begin{figure}[htbp]
\includegraphics{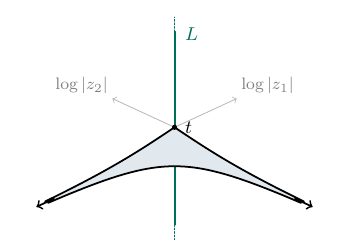}
\caption{The image of a member of the glued family with $r=0.4$.}
\end{figure}

At the other end, $r\to0$, the domain approaches a standard half-cylinder
attached to a three-punctured sphere (Figure~\ref{fig:tropical-limit}). These two limits play different
roles in the argument: we use the $r\to0$ limit to determine the
orientation of the glued family, then read off its smoothing boundary
orientation at $r=1$.

\begin{figure}[htbp]
\begingroup
\definecolor{tropicalL}{RGB}{0,112,91}
\begin{tikzpicture}[
  x={(0.78cm,0.30cm)},y={(-0.40cm,0.50cm)},z={(0cm,1cm)},
  line cap=round,line join=round,
  curve/.style={draw=black,line width=0.85pt},
  ambient/.style={draw=black!28,line width=0.4pt,->},
  every node/.style={font=\small,inner sep=2pt}
]
  \path[use as bounding box] (-3.05cm,-2.65cm) rectangle (2.00cm,2.00cm);

  \draw[ambient] (0,0,0) -- (1.35,0,0);
  \draw[ambient] (0,0,0) -- (0,1.35,0);
  \node[text=black!50,anchor=south west] at (1.35,0,0) {$\epsilon \log|z_1|$};
  \node[text=black!50,anchor=south east] at (0,1.35,0) {$\epsilon \log |z_2|$};

  \draw[tropicalL,line width=0.95pt] (0,0,-2.13) -- (0,0,1.62);
  \draw[tropicalL,line width=0.7pt,densely dotted]
    (0,0,-2.13) -- (0,0,-2.40);
  \draw[tropicalL,line width=0.7pt,densely dotted]
    (0,0,1.62) -- (0,0,1.86);
  \node[tropicalL,anchor=west] at (0.10cm,1.57cm) {$L$};

  \coordinate (attachment) at (0,0,0);
  \coordinate (vertex) at (-1,-1,0);

  \draw[curve] (attachment) -- (vertex);
  \draw[curve,->] (vertex) -- (-3.65,-1,0);
  \draw[white,line width=3.2pt] (-1,-1.82,0) -- (-1,-2.08,0);
  \draw[curve,->] (vertex) -- (-1,-3.65,0);
  \fill (attachment) circle[radius=1.25pt];
  \fill (vertex) circle[radius=1.25pt];

  \node[anchor=east] at (-0.30cm,-0.34cm) {$(-1,-1,0)$};
  \node[anchor=north] at (-1.63cm,-1.52cm) {$(-1,0,0)$};
  \node[anchor=west] at (0.46cm,-1.57cm) {$(0,-1,0)$};
\end{tikzpicture}
\endgroup
\caption{Rescaling coordinates as we approach the $r\to 0$ limit gives a tropical curve with a trivalent vertex. Around this trivalent vertex, our curve is modelled on a holomorphic map of a three-punctured sphere to $(\mathbb C^*)^3$, with one of the punctures joined to a half-cylinder leaving $L$.}
\label{fig:tropical-limit}
\end{figure}
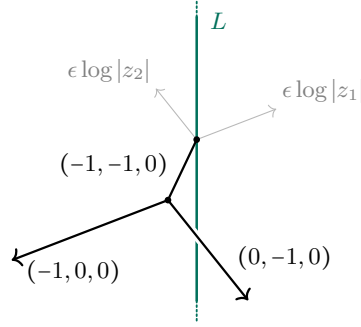

\subsection{Orientation comparison in the model}
\label{subsec:disconnected-model-orientation}

We orient the wall in two ways: as the boundary of the positive-crossing chamber, using the product orientation of the two half-cylinders, and as the smoothing boundary $r=1$ of the glued family, which we orient using its $r\to0$ limit.

\subsubsection*{Positive chamber orientation}

For the normalised family of Section~\ref{subsec:disconnected-normalised-family}, Lemma~\ref{lem:half-cylinder-orientation} and~\eqref{eq:half-cylinder-oriented-parametrisation} give the first half-cylinder (of slope $(1,0)$) the orientation $(\partial_{\theta_2},\partial_{t_1})$ and the second half-cylinder (of slope $(0,1)$) the orientation $(\partial_{t_2},\partial_{\theta_1})$. By Remark~\ref{rem:disjoint-union-orientation}, the normalised family carries the product orientation $(\partial_{\theta_2},\partial_{t_1},\partial_{t_2},\partial_{\theta_1})$.

At the crossing wall $t_1=t_2$, the ordered boundary tangents of the two strands are $w_1=\partial_{\theta_1}$ and $w_2=\partial_{\theta_2}$, and the displacement from strand~$2$ to strand~$1$ is $(t_1-t_2)\partial_t$. Since $(\partial_{\theta_1},\partial_{\theta_2},\partial_t)$ is the orientation of $L$, the positive-crossing chamber is $t_1\geq t_2$, with outward normal $\partial_{t_2}-\partial_{t_1}$ (pointing towards $t_1<t_2$). Placing the outward normal first as in Definition~\ref{def:boundary-orientations}(i), the positive chamber orientation $\mathsf o^+_\partial$ of the wall $t_1=t_2=t$ is given by $(\partial_{\theta_2},\partial_t,\partial_{\theta_1})$, which agrees with the orientation $(\partial_{\theta_1},\partial_{\theta_2},\partial_t)$ of $L$.

\subsubsection*{Orienting the glued family via the elliptic boundary node model}

For $r\in(0,1)$, the domain of the glued map~\eqref{eq:disconnected-glued-family} is the unit disc $\Sigma_r$ with interior marked points at $z=\pm r$, where $z_1$ and $z_2$ vanish. Using our trivialisation of $T^{\log}\bar X$, the pulled-back bundle pair is $(\mathbb C^3,\mathbb R^3)$, and $D_{\Sigma_r}=\dbar$ is surjective with kernel $\mathbb R^3$ spanned by $(\partial_{\theta_1},\partial_{\theta_2},\partial_t)$, while $\Conf_{\Sigma_r}$ is injective with cokernel spanned by the variation $r\partial_r$ of the marked points. So $D_{\Sigma_r,du}$ has cohomology only in degree $0$, namely the tangent space of the glued family, with frame $(r\partial_r,\partial_{\theta_1},\partial_{\theta_2},\partial_t)$.

In the coordinate $w=z/r$, $\Sigma_r$ is the disc $\{|w|\leq 1/r\}$ with marked points at $w=\pm1$, and the limit $r\to0$ stretches the neck along its boundary. The limit $\Sigma'$ is $\mathbb{CP}^1$ with marked points at $w=\pm1$ and an elliptic boundary node at $q=\infty$, with closed normalisation the three-punctured sphere $\Sigma_c=\mathbb{CP}^1_{-1,1,\infty}$. In logarithmic coordinates, $du$ converges to the logarithmic derivative of the three-punctured sphere $w\mapsto(w-1,w+1,1)$, which is injective at $q$. So $q$ is a stable elliptic node in the sense of Remark~\ref{rem:Q-geometry}. As $\Sigma'$ has no smooth boundary, Definition~\ref{def:holomorphic-orientation} orients $\det D_{\Sigma',du}$ directly using~\eqref{eq:en comparison}, without homotoping $du$ to $0$. Here $\partial_\vartheta=iw\partial_w$ is the anticlockwise rotation in $w$, so $v_q=du(\partial_\vartheta)=\partial_{\theta_1}+\partial_{\theta_2}$, and $du(\nu_q)=-Jv_q=\partial_{\rho_1}+\partial_{\rho_2}$, where $\rho_k=\log|z_k|$. This rescales $z_1$ and $z_2$ equally, which is the deformation $r\partial_r$ of the glued family.

As a sphere with three marked points has no infinitesimal automorphisms or moduli, $\Conf_{\Sigma_c}$ is an isomorphism, so $D_{\Sigma_c,du}$ is quasi-isomorphic to $\ker D_{\Sigma_c}=\mathbb C^3$, with its complex orientation. The triple~\eqref{eq:elliptic-orientation-triple} therefore has all its cohomology in degree $0$:
\[
0\longrightarrow H^0(D_{\Sigma',du})\longrightarrow\mathbb C^3\longrightarrow H^0(Q_{q,du})=J\mathbb R^3/\mathbb R(Jv_q)\longrightarrow0,
\]
so $H^0(D_{\Sigma',du})=\mathbb R(du(\nu_q))\oplus\mathbb R^3$. As $(v_q,\partial_{\theta_2},\partial_t)$ is an oriented basis of $\mathbb R^3$, Remark~\ref{rem:Q-geometry} orients $H^0(Q_{q,du})$ by $([J\partial_t],[J\partial_{\theta_2}])$. The basis $(-Jv_q,v_q,\partial_{\theta_2},\partial_t,J\partial_t,J\partial_{\theta_2})$ of $\mathbb C^3$ has the complex orientation. So $(du(\nu_q),v_q,\partial_{\theta_2},\partial_t)$, and hence $(r\partial_r,\partial_{\theta_1},\partial_{\theta_2},\partial_t)$, is the holomorphic orientation of the glued family.

\subsubsection*{Comparison at the nodal boundary}

At the hyperbolic boundary node $r=1$,  the local plumbing coordinate is, to leading order, $\epsilon = (1-r)^2$, so the outward smoothing direction $-\epsilon\partial_\epsilon$ of Remark~\ref{rem:residue-sign} and Definition~\ref{def:boundary-orientations}(ii) points in the direction of $\partial_r$. Removing $\partial_r$ from the first position gives the smoothing boundary orientation $(\partial_{\theta_1},\partial_{\theta_2},\partial_t)$. Under the identification of the common wall with $L$, the two wall orientations in the model are therefore
\begin{equation}\label{eq:disconnected-model-orientations}
\begin{aligned}
\text{smoothing boundary orientation ($\mathsf o^{\mathrm{sm}}_\partial$):}\quad & (\partial_{\theta_1},\partial_{\theta_2},\partial_t),\\
\text{positive chamber orientation ($\mathsf o^+_\partial$):}\quad & (\partial_{\theta_1},\partial_{\theta_2},\partial_t).
\end{aligned}
\end{equation}
They agree. Combining this model calculation with Proposition~\ref{prop:wall-model-reduction} and Remark~\ref{rem:antiholomorphic-skein-signs} yields the general orientation comparison for disconnected hyperbolic crossings, illustrated schematically in Figure~\ref{fig:disconnected-moduli}:

\begin{prop}[Disconnected hyperbolic orientation comparison]
\label{prop:disconnected-hyperbolic-orientation}
Consider a disconnected hyperbolic wall problem $(E,F,D_\Sigma,\phi,p_1,p_2,I)$.
\begin{enumerate}[(i)]
\item In the holomorphic orientation convention, the smoothing boundary orientation $\mathsf o^{\mathrm{sm}}_\partial$ of $\det D_{\mathrm{wall}}$ agrees with the positive chamber orientation $\mathsf o^+_\partial$.
\item In the anti-holomorphic orientation convention, the smoothing boundary orientation $\mathsf o^{\mathrm{sm,ah}}_\partial$ of $\det D_{\mathrm{wall}}$ agrees with the negative chamber orientation $\mathsf o^{-,\mathrm{ah}}_\partial$, matching the hyperbolic skein relation of~\cite[Definition 3.1]{worldsheetskein}.
\end{enumerate}
\end{prop}

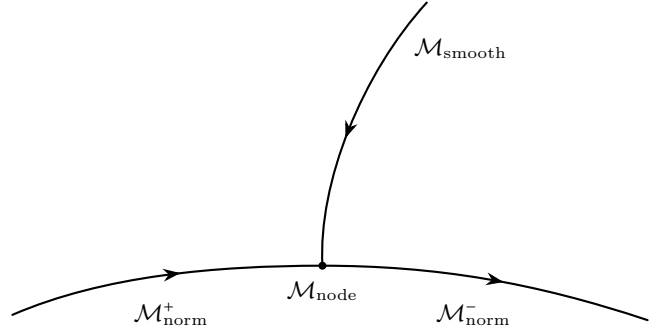
\begin{figure}[htbp]
\centering
\begin{tikzpicture}[
  x=1cm, y=1cm,
  line width=0.8pt,
  line cap=round,
  line join=round,
  every node/.style={font=\small, inner sep=2pt},
  oriented/.style={
    postaction={decorate},
    decoration={markings,
      mark=at position 0.55 with
        {\arrow{Stealth[length=2.3mm,width=1.7mm]}}}
  }
]
  \coordinate (wall) at (0,0);

  \draw[oriented] (-4.1,-0.65)
    .. controls (-2.9,-0.13) and (-1.4,0) .. (wall);
  \draw[oriented] (wall)
    .. controls (1.5,0) and (2.8,-0.22) .. (4.3,-0.72);

  \draw[oriented] (1.38,3.48)
    .. controls (0.53,2.55) and (-0.06,1.10) .. (wall);

  \fill (wall) circle[radius=1.5pt];
  \node[anchor=north] at (0,-0.12) {$\mathcal M_{\mathrm{node}}$};
  \node[anchor=west] at (1.13,2.87)
    {$\mathcal M_{\mathrm{smooth}}$};
  \node[anchor=north] at (-2,-0.4)
    {$\mathcal M_{\mathrm{norm}}^+$};
  \node[anchor=north] at (2,-0.4)
    {$\mathcal M_{\mathrm{norm}}^-$};
\end{tikzpicture}
\caption{A schematic picture of the moduli space near a curve with a hyperbolic node and its normalisation, corresponding to a disconnected hyperbolic wall problem, showing the orientation relationship when we use the holomorphic orientation convention.}
\label{fig:disconnected-moduli}
\end{figure}

\begin{remark}\label{tilted-L}We can perform similarly explicit analysis for a tilted version of $L$ defined by
\[L'=\{(e^{i\theta_1}e^{ta},e^{i\theta_2}e^{tb},e^t)\in X \text{ such that }(\theta_1,\theta_2,t)\in \mathbb T^2\times \mathbb R \}\ .\]
As all curves discussed above have the third coordinate constant, we can shift them by the transformation $(z_1,z_2,z_3)\mapsto (|z_3|^az_1,|z_3|^b z_2,z_3)$ to have boundary on the tilted totally real submanifold $L'$. All calculations of orientations of these moduli spaces of shifted curves are unchanged.
 \end{remark}

\section{Elliptic wall}\label{sec:elliptic-crossing}

At an elliptic incidence, the smoothing boundary orientation agrees
with the boundary orientation of the negative-incidence chamber, which has the
smaller intersection with a bounding chain $C$. We calculate this
sign in a family of half-cylinders degenerating to a plane with an elliptic node, confirming the general comparison of
Lemma~\ref{lem:elliptic comparison}.

Section~\ref{subsec:elliptic-setup} fixes the geometric and Spin data.
Section~\ref{subsec:elliptic-families} describes the two families, and
Section~\ref{subsec:elliptic-model-orientation} brings together the
bounding-chain and Spin corrections in their orientation comparison.

\subsection{Geometric setup and orientation data}
\label{subsec:elliptic-setup}

The ambient space in this case is
\[Y=\left(\mathbb C^2\setminus \{z_1z_2+1=0\}\right)\times  \mathbb C^*\]
with structure at infinity induced from the log structure on a toric compactification $\bar Y$ in which the function $z_1z_2$ is smooth where it is bounded and also has nonvanishing derivative away from $z_1=z_2=0$. (See \cite[Section 2]{tropological} for a discussion of the compactification required and the associated log structure.)

This structure at infinity is compatible with the following holomorphic volume form:
\[\Omega_Y=\frac{dz_1\wedge dz_2}{z_1z_2+1} \wedge\frac{dz_3}{z_3}.\]

Our Lagrangian is
\[L=\{(x+iy,x-iy, e^{i\theta})\in Y \text{ such that }(x,y,\theta)\in \mathbb R^2\times \mathbb T \}\ .\]

An appropriate K\"ahler form comes from the toric compactification $\bar Y$.  It is easy to choose such a K\"ahler form to vanish on $L$, as this toric manifold can be chosen to admit the anti-holomorphic involution $(z_1,z_2,z_3)\mapsto (\bar z_2, \bar z_1, \frac 1{\bar z_3})$, which fixes $L$.  Then $L$ is special Lagrangian in the sense that $\Omega_Y\rvert_L$ is a real volume form. The induced orientation is
compatible with the trivialisation defined by the frame
\[(\partial_{x},\partial_{y},\partial_\theta).\]
Our chosen Spin structure on $L$ is defined using the above frame.

\subsubsection{Comparison with the torus model}
\label{subsec:elliptic-torus-comparison}

To use the half-cylinder orientations of Section~\ref{half-cylinder-section}, compare this model with the torus model: the subset where the real part of $z_1z_2$ is positive embeds into $(\mathbb C^*)^3$ via the map
\begin{equation}\label{eq:elliptic-torus-comparison}
h:(z_1,z_2,z_3)\longmapsto(z_1,z_3,z_1z_2).
\end{equation} The image of this embedding is the subset where the real part of the third coordinate is positive.

Note that $h^*\Omega =\frac{dz_1}{z_1}\wedge \frac{dz_2}{z_2}\wedge \frac{dz_3}{z_3}$, which is a deformation of the above holomorphic volume form when the real part of $z_1z_2$ is positive.

In the region where the real part of $z_1z_2$ is positive, $L$ agrees
with the inverse image of the tilted totally real submanifold $L'$ from
Remark~\ref{tilted-L}, with $(a,b)=(1/2,0)$.

The chosen Spin structure, restricted to $(\mathbb R^2\setminus 0)\times \mathbb T=\mathbb R\times \mathbb T^2$, differs from the trivial Spin structure, which is given by the frame $(x\partial_x+y\partial_y, x\partial_y-y\partial_x, \partial_\theta)$. In particular, relative to the trivial Spin structure, our Spin structure rotates once around a circle where $x^2+y^2=c$, and therefore represents a noncontractible loop in $SO(3)$.

This difference of Spin structures supplies the sign correction when
we orient the half-cylinder family in
Section~\ref{subsec:elliptic-smoothing-model}.

\subsubsection{Tropical pictures}

The circle actions of weight $(1,-1,0)$ and $(0,0,1)$ preserve our Lagrangian $L$, and the corresponding $\mathbb C^*$ actions preserve the function $z_1z_2+1$ and our holomorphic volume form. There is a special Lagrangian torus fibration on $Y$ defined using $\log |z_1z_2+1|$ together with the moment map $\mu$ from the above $\mathbb T^2$ action~\cite[Section 2.2]{tropological}. This has singular fibres  consisting of $S^1$ times a sphere intersecting itself where $z_1=z_2=0$.

Below, we will consider curves that are locally orbits of the $\mathbb C^*$ action of weight $(1,-1,0)$, so their image under $(\log|z_1z_2+1|,\mu)$ will be one dimensional, if not actually straight in the symplectic affine structure on the base. Similarly, the singular locus has image a line, and $L$ has image a ray emanating from this critical line; neither need be straight. The fibre of $L$ where its image meets the critical line  is the circle $(0,0,e^{i\theta})$, whereas the fibre over other points is the torus $(re^{i\alpha},re^{-i\alpha}, e^{i\theta})$.  We will draw schematic pictures in which the singular locus, $L$, and our curves are all straight. These pictures can also be understood in a suitable tropical limit~\cite[Section 3]{tropological}.

\subsection{The families and their degeneration}
\label{subsec:elliptic-families}

\subsubsection{The half-cylinders and their elliptic limit}
\label{subsec:elliptic-half-cylinders}

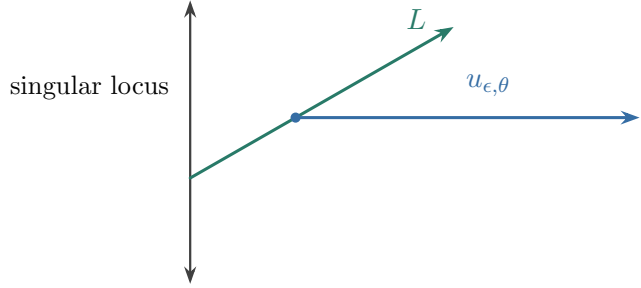
\begin{figure}[htbp]
  \centering

\begingroup
\usetikzlibrary{arrows.meta}
\definecolor{ellipticL}{HTML}{287A68}
\definecolor{ellipticCurve}{HTML}{356DA4}
\begin{tikzpicture}[
  x={(.82cm,.47cm)}, %
  y={(1cm,0cm)},    %
  z={(0cm,1cm)},    %
  line cap=round,line join=round,
  every node/.style={font=\large,inner sep=2pt},
  singular/.style={draw=black!75,line width=.85pt},
  lagrangian/.style={draw=ellipticL,line width=1.15pt},
  curve/.style={draw=ellipticCurve,line width=1.15pt},
  tip/.style={-{Stealth[length=2.5mm,width=1.8mm]}}
]
  \def\rhoE{1.70}
  \coordinate(O)at(0,0,0);
  \coordinate(P)at(\rhoE,0,0);

  \draw[singular,{Stealth[length=2.3mm,width=1.6mm]}-{Stealth[length=2.3mm,width=1.6mm]}]
    (0,0,-1.40)--(0,0,2.35);
  \draw[lagrangian,tip](O)--(4.25,0,0);

  \draw[curve,tip](P)--(\rhoE,4.55,0);
  \fill[ellipticCurve](P)circle[radius=1.9pt];

  \node[anchor=east,align=right,font=\normalsize]at(0,-.20,1.20)
    {singular locus};
  \node[text=ellipticL,above=5pt]at(3.65,0,0){$L$};
  \node[text=ellipticCurve,above=6pt]at(\rhoE,2.55,0){$u_{\epsilon,\theta}$};
\end{tikzpicture}
\endgroup
\caption{Image of the half-cylinder
$u_{\epsilon,\theta}$ in the torus-fibration base.
The singular locus is the image of $\{z_1=z_2=0\}$, and the
green ray is the image of $L$. Each point of the blue ray
represents a circle in the half-cylinder: its endpoint
represents the boundary circle on $L$. As $\epsilon\to0$, the endpoint
moves towards the intersection of $L$ with the singular locus, producing the elliptic nodal limit.
}
  \label{fig:elliptic-tropical-half-cylinder}
\end{figure}
Consider the following moduli space of holomorphic  half-cylinders  $u_{\epsilon,\theta}:\mathbb D^*\longrightarrow Y$
\begin{equation}\label{eq:elliptic-half-cylinder-family}u_{\epsilon, \theta}(z)=(\epsilon z^{-1},\epsilon z, e^{i\theta})\ ,\end{equation}
defined for $\epsilon>0$. In the limit $\epsilon\to 0$, there is an elliptic boundary in the moduli space, where the boundary component of the punctured disc shrinks to $0$, forming an elliptic node on $L$; see Figure~\ref{fig:elliptic-tropical-half-cylinder}. Changing coordinates by $wz=\epsilon$, we see the map
\[w\mapsto (w,\epsilon^2 w^{-1},e^{i\theta})\]
defined on the half-cylinder where $\abs w\geq \epsilon$. So, in the $\epsilon=0$ limit, we get the map
\begin{equation}\label{eq:elliptic-nodal-limit}w\mapsto (w,0,e^{i\theta})\end{equation}
defined for $w\in \mathbb C$, glued at $w=0$ to a constant disc $z\mapsto (0,0,e^{i\theta})$. The resulting map from the nodal curve is unstable, unless we include infinitesimal gluing information at the node.

\subsubsection{The normalised plane family}
\label{subsec:elliptic-crossing-family}

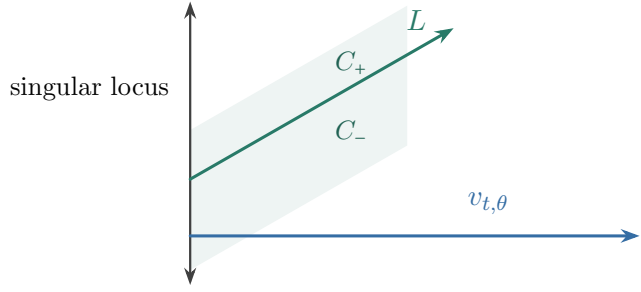
\begin{figure}[htbp]
  \centering
\begingroup
\usetikzlibrary{arrows.meta}
\definecolor{ellipticL}{HTML}{287A68}
\definecolor{ellipticCurve}{HTML}{356DA4}
\begin{tikzpicture}[
  x={(.82cm,.47cm)}, %
  y={(1cm,0cm)},    %
  z={(0cm,1cm)},    %
  line cap=round,line join=round,
  every node/.style={font=\large,inner sep=2pt},
  singular/.style={draw=black!75,line width=.85pt},
  lagrangian/.style={draw=ellipticL,line width=1.15pt},
  curve/.style={draw=ellipticCurve,line width=1.15pt},
  tip/.style={-{Stealth[length=2.5mm,width=1.8mm]}}
]
  \def\tValue{-.75} %
  \coordinate(O)at(0,0,0);
  \coordinate(P)at(0,0,\tValue);

\fill[ellipticL!8]
  (0,0,-1.20)--(3.5,0,-1.20)--
  (3.5,0,.65)--(0,0,.65)--cycle;

  \draw[singular,{Stealth[length=2.3mm,width=1.6mm]}-{Stealth[length=2.3mm,width=1.6mm]}]
    (0,0,-1.40)--(0,0,2.35);
  \draw[lagrangian,tip](O)--(4.25,0,0);

  \draw[curve,tip](P)--(0,5.944,\tValue);

  \node[anchor=east,align=right,font=\normalsize]at(0,-.20,1.20)
    {singular locus};
  \node[text=ellipticL,above=5pt]at(3.65,0,0){$L$};
  \node[text=ellipticCurve,above=6pt]at(0,3.944,\tValue){$v_{t,\theta}$};
  \node[text=ellipticL!80!black,font=\normalsize]
  at(2.6,0,.32){$C_+$};
\node[text=ellipticL!80!black,font=\normalsize]
  at(2.6,0,-.60){$C_-$};
\end{tikzpicture}
\endgroup

  \caption{The normalised plane $v_{t,\theta}$, shown for
$t<0$ in the same projection as
Figure~\ref{fig:elliptic-tropical-half-cylinder}. The pale green regions depict the bounding chains $C_+$ and
$C_-$ for $L$; see Section~\ref{subsec:elliptic-bounding-chain}. The endpoint of $v_{t,\theta}$ on the singular locus represents the smooth interior point $w=0$, and the unique intersection of this plane with $C_-$ of sign $-1$. The remaining points of the blue ray represent
circles in the plane. Increasing $t$ moves the endpoint upwards
along the singular locus so that it intersects $C_+$ with sign $1$ for $t>0$. At $t=0$, it meets $L$.
}
  \label{fig:elliptic-tropical-normalised}
\end{figure}

These holomorphic planes sit inside a moduli space of bounded-energy holomorphic planes
\begin{equation}\label{eq:elliptic-normalised-family}w\mapsto (w,0,e^{t+i\theta})\end{equation}
 parametrised by $t\in \mathbb R$ and $\theta$. This family intersects $L$ at the point $w=0$  when $t=0$; see Figure~\ref{fig:elliptic-tropical-normalised}.

The wall $t=0$ is identified with the nodal boundary $\epsilon=0$ of
the half-cylinder family by attaching the constant disc at $w=0$.
Both are parametrised by $\theta$. To compare their orientations, we
must identify the chamber selected by the elliptic skein relation and
include the Spin-structure correction in the half-cylinder orientation.

\subsection{Orientation comparison in the model}
\label{subsec:elliptic-model-orientation}

We first orient the wall using a bounding chain for $L$, then orient the smoothing family using a comparison to the half-cylinders of Section~\ref{half-cylinder-section} and a Spin correction.

\subsubsection{The negative chamber orientation and bounding chain}
\label{subsec:elliptic-bounding-chain}

To identify the chamber in the elliptic skein relation, choose an appropriate bounding chain $C$ such that $\partial C=2L$ and satisfying \cite[Definition 2.1]{ekholm2019skeins}. The intersection of curves with $C$ gives a topological way of distinguishing the half of the moduli space with $t$ positive from the half of the moduli space with $t$ negative. Let
\begin{equation}\label{eq:elliptic-bounding-chain}C_{+}:=\{(x+iy,x-iy, e^{i\theta}e^{t})\text{ such that } x,y\in\mathbb R,\ \theta\in\mathbb T,\ t\in[0,\infty)\}\end{equation}
 oriented by $(x,y,\theta,t)$, and let $C_-$ be the same expression with $t\leq 0$, oriented by $(x,y,t, \theta)$. Then $\partial C_+=L=\partial C_-$.  Moreover, $C_{\pm}$ is the image of $L$ under the positive flow of the vector field $\pm J (-\partial_{\theta})$. Our required chain is $C=C_+ +C_-$.

Using the vector field $-\partial_\theta$ to determine framings, we see that the boundary of $u_{\epsilon,\theta}$ is the $0$-framed unknot in $L$.

Moreover, the planes $w\mapsto (w,0,e^t)$ intersect $C$ positively when $t>0$ and negatively when $t<0$.

Within our toric compactification $\bar Y$, these maps have trivial normal bundle in the $z_3$ direction, but normal bundle $O(-1)$ in the $z_2$ direction, because of our requirement that $\bar Y$ be compatible with the function $z_1z_2$. Accordingly, this moduli space is transversely cut out, and inherits the complex orientation $(t,\theta)$.

The chamber $t\leq0$ has the smaller intersection with $C$, so it is the negative-incidence chamber. Its
outward normal at $t=0$ is $\partial_t$, so the orientation
$(t,\theta)$ of the plane family induces the negative chamber orientation
$\partial_\theta$.

\subsubsection{The smoothing boundary orientation and Spin correction}
\label{subsec:elliptic-smoothing-model}

For $\epsilon>0$, these half-cylinders are defined in the region where the real part of $z_1z_2$ is greater than $0$. Composing with $h$, we get a family of holomorphic half-cylinders in $(X,L')$, whose orientation was calculated in Section~\ref{half-cylinder-section} and Remark~\ref{tilted-L}. In particular, within $(X,L')$ the orientation of the moduli space with an extra marked point on the boundary agrees with the orientation of $L'$.

The map $h$ changes the Spin structure along the boundary of these curves: the trivialisation of $TL'$ around the boundary pulls back to a trivialisation of $TL$ that differs from the chosen Spin structure by a nontrivial element of $\pi_1(SO(3))$. Therefore, the map $h$ reverses the orientation of the moduli space, so we must orient  the moduli space of curves with one marked point on the boundary so that the evaluation map to $L$ reverses orientations. The boundary evaluation is
\begin{equation}\label{eq:elliptic-boundary-evaluation}
u_{\epsilon,\theta}(e^{i\vartheta})
=(\epsilon e^{-i\vartheta},\epsilon e^{i\vartheta},e^{i\theta}),
\end{equation}
so $(\vartheta,\theta,\epsilon)$ reverses the orientation of $L$.
Removing the boundary marked point therefore gives the orientation
$(\theta,\epsilon)$ on the half-cylinder moduli space.

At $\epsilon=0$, the outward normal is $-\partial_\epsilon$.
The outward-normal-first convention applied to $(\theta,\epsilon)$
therefore gives the smoothing boundary orientation $\partial_\theta$.

\subsubsection*{Comparison at the nodal boundary}

Under the identification of the common wall by $\theta$, the two
calculations give
\begin{equation}\label{eq:elliptic-model-orientations}
\begin{aligned}
\text{smoothing boundary orientation:}\quad & \partial_\theta,\\
\text{negative chamber orientation:}\quad & \partial_\theta.
\end{aligned}
\end{equation}
The orientations agree, confirming the direct calculation of the comparison at all elliptic wall problems; see Lemma~\ref{lem:elliptic comparison}. Figure~\ref{fig:elliptic-moduli} summarises the resulting orientations near the wall.

\begin{figure}[htbp]
\centering
\begin{tikzpicture}[
  x=1cm, y=1cm,
  line width=0.8pt,
  line cap=round,
  line join=round,
  every node/.style={font=\small, inner sep=2pt},
  oriented/.style={
    postaction={decorate},
    decoration={markings,
      mark=at position 0.55 with
        {\arrow{Stealth[length=2.3mm,width=1.7mm]}}}
  }
]
  \coordinate (wall) at (0,0);

  \draw[oriented] (-4.1,-0.65)
    .. controls (-2.9,-0.13) and (-1.4,0) .. (wall);
  \draw[oriented] (wall)
    .. controls (1.5,0) and (2.8,-0.22) .. (4.3,-0.72);

  \draw[oriented] (1.38,3.48)
    .. controls (0.53,2.55) and (-0.06,1.10) .. (wall);

  \fill (wall) circle[radius=1.5pt];
  \node[anchor=north] at (0,-0.12) {$\mathcal M_{\rm node}$};
  \node[anchor=west] at (1.13,2.87)
    {$\mathcal M_{\mathrm{smooth}}$};
  \node[anchor=north] at (-2,-0.4)
    {$\mathcal M^-_{\mathrm{norm}}$};
  \node[anchor=north] at (2,-0.4)
    {$\mathcal M^+_{\mathrm{norm}}$};
\end{tikzpicture}
\caption{A schematic transverse slice of the moduli space near a curve with an elliptic boundary node and its normalisation. The implied orientation comparison is valid in both the holomorphic and anti-holomorphic orientation conventions.}
\label{fig:elliptic-moduli}
\end{figure}
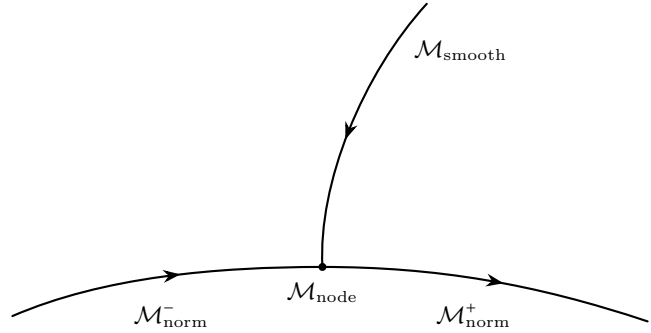

\begin{remark}[Comparison of elliptic skein conventions]
\label{rem:elliptic-skein-convention}
In  \cite[v2, Definition 3.1]{worldsheetskein}, $v_{+}$ intersects the bounding chain $C$ negatively, and $v_-$ intersects the bounding chain positively, whereas our modified elliptic skein relation has $v_1$ intersecting $C$ positively and $v_{-1}$ intersecting $C$ negatively. Thus the chamber $t\leq0$ is the one containing $v_{-1}$ in
our convention. The agreement in
Equation~\eqref{eq:elliptic-model-orientations} gives our elliptic skein
sign, which is opposite to the sign in
\cite[Definition 3.1]{worldsheetskein}. We demonstrate that our signs enable a homomorphism from the worldsheet skein module to the HOMFLYPT skein module of $L$ in Lemma~\ref{skein-comparison}.
\end{remark}

\section{Connected hyperbolic crossing via cusp deformations}\label{connected-hyperbolic-section}

The smoothing boundary orientation at a connected hyperbolic
crossing agrees with   the negative chamber orientation. Following
\cite[Section 5.2 and Lemma 2.10]{ekholm2019skeins}, we compute this
sign in a cusp family joining the connected hyperbolic and elliptic
walls, then apply the elliptic comparison of
Lemma~\ref{lem:elliptic comparison}. Section~\ref{sec:ramified-covers} gives a second calculation of this sign that does not use the elliptic comparison.

\subsection{The normalised disc family and its walls}
\label{subsec:connected-normalised-family}

Our model is a two-parameter family of discs $v_{s,t}$ whose wall $t=0$ is elliptic for $s>0$ and connected hyperbolic for $s<0$. This family is a slice of a $12$-dimensional moduli space, transverse to the action of a  $10$-dimensional space $W$ of infinitesimal automorphisms of $(\mathbb CP^3,\mathbb RP^3)$ (complementary to the Lie algebra of the stabiliser of $v_{0,0}$).

 Explicitly, consider the family of holomorphic maps
\[v_{s,t}: (\mathbb D,\partial \mathbb D)\longrightarrow (\mathbb CP^3,\mathbb RP^3)\]
defined on  affine charts $(\mathbb H,\mathbb R)$ and $(\mathbb C^3,\mathbb R^3)$ as
\begin{equation}\label{eq:connected-cusp-family}v_{s,t}(z)= (z^2,z(z^2+s), tz)  \end{equation}
 When $(s,t)=(0,0)$  the above map has a cusp singularity at $z=0$ and elsewhere this curve is smooth. When $t\neq 0$, the boundary of this curve is smooth and embedded, and the interior does not intersect $\mathbb RP^3$.

\begin{figure}[htbp]
\centering
\begingroup
\usetikzlibrary{arrows.meta}
\definecolor{cuspBoundary}{HTML}{356DA4}
\definecolor{cuspInterior}{HTML}{B66A2C}
\begin{tikzpicture}[x=.78cm,y=.78cm,
  line cap=round,line join=round,
  every node/.style={font=\small,inner sep=2pt},
  boundary/.style={draw=cuspBoundary,line width=1.0pt},
  orientationtip/.style={-{Stealth[length=1.8mm,width=1.25mm]}},
  guide/.style={draw=black!13,line width=.35pt}
]
\node[font=\large]at(4.6,12.20){$v_{s,t}(z)=(z^2,z(z^2+s),tz)$};
\node[font=\normalsize]at(0.0,11.50){$s<0$};
\node[font=\normalsize]at(4.6,11.50){$s=0$};
\node[font=\normalsize]at(9.2,11.50){$s>0$};
\draw[guide](2.3,-2.65)--(2.3,11.12);
\draw[guide](6.9,-2.65)--(6.9,11.12);
\draw[guide](-2.10,6.90)--(11.30,6.90);
\draw[guide](-2.10,2.30)--(11.30,2.30);
\node[anchor=east,font=\normalsize]at(-2.45,9.20){$t>0$};
\begin{scope}[shift={(-0.45,9.20)}]
\clip(-1.55,-1.75)rectangle(2.40,1.75);
\draw[boundary](4.0000000,-5.4000000) .. controls (2.9118899,-2.4893054) and (2.0457767,-0.9105925) .. (1.4016604,-0.1203575);
\draw[boundary](1.2021663,0.1072682) .. controls (0.4007221,0.9502424) and (0.0000000,0.4751212) .. (0.0000000,0.0000000);
\draw[boundary](0.0000000,0.0000000) .. controls (0.0000000,-0.8666666) and (1.3333333,-1.7333334) .. (4.0000000,5.4000000);
\draw[boundary,orientationtip] (2.0736000,-1.1139840) .. controls (1.9392000,-0.8843467) and (1.8113333,-0.6829333) .. (1.6900000,-0.5070000);
\draw[boundary,orientationtip] (1.6900000,0.5070000) .. controls (1.8113333,0.6829333) and (1.9392000,0.8843467) .. (2.0736000,1.1139840);
\end{scope}
\node[text=black!75]at(0.00,7.18){positive crossing};
\begin{scope}[shift={(4.15,9.20)}]
\clip(-1.55,-1.75)rectangle(2.40,1.75);
\draw[boundary](4.0000000,-8.0000000) .. controls (1.3333333,0.0000000) and (0.0000000,0.0000000) .. (0.0000000,0.0000000);
\draw[boundary](0.0000000,0.0000000) .. controls (0.0000000,0.0000000) and (1.3333333,0.0000000) .. (4.0000000,8.0000000);
\draw[boundary,orientationtip] (1.1664000,-1.2597120) .. controls (1.0584000,-1.0847520) and (0.9579000,-0.9340920) .. (0.8649000,-0.8043570);
\draw[boundary,orientationtip] (0.8649000,0.8043570) .. controls (0.9579000,0.9340920) and (1.0584000,1.0847520) .. (1.1664000,1.2597120);
\end{scope}
\node[text=black!75]at(4.60,7.18){framing tangency};
\begin{scope}[shift={(8.75,9.20)}]
\clip(-1.55,-1.75)rectangle(2.40,1.75);
\draw[boundary](4.0000000,-10.0000000) .. controls (1.3333333,-1.3333334) and (0.0000000,-0.6666666) .. (0.0000000,0.0000000);
\draw[boundary](0.0000000,0.0000000) .. controls (0.0000000,0.6666666) and (1.3333333,1.3333334) .. (4.0000000,10.0000000);

\draw[cuspInterior,fill=white,line width=.85pt](-1.0000000,0.0000000)circle[radius=2.1pt];
\draw[boundary,orientationtip] (0.7396000,-1.4960560) .. controls (0.6593333,-1.3458453) and (0.5856000,-1.2124907) .. (0.5184000,-1.0932480);
\draw[boundary,orientationtip] (0.5184000,1.0932480) .. controls (0.5856000,1.2124907) and (0.6593333,1.3458453) .. (0.7396000,1.4960560);
\end{scope}
\node[text=black!75]at(9.20,7.18){intersection with $C_+$};
\node[anchor=east,font=\normalsize]at(-2.45,4.60){$t=0$};
\begin{scope}[shift={(-0.45,4.60)}]
\clip(-1.55,-1.75)rectangle(2.40,1.75);
\draw[boundary](4.0000000,-5.4000000) .. controls (1.3333333,1.7333334) and (0.0000000,0.8666666) .. (0.0000000,0.0000000);
\draw[boundary](0.0000000,0.0000000) .. controls (0.0000000,-0.8666666) and (1.3333333,-1.7333334) .. (4.0000000,5.4000000);
\fill[cuspBoundary](1.3000000,0.0000000)circle[radius=1.6pt];
\draw[boundary,orientationtip] (2.0736000,-1.1139840) .. controls (1.9392000,-0.8843467) and (1.8113333,-0.6829333) .. (1.6900000,-0.5070000);
\draw[boundary,orientationtip] (1.6900000,0.5070000) .. controls (1.8113333,0.6829333) and (1.9392000,0.8843467) .. (2.0736000,1.1139840);
\end{scope}
\node[text=black!75]at(0.00,2.58){hyperbolic crossing};
\begin{scope}[shift={(4.15,4.60)}]
\clip(-1.55,-1.75)rectangle(2.40,1.75);
\draw[boundary](4.0000000,-8.0000000) .. controls (1.3333333,0.0000000) and (0.0000000,0.0000000) .. (0.0000000,0.0000000);
\draw[boundary](0.0000000,0.0000000) .. controls (0.0000000,0.0000000) and (1.3333333,0.0000000) .. (4.0000000,8.0000000);
\draw[boundary,orientationtip] (1.1664000,-1.2597120) .. controls (1.0584000,-1.0847520) and (0.9579000,-0.9340920) .. (0.8649000,-0.8043570);
\draw[boundary,orientationtip] (0.8649000,0.8043570) .. controls (0.9579000,0.9340920) and (1.0584000,1.0847520) .. (1.1664000,1.2597120);
\end{scope}
\node[text=black!75]at(4.60,2.58){cusp};
\begin{scope}[shift={(8.75,4.60)}]
\clip(-1.55,-1.75)rectangle(2.40,1.75);
\draw[boundary](4.0000000,-10.0000000) .. controls (1.3333333,-1.3333334) and (0.0000000,-0.6666666) .. (0.0000000,0.0000000);
\draw[boundary](0.0000000,0.0000000) .. controls (0.0000000,0.6666666) and (1.3333333,1.3333334) .. (4.0000000,10.0000000);

\fill[cuspInterior](-1.0000000,0.0000000)circle[radius=2.0pt];
\draw[boundary,orientationtip] (0.7396000,-1.4960560) .. controls (0.6593333,-1.3458453) and (0.5856000,-1.2124907) .. (0.5184000,-1.0932480);
\draw[boundary,orientationtip] (0.5184000,1.0932480) .. controls (0.5856000,1.2124907) and (0.6593333,1.3458453) .. (0.7396000,1.4960560);
\end{scope}
\node[text=black!75]at(9.20,2.58){elliptic $L$ incidence};
\node[anchor=east,font=\normalsize]at(-2.45,0.00){$t<0$};
\begin{scope}[shift={(-0.45,0.00)}]
\clip(-1.55,-1.75)rectangle(2.40,1.75);
\draw[boundary](4.0000000,-5.4000000) .. controls (1.3333333,1.7333334) and (0.0000000,0.8666666) .. (0.0000000,0.0000000);
\draw[boundary](0.0000000,0.0000000) .. controls (0.0000000,-0.4751212) and (0.4007221,-0.9502424) .. (1.2021663,-0.1072682);
\draw[boundary](1.4016604,0.1203575) .. controls (2.0457767,0.9105925) and (2.9118899,2.4893054) .. (4.0000000,5.4000000);
\draw[boundary,orientationtip] (2.0736000,-1.1139840) .. controls (1.9392000,-0.8843467) and (1.8113333,-0.6829333) .. (1.6900000,-0.5070000);
\draw[boundary,orientationtip] (1.6900000,0.5070000) .. controls (1.8113333,0.6829333) and (1.9392000,0.8843467) .. (2.0736000,1.1139840);
\end{scope}
\node[text=black!75]at(0.00,-2.02){negative crossing};
\begin{scope}[shift={(4.15,0.00)}]
\clip(-1.55,-1.75)rectangle(2.40,1.75);
\draw[boundary](4.0000000,-8.0000000) .. controls (1.3333333,0.0000000) and (0.0000000,0.0000000) .. (0.0000000,0.0000000);
\draw[boundary](0.0000000,0.0000000) .. controls (0.0000000,0.0000000) and (1.3333333,0.0000000) .. (4.0000000,8.0000000);
\draw[boundary,orientationtip] (1.1664000,-1.2597120) .. controls (1.0584000,-1.0847520) and (0.9579000,-0.9340920) .. (0.8649000,-0.8043570);
\draw[boundary,orientationtip] (0.8649000,0.8043570) .. controls (0.9579000,0.9340920) and (1.0584000,1.0847520) .. (1.1664000,1.2597120);
\end{scope}
\node[text=black!75]at(4.60,-2.02){framing tangency};
\begin{scope}[shift={(8.75,0.00)}]
\clip(-1.55,-1.75)rectangle(2.40,1.75);
\draw[boundary](4.0000000,-10.0000000) .. controls (1.3333333,-1.3333334) and (0.0000000,-0.6666666) .. (0.0000000,0.0000000);
\draw[boundary](0.0000000,0.0000000) .. controls (0.0000000,0.6666666) and (1.3333333,1.3333334) .. (4.0000000,10.0000000);

\draw[cuspInterior,fill=white,line width=.85pt](-1.0000000,0.0000000)circle[radius=2.1pt];
\draw[boundary,orientationtip] (0.7396000,-1.4960560) .. controls (0.6593333,-1.3458453) and (0.5856000,-1.2124907) .. (0.5184000,-1.0932480);
\draw[boundary,orientationtip] (0.5184000,1.0932480) .. controls (0.5856000,1.2124907) and (0.6593333,1.3458453) .. (0.7396000,1.4960560);
\end{scope}
\node[text=black!75]at(9.20,-2.02){intersection with $C_-$};
\end{tikzpicture}
\endgroup
\caption{The intersection of  $v_{s,t}$ with $L$ and $C$.
For $s>0$, circles mark the interior point $z=i\sqrt{s}$,
whose image lies in $L$ when $t=0$ and in $C_\pm$ otherwise.}
\label{fig:connected-normalized-oriented}
\end{figure}
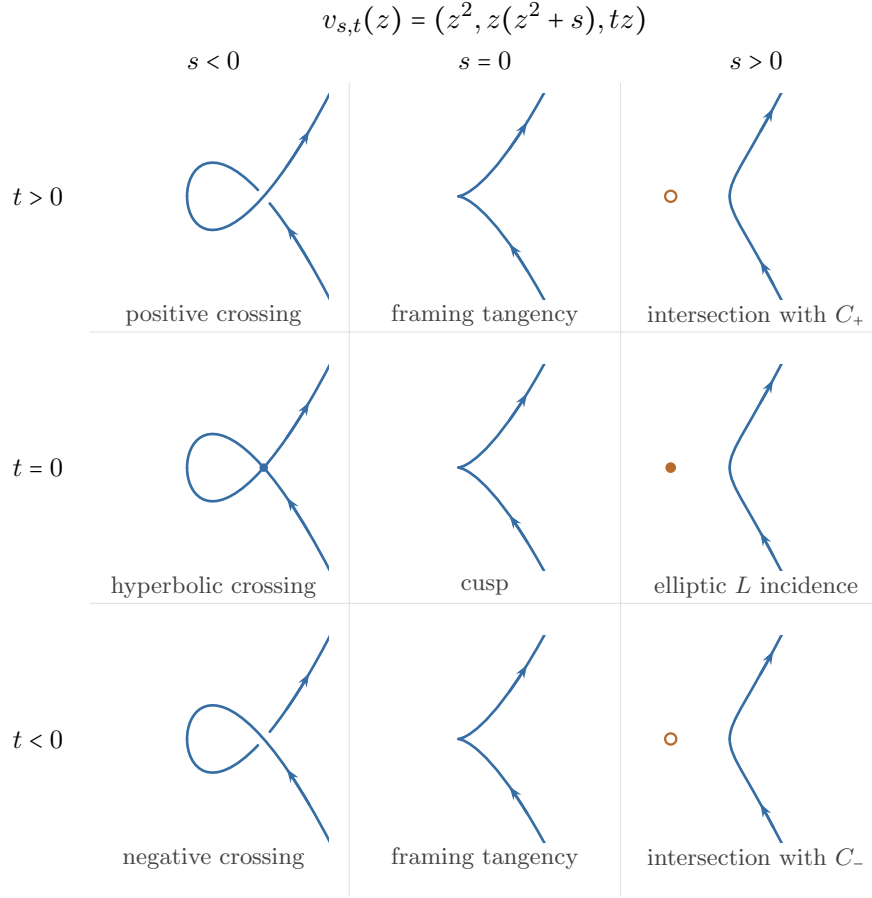

The wall $t=0$ has two parts (Figure~\ref{fig:connected-normalized-oriented}). For $s>0$, the interior point
$z=i\sqrt{s}$ maps to $L$, giving an elliptic wall crossing.
At $t=0$ and $s<0$, there is a hyperbolic crossing at
$z=\pm\sqrt{-s}$. Both points lie on the single boundary component
of the normalised disc, so this is a connected crossing. Using the locally defined projection to the $\mathbb R^2$ plane, the positive resolution of this crossing is $t>0$.
The cusp at $(s,t)=(0,0)$ separates these two types of wall.

\subsection{The nodal curves and their smoothings}
\label{subsec:connected-smoothing-family}

For $s\neq0$, the image of $v_{s,0}$ is half  of a nodal plane cubic in the plane $z_3=0$, locally described by
\begin{equation}\label{eq:connected-nodal-cubic}
z_2^2=z_1(z_1+s)^2.
\end{equation} Use $u_{s,0}$ to denote the corresponding nodal curves with boundary on $L$. Let us interpolate between $u_{1,0}$ and $u_{-1,0}$ in a family smooth curves.
Let $u_{\cos \theta,-\sin\theta}$ for $\theta\in [0,\pi]$ be the family of  curves with boundary on $L$ interpolating $u_{\pm 1,0}$, whose double is  the plane cubic
\[z_2^2=(z_1-r_1(\theta))(z_1-r_2(\theta))(z_1-r_3(\theta))\]
where the roots $r_k(\theta)$ linearly interpolate between $(r_1(0),r_2(0),r_3(0))=(-1,-1,0)$ and $(r_1(\pi),r_2(\pi),r_3(\pi))=(0,1,1)$, so\[(r_1(\theta),r_2(\theta),r_3(\theta))=(\theta/\pi-1,\,2\theta/\pi-1,\,\theta/\pi).\] The domains of these plane cubics are covers of $\mathbb CP^1$ branched at these roots and infinity, and have cross ratio $\theta/\pi$ linearly varying from 0 to $1$. For $\theta\in(0,\pi)$, the derivative of this family in $H^0D_u$ is therefore nonzero, and independent of any direction coming from an infinitesimal real projective automorphism. As $\theta$ approaches $0$ or $\pi$, plumbing coordinates near the colliding branch points are, to leading order $\epsilon=\theta^2$ or $(\pi-\theta)^2$ respectively, so the appropriate logarithmic derivative is also nonzero in $H^0(D_{u_{\pm 1,0}})$.

Consider the real scaling action $(z_1,z_2,z_3)\mapsto(\lambda z_1, \lambda^{3/2} z_2,z_3)$. This action sends $u_{s,0}$ to $u_{\lambda s, 0}$. For $\lambda>0$ and $\theta\in [0,\pi]$ define $u_{\lambda\cos\theta,-\lambda\sin\theta}$ by scaling $u_{\cos\theta,-\sin\theta}$. So in particular, $u_{\lambda\cos\theta,-\lambda\sin\theta}$ is half the  plane cubic defined by
\[z_2^2=(z_1-\lambda r_1(\theta))(z_1-\lambda r_2(\theta))(z_1-\lambda r_3(\theta))\]

This defines a family $u_{s,h}$ of curves for $h\leq 0$ and $(s,h)\neq (0,0)$, such that  $v_{s,0}$ is the normalisation of $u_{s,0}$; see Figure~\ref{fig:connected-smoothing-oriented}.

We will compare the boundary $h=0$  with the wall $t=0$ through their
common normalisation $v_{s,0}$: the node is elliptic for $s>0$ and
connected hyperbolic for $s<0$. Section~\ref{subsec:connected-smoothing-orientation}
compares the corresponding smoothing boundary orientations.

\begin{figure}[htbp]
\centering
\begingroup
\usetikzlibrary{arrows.meta}
\definecolor{smoothingBlue}{HTML}{356DA4}
\definecolor{smoothingOrange}{HTML}{B66A2C}
\begin{tikzpicture}[x=.78cm,y=.78cm,
  line cap=round,line join=round,
  every node/.style={font=\small,inner sep=2pt},
  outer/.style={draw=smoothingBlue,line width=1.0pt},
  oval/.style={draw=smoothingOrange,line width=1.0pt},
  guide/.style={draw=black!13,line width=.35pt}
]

\node[font=\normalsize]at(0.0,6.90){$s<0$};
\node[font=\normalsize]at(4.6,6.90){$s=0$};
\node[font=\normalsize]at(9.2,6.90){$s>0$};
\draw[guide](2.3,-2.65)--(2.3,6.52);
\draw[guide](6.9,-2.65)--(6.9,6.52);
\draw[guide](-2.10,2.30)--(11.30,2.30);
\node[anchor=east,font=\normalsize]at(-2.45,4.60){$h=0$};
\begin{scope}[shift={(-0.45,4.60)}]
\clip(-1.55,-1.75)rectangle(2.40,1.75);
\draw[outer](4.0000000,-5.4000000)
 .. controls (2.8535672,-2.3332924) and (1.9535672,-0.7451813) .. (1.3000000,0.0000000);
\draw[outer](1.3000000,0.0000000)
 .. controls (1.9535672,0.7451813) and (2.8535672,2.3332924) .. (4.0000000,5.4000000);
\draw[oval](1.3000000,0.0000000)
 .. controls (0.4333333,0.9881520) and (0.0000000,0.4940760) .. (0.0000000,0.0000000);
\draw[oval](0.0000000,0.0000000)
 .. controls (0.0000000,-0.4940760) and (0.4333333,-0.9881520) .. (1.3000000,0.0000000);
\fill[black!70](1.3000000,0.0000000)circle[radius=1.7pt];
\end{scope}
\node[text=black!75]at(0.00,2.58){hyperbolic node};
\begin{scope}[shift={(4.15,4.60)}]
\clip(-1.55,-1.75)rectangle(2.40,1.75);
\draw[black!32,line width=.8pt,densely dashed](4.0000000,-8.0000000)
 .. controls (1.3333333,0.0000000) and (0.0000000,0.0000000) .. (0.0000000,0.0000000);
\draw[black!32,line width=.8pt,densely dashed](0.0000000,0.0000000)
 .. controls (0.0000000,0.0000000) and (1.3333333,0.0000000) .. (4.0000000,8.0000000);
\end{scope}
\node[text=black!75]at(4.60,2.58){cusp; see Remark~\ref{rem:connected-cusp-limit}};
\begin{scope}[shift={(8.75,4.60)}]
\clip(-1.55,-1.75)rectangle(2.40,1.75);
\draw[outer](4.0000000,-10.0000000)
 .. controls (1.3333333,-1.3333334) and (0.0000000,-0.6666666) .. (0.0000000,0.0000000);
\draw[outer](0.0000000,0.0000000)
 .. controls (0.0000000,0.6666666) and (1.3333333,1.3333334) .. (4.0000000,10.0000000);
\fill[smoothingOrange](-1.0000000,0.0000000)circle[radius=2.0pt];
\end{scope}
\node[text=black!75]at(9.20,2.58){elliptic node};
\node[anchor=east,font=\normalsize]at(-2.45,0.00){$h<0$};
\begin{scope}[shift={(-0.45,0.00)}]
\clip(-1.55,-1.75)rectangle(2.40,1.75);
\draw[outer](2.5000000,-1.8713632)
 .. controls (2.3259708,-1.5282081) and (2.1736953,-1.2524561) .. (2.0431734,-1.0285271)
 .. controls (1.9126515,-0.8045982) and (1.8038833,-0.6325120) .. (1.7168687,-0.4973375)
 .. controls (1.6733614,-0.4297502) and (1.6352926,-0.3714031) .. (1.6026621,-0.3203921)
 .. controls (1.5700316,-0.2693810) and (1.5428396,-0.2257178) .. (1.5210859,-0.1873072)
 .. controls (1.4993323,-0.1488965) and (1.4830170,-0.1157465) .. (1.4721402,-0.0854074)
 .. controls (1.4612634,-0.0550681) and (1.4558250,-0.0275269) .. (1.4558250,0.0000000)
 .. controls (1.4558250,0.0275269) and (1.4612634,0.0550681) .. (1.4721402,0.0854074)
 .. controls (1.4830170,0.1157465) and (1.4993323,0.1488965) .. (1.5210859,0.1873072)
 .. controls (1.5428396,0.2257178) and (1.5700316,0.2693810) .. (1.6026621,0.3203921)
 .. controls (1.6352926,0.3714031) and (1.6733614,0.4297502) .. (1.7168687,0.4973375)
 .. controls (1.8038833,0.6325120) and (1.9126515,0.8045982) .. (2.0431734,1.0285271)
 .. controls (2.1736953,1.2524561) and (2.3259708,1.5282081) .. (2.5000000,1.8713632);
\draw[oval](1.1745846,0.0000000)
 .. controls (1.1745846,0.0209124) and (1.1707008,0.0418407) .. (1.1630078,0.0636046)
 .. controls (1.1553147,0.0853685) and (1.1438127,0.1079891) .. (1.1287222,0.1318672)
 .. controls (1.1136317,0.1557453) and (1.0949535,0.1808825) .. (1.0730455,0.2070944)
 .. controls (1.0511375,0.2333062) and (1.0260009,0.2605828) .. (0.9981173,0.2882265)
 .. controls (0.9702337,0.3158704) and (0.9396046,0.3438685) .. (0.9068169,0.3711852)
 .. controls (0.8740293,0.3985018) and (0.8390848,0.4251248) .. (0.8026532,0.4498668)
 .. controls (0.7662215,0.4746088) and (0.7283045,0.4974594) .. (0.6896289,0.5172480)
 .. controls (0.6509532,0.5370366) and (0.6115209,0.5537548) .. (0.5720875,0.5663913)
 .. controls (0.5326542,0.5790278) and (0.4932218,0.5875761) .. (0.4545462,0.5912980)
 .. controls (0.4158705,0.5950198) and (0.3779535,0.5939108) .. (0.3415219,0.5875867)
 .. controls (0.3050902,0.5812625) and (0.2701457,0.5697214) .. (0.2373581,0.5529785)
 .. controls (0.2045705,0.5362355) and (0.1739413,0.5142915) .. (0.1460578,0.4875700)
 .. controls (0.1181742,0.4608485) and (0.0930375,0.4293527) .. (0.0711295,0.3938865)
 .. controls (0.0492215,0.3584202) and (0.0305433,0.3189886) .. (0.0154528,0.2767145)
 .. controls (0.0003623,0.2344405) and (-0.0111397,0.1893305) .. (-0.0188327,0.1427369)
 .. controls (-0.0265258,0.0961433) and (-0.0304095,0.0480736) .. (-0.0304095,0.0000000)
 .. controls (-0.0304095,-0.0480736) and (-0.0265258,-0.0961433) .. (-0.0188327,-0.1427369)
 .. controls (-0.0111397,-0.1893305) and (0.0003623,-0.2344405) .. (0.0154528,-0.2767145)
 .. controls (0.0305433,-0.3189886) and (0.0492215,-0.3584202) .. (0.0711295,-0.3938865)
 .. controls (0.0930375,-0.4293527) and (0.1181742,-0.4608485) .. (0.1460578,-0.4875700)
 .. controls (0.1739413,-0.5142915) and (0.2045705,-0.5362355) .. (0.2373581,-0.5529785)
 .. controls (0.2701457,-0.5697214) and (0.3050902,-0.5812625) .. (0.3415219,-0.5875867)
 .. controls (0.3779535,-0.5939108) and (0.4158705,-0.5950198) .. (0.4545462,-0.5912980)
 .. controls (0.4932218,-0.5875761) and (0.5326542,-0.5790278) .. (0.5720875,-0.5663913)
 .. controls (0.6115209,-0.5537548) and (0.6509532,-0.5370366) .. (0.6896289,-0.5172480)
 .. controls (0.7283045,-0.4974594) and (0.7662215,-0.4746088) .. (0.8026532,-0.4498668)
 .. controls (0.8390848,-0.4251248) and (0.8740293,-0.3985018) .. (0.9068169,-0.3711852)
 .. controls (0.9396046,-0.3438685) and (0.9702337,-0.3158704) .. (0.9981173,-0.2882265)
 .. controls (1.0260009,-0.2605828) and (1.0511375,-0.2333062) .. (1.0730455,-0.2070944)
 .. controls (1.0949535,-0.1808825) and (1.1136317,-0.1557453) .. (1.1287222,-0.1318672)
 .. controls (1.1438127,-0.1079891) and (1.1553147,-0.0853685) .. (1.1630078,-0.0636046)
 .. controls (1.1707008,-0.0418407) and (1.1745846,-0.0209124) .. (1.1745846,-0.0000000) -- cycle;
\end{scope}
\begin{scope}[shift={(4.15,0.00)}]
\clip(-1.55,-1.75)rectangle(2.40,1.75);
\draw[outer](2.5000000,-3.9338276)
 .. controls (2.1241582,-3.0409970) and (1.7952966,-2.3597765) .. (1.5134152,-1.8372696)
 .. controls (1.2315338,-1.3147628) and (0.9966326,-0.9509511) .. (0.8087117,-0.6930993)
 .. controls (0.6207908,-0.4352475) and (0.4798501,-0.2833172) .. (0.3858897,-0.1852290)
 .. controls (0.3389094,-0.1361848) and (0.3036743,-0.1006314) .. (0.2801841,-0.0720020)
 .. controls (0.2566940,-0.0433727) and (0.2449490,-0.0216749) .. (0.2449490,0.0000000)
 .. controls (0.2449490,0.0216749) and (0.2566940,0.0433727) .. (0.2801841,0.0720020)
 .. controls (0.3036743,0.1006314) and (0.3389094,0.1361848) .. (0.3858897,0.1852290)
 .. controls (0.4798501,0.2833172) and (0.6207908,0.4352475) .. (0.8087117,0.6930993)
 .. controls (0.9966326,0.9509511) and (1.2315338,1.3147628) .. (1.5134152,1.8372696)
 .. controls (1.7952966,2.3597765) and (2.1241582,3.0409970) .. (2.5000000,3.9338276);
\draw[oval](0.0000000,0.0000000)
 .. controls (0.0000000,0.0079345) and (-0.0031877,0.0158800) .. (-0.0093228,0.0236338)
 .. controls (-0.0154579,0.0313877) and (-0.0245357,0.0389592) .. (-0.0358719,0.0458929)
 .. controls (-0.0472082,0.0528265) and (-0.0607940,0.0591101) .. (-0.0756055,0.0640637)
 .. controls (-0.0904170,0.0690173) and (-0.1064426,0.0726188) .. (-0.1224745,0.0742385)
 .. controls (-0.1385064,0.0758581) and (-0.1545319,0.0754777) .. (-0.1693434,0.0728307)
 .. controls (-0.1841550,0.0701837) and (-0.1977408,0.0652641) .. (-0.2090770,0.0583541)
 .. controls (-0.2204133,0.0514440) and (-0.2294910,0.0425517) .. (-0.2356262,0.0324912)
 .. controls (-0.2417613,0.0224307) and (-0.2449490,0.0112212) .. (-0.2449490,0.0000000)
 .. controls (-0.2449490,-0.0112212) and (-0.2417613,-0.0224307) .. (-0.2356262,-0.0324912)
 .. controls (-0.2294910,-0.0425517) and (-0.2204133,-0.0514440) .. (-0.2090770,-0.0583541)
 .. controls (-0.1977408,-0.0652641) and (-0.1841550,-0.0701837) .. (-0.1693434,-0.0728307)
 .. controls (-0.1545319,-0.0754777) and (-0.1385064,-0.0758581) .. (-0.1224745,-0.0742385)
 .. controls (-0.1064426,-0.0726188) and (-0.0904170,-0.0690173) .. (-0.0756055,-0.0640637)
 .. controls (-0.0607940,-0.0591101) and (-0.0472082,-0.0528265) .. (-0.0358719,-0.0458929)
 .. controls (-0.0245357,-0.0389592) and (-0.0154579,-0.0313877) .. (-0.0093228,-0.0236338)
 .. controls (-0.0031877,-0.0158800) and (-0.0000000,-0.0079345) .. (0.0000000,-0.0000000) -- cycle;
\end{scope}
\begin{scope}[shift={(8.75,0.00)}]
\clip(-1.55,-1.75)rectangle(2.40,1.75);
\draw[outer](2.5000000,-5.5167925)
 .. controls (2.0898695,-4.4131825) and (1.7310052,-3.5508430) .. (1.4234073,-2.8695172)
 .. controls (1.1158094,-2.1881915) and (0.8594779,-1.6878769) .. (0.6544126,-1.3084019)
 .. controls (0.4493473,-0.9289268) and (0.2955484,-0.6703115) .. (0.1930157,-0.4723725)
 .. controls (0.0904831,-0.2744335) and (0.0392168,-0.1371968) .. (0.0392168,0.0000000)
 .. controls (0.0392168,0.1371968) and (0.0904831,0.2744335) .. (0.1930157,0.4723725)
 .. controls (0.2955484,0.6703115) and (0.4493473,0.9289268) .. (0.6544126,1.3084019)
 .. controls (0.8594779,1.6878769) and (1.1158094,2.1881915) .. (1.4234073,2.8695172)
 .. controls (1.7310052,3.5508430) and (2.0898695,4.4131825) .. (2.5000000,5.5167925);
\draw[oval](-0.8795022,0.0000000)
 .. controls (-0.8795022,0.0175787) and (-0.8831488,0.0351588) .. (-0.8901671,0.0516885)
 .. controls (-0.8971855,0.0682184) and (-0.9075701,0.0836957) .. (-0.9205383,0.0970561)
 .. controls (-0.9335066,0.1104164) and (-0.9490482,0.1216475) .. (-0.9659921,0.1297779)
 .. controls (-0.9829359,0.1379084) and (-1.0012685,0.1429196) .. (-1.0196084,0.1441682)
 .. controls (-1.0379483,0.1454168) and (-1.0562809,0.1428875) .. (-1.0732247,0.1365247)
 .. controls (-1.0901685,0.1301620) and (-1.1057102,0.1199625) .. (-1.1186784,0.1066049)
 .. controls (-1.1316467,0.0932473) and (-1.1420313,0.0767434) .. (-1.1490497,0.0584459)
 .. controls (-1.1560680,0.0401485) and (-1.1597146,0.0200814) .. (-1.1597146,0.0000000)
 .. controls (-1.1597146,-0.0200814) and (-1.1560680,-0.0401485) .. (-1.1490497,-0.0584459)
 .. controls (-1.1420313,-0.0767434) and (-1.1316467,-0.0932473) .. (-1.1186784,-0.1066049)
 .. controls (-1.1057102,-0.1199625) and (-1.0901685,-0.1301620) .. (-1.0732247,-0.1365247)
 .. controls (-1.0562809,-0.1428875) and (-1.0379483,-0.1454168) .. (-1.0196084,-0.1441682)
 .. controls (-1.0012685,-0.1429196) and (-0.9829359,-0.1379084) .. (-0.9659921,-0.1297779)
 .. controls (-0.9490482,-0.1216475) and (-0.9335066,-0.1104164) .. (-0.9205383,-0.0970561)
 .. controls (-0.9075701,-0.0836957) and (-0.8971855,-0.0682184) .. (-0.8901671,-0.0516885)
 .. controls (-0.8831488,-0.0351588) and (-0.8795022,-0.0175787) .. (-0.8795022,-0.0000000) -- cycle;
\end{scope}
\end{tikzpicture}
\endgroup
\caption{Boundary images of $u_{s,h}$ in the same
$(z_1,z_2)$ projection.
The dashed cusp only determines the limiting image; the stable-map limits
at $(s,h)=(0,0)$ are described in
Remark~\ref{rem:connected-cusp-limit}.}
\label{fig:connected-smoothing-oriented}
\end{figure}
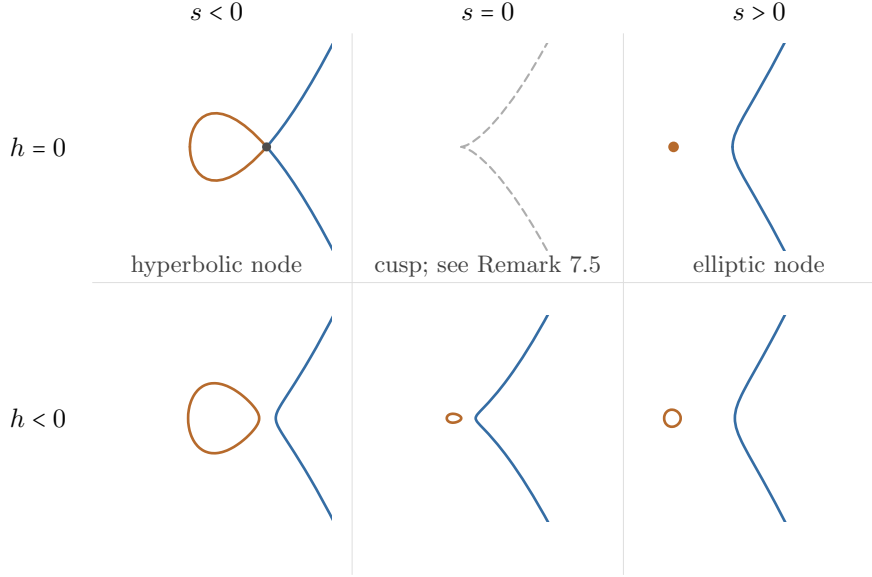

\subsection{Orientation comparison in the model}
\label{subsec:connected-model-orientation}

Figure~\ref{fig:connected-moduli} shows the walls, chambers and smoothing families in a transverse slice of the moduli space. We first compare the two chamber orientations using their normal lines.
We then compare their smoothing boundary orientations using the index bundle
of the family $u_{s,h}$.

\subsubsection{The bounding chain and elliptic chambers}
\label{subsec:connected-bounding-chain}

We can choose a bounding chain $C=C_++ C_-$ such that
\begin{itemize}\item  near $(0,0,0)$, $C_+$ agrees with  $\mathbb R^2\times \mathbb H$, and $C_-$ agrees with $\mathbb R^2\times \overline{\mathbb H}$,
\item the cusp curve $v_{0,0}$ intersects $C$ only at its boundary,
\item $v_{0,0}$ is tangent to $C$ only at the cusp.
\end{itemize}

A projective construction of this bounding chain is given in
Remark~\ref{rem:connected-bounding-chain-construction}.

Away from the cusp, the lifts of $v_{0,0}$ and $C$ to
the real oriented blowup of $\mathbb{CP}^3$ along $\mathbb{RP}^3$ are
disjoint. So, after restricting to sufficiently small
$s$ and $t$, every interior intersection of $v_{s,t}$ with $C$, and
every tangency, occurs in the neighbourhood
of the cusp image $(0,0,0)$ where $C_\pm$ agree with the standard local model.

It follows immediately from this local model that, for $s>0$ and $s,t$ sufficiently small, $v_{s,t}$ has a unique interior
intersection with $C$, occurring at
\begin{equation}\label{eq:connected-elliptic-intersection}
z=i\sqrt{s}.
\end{equation}
This intersection is positive and lies in $C_+$ when $t>0$, and is
negative and lies in $C_-$ when $t<0$.

\begin{figure}[htbp]
\begingroup
\usetikzlibrary{arrows.meta}
\definecolor{hyperbolicWall}{HTML}{356DA4}
\definecolor{ellipticWall}{HTML}{B66A2C}
\definecolor{normalisedFill}{HTML}{F3F7FA}
\definecolor{smoothingFill}{HTML}{F4F2EF}
\begin{tikzpicture}[
  x={(1cm,0cm)},y={(.52cm,.90cm)},z={(0cm,1cm)},
  line cap=round,line join=round,
  every node/.style={font=\small,inner sep=2pt},
  edge/.style={draw=black!37,line width=.55pt},
  coordaxis/.style={draw=black!55,line width=.65pt},
  tip/.style={-{Stealth[length=2.2mm,width=1.55mm]}},
  chamber/.style={align=center,font=\normalsize,text=black!85}
]
  \path[fill=smoothingFill]
    (-5,0,0)--(5,0,0)--(5,0,-4.05)--(-5,0,-4.05)--cycle;
  \draw[edge,densely dashed](-5,0,0)--(-5,0,-4.05);
  \draw[edge](-5,0,-1.98)--(-5,0,-4.05)--(5,0,-4.05)--(5,0,0);

  \draw[coordaxis,densely dashed](0,0,0)--(0,0,-1.98);
  \draw[coordaxis,tip](0,0,-1.98)--(0,0,-4.46)
    node[below,text=black]{$-h$};

  \path[fill=normalisedFill]
    (-5,-2.2,0)--(5,-2.2,0)--(5,2.2,0)--(-5,2.2,0)--cycle;
  \draw[edge]
    (-5,-2.2,0)--(5,-2.2,0)--(5,2.2,0)--(-5,2.2,0)--cycle;

  \draw[coordaxis,densely dashed](0,0,-.13)--(0,0,-1.98);
  \draw[edge,densely dashed](-5,0,0)--(-5,0,-1.98);
  \draw[coordaxis,tip](0,-2.2,0)--(0,2.72,0)
    node[above,text=black]{$t$};
  \draw[coordaxis,tip](5,0,0)--(5.78,0,0)
    node[right,text=black]{$s$};

  \draw[hyperbolicWall,line width=1.35pt](-5,0,0)--(0,0,0);
  \draw[ellipticWall,line width=1.35pt](0,0,0)--(5,0,0);
  \node[above=5pt,text=hyperbolicWall,align=center]at(-2.75,0,0)
    {hyperbolic wall};
  \node[above=5pt,text=ellipticWall,align=center]at(2.75,0,0)
    {elliptic wall};

  \node[chamber]at(-2.95,1.38,0)
    {positive crossing};
  \node[chamber]at(2.65,1.38,0)
    {positive intersection\\with $C$};
  \node[chamber]at(-2.50,-1.32,0)
    {negative crossing};
  \node[chamber]at(2.95,-1.32,0)
    {negative intersection\\with $C$};

  \node[above=8pt,align=center,font=\normalsize]at(-2.75,2.2,0)
    {normalised family $v_{s,t}$\\[-1pt]$(s,t)\text{-plane}$};
  \node[align=center,font=\normalsize,text=black!80]at(2.45,0,-3.02)
    {smoothing family $u_{s,h}$\\[3pt]$(s,h)\text{-half-plane},\quad h\leq0$};
  \node[align=center,text=black!65]at(-2.6,0,-3.02)
    {smooth curves\\[2pt]$h<0$};

  \fill[black!85](0,0,0)circle[radius=2.0pt];
  \node[anchor=east,text=black!85]at(-.12,.56,0){cusp};
\end{tikzpicture}
\endgroup
\caption{The wall and chamber structure of a transverse slice of the moduli space near $v_{0,0}$, with curves identified if they agree off $L$. (See Remark~\ref{rem:connected-cusp-limit} for the curves identified with $v_{0,0}$.) This picture suppresses a common oriented rank $10$ subbundle $W$. }
\label{fig:connected-moduli}
\end{figure}
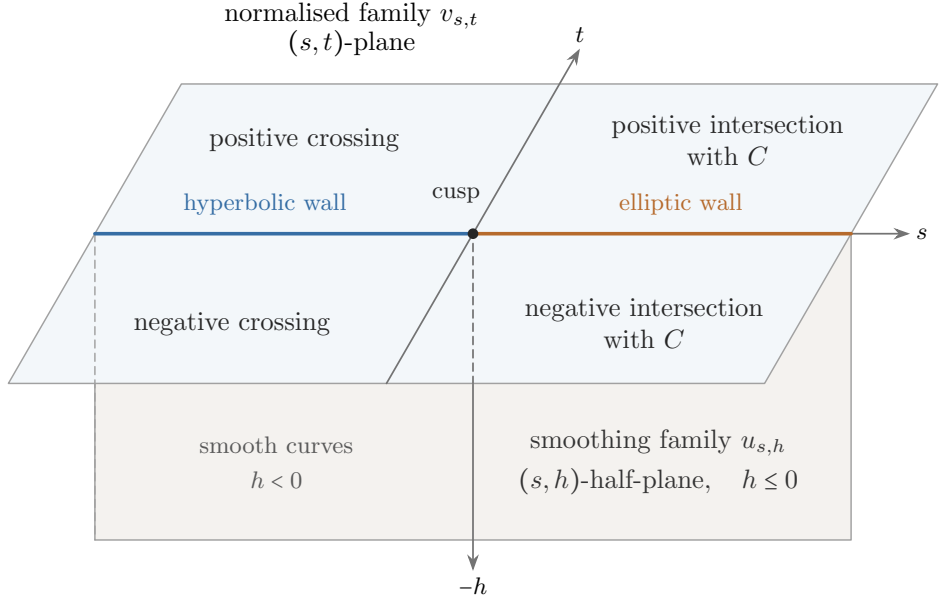

\subsubsection{The chamber orientations}
\label{subsec:connected-wall-orientation}

The normalised deformation complex $D_{v_{s,t}}$ is unobstructed, with homology concentrated in degree zero. Its degree-zero cohomology is described by variations of four real homogeneous cubics, modulo infinitesimal scaling and  disc reparametrisations, and therefore has dimension \(16-1-3=12\).

The $15$-dimensional real projective group acts on the space of curves, and the stabiliser subgroup of $v_{0,0}$ is $5$ dimensional, consisting of the $4$-dimensional subgroup fixing the plane $\{z_3=0\}$ pointwise and the $1$-dimensional scaling action $(z_1,z_2,z_3)\mapsto (\lambda^2z_1,\lambda^3z_2,z_3)$. Choose a $10$-dimensional complementary subspace $W$ to the Lie algebra of this stabiliser subgroup, corresponding to $10$ linearly independent holomorphic vector fields tangent to $\mathbb RP^3$. Given any holomorphic curve $u$ with boundary on $\mathbb RP^3$, there is an induced map $W\to H^0 D_u$.  Since \(s,t\) are transverse to the cusp-preserving locus, these variations give, after restricting to a neighbourhood of the origin, a splitting
\[
H^0D_{v_{s,t}}
=\mathbb R\partial_s\oplus\mathbb R\partial_t\oplus W.
\]
Orient \(W\) so that this ordered splitting has the holomorphic orientation.

For \(s\ne0\), the incidence or crossing is transverse within the plane $\{z_3=0\}$. In both cases, the wall deformation space consists of the eight deformations within that plane and the three deformations of the plane itself. Hence
\[
H^0D_{\mathrm{wall}}=\mathbb R\partial_s\oplus W.
\]
The frame \((\partial_s,W)\) identifies the wall determinant lines on the two sides of the cusp.

\begin{lemma} \label{Lcross} Use the above trivialisation to identify $\det D_{\rm wall}$ at the elliptic wall ($s>0$) with $\det D_{\rm wall}$ at the connected hyperbolic wall ($s<0$). Then the negative chamber orientations at the two walls agree.
\end{lemma}
\begin{proof}
As all homology of $D_{\rm wall}$ is concentrated in degree $0$, Remark~\ref{rem:transverse-wall-recipe} allows us to orient $D_{\rm wall}$ as the boundary of the negative chamber. In both cases the  negative chamber is $t<0$: at the elliptic wall ($s>0$), the negative-incidence chamber is where $v_{s,t}$ intersects $C_-$, which is $t<0$ by Section~\ref{subsec:connected-bounding-chain}; at the hyperbolic wall ($s<0$), the negative-crossing resolution is $t<0$ by Section~\ref{subsec:connected-normalised-family}. So at both walls the outward normal   is  $\partial_t$. It follows that the negative chamber orientations at the two walls agree.
\end{proof}

\subsubsection{The smoothing boundary orientations}
\label{subsec:connected-smoothing-orientation}

The infinitesimal projective action of $W$ defines  a map $W\to H^0 D_u$. At each stable cusp limit $u$ (seen as a limit of $u_{s,h}$ as $s$ and $h$ approach zero), this map is injective: restriction to the nonconstant cusp component recovers the original ten variations. Similarly, at a nodal curve $u_{s,0}$, the composition $W\to H^0D_{u_{s,0}}\to H^0 D'_{\rm wall}=  H^0 D_{\rm wall}$ agrees with the inclusion of $W$ into $D_{\rm wall}$. So, for $s,h$ sufficiently small,  $W\to H^0D_{u_{s,h}}$ defines a  rank $10$ subbundle \(W\),  whose projection to the wall deformation space agrees with the \(W\) chosen above. By our construction of $u_{s,h}$, for $(s,h)\neq (0,0)$,  $W$ is also transverse to the span of $\partial_s, h\partial_h$ within $H^0 D_{u_{s,h}}$, because the angular direction changes conformal modulus and the radial direction is the cusp-preserving scaling, whose action is independent from $W$.

 Away from the stable cusp limits, the deformation complexes are unobstructed, and we apply Remark~\ref{rem:transverse-wall-recipe} to obtain the boundary orientation.
 Orient $\det D_{u_{s,h}}=\det H^0 D_{u_{s,h}}$ using the holomorphic orientation convention. Then,  quotienting by the oriented subbundle \(W\) leaves an oriented two-dimensional deformation space, parametrised by $(s,h)\neq (0,0)$, with $h\leq 0$.
The logarithmic outward normal $-h\partial_h$ points in the same direction as the outward normal $\partial_h$.
 The induced boundary orientations give the same orientation of the \(s\)-axis for \(s>0\) and \(s<0\). Therefore, the smoothing boundary orientations for $s>0$ and $s<0$  agree under the identification of Lemma~\ref{Lcross}.

\begin{lemma}\label{lem:connected-hyperbolic-orientation}
At a connected hyperbolic crossing, whose two preimages lie on the
same boundary component, the smoothing
boundary and negative chamber orientations of
$\det D_{\rm wall}$ agree.
\end{lemma}
\begin{proof}

By the reduction to local standard models, Proposition~\ref{prop:wall-model-reduction},  it suffices to consider the above cusp family. Lemma~\ref{Lcross} identifies the negative chamber orientations at its elliptic and connected hyperbolic walls, while the preceding calculation identifies their smoothing boundary orientations. The elliptic comparison Lemma~\ref{lem:elliptic comparison} therefore proves the connected hyperbolic comparison.

\end{proof}

\subsection{Bounding chain and framing}

\begin{remark}[Construction of the bounding chain, by ChatGPT]
\label{rem:connected-bounding-chain-construction}
Here is the algebraic construction of the bounding chain used in
Section~\ref{subsec:connected-bounding-chain}.  We can take the image of   $C$ to be an appropriate perturbation of the submanifold
\begin{equation}\label{eq:connected-projective-chain}C':=\{[z_0,z_1,z_2,z_3]\text{ such that }(z_1+iz_2)(\bar z_0+i\bar z_3)=(\bar z_1+i\bar z_2)(z_0+i z_3)\}\end{equation}
subdivided into positive and negative halves by the sign of
\begin{equation}\label{eq:connected-projective-chain-halves}|z_1-iz_2|^2+|z_0-iz_3|^2 - |z_1+iz_2|^2-|z_0+iz_3|^2\ .\end{equation}
Writing $z=x+iy$, substitution of
$v_{0,0}(z)=[1:z^2:z^3:0]$ into the defining equation for
$ C'$ gives
\begin{equation}\label{eq:connected-projective-chain-check}
2y\bigl(y^2-3x^2+2ix\bigr)=0.
\end{equation}
For $y>0$, its imaginary part forces $x=0$, after which its real part
forces $y=0$. Thus there are no interior intersections. Along the
finite part of the boundary, the normal derivative vanishes only at
$x=0$; the second boundary tangency occurs at $z=\infty$.

Thus  $v_{0,0}$ intersects $C'$ only at its boundary, and is tangent at the cusp and $[0,0,1,0]$. An appropriate perturbation of $C'$ gives a bounding chain $C$ satisfying our conditions.
\end{remark}

\begin{remark}[Framing of $\partial v_{s,t}$]\label{framing-model}

Consider the framing of $\partial v_{s,t}$ determined by the vector field $\xi$.  At $s=0$ and $z=0$, the boundary tangent is $(0,0,t)$, so, for $t\neq 0$, it is parallel to the framing vector field $\xi$, with direction determined by the sign of $t$.  When $s<0$, the intersection
with $C$ disappears, while the framing of $\partial v_{s,t}$ changes
by a positive twist when $t>0$ and by a negative twist when $t<0$.
Consequently,
\begin{equation}\label{eq:connected-framing-invariance}
a^{v_{s,t}\cap C}\langle\partial v_{s,t}\rangle
\end{equation}
is independent of $s\neq 0$ in the HOMFLYPT skein module.
\end{remark}

\subsection{The cusp limit}
\begin{remark}[The cusp limit]\label{rem:connected-cusp-limit}
The stable-map limit of $u_{s,h}$ as $(s,h)$ approaches $(0,0)$ depends
on the path of approach. The limit of $u_{\lambda s,\lambda h}$ as
$\lambda\to0$ consists of the cusp curve $v_{0,0}$ attached at a
hyperbolic node to a (possibly nodal) constant annulus, with the same conformal modulus as $u_{s, h}$.
The extra boundary is described by real oriented blowup
of $\mathbb R\times(-\infty,0]$ at $(0,0)$. In the local quotient that forgets constant components on $L$, the exceptional boundary collapses to the cusp locus, which has real codimension two. So, these extra boundary components are not a problem for defining invariant skein-valued counts of curves.
\end{remark}

\section{Orienting ramified covers}\label{sec:ramified-covers}

This section orients covers of an immersed holomorphic curve using a choice of auxiliary complex structure $\mathcal J$ on the normal bundle, so that the boundary condition becomes $\mathcal J$-complex.  Katz and Liu used a similar method in \cite[Example 5.2.3 and Section 8.2]{KatzLiuDiscCovers}. We compare this orientation with the holomorphic and anti-holomorphic orientation conventions, first for an immersed curve (Lemma~\ref{Jcomparison1}) and then for its covers (Lemma~\ref{Jcomparison2}), and use it in Section~\ref{subsec:ramified-half-cylinders} to give an independent comparison of orientations at hyperbolic wall problems (Proposition~\ref{prop:ramified-hyperbolic}).

\smallskip

Consider a (logarithmically) immersed holomorphic curve $u\colon (\Sigma,\partial\Sigma)\to (\bar X, L)$ from a smooth bordered Riemann surface with no boundary marked points, with associated coupled bundle pair $(E,F,D_\Sigma,\phi) = (u^*T^{\log}\bar X, u^*T^{\log}L, D_\Sigma, du)$ as in Definition~\ref{def:coupled-bundle-pair}. So the coupling map $\phi=du^{\log}$ is injective,\footnote{Injectivity on logarithmic tangent spaces implies that $\Sigma$ has no nodes, unless we are working suitably in the logarithmic or exploded categories (cf.\ Section~\ref{sec:gluing}).} and we have a short exact sequence of bundle pairs
\[
0 \longrightarrow (T^{\log}\Sigma,T^{\log}\partial\Sigma) \xrightarrow{\;\phi\;} (E,F) \longrightarrow (N, N_{\mathbb R}) \longrightarrow 0.
\]
Here $N_{\mathbb R}=F/\phi(T^{\log}\partial\Sigma)$ is oriented by the boundary orientation $\partial_\vartheta$ of $T^{\log}\partial\Sigma$ and the orientation of $F$. Define
\[
D^N_\Sigma\colon \Omega^0(N,N_{\mathbb R}) \longrightarrow \Omega^{0,1}(N)
\]
as the corresponding quotient Cauchy--Riemann operator. This fits into an exact triple
\[
0 \longrightarrow \Conf_{\Sigma} \xrightarrow{\;\phi\;} D_{\Sigma} \longrightarrow D^N_\Sigma \longrightarrow 0,
\]
and the canonical projection $ D_{\Sigma,\phi} = \Cone(\Conf_\Sigma\xrightarrow{\,\phi\,} D_\Sigma) \to D^N_\Sigma$ is a quasi-isomorphism.

Now consider a proper stable holomorphic map
\[
f\colon (\widetilde\Sigma,\partial\widetilde\Sigma)\longrightarrow (\Sigma,\partial \Sigma)
\]
from a smooth bordered Riemann surface $\widetilde\Sigma$, together with the pulled-back coupled bundle pair $(f^*E,f^*F,D_{\widetilde\Sigma},f^*\phi\circ df^{\log})$, whose deformation complex is $D_{u\circ f} = D_{\widetilde\Sigma, f^*\phi\circ df^{\log}}$. For compatibility with the logarithmic structure, we require that the interior marked points of $\widetilde\Sigma$ are exactly the inverse images of the interior marked points of $\Sigma$ (in particular, no components are contracted to a marked point of $\Sigma$). Setting
\[
D_f := \Cone\bigl(\Conf_{\widetilde\Sigma} \xrightarrow{\;df^{\log}\;} f^*\Conf_\Sigma\bigr),
\]
we obtain an exact triple of deformation complexes
\begin{equation*}\label{Psif}\tag{$\mathfrak T_{f}$}
0 \longrightarrow D_{f} \longrightarrow D_{u\circ f} \longrightarrow D^{f^*N}_{\widetilde\Sigma} \longrightarrow 0,
\end{equation*}
where $D^{f^*N}_{\widetilde\Sigma}$ is the pullback of $D^N_\Sigma$ to the bundle pair $f^*(N,N_{\mathbb R})$ over $\widetilde\Sigma$. Orienting $D_{u\circ f}$ therefore reduces to orienting $D_{f}$ and the normal operator $D^{f^*N}_{\widetilde\Sigma}$. Below, we choose an auxiliary complex structure on $N$ to orient $D^{f^*N}_{\widetilde\Sigma}$.

Given a Spin structure $\mathfrak s$ on $F$, there is an induced Spin structure on $N_{\mathbb R}$. We choose an auxiliary complex structure $\mathcal J$ on $N$ satisfying the following conditions:
\begin{itemize}
\item $\mathcal J$ commutes with $J$;
\item $\mathcal J$ preserves $N_{\mathbb R}$ and induces the given quotient orientation on $N_{\mathbb R}=F/\phi(T^{\log}\partial\Sigma)$;
\item $\det_{\mathbb C}\mathcal J=1$, so that $N=N^+\oplus N^-$ splits into $J$-invariant, $\mathcal J$-complex line bundles on which $\mathcal J$ acts as $J$ on $N^+$ and as $-J$ on $N^-$; and
\item the negative eigenbundle $N^-$ admits a global  $\mathcal J$-complex trivialisation over $\Sigma$ whose composition with the $\mathcal J$-complex isomorphism $N_{\mathbb R}\xrightarrow{\;\cong\;} N^-|_{\partial\Sigma}$ (given by projection along $N^+|_{\partial\Sigma}$) induces a trivialisation of $N_{\mathbb R}$ compatible with the Spin structure $\mathfrak s$. So, there exists a frame $(\phi(\partial_{\vartheta}), e, \mathcal J e)$ compatible with $\mathfrak s$, and an extension of the projection of $e$ to $N^-\rvert_{\partial\Sigma}$ to a nonvanishing section of  $N^-$.\end{itemize}

Such an auxiliary complex structure can be constructed as follows. Choose a metric on $N_{\mathbb R}$, define $\mathcal J|_{N_{\mathbb R}}$ as the positive quarter-turn in this metric, and extend $\mathcal J$ to $N|_{\partial\Sigma} = N_{\mathbb R}\oplus JN_{\mathbb R}$ to commute with $J$. Next, choose a unit section $e$ of $N_{\mathbb R}$ such that $(\phi(\partial_\vartheta),e,\mathcal Je)$ is a frame of $F$ compatible with the Spin structure $\mathfrak s$. Then the projection  $\frac 12(e+J\mathcal Je)$ of $e$ to $N^-$ is a nonvanishing section of $N^-|_{\partial\Sigma}$. Extend this section to a nowhere-vanishing section of $N$ over all of $\Sigma$ and define $N^-$ to be its $J$-complex span. Finally, define $N^+$ as the orthogonal complement of $N^-$ with respect to a Hermitian metric on $(N,J)$ extending the chosen metric on $N_{\mathbb R}$.

Because $\mathcal J$ preserves $N_{\mathbb R}$ and commutes with $J$, the operator $\frac 12(D^N_\Sigma - \mathcal J D^N_\Sigma\mathcal J)$ is a $\mathcal J$-complex-linear Cauchy--Riemann operator on $(N,N_{\mathbb R})$ with the same principal symbol as $D^N_\Sigma$.  The determinant line of such a $\mathcal J$-complex-linear operator carries a canonical $\mathcal J$-complex orientation. Transporting this orientation back along the linear homotopy to $D^N_\Sigma$ defines an orientation $\mathsf o^{\mathcal J}(\det D^N_\Sigma)$, and similarly $\mathsf o^{\mathcal J}(\det D^{f^*N}_{\widetilde\Sigma})$ on the pullback.

\begin{lemma}\label{Jcomparison1}
Suppose that $\Sigma$ is connected of genus $g$ with $n$ interior marked points. Under the quasi-isomorphism $D_{\Sigma,\phi}\xrightarrow{\;\simeq\;} D^N_\Sigma$, the $\mathcal J$-complex orientation $\mathsf o^{\mathcal J}$ is related to the holomorphic orientation $\mathsf o_{\mathfrak s}$ (Definition~\ref{def:holomorphic-orientation}) and anti-holomorphic orientation $\mathsf o_{\mathfrak s}^{\mathrm{ah}}$ (Definition~\ref{def:antiholomorphic-orientation}) induced by the Spin structure $\mathfrak s$ by
\[
\mathsf o^{\mathcal J}
= (-1)^{1-g+b_{\mathbb R}(\Sigma)}\mathsf o_{\mathfrak s}
= (-1)^{b_{\mathbb R}(\Sigma)+n}\mathsf o_{\mathfrak s}^{\mathrm{ah}}.
\]
\end{lemma}

\begin{proof}
After stretching the neck, the comparison splits into a factor $(-1)^{1-g}$ from reversing the complex structure on $N^-$ over the closed normalisation, and a factor $(-1)^{b_{\mathbb R}(\Sigma)}$ from the node factors; the anti-holomorphic comparison then follows from~\eqref{eq:antiholomorphic-sign-index}.

By construction, the Spin structure $\mathfrak s$ on $F$ is represented by a trivialisation of $E|_{\partial \Sigma}$ preserving the splitting $E = \phi(T^{\log}\Sigma)\oplus N^-\oplus N^+$. Deform $D^N_\Sigma$ to a $\mathcal J$-complex-linear normal operator, and stretch the neck along each component of $\partial \Sigma$ as in Definition~\ref{def:holomorphic-orientation}(ii), keeping the splitting and trivialisation in standard form on the collars. Let $\Sigma'$ be the limiting nodal surface, which has an elliptic boundary node $q$ for each of the $b_{\mathbb R}(\Sigma)$ boundary components of $\Sigma$ and closed normalisation $\Sigma_c$. Since $\phi$ is injective at each $q$, every node $q$ is stable in the sense of Remark~\ref{rem:Q-geometry}. The elliptic boundary node exact triple~\eqref{eq:elliptic-orientation-triple} is therefore quasi-isomorphic to the $\mathcal J$-complex-linear exact triple
\[
0 \longrightarrow D^N_{\Sigma'} \longrightarrow D^N_{\Sigma_c} \longrightarrow \bigoplus_q H^0(Q_{q,\phi}) \longrightarrow 0,
\]
where the map to $H^0(Q_{q,\phi})$ restricts a section of $N$ to $q$ and projects to $H^0(Q_{q,\phi}) \cong JN_{\mathbb R,q}$ along $N_{\mathbb R,q}$.

On the closed surface $\Sigma_c$, we can deform $D^N_{\Sigma_c}$ through $\mathcal J$-complex-linear Fredholm operators to an operator preserving the splitting $N = N^-\oplus N^+$. On $N^+$, $\mathcal J=J$, so the $\mathcal J$-complex and $J$-complex orientations agree. On $N^-$, $\mathcal J=-J$; since $N^-$ is globally trivial, its complex Riemann--Roch index on $\Sigma_c$ is $1-g$. Reversing the complex structure on $N^-$ therefore changes the orientation of $\det D^N_{\Sigma_c}$ by $(-1)^{1-g}$:
\[
\mathsf o^{\mathcal J}(\det D^N_{\Sigma_c}) = (-1)^{1-g}\mathsf o(\det D^N_{\Sigma_c}),
\]
where $\mathsf o(\det D^N_{\Sigma_c})$ is the standard $J$-complex orientation.

At each of the $b_{\mathbb R}(\Sigma)$ elliptic nodes $q$, the outward normal is $\nu_q = -j\partial_\vartheta$, so $\phi(\nu_q) = -J\phi(\partial_\vartheta)$. By Definition~\ref{def:Q orientation} and Remark~\ref{rem:Q-geometry},  $H^0(Q_{q,\phi}) $ is given the orientation of the quotient $JF_q/\phi(\nu_q)$, which is opposite to the $\mathcal J$-complex orientation, defined by $JN_{\mathbb R}=J(F_q/\phi(\partial_\vartheta))$.  Across all $b_{\mathbb R}(\Sigma)$ nodes, this contributes a sign of $(-1)^{b_{\mathbb R}(\Sigma)}$. Via the additivity isomorphism of \eqref{eq:elliptic-orientation-triple}, we obtain
\[
\mathsf o^{\mathcal J}(\det D_{\Sigma'}^N) = (-1)^{1-g+b_{\mathbb R}(\Sigma)}\mathsf o_{\mathfrak s}(\det D_{\Sigma',\phi}).
\]
Transporting this relation back along the neck-stretching homotopy proves the first equality. For a smooth connected surface of genus $g$ with $n$ interior marked points, \eqref{eq:antiholomorphic-sign-index} gives the anti-holomorphic correction sign $\epsilon_{\mathrm{ah}}(\Sigma) = (-1)^{1-g-n}$, from which the second equality follows.
\end{proof}

To orient $D_{u\circ f}$ via~\eqref{Psif}, we also need to orient $D_f$. Assume that $f$ is an unramified cover over a collar neighbourhood  of $\partial \Sigma$, so that $df^{\log}\colon T^{\log}\widetilde\Sigma \to f^*T^{\log}\Sigma$ is an isomorphism near $\partial\widetilde \Sigma$. Consequently, there is a quasi-isomorphic subcomplex of $D_f$ in which all vector fields are required to vanish on a fixed neighbourhood of $\partial\widetilde \Sigma$.  Because the boundary condition is now zero, this is a complex-linear Fredholm complex, so it has a canonical complex orientation $\mathsf o(\det D_f)$. We can similarly excise regions in which $df$ is an isomorphism to obtain a quasi-isomorphic Fredholm complex supported in a neighbourhood of $df=0$. For instance, when $f$ is simply ramified over interior points $p_k\in \Sigma^\circ\setminus\{\text{marked points}\}$, $D_f$ has cohomology concentrated in degree $0$, canonically isomorphic to $\bigoplus_k T_{p_k}\Sigma$.

Define the orientation $\mathsf o^{\mathcal J}(\det D_{u\circ f})$ via the additivity isomorphism $\Psi_{\mathfrak T_f}$ of the exact triple~\eqref{Psif}, combining the complex orientation $\mathsf o(\det D_f)$ and the $\mathcal J$-complex orientation $\mathsf o^{\mathcal J}(\det D^{f^*N}_{\widetilde\Sigma})$:
\[
\mathsf o^{\mathcal J}(\det D_{u\circ f}) := \Psi_{\mathfrak T_f}\bigl(\mathsf o(\det D_f)\otimes\mathsf o^{\mathcal J}(\det D^{f^*N}_{\widetilde\Sigma})\bigr).
\]

\begin{lemma}\label{Jcomparison2}
Suppose that $\widetilde \Sigma$ is connected of genus $\widetilde g$ with $\widetilde n$ interior marked points, and that $f:\widetilde\Sigma\to \Sigma$ is an unramified cover near $\partial \Sigma$. Then
\[
\mathsf o^{\mathcal J}(\det D_{u\circ f})
= (-1)^{1-\widetilde g+b_{\mathbb R}(\widetilde \Sigma)}\mathsf o_{\mathfrak s}(\det D_{u\circ f})
= (-1)^{b_{\mathbb R}(\widetilde \Sigma)+\widetilde n}\mathsf o_{\mathfrak s}^{\mathrm{ah}}(\det D_{u\circ f}).
\]
\end{lemma}

\begin{proof}
Pulling back the collar trivialisation and splitting from the proof of Lemma~\ref{Jcomparison1}, we stretch the neck along each component of $\partial\widetilde\Sigma$. Throughout this homotopy, $\mathsf o^{\mathcal J}$ remains well-defined. In the limit, the nodal surface $\widetilde \Sigma'$ consists of a closed surface $\widetilde \Sigma_c$ of genus $\widetilde g$ with $b_{\mathbb R}(\widetilde\Sigma)$ elliptic boundary nodes $q$. Writing $(E', F', \phi')$ for the limiting coupled bundle pair on $\widetilde\Sigma'$, we have a commuting $3\times 3$ grid of exact triples:
\[
\begin{tikzcd}
0 \rar & D_{f'}\dar \rar & D_{f'_c} \rar\dar & \bigoplus_q \bigl(\mathbb R(\nu_q)\xrightarrow{\;\cong\;}\mathbb R(\nu_q)\bigr)[-1] \rar\dar & 0 \\
0 \rar & D_{\widetilde\Sigma',\phi'}\rar\dar & D_{\widetilde\Sigma_c,\phi'_c}\rar\dar & \bigoplus_q Q_{q,\phi'}\rar\dar & 0 \\
0 \rar & D_{\widetilde\Sigma'}^{f^*N}\rar & D_{\widetilde \Sigma_c}^{f^*N} \rar & \bigoplus_q H^0(Q_{q,\phi'}) \rar & 0
\end{tikzcd}
\]
In the top row, the quotient at each elliptic node $q$ is the acyclic complex $\mathbb R(\nu_q)\xrightarrow{\,\cong\,}\mathbb R(\nu_q)$ (oriented by coevaluation as in Appendix~\ref{app:coevaluation}), so the top row gives an orientation-preserving isomorphism $\det D_{f'}\cong \det D_{f'_c}$. The middle row is the elliptic boundary node exact triple~\eqref{eq:elliptic-orientation-triple} used to define the holomorphic orientation $\mathsf o_{\mathfrak s}$ in Definition~\ref{def:holomorphic-orientation}, and the bottom row is $\mathcal J$-complex-linear. Since all complexes in the right-hand column have even index, Lemma~\ref{lem:exact-squares} shows that commuting these factors introduces no Koszul signs.

After deforming $D^{f^*N}_{\widetilde\Sigma_c}$ through $\mathcal J$-complex-linear operators to a split operator on $f^*N = f^*N^-\oplus f^*N^+$, the complex Riemann--Roch index of the trivialised $f^*N^-$ summand on $\widetilde\Sigma_c$ is $1-\widetilde g$. Thus the $\mathcal J$-complex and $J$-complex orientations on the closed surface $\widetilde\Sigma_c$ satisfy
\[
\mathsf o^{\mathcal J}(\det D_{\widetilde\Sigma_c}^{f^*N}) = (-1)^{1-\widetilde g}\mathsf o(\det D_{\widetilde\Sigma_c}^{f^*N}).
\]
Furthermore, as in the proof of Lemma~\ref{Jcomparison1}, the $\mathcal J$-complex orientation of $\bigoplus_q H^0(Q_{q,\phi'})$ differs by $(-1)^{b_{\mathbb R}(\widetilde\Sigma)}$ from the orientation of $\bigotimes_q \det Q_{q,\phi'}$ specified in Definition~\ref{def:Q orientation}. Combining the columns via Lemma~\ref{lem:exact-squares} yields
\[
\mathsf o^{\mathcal J}(\det D_{\widetilde\Sigma',\phi'}) = (-1)^{1-\widetilde g+b_{\mathbb R}(\widetilde \Sigma)}\mathsf o_{\mathfrak s}(\det D_{\widetilde\Sigma',\phi'}).
\]
Transporting this relation back along the homotopy gives the comparison with $\mathsf o_{\mathfrak s}$, and multiplying by $\epsilon_{\mathrm{ah}}(\widetilde\Sigma) = (-1)^{1-\widetilde g-\widetilde n}$ from~\eqref{eq:antiholomorphic-sign-index} gives the comparison with $\mathsf o_{\mathfrak s}^{\mathrm{ah}}$.
\end{proof}

\subsection{Example: simply ramified covers of half-cylinders}\label{subsec:ramified-half-cylinders}

Let $u(z)=(z,1,1)$ be one of the holomorphic half-cylinders in $(X,L)=((\mathbb C^*)^3, \mathbb T^2\times\mathbb R)$ from Section~\ref{half-cylinder-section}, with domain $\mathbb D_0$. Using the $(\mathbb C^*)^3$-invariant trivialisation $(a,b,c)\mapsto ia\,z_1\partial_{z_1} + ib\,z_2\partial_{z_2} + c\,z_3\partial_{z_3}$ of $T^{\log}\bar X$ from Section~\ref{subsec:disconnected-setup}, the normal bundle pair $(N,N_{\mathbb R})$ is canonically trivialised by $(\partial_{\theta_2},\partial_t)$, and we may take the auxiliary complex structure $\mathcal J$ on $N$ to be the projection of $(a,b,c)\mapsto (ia,-c,b)$, which is compatible with the trivial Spin structure $\mathfrak s$ on $L$.

Now suppose that $f\colon\widetilde \Sigma\to \mathbb D_0$ is a connected ramified cover of $\mathbb D_0$, simply branched away from $0$ over a finite subset $S\subset \mathbb D_0^\circ\setminus \{0\}$, with arbitrary ramification profile over $0$ (so $\widetilde n = |f^{-1}(0)|$), and unramified elsewhere. Setting $W := H^0(D_u) = \mathbb R\partial_{\theta_2}\oplus\mathbb R\partial_t$, we have
\[
H^0(D_{u\circ f}) = H^0(D_f)\oplus H^0(D^{f^*N}_{\widetilde\Sigma}) = \left(\bigoplus_{p\in S} T_{p}\mathbb D_0\right)\oplus W,
\]
and
\[
H^1(D_{u\circ f}) = H^1(D^{f^*N}_{\widetilde\Sigma}) = W\otimes_{\mathbb R} H^1(\widetilde \Sigma;\mathbb R).
\]
Here, $H^1(\widetilde \Sigma;\mathbb R)$ is the cohomology of the compact surface $\widetilde \Sigma$ with marked points retained. The $\mathcal J$-complex orientation $\mathsf o^{\mathcal J}(\det D_{u\circ f})$ is induced by giving $W$ and $H^1(\widetilde\Sigma;\mathbb R)\otimes_{\mathbb R} W$ their $\mathcal J$-complex orientations and each $T_p\mathbb D_0$ its standard complex orientation. By the Riemann--Hurwitz formula for the doubled cover $2\widetilde\Sigma\to 2\mathbb D_0$, the Euler characteristic of $\widetilde\Sigma$ satisfies $\chi(\widetilde\Sigma) = -|S|$; since $\chi(\widetilde\Sigma) = 2-2\widetilde g-b_{\mathbb R}(\widetilde\Sigma)-\widetilde n\equiv b_{\mathbb R}(\widetilde\Sigma)+\widetilde n\pmod 2$, Lemma~\ref{Jcomparison2} gives
\[
\mathsf o_{\mathfrak s}^{\mathrm{ah}}(\det D_{u\circ f}) = (-1)^{|S|}\mathsf o^{\mathcal J}(\det D_{u\circ f}).
\]

We can specify a hyperbolic wall problem (Definition~\ref{def:wall-problem}) on $\widetilde\Sigma$ by choosing two distinct points $p_1, p_2\in \partial \widetilde \Sigma$ in the inverse image of a point $p_0\in \partial\mathbb D_0$, with $I$ the identity on $E_{p_1}=E_{p_2}$. Let $D_{\mathrm{wall}}$ be the corresponding wall complex~\eqref{eq:D-wall-def}. Varying $p_0\in \partial \mathbb D_0$ moves $p_1$ and $p_2$ in tandem, so
\[
H^0(D_{\mathrm{wall}}) = H^0(D_{u\circ f})\oplus T_{p_0}\partial\mathbb D_0,
\]
while the matching constraint at $(p_1,p_2)$ in the normal directions gives an exact sequence
\begin{equation}\label{extraW}
0\longrightarrow W \longrightarrow H^1(D_{\mathrm{wall}}) \longrightarrow H^1(D_{u\circ f}) \longrightarrow0.
\end{equation}
Identifying both $T_{p_1}\partial \widetilde \Sigma$ and $T_{p_2}\partial \widetilde \Sigma$ with $T_{p_0}\partial\mathbb D_0$ via $df$, the incidence map $K$ of~\eqref{eq:K-hyp-def} has kernel $T_{p_0}\partial\mathbb D_0$ (embedded diagonally) and cokernel $W$. Neither this identification nor the relation between the orientations of $\det D_{\mathrm{wall}}$ and $\det D_{u\circ f}$ induced by the incidence triple~\eqref{eq:wall-triangle} depends on whether $p_1$ and $p_2$ lie on the same boundary component of $\widetilde\Sigma$ or on distinct components.

Define the $\mathcal J$-complex positive chamber orientation $\mathsf o^{+,\mathcal J}_\partial$ and smoothing boundary orientation $\mathsf o^{\mathrm{sm},\mathcal J}_\partial$ of $\det D_{\mathrm{wall}}$ analogously to the holomorphic orientations (Definition~\ref{def:boundary-orientations}), by
\[
(-1)^{\ind D_{\mathrm{wall}}}
\Psi_{\mathfrak T_{\mathrm{inc}}}
\bigl(\mathsf o^{+,\mathcal J}_\partial\otimes\mathsf o_{\mathrm{inc}}\bigr)
=\mathsf o^{\mathcal J}_{\mathrm{norm}},
\qquad
\mathsf o^{\mathrm{sm},\mathcal J}_\partial
=\Psi_{\mathfrak T_{\mathrm{sm}}}
\bigl(\mathsf o^{\mathcal J}_{\mathrm{node}}\otimes r_{\mathrm{out}}^{[-1]}\bigr),
\]
where $\mathsf o^{\mathcal J}_{\rm node}$ is defined on the nodal curve by transporting the $\mathcal J$-complex orientation from covers obtained by smoothing this boundary node.
Since smoothing the hyperbolic node changes $b_{\mathbb R}(\widetilde\Sigma)$ by $\pm 1$ while preserving $\widetilde n$, Lemma~\ref{Jcomparison2} implies that these orientations are related to their anti-holomorphic counterparts by
\[
\mathsf o^{+,\mathcal J}_\partial
=(-1)^{b_{\mathbb R}(\widetilde \Sigma)+\widetilde n}\mathsf o^{+,\mathrm{ah}}_\partial,
\qquad
\mathsf o^{\mathrm{sm},\mathcal J}_\partial
=(-1)^{b_{\mathbb R}(\widetilde \Sigma)+1+\widetilde n}\mathsf o^{\mathrm{sm},\mathrm{ah}}_\partial,
\]
whereas the comparison between $\mathsf o^{\mathrm{sm},\mathcal J}_\partial$ and the holomorphic smoothing boundary orientation $\mathsf o^{\mathrm{sm}}_\partial$ depends on whether the hyperbolic wall problem is connected or disconnected.

\begin{lemma}\label{lem:J-positive-chamber}
The $\mathcal J$-complex positive chamber orientation $\mathsf o^{+,\mathcal J}_\partial$ of $\det D_{\mathrm{wall}}$ is the orientation determined by giving $T_{p_0}\partial \mathbb D_0$ the boundary orientation, the factors of $W$ the $\mathcal J$-complex orientation, and each $T_p\mathbb D_0$ ($p\in S$) the complex orientation.
\end{lemma}

\begin{proof}
 By choosing a $\mathcal J$-complex splitting of (\ref{extraW}), we write $H^1(D_{\rm wall})=H^1(D_{u\circ f})\oplus W$.

The incidence exact triple~\eqref{eq:wall-triangle} is quasi-isomorphic to the following exact triple:
\[
\begin{tikzcd}
0\rar\dar
  & T_{p_1}\partial \widetilde \Sigma\oplus T_{p_2}\partial \widetilde \Sigma \rar{=} \dar{\bigl(0,\frac 12(df_{p_1}+df_{p_2}),\phi(df_{p_1}-df_{p_2})\bigr)}
  & T_{p_1}\partial \widetilde \Sigma\oplus T_{p_2}\partial \widetilde \Sigma\dar{K} \\
H^0(D_{u\circ f})\oplus T_{p_0}\partial \mathbb D_0 \rar \dar{0}
  & H^0(D_{u\circ f}) \oplus T_{p_0}\partial \mathbb D_0 \oplus F_{p_0} \rar \dar{(h,t,b)\mapsto (0,[b])}
  & F_{p_0} \dar \\
H^1(D_{u\circ f})\oplus W \rar{=}
  & H^1(D_{u\circ f})\oplus F_{p_0}/\mathbb R\phi(\partial_{\vartheta}) \rar
  & 0
\end{tikzcd}
\]
Here the middle column has cohomology $H^*(D_{u\circ f})$, with its acyclic summand split by the basis $\bigl((\partial_\vartheta,\partial_\vartheta),\, \frac 12(\partial_\vartheta,-\partial_\vartheta)\bigr)$ of $T_{p_1}\partial\widetilde\Sigma\oplus T_{p_2}\partial\widetilde\Sigma$, which maps to $(\partial_\vartheta, 0)$ in $T_{p_0}\partial\mathbb D_0\oplus F_{p_0}$ and $(0, \phi(\partial_\vartheta))$ in $T_{p_0}\partial\mathbb D_0\oplus F_{p_0}$. Since $\bigl((\partial_\vartheta,\partial_\vartheta),\, (\partial_\vartheta,-\partial_\vartheta)\bigr)$ is a negatively oriented basis of $T_{p_1}\partial \widetilde \Sigma\oplus T_{p_2}\partial \widetilde \Sigma$, the Euler isomorphism applied to the middle column reverses orientations when we use $\mathsf o^{\mathcal J}_{\mathrm{norm}}$ on $\det D_{u\circ f}$, the boundary orientation on $T_{p_0}\partial \mathbb D_0$, the $\mathcal J$-complex orientation on $F_{p_0}/\mathbb R\phi(\partial_{\vartheta})=W$, and the incidence orientation (Definition~\ref{def:incidence-orientation}) on the right-hand column.  Since $\ind D_{\mathrm{wall}}$ is odd, the defining relation $(-1)^{\ind D_{\mathrm{wall}}}\Psi_{\mathfrak T_{\mathrm{inc}}}\bigl(\mathsf o^{+,\mathcal J}_\partial\otimes\mathsf o_{\mathrm{inc}}\bigr)  = \mathsf o^{\mathcal J}_{\mathrm{norm}}$ implies that $\mathsf o^{+,\mathcal J}_\partial$ is the orientation of the left-hand column obtained from the boundary orientation of $T_{p_0}\partial \mathbb D_0$, the $\mathcal J$-complex orientation of $W$, and the $\mathsf o^{\mathcal J}_{\mathrm{norm}}$ orientation on $H^*(D_{u\circ f})$.
\end{proof}

\begin{lemma}\label{lem:J-smoothing}
The $\mathcal J$-complex smoothing boundary orientation $\mathsf o^{\mathrm{sm},\mathcal J}_\partial$ of $\det D_{\mathrm{wall}}$ agrees with the $\mathcal J$-complex positive chamber orientation $\mathsf o^{+,\mathcal J}_\partial$.
\end{lemma}

\begin{proof}
Smoothing the hyperbolic node of $\widetilde\Sigma_{\mathrm{node}}=\widetilde\Sigma/(p_1\sim p_2)$ resolves $f_{\rm node}$ into a family of smooth ramified covers $f_\epsilon\colon\widetilde\Sigma_\epsilon\to\mathbb D_0$ ($\epsilon>0$) with one additional simple interior branch point $p(\epsilon)\in\mathbb D_0^\circ$ approaching $p_0$ as $\epsilon\to 0$. In a local model, the distance of this new branch point to the boundary is proportional to  $\sqrt \epsilon$ to leading order. So, to leading order, the outward normal direction $-\epsilon\partial_{\epsilon}$ is sent to $\frac {(1-\abs z)}{2}\nu$, where $\nu$ is an outward normal to $\mathbb D_0$.

On the normal bundle pair, identifying $p_1$ and $p_2$ simply adds the matching cokernel $W$ to $H^1$, so the normal operator $D_{\widetilde\Sigma_{\mathrm{node}}}^{f^*N}$ has cohomology equal to that of the wall normal part, with all factors of $W$ carrying their $\mathcal J$-complex orientations.

On the covering factor, the new branch point contributes $T_{p(\epsilon)}\mathbb D_0$ to $H^0(D_{f_\epsilon})$, with its standard complex orientation, given by $(\nu,\partial_{\vartheta})$ as $\epsilon\to 0$. Contracting $\mathsf o^{\mathcal J}_{\mathrm{node}}$ with $r_{\mathrm{out}}^{[-1]}$ via $\Psi_{\mathfrak T_{\mathrm{sm}}}\bigl(\mathsf o^{\mathcal J}_{\mathrm{node}}\otimes r_{\mathrm{out}}^{[-1]}\bigr)$ removes the outward normal $\nu$ from the first slot and leaves the positive boundary orientation $\partial_\vartheta$ on $T_{p_0}\partial\mathbb D_0$, matching $\mathsf o^{+,\mathcal J}_\partial$.
\end{proof}

\begin{prop}[Hyperbolic comparisons via ramified covers]\label{prop:ramified-hyperbolic}
For the hyperbolic wall problems above, $\mathsf o^{\mathrm{sm},\mathcal J}_\partial = \mathsf o^{+,\mathcal J}_\partial$, whether $p_1$ and $p_2$ lie on the same boundary component of $\widetilde\Sigma$ or on distinct components. Consequently, $\mathsf o^{\mathrm{sm},\mathrm{ah}}_\partial=\mathsf o^{-,\mathrm{ah}}_\partial$ in both cases, whereas $\mathsf o^{\mathrm{sm}}_\partial=\mathsf o^-_\partial$ in the connected case and $\mathsf o^{\mathrm{sm}}_\partial=\mathsf o^+_\partial$ in the disconnected case. By Proposition~\ref{prop:wall-model-reduction}, the same comparisons hold at all hyperbolic wall problems.
\end{prop}

\begin{proof}
The first statement combines Lemmas~\ref{lem:J-positive-chamber} and~\ref{lem:J-smoothing}. Applying the sign relations from Lemma~\ref{Jcomparison2}:
\begin{itemize}
\item In the anti-holomorphic convention,
\[
\mathsf o^{\mathrm{sm},\mathrm{ah}}_\partial
= (-1)^{b_{\mathbb R}(\widetilde\Sigma)+1+\widetilde n}\mathsf o^{\mathrm{sm},\mathcal J}_\partial
= (-1)^{b_{\mathbb R}(\widetilde\Sigma)+1+\widetilde n}\mathsf o^{+,\mathcal J}_\partial
= -\mathsf o^{+,\mathrm{ah}}_\partial
= \mathsf o^{-,\mathrm{ah}}_\partial
\]
in both the connected and disconnected cases.
\item In the holomorphic convention, since smoothing changes $1-\widetilde g+b_{\mathbb R}$ by $+1$ in the connected case and by $-2$ in the disconnected case, we obtain $\mathsf o^{\mathrm{sm}}_\partial = -\mathsf o^+_\partial = \mathsf o^-_\partial$ in the connected case and $\mathsf o^{\mathrm{sm}}_\partial = \mathsf o^+_\partial$ in the disconnected case.
\end{itemize}
The final statement follows from Proposition~\ref{prop:wall-model-reduction}.
\end{proof}

\appendix
\numberwithin{equation}{section}
\section{Determinant-line conventions}\label{app:determinant-lines}

This appendix records the determinant-line conventions used in the orientation comparisons.

Our orientation conventions depend on isomorphisms of determinant lines induced from short exact sequences of Fredholm operators and three-term Fredholm complexes, and our proofs rely on such additivity isomorphisms being continuous under homotopies.   A determinant-functor formalism was developed by Knudsen, Mumford and Deligne \cite{KnudsenDetFunctors,DeligneDetCohomology,KnudsenMumfordDetDiv} using abstract Picard categories.  Zinger constructed a concrete realisation~\cite{ZingerDetLine}\footnote{An updated version of Zinger's paper is available on  \href{https://www.math.stonybrook.edu/~azinger/research/DetLineBDupdated.pdf}{his website}.} of this abstract approach for Fredholm operators. Below, we extend Zinger's concrete conventions to bounded Fredholm complexes.  Our conventions were formulated in consultation with ChatGPT given access to the above papers.

\subsection{Graded determinant lines and duals}\label{app:graded-lines}

We use the convention that for a finite-dimensional real vector space $V$,  $\det V$ is the $1$-dimensional vector space $\Lambda^{\dim V} V$, with index
\[\ind \det V:=\dim V=:\ind V.\]   In the special case that $V=\{0\}$, $\det \{0\}=\mathbb R$. We consider the tensor product of such graded lines to be a graded commutative operation, with $\det \{0\}$ acting as a unit. So
\[\det V\otimes \det W\cong (-1)^{\ind V\ind W}\det W\otimes \det V\ .\] More precisely, there is a canonical isomorphism
\[\operatorname{Wedge}:\det V\otimes \det W\longrightarrow \det (V\oplus W)\]
sending $\alpha\otimes \beta$ to $(\iota_V)_*\alpha\wedge(\iota_W)_*\beta$, where $\iota_V:V\longrightarrow V\oplus W$ is the inclusion. The Koszul  sign $(-1)^{\ind V\ind W}$ appears when we reverse tensor factors and then wedge the result.
\[\begin{tikzcd} \det V\otimes \det W \rar{\operatorname{Wedge}}\dar{\alpha\otimes \beta\mapsto \beta\otimes \alpha} & \det (V\oplus W)\dar{(-1)^{\dim V\dim W}}
\\ \det W\otimes \det V\rar{\operatorname{Wedge}} & \det(W\oplus V)  \end{tikzcd}\]

 The right dual $\det V^{[1]}$ of $\det V$ means $\det (V^*)$ with index
  \[\ind \det V^{[1]}:=-\dim V\] together with the isomorphism
\[\operatorname{eval}:\det V\otimes \det V^{[1]}\cong \det \{0\}\] sending $(e_1\wedge\dotsb \wedge e_n)\otimes (e_n^*\wedge\dotsb \wedge e_1^*)$ to  $1\in \det\{0\}$.

Use \[\operatorname{coev}_V: \det \{0\}\longrightarrow \det V\otimes \det V^{[1]}\] for the inverse map. If $\alpha=e_1\wedge\dotsb \wedge e_n\in \det V$, use the notation $\alpha^{[1]}:=e_n^*\wedge \dotsb \wedge e_1^*\in \det V^{[1]}$. So, $\operatorname{eval}\alpha\otimes \alpha^{[1]}=1$.

An orientation of $\det V$ induces a dual orientation on $\det V^{[1]}$ so that $\operatorname{eval}$ is positive. In particular, if $\alpha\in \det V$ is positive, then $\alpha^{[1]}\in \det V^{[1]}$ is positive. So, if $V$ has a complex structure $J$, the complex orientation of $\det V^{[1]}$ is determined by
\[(e_1\wedge J e_1\wedge \dotsb )^{[1]}=(Je_1)^{[1]}\wedge e_1^{[1]}\wedge\dotsb \]

The left dual $\det V^{[-1]}$ is a different object from the right dual, and given by $\det V^*$ with index $-\dim V$ and the pairing
\[\operatorname{eval}:\det V^{[-1]}\otimes \det V \cong \det \{0\}\] sending $(e_1^*\wedge\dotsb \wedge e_n^*)\otimes (e_n\wedge\dotsb \wedge e_1)$ to  $1\in \det\{0\}$.
Similarly $\det V^{[k]}$ is a separate object for each $k$, with index $(-1)^k \dim V$ defined inductively so that $\det V^{[k+1]}$ and
$\det V^{[k-1]}$ are respectively the right and left dual of  $\det V^{[k]}$. This formalism was considered by Knudsen in \cite[Definition 2.20]{KnudsenDetFunctors}.  There are two canonical isomorphisms
\[\begin{tikzcd}\det V^{[2]}\rar[shift right]\rar[shift left]&  \det V\end{tikzcd}  \]
differing by  $(-1)^{\dim V}$: one induced from the identity map from $V^{**}\cong V$, and the other
sending $e_1\wedge\dotsb \wedge  e_n$ to $(-e_1)\wedge\dotsb \wedge (-e_n)$, compatible with $\det V^{[2]}$ being the right dual of $\det V^{[1]}$.   It is important to distinguish the right dual $\det V^{[1]}$ from the left dual $\det V^{[-1]}$ so that we can safely cancel $\det V\otimes \det V^{[1]}$ without introducing sign errors.  Accordingly, we consider $\det V^{[2]}$ and $\det V$ as separate objects, which are unsafe to identify unless $V$ is even dimensional.

 A short exact sequence $0\to V_1\to V_2\to V_3\to 0$ induces a canonical isomorphism $\det V_2\cong \det V_1\otimes \det V_3$, computed using any splitting $V_2\cong V_1\oplus V_3$.  We refer to this isomorphism as $\operatorname{Wedge}:\det V_1\otimes \det V_3\longrightarrow \det V_2$.

\subsection{Complexes, shifts and cones}\label{app:complex-conventions}

For $V_\bullet$  a finite-dimensional $\mathbb Z$-graded vector space,  define
\begin{equation}\label{eq:graded-complex-determinant}\det V_\bullet:=\bigotimes_k \det (V_k)^{[k]}\ ,\quad \ind V_\bullet:=\ind\det V_\bullet:=\sum_k(-1)^k\dim V_k, \end{equation} where the tensor product is ordered by increasing $k$.

It is helpful to think of $\det V_\bullet$ as the top exterior power of the vector space
\[
\widetilde V=\bigoplus_k\widetilde V_k,
\qquad
\widetilde V_k=
\begin{cases}
V_k,&k\text{ even},\\
V_k^*,&k\text{ odd},
\end{cases}
\]
(with $\det \widetilde V$ assigned the index $\sum_k(-1)^k\dim V_k$),
using the isomorphism
\[\operatorname{WEDGE}:\det V_\bullet\to \det \widetilde V, \quad \bigotimes_k\alpha_k\mapsto \bigwedge_k \alpha_k\]

Given an isomorphism $h:V_\bullet\to V'_\bullet$, there is an obvious  isomorphism $h_*: \det V_\bullet\to \det  V'_\bullet$ making the assignment $V_\bullet \mapsto \det V_\bullet$ functorial.

Define
\[
\operatorname{Wedge}:\det V_\bullet\otimes\det W_\bullet
\longrightarrow\det(V_\bullet\oplus W_\bullet)
\]
by regrouping factors of equal (cochain) degree using Koszul symmetry, then applying ordinary exterior multiplication in each degree, including ordinary exterior multiplication of covectors in odd degrees. So, the following diagram commutes.
\[\begin{tikzcd} \det V_\bullet\otimes\det W_\bullet \dar[swap]{\operatorname{WEDGE}\otimes\operatorname{WEDGE}} \rar{\operatorname{Wedge}} & \det (V_\bullet\oplus W_\bullet) \dar{\operatorname{WEDGE}}
\\  \det \widetilde V\otimes \det \widetilde W \rar{\operatorname{Wedge}} & \det \widetilde {V\oplus  W} \end{tikzcd}\]

Explicitly, writing \(a_i=\dim V_i\), \(b_i=\dim W_i\),
\begin{equation}\label{gwedge}
\begin{split}\operatorname{Wedge}
\left(
\bigotimes_i\alpha_i^{[i]}
\otimes
\bigotimes_i\beta_i^{[i]}
\right)
&=
(-1)^{\,\sum_{i>j}a_i b_j}
\bigotimes_i\left(\alpha_i^{[i]}\wedge\beta_i^{[i]}\right),
\\ & = (-1)^{\,\sum_{i>j}a_i b_j+\sum_{k\text{ odd}} a_kb_k }
\bigotimes_i\left(\alpha_i\wedge\beta_i\right)^{[i]}
\end{split}
\end{equation}
with all tensor products ordered by increasing degree.
 This defines an associative, graded commutative operation, because ordinary wedge products on $\tilde V$ are associative and graded commutative.
Note that the sign correction vanishes in the case that $a_i$ and $b_i$ are even for all nonzero $i$.
For a degreewise short exact sequence, define its graded wedge using any graded splitting; the resulting map is independent of that choice.

\begin{defn}[Safe complexes]\label{def:safe-complex}We say a cochain complex $(V_\bullet,d_\bullet)$ is safe if $\dim V_k$ is even for all $k\neq 0$, and otherwise refer to this cochain complex as unsafe. \end{defn}
We adopt this normative terminology to indicate that manipulations with determinant lines of unsafe cochain complexes often carry hidden Koszul signs.

Note the difference from conventions which identify $\det V^{[-1]}$ with $\det V^{[1]}$. Such an identification is not too dangerous for safe cochain complexes, because the two canonical isomorphisms between $\det V^{[-1]}$ and $\det V^{[1]}$ agree when $\dim V$ is even, but for unsafe cochain complexes, such an identification requires careful formulation in a Picard category~\cite{KnudsenDetFunctors}.

 With our conventions the shifted complex $(V_\bullet,d_\bullet)[n]$ has $V[n]_k:=V_{k-n}$ and $d[n]_{k}:=d_{k-n}$, so there are no signs introduced in the shifted complex.

 Given a chain map $u_\bullet:(V_\bullet,d_\bullet)\to (V'_\bullet, d'_\bullet)$, our convention for the cone $\operatorname{Cone}(u_\bullet) $ is the cochain complex $(C_\bullet, d''_\bullet)$ with
 \[C_k:=V'_k\oplus V_{k+1},\quad\text{ and } d''_k(v',v):=(d'_kv' +(-1)^{k+1}u_{k+1}v  , d_{k+1}v)\ .\] So, $\ind \operatorname{Cone}(u_\bullet)=\ind V'_{\bullet}-\ind V_{\bullet}$ and $V'_\bullet\to \operatorname{Cone}(u_\bullet)\to V_\bullet[-1]$ is a short exact sequence of cochain complexes.

 Note that the shift of a safe complex may not be safe, and similarly, the cone of a map between safe complexes need not be safe.

\subsection{Acyclic complexes and coevaluation}\label{app:coevaluation}

If the finite-dimensional cochain complex $(E_\bullet,d_\bullet)$ is acyclic, there is a natural isomorphism
\[\operatorname{coev}: \det \{0\}\longrightarrow \det E_\bullet \]
(depending on $d_\bullet$),  defined as follows:
Let $(V\to V)[k]$ indicate the complex with $V$ in degrees $k$ and $k+1$, with the identity map between them. The isomorphism
\[\operatorname{coev}_{V[k]}:\det\{0\}\longrightarrow \det V^{[k]}\otimes \det V^{[k+1]}=\det (V\to V)[k] \]
 sends $1$ to $\alpha^{[k]} \otimes \alpha^{[k+1]}$, where $\alpha^{[k]}$ is $e_1\wedge\dotsb \wedge e_n$ for $k$ even and $e_n^*\wedge \dotsb \wedge e_1^*$ for $k$ odd.
If
\[E_\bullet=\bigoplus_{k} (V_k\to V_k)[k] \]
we define
\[\operatorname{coev}_{E_\bullet}:\det\{0\}\longrightarrow \det E_\bullet \]
by
\[\operatorname{coev}_{E_\bullet}(1)=\bigotimes_{k}\left( \alpha_{k-1}^{[k]}\wedge \alpha_{k}^{[k]}\right)=\operatorname{Wedge}\left( \bigotimes_k \operatorname{coev}_{V_k[k]}(1)\right) \]
where $\alpha_{k}\in \det V_k\setminus 0$.  As $\operatorname{Wedge}$ is graded commutative, and the index of $(V_k\to V_k)[k]$ is $0$, the order does not matter in the above formula.

Similarly, we define the coevaluation for any acyclic complex $(C_\bullet, d_\bullet)$ by splitting it as above. This does not depend on the splitting, because
\[V_{2k-1}^*\oplus V_{2k-1}= \ker (d_{2k-2}^*\oplus d_{2k})\subset C_{2k-1}^*\oplus C_{2k} ,\] and  $(\alpha_{2k-2}^{[2k-1]}\wedge \alpha_{2k-1}^{[2k-1]})\otimes (\alpha_{2k-1}^{[2k]}\wedge \alpha_{2k}^{[2k]})$ does not depend on the choice of inclusion of $V_{2k}\subset C_{2k}$ or $V_{2k-2}^*\subset C_{2k-1}^*$  from the splitting. Graded commutativity of $\operatorname{Wedge}$ implies that for acyclic cochain complexes $C$ and $C'$,
\[
\operatorname{coev}_{C\oplus C'}(1)
=
\operatorname{Wedge}\bigl(
\operatorname{coev}_{C}(1)\otimes
\operatorname{coev}_{C'}(1)\bigr).
\]

\subsection{Fredholm determinants and the Euler isomorphism}\label{app:euler-isomorphism}

If  $(A_\bullet, D_\bullet)$ is a  bounded Fredholm complex, define
\begin{equation}\label{eq:fredholm-complex-determinant}\det D_\bullet :=\bigotimes_k \det H^k(A_{\bullet},D_\bullet)^{[k]}=\det (H^\bullet ((A_{\bullet},D_\bullet)),0)\ .\end{equation}
We say that $D_\bullet$ has \emph{safe} cohomology if $\dim H^k$ is even for all nonzero $k$.
In the particular case of a Fredholm map $D$, we consider the domain of $D$ to be in degree $0$, and the codomain to be in degree $1$, so
\begin{equation}\label{eq:fredholm-map-determinant}\det D:=\det \ker D\otimes \det(\coker D)^{[1]}.\end{equation}
A Fredholm map has safe cohomology if it is surjective, or has even-dimensional cokernel. If $D$ is infinite dimensional, we can obtain a safe finite-dimensional approximation of $D$ by choosing an even-dimensional subspace $V_1$ of the codomain, transverse to $D$, and restricting $D$ to $V_0:=D^{-1}(V_1)$.

Generalising the coevaluation map for acyclic complexes, for a finite-dimensional cochain complex $(V_\bullet, d_\bullet)$, there is a canonical isomorphism, called the Euler isomorphism
\begin{equation}\label{eq:euler-isomorphism}\Phi_{V_\bullet}:\det d_\bullet\longrightarrow \det V_\bullet\ .\end{equation}
The Euler isomorphism carries hidden Koszul signs, and different natural conventions are possible unless $d_\bullet$ has safe cohomology and $V_\bullet$ is safe.   We define $\Phi_{V_\bullet}$  by splitting $V_\bullet$ as
\[V_\bullet =H^\bullet\oplus \left (\bigoplus_{k} (E_k\to E_k)[k]\right)  \]
then
\begin{equation}\label{eq:euler-isomorphism-formula}\Phi_{V_\bullet} \left( \bigotimes_k h_k^{[k]}\right):= \bigotimes_{k} \left(\alpha_{k-1}^{[k]}\wedge h_{k}^{[k]} \wedge \alpha_{k}^{[k]}\right) \end{equation}
where $\alpha_k\in \det E_k\setminus 0$. To see that this does not depend on the splitting, use the canonical filtration \[ \operatorname{im}(d_{2k-1}^*\oplus d_{2k-1}) \subset \ker(d_{2k-2}^*\oplus d_{2k}) \subset V_{2k-1}^*\oplus V_{2k}. \]
Then,
\begin{itemize}\item
$ E_{2k-1}^{*}\oplus E_{2k-1} \subset V_{2k-1}^*\oplus V_{2k}$ is the image of $d_{2k-1}^*\oplus d_{2k-1}$,
\item $ (h_{2k-1}^{[2k-1]}\wedge \alpha_{2k-1}^{[2k-1]})\wedge (\alpha_{2k-1}^{[2k]}\wedge h_{2k}^{[2k]})$ in $\bigwedge \ker (d_{2k-2}^*\oplus d_{2k})$ does not depend on the choice of splitting into  $ (H^{2k-1}\oplus E_{2k-1})^*\oplus E_{2k-1}\oplus H^{2k}$,
\item and therefore $\alpha_{k-1}^{[k]}\wedge h_{k}^{[k]} \wedge \alpha_{k}^{[k]}$ does not depend on the splitting  of $V_{2m}=\ker d_{2m}\oplus E_{2m}$ or $V_{2m-1}^*=\ker d^*_{2m-2}\oplus E_{2m-2}^*$.
\end{itemize}

This Euler isomorphism  coincides with coevaluation for acyclic cochain complexes, and is compatible with $\operatorname{Wedge}$. In particular,
\[
\Phi_{V_\bullet} \left( \bigotimes_k h_k^{[k]}\right)=\operatorname{Wedge}\left(
\bigotimes_k h_k^{[k]}\otimes\operatorname{coev}_{C}(1)
\right),
\qquad
C=\bigoplus_k(E_k\to E_k)[k],
\]
because
\[
\begin{aligned}
\Phi_V\left(\bigotimes_k h_k^{[k]}\right)
&=
(-1)^{\sum_k\dim H^k\,\dim E_{k-1}}
\bigotimes_k
\left(
h_k^{[k]}\wedge\alpha_{k-1}^{[k]}
\wedge\alpha_k^{[k]}
\right)\\
&=
\operatorname{Wedge}\left(
\bigotimes_k h_k^{[k]}
\otimes\operatorname{coev}_{C}(1)
\right),
\end{aligned}
\]
where the second equality follows from the telescoping identity
\[
\sum_{i>j}\dim H^i\,
\bigl(\dim E_{j-1}+\dim E_j\bigr)
\equiv
\sum_i\dim H^i\,\dim E_{i-1}
\pmod2.
\]

 As $\operatorname{Wedge}$ is a graded commutative operation, and acyclic complexes have index $0$, it does not matter on which side we place the acyclic complex. Similarly,
 \[\operatorname{Wedge}\left(
h
\otimes\operatorname{coev}_{C}(1)
\otimes  h'
\otimes\operatorname{coev}_{C'}(1)\right)=\operatorname{Wedge}\left(
h \otimes  h'
\otimes\operatorname{coev}_{C\oplus C'}(1)\right)\] for acyclic complexes $C$ and $C'$, so   the Euler isomorphism commutes with $\operatorname{Wedge}$, and the following diagram commutes
\[\begin{tikzcd} \det V_\bullet \otimes \det V'_{\bullet} \ar{rr}{\operatorname{Wedge}} && \det (V_\bullet\oplus V'_\bullet)
\\ \det d_\bullet \otimes \det d'_\bullet \uar{\Phi_{V_\bullet}\otimes \Phi_{V'_\bullet}} \ar{rr}{\operatorname{Wedge}} && \det (d_\bullet\oplus d'_\bullet)\uar{\Phi_{V_\bullet\oplus V'_{\bullet}}}
\end{tikzcd}
\]

In our convention, given a map $f:V\to W$, the Euler isomorphism $\Phi: \det\ker f\otimes \det \coker f^{[1]}\to \det V\otimes \det W^{[1]}$ implies the same orientation relationship as writing $V=\ker f\oplus A$ and $W= \coker f\oplus f(A)$, so
\begin{equation}\label{eq:euler-splitting-sign}\det V\otimes \det W^{[1]}\cong \det\ker f\otimes \det A\otimes \det f(A)^{[1]}\otimes \det \coker f^{[1]}\ .\end{equation} This differs by a sign $(-1)^{\dim f(A)\dim \coker f}$ from the convention implied by writing $V=\ker f\oplus A$ and $W=f(A)\oplus \coker f$. If we shift this map so that $V$ is odd degree instead of degree $0$, the Euler isomorphism implies the reversed orientation relationship corresponding to $V=A\oplus \ker f$ and $W=f(A)\oplus \coker f$.

As explained in Remark~\ref{rem:smoothing-family-regularity}, families of operators $D_{\Sigma_t}\colon \Omega^{0}(E_t,F_t)\longrightarrow \Omega^{0,1}(E_t)$ over node-smoothing families do not have a fixed domain and codomain, but locally admit continuous finite-dimensional approximations $\bigl(D_{\Sigma_t}^{-1}(V_t)\to V_t\bigr)\hookrightarrow D_{\Sigma_t}$, unique up to stabilisation and homotopy. More generally, any family of Fredholm complexes $(A_{\bullet,t}, D_{\bullet,t})$ considered in this paper locally admits a continuous family of finite-dimensional cochain complexes $(V_{\bullet,t},d_{\bullet,t})$ with an injective quasi-isomorphism $(V_{\bullet,t}, d_{\bullet,t})\hookrightarrow (A_{\bullet,t}, D_{\bullet,t})$, again unique up to stabilisation and homotopy. We equip $\det D_{\bullet,t}$ with the structure of a continuous real line bundle by requiring the Euler isomorphism
\[
\Phi_{V_{\bullet,t}}\colon \det D_{\bullet,t} = \det d_{\bullet,t} \xrightarrow{\;\cong\;} \det V_{\bullet,t}
\]
to be continuous. Because $\Phi_{V_{\bullet,t}}$ commutes with $\operatorname{Wedge}$ and coevaluation on acyclic complexes, this topology is independent of the choice of finite-dimensional approximation.

\subsection{Additivity for exact triples}\label{app:additivity}

For a short exact sequence of Fredholm complexes,
\[\begin{tikzcd}0\rar & D'_\bullet \rar & D_\bullet \rar & D_\bullet''\rar& 0\end{tikzcd}\]
(also called an exact triple, $\mathfrak T$), we want a natural  isomorphism
\begin{equation}\label{eq:additivity-isomorphism}\begin{tikzcd}\det D'_\bullet \otimes \det D''_\bullet \rar{\Psi_{\mathfrak T}} &\det D_\bullet\end{tikzcd}\end{equation}
(called the additivity isomorphism) extending the case when all cohomology is supported in degree $0$. In the case of Fredholm maps, cohomology supported in degree 0 is equivalent to the transverse case that each Fredholm map is surjective. We need this isomorphism to be continuous under homotopies of exact triples (and hence compatible with the Euler isomorphisms) and natural with respect to quasi-isomorphisms of exact triples.

 Given an exact triple consisting of finite-dimensional complexes
\begin{equation*}\tag{$\mathfrak t$}\begin{tikzcd}(V'_\bullet,d'_\bullet) \rar & (V_\bullet,d_\bullet)\rar & (V''_\bullet,d''_\bullet)\end{tikzcd}\end{equation*}
we define $\Psi_{\mathfrak t}: \det d'\otimes \det d''\longrightarrow \det d$ as the isomorphism fitting into the following commutative diagram.
\begin{equation}\label{eq:finite-additivity-diagram}\begin{tikzcd} \det d_\bullet '\otimes\det d_\bullet ''\dar{\Phi_{V'}\otimes\Phi_{V''}} \rar{\Psi_{\mathfrak t}} & \det d_\bullet \dar{\Phi_{V}}
\\  \det V'_\bullet\otimes \det V''_{\bullet}\rar{\operatorname{Wedge}}& \det V_\bullet  \end{tikzcd}\end{equation}
This is easiest to interpret in the case that these are safe complexes, so the sign in $\operatorname{Wedge}$ is trivial.

For an exact triple $\mathfrak T$ of Fredholm complexes
\begin{equation*}\tag{$\mathfrak T$}\begin{tikzcd} 0 \rar & D'_\bullet \rar & D_\bullet \rar & D''_\bullet \rar & 0\end{tikzcd}\end{equation*}
we can construct a quasi-isomorphic exact triple $\mathfrak t$ of finite-dimensional complexes
\begin{equation*}\tag{$\mathfrak t$}\begin{tikzcd}0 \rar & (V'_\bullet,d'_\bullet) \rar  & (V_\bullet,d_\bullet)\rar & (V''_\bullet,d''_\bullet) \rar & 0
\end{tikzcd}\end{equation*}
and define $\Psi_{\mathfrak T}$ as $\Psi_{\mathfrak t}$ composed with the isomorphisms $\det D'_\bullet\otimes \det D''_\bullet\cong \det d'_{\bullet}\otimes \det d''_{\bullet}$ and $\det D_\bullet \cong \det d_\bullet$.

The resulting map does not depend on the chosen finite-dimensional model, because any two are related by stabilising compatibly to a larger exact triple of cochain complexes. This stabilisation  involves taking the direct sum with a triple of acyclic complexes, which introduces no further Koszul signs because $\operatorname{Wedge}$ commutes with the Euler isomorphism, and acyclic complexes have even index.

  For a continuous family of exact Fredholm triples, it is possible to locally choose a continuous injective quasi-isomorphism from an exact triple of finite-dimensional complexes, so the resulting construction is continuous.
Explicitly, for a continuous family of exact triples $\mathfrak T_t$ of two-term Fredholm operators $0\to D'_t\to D_t\to D''_t\to 0$, we locally choose continuous families of transverse target subspaces $V'_{1,t}$ for $D'_t$ and $V''_{1,t}$ for $D''_t$, choose a continuous lift $\tilde V''_{1,t}$ of $V''_{1,t}$ to the codomain of $D_t$, and set $V_{1,t}:=V'_{1,t}\oplus \tilde V''_{1,t}$. Taking the preimages $V'_{0,t}:=(D'_t)^{-1}(V'_{1,t})$, $V_{0,t}:=D_t^{-1}(V_{1,t})$, and $V''_{0,t}:=(D''_t)^{-1}(V''_{1,t})$ yields a continuous family of exact triples $\mathfrak t_t$ of finite-dimensional complexes with an injective quasi-isomorphism $\mathfrak t_t\hookrightarrow \mathfrak T_t$. More generally, when $\mathfrak T_t$ is a family of mapping cones of a coupling map $\phi_t$ between exact triples of $\Conf$ and $D$ operators (as in $\mathfrak T_{\mathrm{inc}}$, $\mathfrak T_{\mathrm{sm}}$, $\mathfrak T_{\mathrm{ell}}$, and $\mathfrak T_{\mathrm{n}}$), we first apply this construction to the $\Conf$ triple, and then choose the target subspaces for the $D$ triple to contain the image under $\phi_{1,t}$ of the chosen target subspaces for $\Conf$.

\subsection{Associativity and exact squares}\label{app:exact-squares}

When comparing orientations across simultaneous degenerations or excisions (such as in Lemma~\ref{lem:gluing-orientation} and Lemma~\ref{lem:wall-excision}), we use the compatibility of the additivity isomorphism with commuting $3\times 3$ grids of exact triples.

Consider a commutative diagram of bounded Fredholm complexes with exact rows and columns:
\begin{equation}\label{eq:exact-grid}
\begin{tikzcd}
 & 0 \dar & 0 \dar & 0 \dar & \\
0 \rar & D'_{1} \rar\dar & D_{1} \rar\dar & D''_{1} \rar\dar & 0 \\
0 \rar & D'_{2} \rar\dar & D_{2} \rar\dar & D''_{2} \rar\dar & 0 \\
0 \rar & D'_{3} \rar\dar & D_{3} \rar\dar & D''_{3} \rar\dar & 0 \\
 & 0 & 0 & 0 &
\end{tikzcd}
\end{equation}
Use $\Psi_i$ for the additivity isomorphisms from the rows, and $\Psi'$,$\Psi$,$\Psi''$ for the additivity isomorphisms from the columns. Then, there are two ways to calculate $\det D_2$ from the determinants of the corner complexes:
\[\Psi(\Psi_1 \otimes \Psi_3):(\det D'_1\otimes \det D''_1)\otimes (\det D'_3\otimes \det D''_{3})\longrightarrow \det D_2\]
\[\Psi_2(\Psi'\otimes \Psi''):(\det D'_1\otimes \det D'_3)\otimes (\det D''_1\otimes \det D''_{3})\longrightarrow \det D_2\]
The associativity and graded commutativity of $\operatorname{Wedge}$ imply that these two isomorphisms agree, once we identify $\det D''_1 \otimes \det D'_3$ with $\det D'_3\otimes \det D''_1$ with the appropriate Koszul sign.

\begin{lemma}[Additivity for exact squares]\label{lem:exact-squares}
Given an exact square and associated additivity isomorphisms as above, the following diagram commutes
\[\begin{tikzcd} \det D'_1\otimes \det D''_1\otimes \det D'_3\otimes \det D''_{3}\dar[swap]{(-1)^{\ind D''_1\ind D'_3}R } \ar{rr}{\Psi(\Psi_1 \otimes \Psi_3)}& & \det D_2
\\ \det D'_1\otimes \det D'_3\otimes \det D''_1\otimes \det D''_{3} \ar{urr}[swap]{\Psi_2(\Psi'\otimes \Psi'')}\end{tikzcd}\]
where $R$ is the unsigned isomorphism swapping the second and third tensor factors.

\end{lemma}

\begin{proof}
By choosing a compatible quasi-isomorphic $3\times 3$ grid of bounded finite-dimensional cochain complexes (for instance, by inductively selecting transverse finite-dimensional approximations and compatible splittings), we reduce to the  finite-dimensional case. Then, by continuity of additivity isomorphisms using a simultaneous linear homotopy of the differentials to $0$, or by applying the Euler isomorphism,  we reduce to the case that all differentials are $0$, and all additivity isomorphisms coincide with $\operatorname{Wedge}$ in a splitting. So, we have reduced to the case of a commutative diagram of graded vector spaces as follows
\[\begin{tikzcd}
 & 0 \dar & 0 \dar & 0 \dar & \\
0 \rar & V'_{1,\bullet} \rar\dar & V'_{1,\bullet}\oplus V''_{1,\bullet} \rar\dar & V''_{1,\bullet} \rar\dar & 0 \\
0 \rar & V'_{1,\bullet}\oplus V'_{3,\bullet}  \rar\dar & V_{2,\bullet} \rar\dar & V''_{1,\bullet}\oplus V''_{3,\bullet} \rar\dar & 0 \\
0 \rar & V'_{3,\bullet} \rar\dar & V'_{3,\bullet}\oplus V''_{3,\bullet}  \rar\dar & V''_{3,\bullet} \rar\dar & 0 \\
 & 0 & 0 & 0 &
\end{tikzcd}\]

Given splittings as above, there is an induced inclusion  of $V'_{1}$, $V''_{1}$, and $V'_{3}$ into  $V_{2}$, and a contractible choice of inclusion $V''_3\subset V_2$ splitting the middle horizontal and vertical exact sequences. Then, for
\[\alpha'_1\in \det V'_{1,\bullet},\quad \alpha''_1\in \det V''_{1,\bullet},\quad \alpha'_3\in \det V'_{3,\bullet},\quad \alpha''_3\in \det V''_{3,\bullet},  \]
\[\Psi(\Psi_1 \otimes \Psi_3)(\alpha'_1\otimes \alpha''_1\otimes \alpha'_3\otimes \alpha''_3)=\operatorname{Wedge}(\alpha'_1\otimes \alpha''_1\otimes \alpha'_3\otimes \alpha''_3)\]
and
\[\Psi_2(\Psi' \otimes \Psi'')(\alpha'_1\otimes  \alpha'_3\otimes \alpha''_1\otimes \alpha''_3)=\operatorname{Wedge}(\alpha'_1\otimes \alpha'_3\otimes \alpha''_1 \otimes \alpha''_3)\]
which differ by the Koszul sign
 \[(-1)^{\ind V''_{1,\bullet}\ind V'_{3,\bullet}  }=(-1)^{\ind D''_{1,\bullet}\,\ind D'_{3,\bullet}}.\]
\end{proof}

\clearpage
\section{Notation}\label{app:notation}

\begin{center}
\captionsetup{type=table}
\footnotesize
\renewcommand{\arraystretch}{1.12}
\begin{tabular}{@{}>{\raggedright\arraybackslash}p{0.18\linewidth}>{\raggedright\arraybackslash}p{0.62\linewidth}>{\raggedright\arraybackslash}p{0.13\linewidth}@{}}
\toprule
\textbf{Notation} & \textbf{Meaning} & \textbf{Where}\\
\midrule
\multicolumn{3}{@{}l}{\emph{Complexes}}\\
$\Conf_\Sigma$ & $\dbar$ on $(T^{\log}\Sigma,T^{\log}\partial\Sigma)$; kernel: infinitesimal automorphisms, cokernel: variations of complex structure and gluing data & \S\ref{sec:conf}\\
$D_\Sigma$ & real Cauchy--Riemann operator on the bundle pair $(E,F)$, with the domain fixed & \S\ref{subsec:intro-linear-setup}\\
$D_{\Sigma,\phi}$ & $\Cone(\Conf_\Sigma\xrightarrow{\phi}D_\Sigma)$ in degrees $-1,0,1$: the full deformation complex of a coupled bundle pair & Def.~\ref{def:coupled-bundle-pair}\\
$D_{\mathrm{norm}}$ & $D_{\Sigma,\phi}$ for the normalised surface of a wall problem & Def.~\ref{def:wall-problem}\\
$K$, $\Cone(K)$ & incidence map; when $K$ is injective, $\coker K$ is the normal line to the wall & \eqref{eq:K-hyp-def}, \eqref{eq:K-ell-def}\\
$D_{\mathrm{wall}}\simeq D'_{\mathrm{wall}}$ & deformations that preserve the crossing or incidence & \eqref{eq:D-wall-def}, \S\ref{subsec:nodal-smoothing}\\
$D_{\mathrm{node}}$ & deformation complex of the nodal curve $\Sigma_{\mathrm{node}}$, including the smoothing direction & \S\ref{subsec:nodal-smoothing}\\
$\mathbb R_{\mathrm{res}}$, $r_{\mathrm{out}}$ & residue line in degree $-1$; $r_{\mathrm{out}}=-1$ gives the outward smoothing direction & \S\ref{subsec:nodal-smoothing}, Rem.~\ref{rem:residue-sign}\\
$Q_{q,\phi}$ & $\bigl(\mathbb R(\nu_q)\xrightarrow{\phi}JF_q\bigr)$, the factor at an elliptic boundary node; index $2$ & \S\ref{subsec:elliptic-orientation-triple}, Def.~\ref{def:Q orientation}\\
$C_{\mathrm n,\phi}$ & $\bigl(T^{\log}_{\mathrm n}\Sigma\xrightarrow{\phi}E_{\mathrm n}\bigr)$, the factor at an interior node; index $4$ & Lemma~\ref{lem:gluing-orientation}\\
$D_f$, $D^N_\Sigma$ & for a cover $f$ of an immersed curve: $D_f=\Cone(\Conf_{\widetilde\Sigma}\to f^*\Conf_\Sigma)$, and the normal operator $D^N_\Sigma$ & \S\ref{sec:ramified-covers}\\
\midrule
\multicolumn{3}{@{}l}{\emph{Short exact sequences (exact triples)}}\\
$\mathfrak T_{\mathrm{inc}}$ & $D_{\mathrm{wall}}\to D_{\mathrm{norm}}\to\Cone(K)$: incidence triple, used for the chamber orientations & \S\ref{subsec:incidence-triple}\\
$\mathfrak T_{\mathrm{sm}}$ & $D_{\mathrm{node}}\to D'_{\mathrm{wall}}\to\mathbb R_{\mathrm{res}}$: smoothing triple, used for the smoothing boundary orientation & \S\ref{subsec:nodal-smoothing}\\
$\mathfrak T_{\mathrm{ell}}$ & $D_{\Sigma,\phi}\to D_{\Sigma_c,\phi_c}\to\bigoplus_q Q_{q,\phi}$: elliptic boundary node triple, used to define the holomorphic orientation & \S\ref{subsec:elliptic-orientation-triple}\\
$\mathfrak T_{\mathrm n}$ & $D_{\Sigma,\phi}\to D_{\Sigma',\phi'}\to C_{\mathrm n,\phi}$: normalising an interior node & Lemma~\ref{lem:gluing-orientation}\\
\midrule
\multicolumn{3}{@{}l}{\emph{Orientations}}\\
$\mathfrak o$ & orientation of the bundle pair: a class of trivialisations of $F$, e.g.\ from a Spin structure & Def.~\ref{def:bundle-pair-orientation}\\
$\mathsf o_{\mathfrak o}$ & holomorphic orientation of $\det D_{\Sigma,\phi}$; written $\mathsf o_{\mathfrak s}$ in \S\ref{sec:ramified-covers}, where $\mathfrak o$ comes from a Spin structure $\mathfrak s$ & Def.~\ref{def:holomorphic-orientation}\\
$\mathsf o^{\mathrm{ah}}_{\mathfrak o}$ & anti-holomorphic orientation, $\mathsf o^{\mathrm{ah}}_{\mathfrak o}=\epsilon_{\mathrm{ah}}(\Sigma)\,\mathsf o_{\mathfrak o}$ & Def.~\ref{def:antiholomorphic-orientation}, \eqref{eq:antiholomorphic-sign-index}\\
$\mathsf o_{\mathrm{norm}}$, $\mathsf o_{\mathrm{node}}$ & orientations of $D_{\mathrm{norm}}$ and $D_{\mathrm{node}}$ from compatible orientation data & Defs.~\ref{def:compatible-orientation-data}, \ref{def:norm-node-orientations}\\
$\mathsf o_{\mathrm{inc}}$ & orientation of $\det\Cone(K)$, corresponding to the outward normal $-[n_+]$ of the positive chamber & Def.~\ref{def:incidence-orientation}\\
$\mathsf o^\pm_\partial$ & chamber orientations of $\det D_{\mathrm{wall}}$: $\Psi_{\mathfrak T_{\mathrm{inc}}}(\mathsf o^-_\partial\otimes\mathsf o_{\mathrm{inc}})=\mathsf o_{\mathrm{norm}}$ and $\mathsf o^+_\partial=-\mathsf o^-_\partial$ & Def.~\ref{def:boundary-orientations}(i)\\
$\mathsf o^{\mathrm{sm}}_\partial$ & smoothing boundary orientation: $\mathsf o^{\mathrm{sm}}_\partial=\Psi_{\mathfrak T_{\mathrm{sm}}}(\mathsf o_{\mathrm{node}}\otimes r_{\mathrm{out}}^{[-1]})$ & Def.~\ref{def:boundary-orientations}(ii)\\
$\mathsf o^{\pm,\mathrm{ah}}_\partial$, $\mathsf o^{\mathrm{sm},\mathrm{ah}}_\partial$ & the same, using $\mathsf o^{\mathrm{ah}}_{\mathrm{norm}}$ and $\mathsf o^{\mathrm{ah}}_{\mathrm{node}}$ & Def.~\ref{def:boundary-orientations}\\
$\mathsf o^{\mathcal J}$, $\mathsf o^{\pm,\mathcal J}_\partial$, $\mathsf o^{\mathrm{sm},\mathcal J}_\partial$ & auxiliary $\mathcal J$-complex orientations for covers of an immersed curve & \S\ref{sec:ramified-covers}\\
\bottomrule
\end{tabular}
\caption{Notation for the deformation complexes, exact triples and orientations used in Sections~\ref{sec:linear-setting}--\ref{sec:ramified-covers}.}
\label{tab:notation}
\end{center}

\bibliographystyle{plain}
\bibliography{skein-orientations}
 \end{document}